\documentclass{amsart}

\usepackage{etex,amsfonts, amsmath, amsthm, amssymb, stmaryrd, epsfig, graphics,
  psfrag, latexsym, xcolor, mathtools, mathrsfs,enumitem, subfig,
  longtable, booktabs, yfonts, centernot} 
\usepackage[all,cmtip]{xy}
\usepackage[percent]{overpic}
\usepackage{tikz} \usetikzlibrary{matrix,arrows,calc}
\definecolor{darkblue}{rgb}{0,0,0.4} 
\usepackage[colorlinks=true, citecolor=darkblue, filecolor=darkblue, linkcolor=darkblue,urlcolor=darkblue]{hyperref} 
\usepackage[all]{hypcap}
\usepackage{xr-hyper}

\graphicspath{{draws/}{}}

\numberwithin{equation}{section}

\makeatletter
\newtheorem*{rep@theorem}{\rep@title}
\newcommand{\newreptheorem}[2]{%
\newenvironment{rep#1}[1]{%
 \def\rep@title{#2 \ref{##1}}%
 \begin{rep@theorem}}%
 {\end{rep@theorem}}}
\makeatother

\newtheorem{thm}{Theorem}
\newtheorem{theorem}[thm]{Theorem}
\newreptheorem{theorem}{Theorem}

\newtheorem{lemma}[equation]{Lemma}

\newtheorem{corollary}[equation]{Corollary}               
\newreptheorem{corollary}{Corollary}

\newtheorem{proposition}[equation]{Proposition}
\theoremstyle{definition}

\newtheorem{definition}[equation]{Definition}

\newtheorem{construction}[equation]{Construction}
\newtheorem{notation}[equation]{Notation}
\theoremstyle{remark}     

\newtheorem{remark}[equation]{Remark}

\newtheorem{example}[equation]{Example}
\newtheorem{observation}[equation]{Observation}
\newtheorem{convention}[equation]{Convention}

\numberwithin{figure}{section}
\numberwithin{table}{section}

\newcommand{\Appendix}[1]{Appendix~\ref{app:#1}}
\newcommand{\Section}[1]{Section~\ref{sec:#1}}

\newcommand{\Lemma}[1]{Lemma~\ref{lem:#1}}
\newcommand{\Theorem}[1]{Theorem~\ref{thm:#1}}
\newcommand{\Definition}[1]{Definition~\ref{def:#1}}
\newcommand{\Remark}[1]{Remark~\ref{rem:#1}}
\newcommand{\Figure}[1]{Figure~\ref{fig:#1}}

\newcommand{\Corollary}[1]{Corollary~\ref{cor:#1}}
\newcommand{\Proposition}[1]{Proposition~\ref{prop:#1}}

\newcommand{\Example}[1]{Example~\ref{exam:#1}}
\newcommand{\Table}[1]{Table~\ref{table:#1}}

\newcommand{\Construction}[1]{Construction~\ref{construction:#1}}
\newcommand{\Observation}[1]{Observation~\ref{obs:#1}}
\newcommand{\Convention}[1]{Convention~\ref{conv:#1}}

\newcommand{\Equation}[2][{}]{Equation#1~(\ref{eq:#2})}
\newcommand{\Formula}[2][{}]{Formula#1~(\ref{eq:#2})}

\newcommand{\R}{\mathbb{R}}
\newcommand{\Z}{\mathbb{Z}}
\newcommand{\F}{\mathbb{F}}

\newcommand{\N}{\mathbb{N}}

\newcommand{\QQ}{\mathbb{Q}}

\newcommand{\mc}{\mathcal}
\newcommand{\mb}{\mathbb}

\newcommand{\bm}{\mathbf}

\newcommand{\wh}{\widehat}
\newcommand{\wt}{\widetilde}
\newcommand{\ol}{\overline}
\newcommand{\ul}{\underline}

\newcommand{\del}{\partial}
\newcommand{\sbs}{\subset}

\newcommand{\sm}{\setminus}
\renewcommand{\emptyset}{\varnothing}
\newcommand{\interior}{\mathring}

\newcommand{\smas}{\wedge}

\newcommand{\card}[1]{\left\vert{#1}\right\vert}
\newcommand{\Set}[2]{\{#1\mid #2\}}

\newcommand{\si}{\sigma}

\newcommand{\ep}{\epsilon}

\newcommand{\from}{\colon}
\newcommand{\into}{\hookrightarrow}

\newcommand{\onto}{\twoheadrightarrow}

\newcommand{\set}[2]{\{#1\mid#2\}}

\renewcommand{\th}{^{\text{th}}}
\newcommand{\st}{^{\text{st}}}

\DeclareMathOperator{\Ob}{Ob}
\DeclareMathOperator{\Id}{Id}
\DeclareMathOperator{\Sq}{Sq}

\DeclareMathOperator{\Hom}{Hom}

\newcommand{\Kh}{\mathit{Kh}}
\newcommand{\KA}[1]{\mathbb{H}^1}
\newcommand{\rKh}{\widetilde{\Kh}}
\newcommand{\KhCx}{\mc{C}_{\mathit{Kh}}}
\newcommand{\rKhCx}{\widetilde{\mathcal{C}}_{\mathit{Kh}}}
\newcommand{\Cx}{\mathcal{C}}
\newcommand{\Cat}{\mathscr{C}}
\newcommand{\Dat}{\mathscr{D}}
\newcommand{\Eat}{\mathscr{E}}
\newcommand{\Funky}{\mathfrak{f}}

\newcommand{\Realize}[2][{}]{|#2|_{#1}}
\newcommand{\CRealize}[2][{}]{\|#2\|_{#1}}
\newcommand{\Codim}[1]{\langle#1\rangle}

\newcommand{\FlowCat}{\mathscr{C}}
\newcommand{\CubeFlowCat}[1]{\mathscr{C}_C(#1)}
\newcommand{\KhFlowCat}{\mathscr{C}_K}

\newcommand{\Moduli}{\mathcal{M}}

\newcommand{\op}{\mathrm{op}}
\newcommand{\ob}[1]{\mathbf{#1}}
\newcommand{\gr}{\mathrm{gr}}

\newcommand{\FlowCatSpace}[1]{\mathcal{X}(#1)}
\newcommand{\KhSpace}{\mathcal{X}_\mathit{Kh}}
\newcommand{\rKhSpace}{\widetilde{\mathcal{X}}_\mathit{Kh}}

\newcommand{\Cell}[2][{}]{\mathcal{C}_{#1}(#2)}

\newcommand{\TupV}{\mathbf}

\newcommand{\CubeCat}[1]{\underline{2}^{#1}}

\newcommand{\diff}{\delta}
\newcommand{\diffKh}{\delta_{\Kh}}

\newcommand{\vect}{\vec}
\newcommand{\Permu}[1]{{\Pi}_{#1}}
\newcommand{\permu}[1]{{\mathcal{S}}_{#1}}

\newcommand{\Filt}{\mathcal{F}}
\newcommand{\BNcx}{\mc{C}_{\mathit{BN}}}

\newcommand{\co}{\colon}
\newcommand{\bdy}{\partial}
\newcommand{\RR}{\R}

\newcommand{\pt}{\mathrm{pt}}

\newcommand{\NN}{\N}
\newcommand{\ZZ}{\mathbb{Z}}

\DeclareMathOperator{\image}{im}

\newcommand{\mathcenter}[1]{\vcenter{\hbox{$#1$}}}

\newcommand*{\defeq}{\mathrel{\vcenter{\baselineskip0.5ex \lineskiplimit0pt
                     \hbox{\scriptsize.}\hbox{\scriptsize.}}}%
                     =}

\newcommand{\SWdual}[1]{#1^\vee}
\newcommand{\mirror}[1]{m(#1)}
\newcommand{\cellC}[1]{C^{\mathit{cell}}_{#1}}
\newcommand{\cocellC}[1]{C_{\mathit{cell}}^{#1}}
\DeclareMathOperator{\Cone}{Cone}
\newcommand{\Precone}{\widetilde}

\newcommand{\BurnsideCat}{\mathscr{B}}
\newcommand{\CCat}[1]{\CubeCat{#1}}

\newcommand{\thic}[1]{\widehat{#1}}
\newcommand{\thicf}[2][{}]{\widehat{#2}_{#1}}
\newcommand{\othplus}{\dagger}
\newcommand{\SpDiag}[2][{}]{\widetilde{#2}_{#1}^+}
\newcommand{\SpDiagOth}[2][{}]{\widetilde{#2}_{#1}^\othplus}
\newcommand{\SpRefine}[2][{}]{\widetilde{#2}_{#1}}

\newcommand{\BSpaces}{\mathsf{Top}_\bullet}
\newcommand{\Spectra}{\mathscr{S}}

\newcommand{\CW}{\mathsf{CW}_\bullet}
\newcommand{\CWSpectra}{\mathscr{CW}}

\newcommand{\Sets}{\mathsf{Sets}}
\newcommand{\PermuCat}{\mathsf{Permu}}
\newcommand{\ArrowCat}{\mathsf{Arr}}
\newcommand{\SphereS}{\mathbb{S}} 
\newcommand{\Map}{\mathit{Map}}
\newcommand{\Cob}{\mathsf{Cob}}
\newcommand{\emb}{\mathrm{emb}}

\newcommand{\hyModCat}{\text{-}\mathsf{Mod}}
\newcommand{\DeltaInj}{\Delta_{\mathit{inj}}}
\newcommand{\olDeltaInj}{\overline{\Delta}_{\mathit{inj}}}

\newcommand{\KhFunc}{F_{\Kh}}
\newcommand{\rKhFunc}{F_{\rKh}}
\newcommand{\HKK}{M}
\newcommand{\HKKa}{\mathit{HKK}}

\newcommand{\HKKfun}{F_{\mathit{HKK}}}

\DeclareMathOperator{\hocolim}{hocolim}

\newcommand{\wtv}[1]{\stackrel{\raisebox{1pt}{\scriptsize$\rightsquigarrow$}}{\smash{#1}\vphantom{a}}} 
\newcommand{\wtvdag}[1]{{\wtv{#1}}^{\raisebox{-2.5pt}{\scriptsize$\othplus$}}}

\newcommand{\bul}[1]{\ul{\bm{#1}}}

\newcommand{\cmorph}[2]{\varphi_{#1,#2}}

\begin{document}

 
\title{Khovanov homotopy type, Burnside category, and products}

\author{Tyler Lawson}
\thanks{TL was supported by NSF Grant DMS-1206008.}
\email{\href{mailto:tlawson@math.umn.edu}{tlawson@math.umn.edu}}
\address{Department of Mathematics, University of Minnesota, Minneapolis, MN 55455}

\author{Robert Lipshitz}
\thanks{RL was supported by NSF Grant DMS-1149800.}
\email{\href{mailto:lipshitz@uoregon.edu}{lipshitz@uoregon.edu}}
\address{Department of Mathematics, University of Oregon, Eugene, OR 97403}

\author{Sucharit Sarkar}
\thanks{SS was supported by NSF Grant DMS-1350037.}
\email{\href{mailto:sucharit@math.ucla.edu}{sucharit@math.ucla.edu}}
\address{Department of Mathematics, University of California, Los Angeles, CA 90095}

\subjclass[2010]{\href{http://www.ams.org/mathscinet/search/mscdoc.html?code=57M25,55P42}{57M25,
    55P42}}


\date{\today}

\begin{abstract}
  In this paper, we give a new construction of a Khovanov stable homotopy
  type, or spectrum. We show that this construction gives a space stably homotopy
  equivalent to the Khovanov spectra constructed
  in~\cite{RS-khovanov} and~\cite{HKK-Kh-htpy} and, as a corollary,
  that those two constructions give equivalent spectra. We show that
  the construction behaves well with respect to disjoint unions,
  connected sums and mirrors, verifying several conjectures
  from~\cite{RS-khovanov}. Finally, combining these results with
  computations from~\cite{RS-steenrod} and the refined $s$-invariant
  from~\cite{RS-s-invariant} we obtain new results about the slice
  genera of certain knots.
\end{abstract}

\maketitle

\tableofcontents


\section{Introduction}
In a sequence of revolutionary papers in the 1980s, Floer introduced a
family of invariants in low-dimensional and symplectic topology,
now known as Floer
homologies~\cite{Floer-top-instanton,Floer-top-Lag,Floer-top-unregularized}. While
the definitions of these invariants appear to be a semi-infinite
dimensional version of Morse homology, unlike classical (finite- or
infinite-dimensional) Morse homology, Floer homologies were not
apparently isomorphic to the singular homologies of any natural space. In the
1990s, under appropriate hypotheses Cohen-Jones-Segal proposed a
construction of a stable homotopy type whose singular homology was
Floer homology~\cite{CJS-gauge-floerhomotopy}. Their construction
builds a CW complex cell-by-cell, using a version of the
Pontrjagin-Thom construction for manifolds with corners to define the
attaching maps; the input data to this Pontrjagin-Thom construction is
what Cohen-Jones-Segal call a \emph{framed flow category}. Analytic
difficulties mean that their proposal has not yet been carried out
rigorously, though Manolescu constructed a stable homotopy refinement
of the Seiberg-Witten Floer homology of rational homology 3-spheres by
other techniques~\cite{Man-gauge-swspectrum}
(compare~\cite{LM:SW-is-SW}).

Around the turn of the millennium, in a seminal paper Khovanov gave a
bigraded homology theory whose graded Euler characteristic is the
Jones polynomial~\cite{Kho-kh-categorification}. The definition of
this \emph{Khovanov homology} is purely combinatorial, though
equivalent or conjecturally equivalent definitions have been given
using Floer homology~\cite{SS-kh-sympl} (see
also~\cite{As-kh-symp-iso}) and algebraic
geometry~\cite{CK-kh-sheaves1}.

Inspired by Cohen-Jones-Segal's work and the relation with Floer homology,
in~\cite{RS-khovanov} we introduced a stable homotopy refinement of
Khovanov homology. That is, to
each link diagram $K$ we associated a spectrum
$\KhSpace(K)$, with the following properties:
\begin{enumerate}
\item\label{item:khsp-first} The spectrum $\KhSpace(K)$ is a finite CW
  spectrum, that is, a formal desuspension of a finite CW complex.
\item The spectrum $\KhSpace(K)$ comes with a wedge-sum decomposition
  $\KhSpace(K)=\bigvee_j \KhSpace^j(K)$.
\item For each $j$, the cellular cochain complex of $\KhSpace^j(K)$ is
  isomorphic to the Khovanov complex $\KhCx^{*,j}(K)$ in quantum
  grading $j$, via an isomorphism taking the standard generators for
  $\cocellC{*}(\KhSpace^j(K))$ to the standard generators for
  $\KhCx^{*,j}(K)$.
\item\label{item:khsp-last} For each $j$, the stable homotopy type of
  $\KhSpace^j(K)$ is an invariant of the isotopy class of the link
  represented by the diagram $K$.
\end{enumerate}
There is also a reduced version $\rKhSpace(K)$ of $\KhSpace(K)$, which
satisfies properties~(\ref{item:khsp-first})--(\ref{item:khsp-last})
with $\KhCx$ replaced by the reduced Khovanov complex $\rKhCx$. The
constructions of $\KhSpace(K)$ and $\rKhSpace(K)$ involve defining a
framed flow category combinatorially, and then applying
Cohen-Jones-Segal's Pontrjagin-Thom construction.

Largely independently, inspired by their homotopy-theoretic investigations of topological and conformal
field theories, Hu-Kriz-Kriz gave another
construction of a Khovanov stable homotopy type with the same basic
properties~\cite{HKK-Kh-htpy}. Roughly, they turn the Khovanov cube into a functor
from the cube category to the Burnside category of finite sets and
correspondences. They then apply the Elmendorf-Mandell infinite loop
space machine to obtain a functor from the cube to symmetric spectra,
and then take an iterated mapping cone 
to obtain a spectrum.

These refinements give extra structure to the Khovanov homology
groups. Because the Khovanov homology groups of $K$ are cohomology
groups of $\KhSpace(K)$, the mod-$p$ cohomology groups possess
Steenrod power operations such as the squares $\Sq^n$
at $p=2$. It also allows one to apply generalized comohology theories, such as
topological $K$-theory, Morava $K$-theory,
and stable homotopy groups.
Each of these has its own
Atiyah-Hirzebruch spectral sequence determining a potentially new
collection of knot invariants, and maps in the stable homotopy
category give natural transformations relating these different invariants.
In the
philosophically related case of Seiberg-Witten Floer theory, for
instance, equivariant $KO$-theory has been used to obtain interesting
low-dimensional applications~\cite{Man-gauge-KO,Lin-gauge-KO}.

One goal of the present paper is to study the behavior of $\KhSpace(K)$ under
disjoint unions and connected sums of links. In particular, we will
prove:
\begin{theorem}\label{thm:disjoint-union}\cite[Conjecture~10.3]{RS-khovanov}
  Let $L_1$ and $L_2$ be links, and $L_1\amalg L_2$ their disjoint union. Then
  \begin{equation}\label{eq:Kh-disjoint}
  \KhSpace^j(L_1\amalg L_2)\simeq \bigvee_{j_1+j_2=j}\KhSpace^{j_1}(L_1)\smas\KhSpace^{j_2}(L_2).
  \end{equation}
  Moreover, if we fix a basepoint $p$ in $L_1$, not at a crossing, and
  consider the corresponding basepoint for $L_1\amalg L_2$, then
  \begin{equation}\label{eq:rKh-disjoint}
  \rKhSpace^j(L_1\amalg L_2)\simeq \bigvee_{j_1+j_2=j}\rKhSpace^{j_1}(L_1)\smas\KhSpace^{j_2}(L_2).
  \end{equation}
\end{theorem}

\begin{theorem}\label{thm:connect-sum}\cite[Conjecture~10.4]{RS-khovanov}
  Let $L_1$ and $L_2$ be based links and $L_1\# L_2$ the connected sum
  of $L_1$ and $L_2$, where we take the connected sum near the
  basepoints. Then
  \begin{equation}\label{eq:rKh-conn-sum}
  \rKhSpace^j(L_1\# L_2)\simeq \bigvee_{j_1+j_2=j}\rKhSpace^{j_1}(L_1)\smas\rKhSpace^{j_2}(L_2).
  \end{equation}
\end{theorem}
(We also compute the unreduced Khovanov spectrum of a connected
sum~\cite[Conjecture~10.6]{RS-khovanov}, as \Theorem{unred-con-sum}.)

These theorems, though unsurprising themselves, have some interesting
corollaries:

\begin{corollary}\label{cor:Kh-squares}
  For any $n$ there exists an $n$-component link $L_n$ so that the operation
  \[
    \Sq^n\co \Kh^{i,j}(L_n)\to \Kh^{i+n,j}(L_n)
  \]
  is non-zero, for some $i,j\in\ZZ$. Similarly, there exists a knot $K_n$ so that the operation
  \[
  \Sq^n\co \rKh^{i,j}(K_n)\to \rKh^{i+n,j}(K_n)
  \]
  is non-zero, for some $i,j\in\ZZ$. Further, for this knot, the operation 
  \[
  \Sq^n\co \Kh^{i,j}(K_n)\to \Kh^{i+n,j}(K_n)
  \]
  is also non-zero for some $i,j\in\ZZ$.
\end{corollary}

\captionsetup[subfloat]{width=0.28\textwidth}
\begin{figure}
\subfloat[$(9_{42},1)$]{\includegraphics[width=0.28\textwidth]{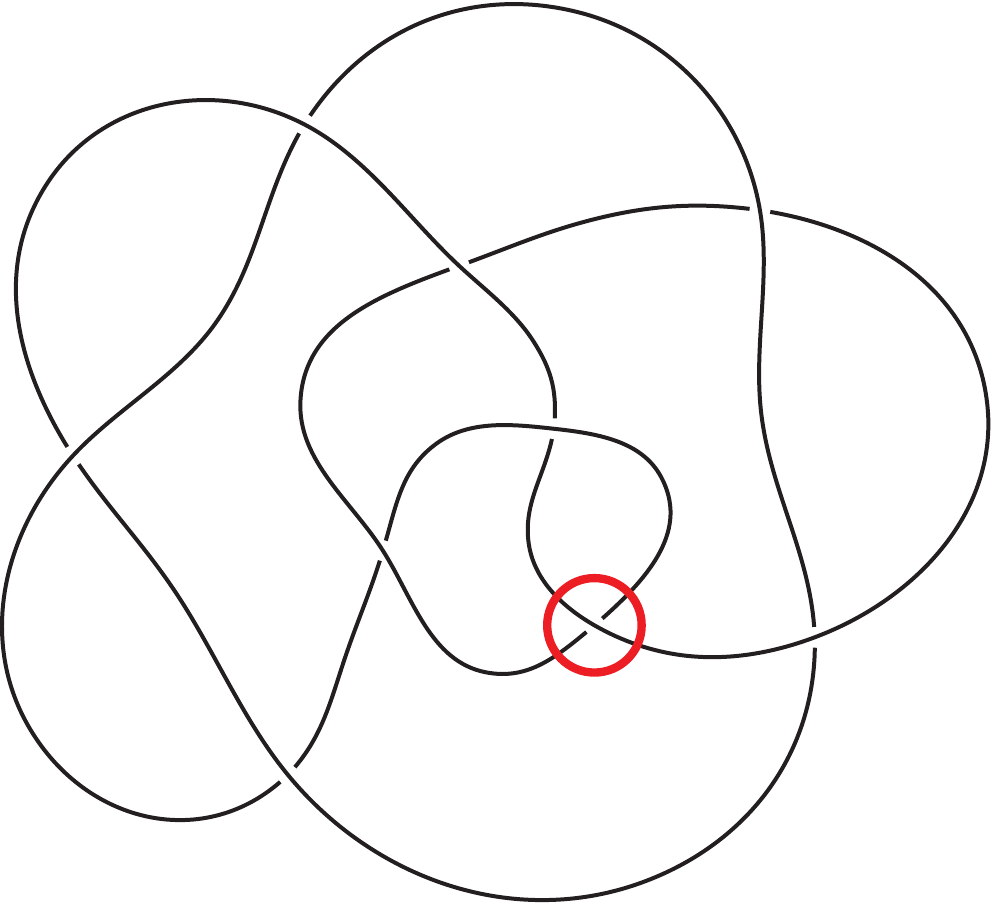}}
\hspace{0.01\textwidth}
\subfloat[$(10_{136},1)$]{\includegraphics[width=0.28\textwidth]{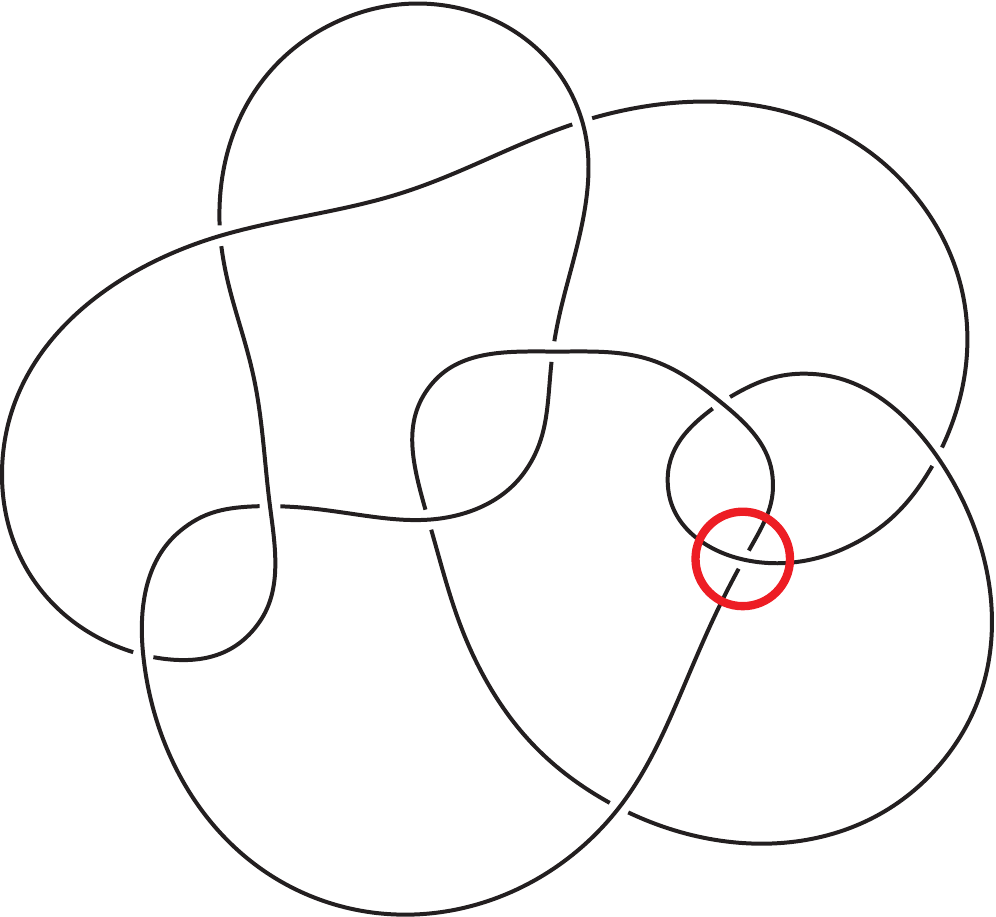}}
\hspace{0.01\textwidth}
\subfloat[$(m(11^n_{19}),2)$]{\includegraphics[width=0.28\textwidth]{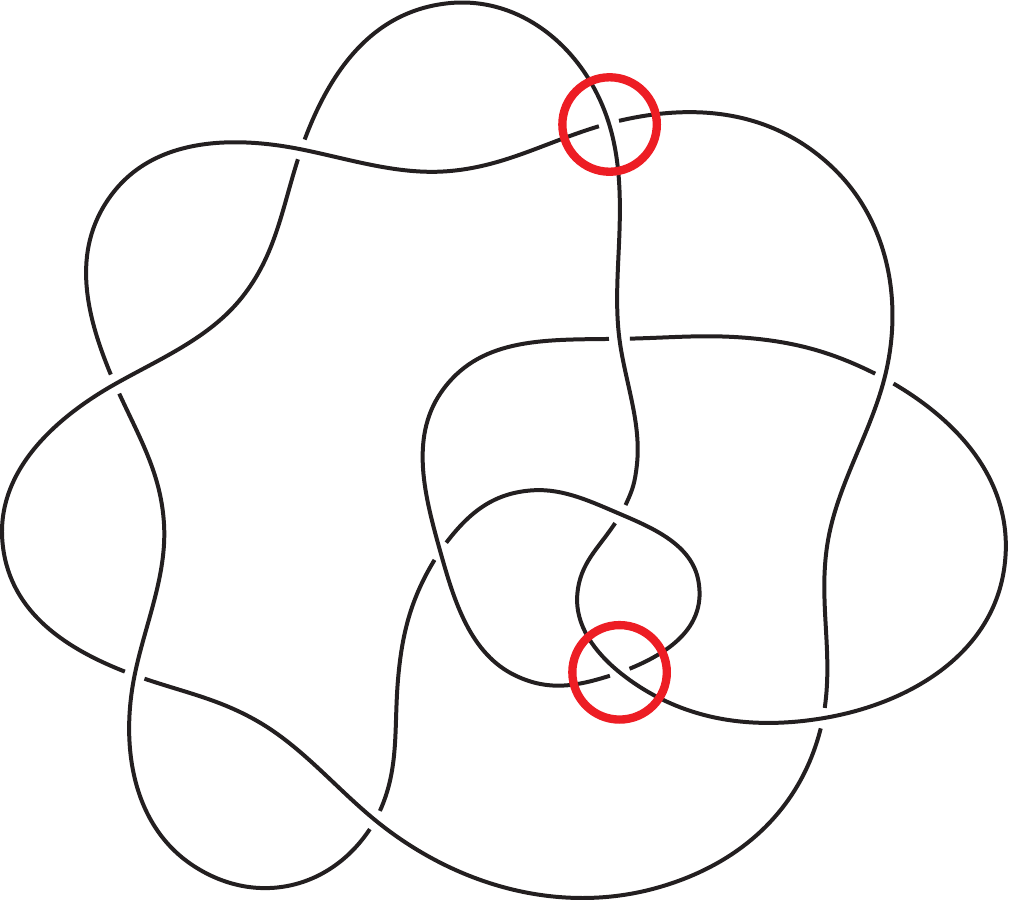}}\\
\subfloat[$(m(11^n_{20}),1)$]{\includegraphics[width=0.28\textwidth]{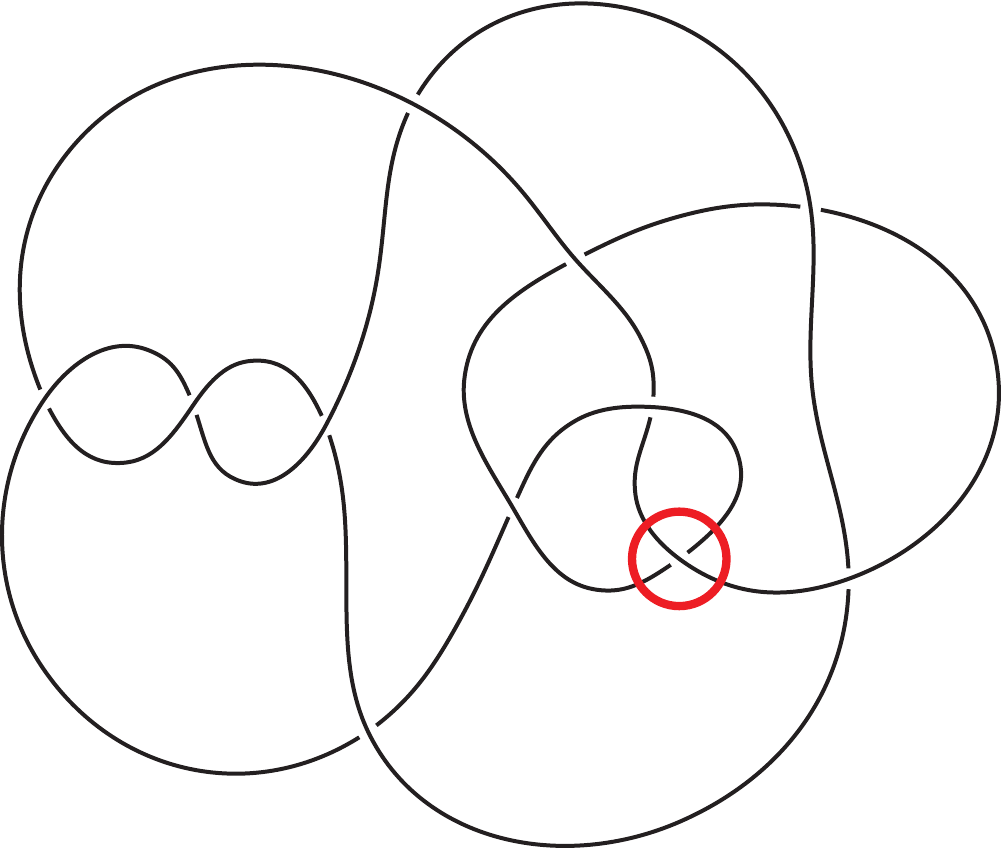}}
\hspace{0.01\textwidth}
\subfloat[$(11^n_{70},2)$]{\includegraphics[width=0.28\textwidth]{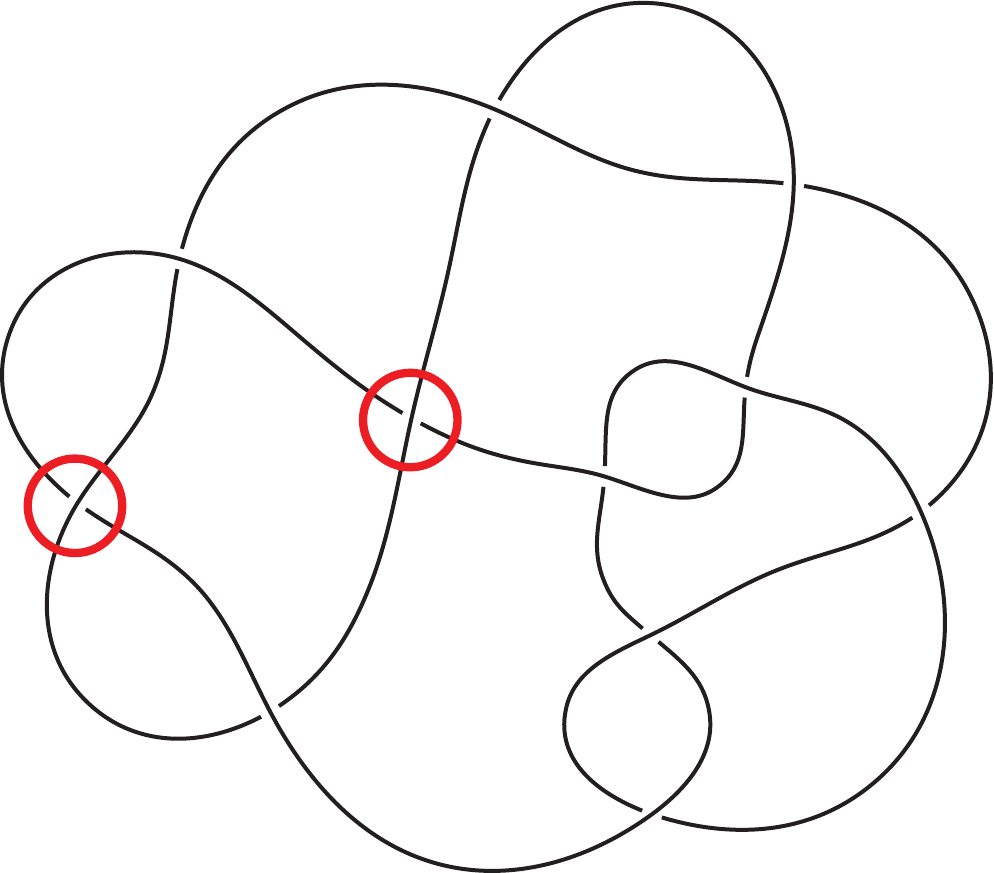}}
\hspace{0.01\textwidth}
\subfloat[$(11^n_{96},1)$]{\includegraphics[width=0.28\textwidth]{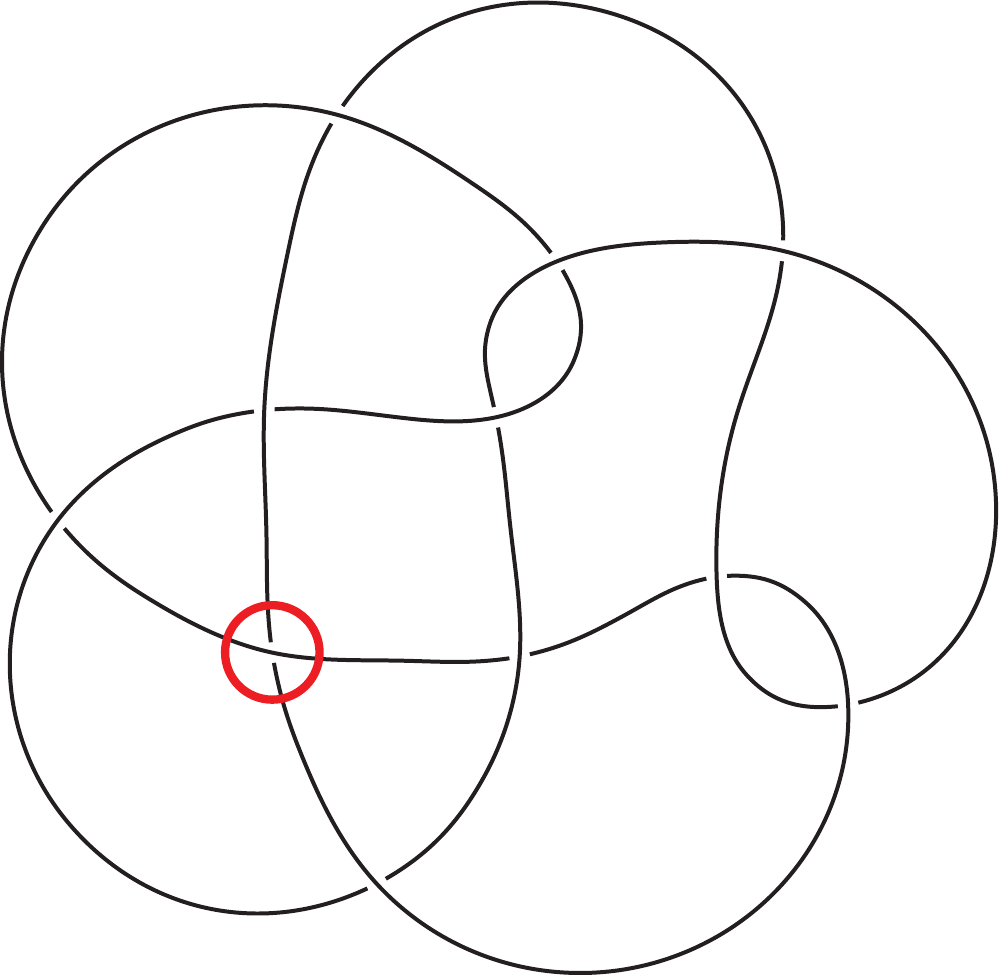}}
\caption{\textbf{The knots appearing in \Corollary{top-appl}.} We have labeled
  the knot $K$ by the pair $(K,g_4(K))$. The value of $g_4(K)$ is
  extracted from Knotinfo \cite{KIN-kh-knotinfo}; the knot diagrams
  have been produced using Knotilus \cite{KIT-kh-knotilus}. Crossings
  giving a minimal unknotting for each knot are circled.}\label{fig:knot-diagrams}
\end{figure}

\begin{corollary}\label{cor:top-appl}
  Let $K$ be one of the knots $9_{42}$, $10_{136}$, $m(11^n_{19})$,
  $m(11^n_{20})$, $11^n_{70}$, or $11^n_{96}$.  (Here $m$ denotes the
  mirror.)  Let $L$ be a knot which is the closure of a positive
  braid. Letting $g_4$ denote the four-ball genus, we have
  \[
  g_4(K\#L)=g_4(K)+g_4(L).
  \]
\end{corollary}
\begin{remark}
  There is some confusion for nomenclature of knots regarding
  mirrors. See \Figure{knot-diagrams} for our convention, and compare
  with the
  discussion around \Table{invariants}. The value of $g_4(K)$ can be
  extracted from \Figure{knot-diagrams} or \Table{invariants}; and
  $g_4(L)$ equals the genus of the Seifert surface obtained by
  applying Seifert's algorithm to the positive braid closure knot
  diagram for $L$~\cite[Theorem 4]{Ras-kh-slice}.
\end{remark}

\Corollary{top-appl} is not implied by computations of Rasmussen's
$s$-invariant or the Heegaard Floer $\tau$-invariant or the signature;
see \Table{invariants}. At least for $9_{42}$, the result is not
implied by the Heegaard Floer concordance invariant $\Upsilon$; the
Heegaard Floer $d$ invariant of $+1$
surgery~\cite{Krcatovich-personal}; or Hom-Rasmussen-Wu's
$\nu^+$-invariant~\cite{Ras-hf-knotcomplements,HW-hf-nu-invt,Hom-personal}.
(These observations are not entirely independent.)

The construction of $\KhSpace(K)$ from~\cite{RS-khovanov} uses
Cohen-Jones-Segal's notion of flow categories. 
To give a proof of Theorems~\ref{thm:disjoint-union} and~\ref{thm:connect-sum}
in this language seems tedious at best: it involves understanding the
combinatorics of (broken) Morse flows on product manifolds, which turns out to
be rather intricate.

Fortunately, as noted above, Hu-Kriz-Kriz gave another construction of
a Khovanov stable homotopy type, and from their construction the
behavior under connected sums and disjoint unions is clear.
%
In this paper, we describe three additional constructions:
\begin{enumerate}
\item A framing-free reformulation of the construction
  from~\cite{RS-khovanov} for a special family of flow categories,
  called cubical flow categories,
  in \Section{cubical}.
\item A version similar to the construction from~\cite{HKK-Kh-htpy} as
  a homotopy colimit, but using the thickened cube instead of the
  Elmendorf-Mandell machinery, in \Section{realize-functor}.
\item An intermediate object between the two
  in \Section{smaller-cube}, using little $k$-cubes.
\end{enumerate}
We then prove:
\begin{theorem}\label{thm:Kh-htpy-agrees}
  The Khovanov stable homotopy types constructed
  in \Section{cubical}, \Section{realize-functor},
  and \Section{smaller-cube}, the
  Khovanov stable homotopy type constructed in~\cite{HKK-Kh-htpy}, and
  the Khovanov stable homotopy type constructed in~\cite{RS-khovanov}
  are all stably homotopy equivalent.
\end{theorem}
This theorem is proved in parts:
\begin{itemize}
\item \Theorem{cubical-CJS-realization} asserts that the cubical flow
  category realization (\Section{cubical}) agrees with the
  Cohen-Jones-Segal construction used in~\cite{RS-khovanov}.
\item \Theorem{box-equals-cubical} asserts that the cubical flow
  category realization (\Section{cubical}) agrees with the little
  $k$-cubes realization (\Section{smaller-cube}).
\item \Theorem{smaller-diag} asserts that the little $k$-cubes
  realization (\Section{smaller-cube}) agrees with the thickening construction
  (\Section{realize-functor}).
\item \Theorem{agree-with-HKK} asserts that the thickening construction (\Section{realize-functor}) agrees with the construction
  in~\cite{HKK-Kh-htpy}.
\end{itemize}

Theorems~\ref{thm:disjoint-union} and~\ref{thm:connect-sum} follow
easily. \Corollary{Kh-squares} follows immediately from these theorems
and computations in~\cite{RS-steenrod}.  We also obtain, via a
TQFT-style argument, that the Khovanov homotopy type of the mirror
knot $m(K)$ is the Spanier-Whitehead dual to the Khovanov homotopy
type of $K$ (\Theorem{mirror}). Finally, \Corollary{top-appl} follows
from these results, the refined $s$ invariant
in~\cite{RS-s-invariant}, the computations in~\cite{RS-steenrod}, and
a brief further argument.

This paper is organized as follows. We collect some basic notation in \Section{basic-notation}. \Section{review} has background on the cube category and the Khovanov construction, the latter of which is not needed again until \Section{Kh-htpy} except in some examples. Other background appears at the beginning of the section in which it is first used.
In \Section{cubical}, after recalling the basics of flow categories, we introduce a special
class of them, cubical flow categories, which live over the
cube, and give a reformulation of the Cohen-Jones-Segal realization
for this class of flow categories (``cubical realization''). (The
Khovanov flow category of~\cite{RS-khovanov} is a cubical flow
category; this is a crucial tool in its construction.)
In \Section{realize-functor} we show that cubical flow categories are
equivalent to $2$-functors from the cube to the Burnside category, and
give a different, choice-free way to realize such a
functor. In \Section{smaller-cube} we give a smaller but less
canonical way to realize a $2$-functor from the cube to the Burnside
category, and prove the two ways to realize such a functor are
equivalent. \Section{cubical-to-cubes} shows that the realization from
\Section{smaller-cube} agrees with the cubical realization from 
\Section{cubical}. \Section{Kh-htpy} is a brief interlude to summarize
these results and recall the Khovanov homotopy type. \Section{HKK}
shows that these realizations agree with the
Hu-Kriz-Kriz-construction~\cite{HKK-Kh-htpy}.

In Sections~\ref{sec:Kh-disj-union}--\ref{sec:applications} we use
these reformulated realizations to prove new properties of the
Khovanov homotopy type. The realizations of the product (smash
product) and disjoint union (wedge sum) of functors from the cube to
the Burnside category have properties as one would
expect; \Section{Kh-disj-union} uses these properties to study the
Khovanov homotopy type of a disjoint union and connected sum,
Theorems~\ref{thm:disjoint-union},~\ref{thm:connect-sum},
and~\ref{thm:unred-con-sum}, and verifies
\Corollary{Kh-squares}. \Section{mirror} uses a TQFT-style argument
suggested by the referee for~\cite{RS-steenrod} to deduce a formula
for the Khovanov homotopy type of a mirror, verifying another
conjecture from~\cite{RS-khovanov}. Finally, \Section{applications}
gives an additivity property for the refined $s$ invariant introduced
in~\cite{RS-s-invariant} and obtains \Corollary{top-appl}.

\Appendix{flow-chart} has a flow chart of how the different sections
depend on each other, so a reader only interested in a particular result can
choose the most efficient path to it.

\begin{remark}
  It may be interesting to compare the homotopy colimit definition of
  the Khovanov homotopy type with~\cite{ET-kh-spectrum}; but see also~\cite{ELST-trivial}.
\end{remark}

\textit{Acknowledgments.} We thank Brent Everitt, Jennifer Hom, David Krcatovich,
Peter Ozsv\'ath and Paul Turner for interesting conversations. We also thank the
referee for~\cite{RS-steenrod} for suggesting the argument
in \Section{mirror}. RL thanks Princeton University for its
hospitality; most of this research was completed while he visited
there. We thank Michael Willis for corrections to an early version of
this manuscript. We also thank the contributors to $n$Lab: while we have
cited published references for the relevant background, $n$Lab has
often been helpful in finding those references. Finally, we thank the referees
for their helpful comments.

\subsection{Basic notation}\label{sec:basic-notation}
The ``cube'' $\{0,1\}^n$ will appear in a number of contexts in this
paper, as will some auxiliary notions related to it:
\begin{itemize}
\item There is a partial order on $\{0,1\}^n$ defined by $v\geq w$ if
  $v$ is obtained from $w$ by replacing some $0$'s by $1$'s.  Define
  $v>w$ if $v\geq w$ and $v\neq w$, and define $\leq$ and $<$ in the
  corresponding ways. The maximum and minimum elements under
  this partial order are denoted $\vec{1}$ and $\vec{0}$,
  respectively.
\item We denote the Manhattan (or $\ell^1$) norm on $\{0,1\}^n$ by
  $\card{v}=\sum_{i=1}^n v_i$. 
\item A \emph{sign assignment} $s$ on the cube is the following: for
  every $u>v$ with $\card{u}-\card{v}=1$, we associate an element
  $s_{u,v}\in\F_2$ such that for any $u>w$ with $\card{u}-\card{w}=2$,
  we have
  \[
  \sum_{\substack{v\\u>v>w}}(s_{u,v}+s_{v,w})=1.
  \]
\end{itemize}

A number of categories will appear in this paper:
\begin{itemize}
\item The category $\Sets$ of finite sets and set maps.
\item The Burnside (2-)category $\BurnsideCat$ of sets, correspondences, and
  isomorphisms of correspondences (see \Section{Burnside}).
\item The cube category $\CCat{n}=\bigl\{1\longrightarrow 0\bigr\}^n$
  (see \Section{cube}).
\item The category $\BSpaces$ of (well) based topological spaces.
\item The subcategory $\CW$ of $\BSpaces$ generated by based CW
  complexes and cellular maps.
\item The category $\Spectra$ of spectra. For concreteness, we can
  take the category of symmetric spectra in topological spaces
  \cite{MMSS-diagram-spectra}.
\item The category $R\hyModCat$ of $R$-modules.
\item The category $\PermuCat$ of permutative categories
  (see \Section{elmendorff-mandell-machine}).
\item The category $\Cob_{\emb}^{1+1}$ of oriented $1$-manifolds
  embedded in $S^2$ and oriented cobordisms embedded in $[0,1]\times
  S^2$ (see \Section{two-functors-same}).
\end{itemize}

Some other notation:
\begin{itemize}
\item Let $\RR_+=[0,\infty)$.
\item Let $\NN=\{0,1,2,\dots\}$ denote the set of non-negative integers.
\end{itemize}

\section{Review of Khovanov homology}\label{sec:review}

\subsection{The cube category}\label{sec:cube}
Let $\CCat{1}$ denote the category with two objects, denoted $0$ and
$1$, and a single non-identity morphism, from $1$ to $0$:
\[
\CCat{1}=\bigl\{1\longrightarrow 0\bigr\}.
\]
For $n\in\ZZ$, $n>1$ let $\CCat{n}=\CCat{1}\times\CCat{n-1}$. That is,
$\CCat{n}$ is the small category with object set $\{0,1\}^n$. Given objects
$v,w\in\{0,1\}^n$ the morphism set $\Hom(v,w)$ is empty unless $v\geq
w$ in the partial order induced by the relation $1>0$ on $\CCat{1}$. 
If $v\geq w$ the set $\Hom(v,w)$ has a single element, which
we will denote $\cmorph{v}{w}$.

The reader is warned that many authors refer to the opposite category
$(\CCat{n})^\op$ as the cube category. (Our choice is made to agree with the
grading conventions on Khovanov homology~\cite{Kho-kh-categorification} and the
Khovanov stable homotopy type~\cite{RS-khovanov}.)

Often, we will view $\CCat{n}$ as a $2$-category with no
non-identity $2$-morphisms. That is, for $f,g\in\Hom(v,w)$, we define
$\Hom(f,g)$ to be empty unless $f=g$ and to have a single element if
$f=g$.

\subsection{The Khovanov construction}\label{sec:khovanov-basic}
The material in the rest of this section is used in
Examples~\ref{exam:KhFlowCat} and~\ref{exam:KhFunc} and
\Remark{no-Lee}, and then is not used again until
Section~\ref{sec:Kh-htpy}.

Khovanov homology, and several of its
generalizations, are all constructed from the cube of
resolutions of a link diagram. Let $K$ be a link diagram in $S^2$ with
$n$ crossings, numbered $c_1,\dots,c_n$. Each of these crossings can
be resolved locally in two different ways, called the $0$-resolution
and the $1$-resolution (see, for instance,
\cite[Figure~14]{Kho-kh-categorification}). Therefore, to each
$v\in\{0,1\}^n$ there is an associated complete resolution $\mc{P}(v)$
obtained by replacing the crossing $c_i$ by its $0$-resolution if
$v_i=0$ and its $1$-resolution if $v_i=1$.
The complete resolution $\mc{P}(v)$ consists
of a collection of disjoint circles in $S^2$.

\begin{figure}
\begin{overpic}[width=0.95\textwidth]{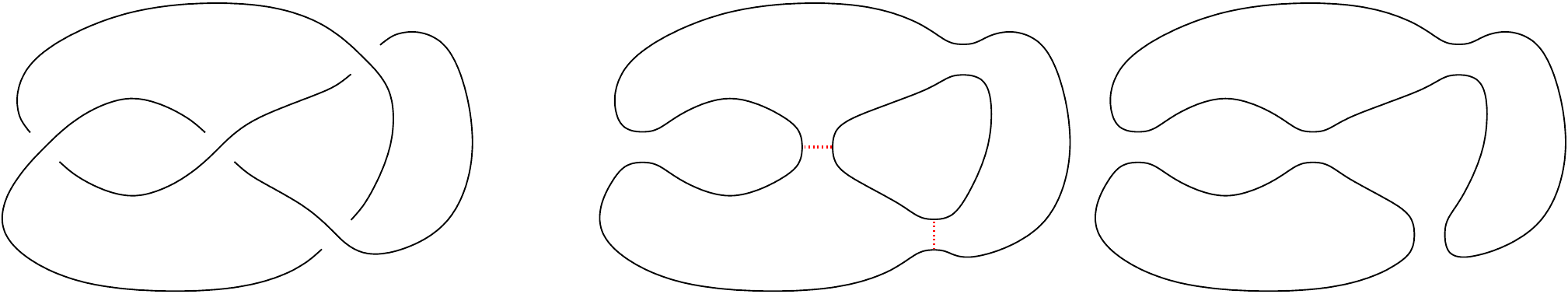}
\put(2.5,6){$c_1$}
\put(13,6){$c_2$}
\put(20.5,5.5){$c_3$}
\put(21.5,11.5){$c_4$}
\put(45,3){$\mc{P}(v)$}
\put(78.5,3){$\mc{P}(u)$}
\end{overpic}
\caption{\textbf{Some resolutions of a knot diagram.} Left: A knot
  diagram for the figure-eight knot with crossings
  $c_1,\dots,c_4$. Right: The complete resolutions at the vertices
  $v=(1,0,0,0)$ and $u=(1,1,1,0)$, along with the embedded $1$-handle
  cores representing the embedded cobordism between $\mc{P}(u)$ and
  $\mc{P}(v)$.}\label{fig:khovanov-basic}
\end{figure}
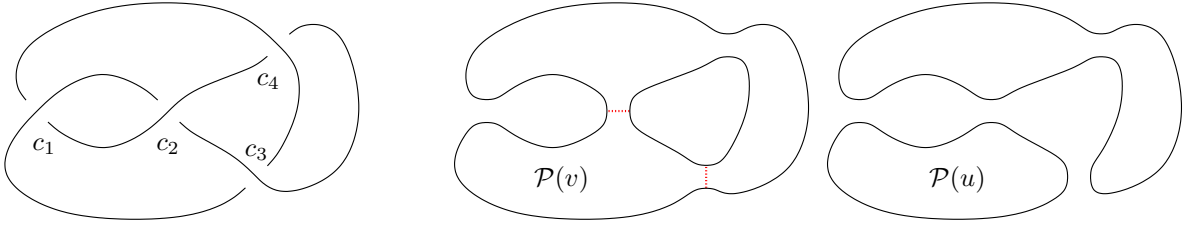

For any $u\geq v\in\{0,1\}^n$, there is a cobordism embedded in
$[0,1]\times S^2$ that connects the resolutions $\mc{P}(u)$ and
$\mc{P}(v)$: to obtain $\mc{P}(u)$, one attaches embedded $1$-handles
to $\mc{P}(v)$ with cores certain arcs near the crossings $c_i$
for all $i$ with $u_i>v_i$. These arcs are illustrated in \Figure{khovanov-basic}. (See also
\cite[Figure~18]{Kho-kh-categorification}. This is also the first
step in Bar-Natan's ``picture world'' approach to Khovanov
homology in~\cite{Bar-kh-tangle-cob}.)  For the special case when
$|u|-|v|=1$, the cobordism either \emph{merges} two circles into one,
or \emph{splits} a single circle into two.

To construct the Khovanov complex, we apply a $(1+1)$-dimensional TQFT
to this cube of cobordisms to obtain a commutative cube of abelian
groups. Specifically, consider the rank-2 Frobenius algebras over
$\ZZ[h,t]$ with basis $\{x_+,x_-\}$ and multiplication and
comultiplication given by
\begin{gather*}
x_+\otimes x_+\stackrel{m}{\mapsto} x_+\qquad x_+\otimes x_-\stackrel{m}{\mapsto} x_-\qquad
x_-\otimes x_+\stackrel{m}{\mapsto} x_-\qquad x_-\otimes x_-\stackrel{m}{\mapsto} hx_-+tx_+ \\
x_+\stackrel{S}{\mapsto} x_+\otimes x_-+ x_-\otimes x_+-hx_+\otimes
x_+\qquad x_-\stackrel{S}{\mapsto} x_-\otimes x_-+tx_+\otimes x_+,
\end{gather*}
and the corresponding $(1+1)$-dimensional TQFT.  Applying this TQFT to
the cube of resolutions of $K$ gives a commutative cube $A\co
(\CCat{n})^{\op}\to \ZZ[h,t]\hyModCat$.
Explicitly, given a vertex $v$, a \emph{Khovanov generator
  over $v$} is a labeling of the circles in $\mc{P}(v)$ by elements of
the set $\{x_+,x_-\}$. The module $A(v)$ is freely generated by the
set of Khovanov generators $F(v)$ over $v$. For an edge of the cube
$\cmorph{u}{v}$ (where $u>v$ and $|u|-|v|=1$), $A(\cmorph{u}{v})\from A(v)\to
A(u)$ is the multiplication or comultiplication above, depending on
whether the cobordism from $\mc{P}(v)$ to $\mc{P}(u)$ is a merge or split, respectively.

\begin{definition}\label{def:KhCx-diffKh}
  Fix a sign assignment $s$ on the cube (in the sense
  of \Section{basic-notation}). The chain complex $\Cx(K)$ is the
  totalization of $A$ with respect to $s$. That is, the chain group is
  defined to be $\Cx(K)=\oplus_{v\in\{0,1\}^n}A(v)$ and the
  differential $\diff\from\Cx(K)\to\Cx(K)$ is defined by stipulating
  the component $\diff_{u,v}$ of $\diff$ that maps from $A(v)$ to
  $A(u)$ to be
  \[
  \diff_{u,v}=\begin{cases}
    (-1)^{s_{u,v}}A(\cmorph{u}{v}^\op)&\text{if $u>v$ and $|u|-|v|=1$,}\\
    0&\text{otherwise.}
  \end{cases}
  \]

  The \emph{Khovanov chain complex} $(\KhCx(K),\diffKh)$ is the
  specialization $h=t=0$. The homology of $\KhCx(K)$ is the
  \emph{Khovanov homology} $\Kh(K)$.
\end{definition}
(The Khovanov complex was introduced by
Khovanov~\cite{Kho-kh-categorification}. The specializations
$(h,t)=(0,1)$ and $(h,t)=(1,0)$ were studied by
Lee~\cite{Lee-kh-endomorphism} and Bar-Natan~\cite{Bar-kh-tangle-cob},
respectively; the case of general $(h,t)$ was studied by
Khovanov~\cite{Kho-kh-Frobenius}, Naot~\cite{Nao-kh-universal}, and
others.)

The homological grading of the summand $A(v)\subseteq \Cx(K)$ is
$|v|-n_-$, where $n_-$ denotes the number of negative crossings in the
link diagram.  There is additionally an internal grading, called the
\emph{quantum grading}, that persists throughout; the quantum grading
of any Khovanov generator in $F(v)$ is
 \[
 n-3n_-+|v|+\#\{\text{circles in $\mc{P}(v)$ labeled $x_+$}\}-\#\{\text{circles in $\mc{P}(v)$ labeled $x_-$}\}.
 \]
 The quantum gradings of the formal variables $(h,t)$ are
 $(-2,-4)$.

 In the presence of a basepoint on the link diagram, and after setting
 $t=0$, there is also a reduced theory. In the reduced theory, for any
 $v\in\{0,1\}^n$, we only consider the Khovanov generators in $F(v)$
 that label the pointed circle in $\mc{P}(v)$ as $x_-$, and shift
 quantum gradings up by $1$. The reduced Khovanov chain complex and
 the reduced Khovanov homology are denoted $\rKhCx$ and $\rKh$,
 respectively. (Alternatively, we could have defined a reduced theory
 by setting $t=0$ and only considering the Khovanov generators that
 label the pointed circle as $x_+$, and shifting quantum gradings down
 by $1$. When $h=t=0$, these two reduced theories are canonically
 isomorphic; see also~\cite{Wig-kh-bn}.)

 The Khovanov homology is an invariant of the link, not just of the link
 diagram, and the reduced Khovanov homology is an invariant of the pointed link.
 More generally, the chain homotopy type of the chain complex $\Cx(K)$ over
 $\Z[h,t]$ is a link invariant as well.

\section{Cubical flow categories}\label{sec:cubical}
The Khovanov stable homotopy type is defined using an auxiliary
object, the Khovanov flow category. We review flow categories
in \Section{flow-cat}, after first recalling some notions related to
manifolds with corners, on which flow categories are based.  The
Khovanov flow category from~\cite{RS-khovanov} is defined as a kind of
cover of another flow category, the cube flow category
(\Definition{cube-flow-cat}), which in turn is based on permutohedra,
a family of polytopes reviewed in \Section{permutohedra}. After
reviewing the cube flow category and some of its basic properties
in \Section{cube-flow-cat}, new material begins in \Section{cubical-flow-cat},
where we abstract the notion of a cover of the cube flow category, into a cubical flow
category (\Definition{cubical}). In \Section{cubical-neat} we give a
slightly different notion of neat embeddings for cubical flow
categories. Using this notion, in \Section{cubical-realize} we give a
realization procedure for a cubical flow category, the cubical
realization. In \Section{cubical-CJS-same} we show that the cubical
realization and the Cohen-Jones-Segal realization (reviewed
in \Section{flow-cat}) give homotopy equivalent spaces.

 \subsection{Manifolds with corners and \texorpdfstring{$\Codim{n}$}{<n>}-manifolds}\label{sec:manifold-corners}
The original construction of the Khovanov stable homotopy type
from~\cite{RS-khovanov} relies on Cohen-Jones-Segal's notion of flow
categories~\cite{CJS-gauge-floerhomotopy}, which in turn
uses a particular notion of manifolds with corners, called
$\Codim{n}$-manifolds~\cite{Jan-top-o(n)manifolds,Lau-top-cobordismcorners}. For
the reader's convenience, we review the relevant definitions here.

\begin{definition}
  A \emph{$k$-dimensional manifold with corners} is a topological
  space $X$ together with an atlas
  $\{(U_\alpha,\phi_\alpha\co U_\alpha\to (\RR_+)^k)\}$ modeled on
  open subsets of $(\RR_+)^k$, so that the transition functions are
  smooth. Given a point $x$ in a chart $(U,\phi)$ let $c(x)$ be the
  number of coordinates in $\phi(x)$ which are $0$; $c(x)$ is
  independent of the choice of chart. The \emph{codimension-$i$
    boundary} of $X$ is $\{x\in X\mid c(x)= i\}$. A $k$-dimensional
  manifold with corners $X$ has a well-defined tangent space $TX$,
  which is an $\RR^k$-plane bundle; a \emph{Riemannian metric} on $X$
  means a Riemannian metric on $TX$.
\end{definition}

\begin{definition}
  A \emph{facet} of $X$ is the closure of a connected component of the
  codimension-$1$ boundary of $X$. (If $X$ is a polytope then this
  agrees with the usual definition of facets.) A \emph{multifacet} of
  $X$ is a (possibly empty) union of disjoint facets of $X$.  A
  manifold with corners $X$ is a \emph{multifaceted manifold} if every
  $x\in X$ belongs to exactly $c(x)$ facets of $X$. An
  \emph{$\Codim{n}$-manifold} is a multifaceted manifold $X$ along
  with an ordered $n$-tuple $(\del_1X,\dots,\del_nX)$ of multifacets
  of $X$ such that: $\bigcup_i\bdy_iX=\bdy X$; and for all distinct
  $i,j$, $\del_iX\cap\del_jX$ is a multifacet of both $\del_iX$ and
  $\del_jX$. (The number $n$ need not be the dimension of $X$.)
  See Laures~\cite{Lau-top-cobordismcorners} for more
  details. (Laures uses the terms `connected
  face', `face', and `manifold with faces' for `facet', `multifacet',
  and `multifaceted manifold', respectively. We have changed the
  terminology since `face' means something different for polytopes
  in \Section{permutohedra}.) Given a $\Codim{n}$-manifold $X$ and a
  vector $v\in\{0,1\}^n$ let $X(v)=\bigcap_{v_i=0}\bdy_iX$, with the
  convention that $X(\vec{1})=X$.
\end{definition}

\begin{example}
  An $n$-gon (polygon with $n$ sides) is a multifaceted manifold if
  $n>1$, while a $1$-gon (disk with one corner on the boundary) is a
  manifold with corners but not a multifaceted manifold.  Only the
  $2n$ gons can be made into $\Codim{2}$-manifolds, though
  $(2n+1)$-gons can be viewed as $2$-dimensional
  $\Codim{3}$-manifolds.  Of the Platonic solids only the tetrahedron,
  cube, and dodecahedron are manifolds with corners, and all three are
  multifaceted manifolds. The cube can be made into a
  $\Codim{3}$-manifold by defining $\bdy_1X$ to be the front and back
  facets, $\bdy_2X$ to be the top and bottom facets, and $\bdy_3X$ to
  be the left and right facets. The tetrahedron and dodecahedron cannot 
  be given the structure of $\Codim{3}$-manifolds, although both
  can be made into $\Codim{4}$-manifolds.  An even more fundamental
  example is $(\RR_+)^n$ itself, which is a $\Codim{n}$-manifold by
  setting $\bdy_i(\RR_+)^n=\{v\in \RR_+^n\mid v_i=0\}$. Similarly,
  $\RR^N\times (\RR_+^n)$ is an $(n+N)$-dimensional
  $\Codim{n}$-manifold.
\end{example}

\begin{construction}
  Given an $\Codim{n}$-manifold $X$ and an $\Codim{m}$-manifold $Y$,
  the product $X\times Y$ inherits the structure of a
  $\Codim{n+m}$-manifold, by declaring
  \[
    \bdy_i(X\times Y)=
    \begin{cases}
      (\bdy_iX)\times Y & 1\leq i\leq n\\
      X\times (\bdy_{i-n}Y) & n+1\leq i\leq n+m.
    \end{cases}
  \]
\end{construction}

We end this subsection with the definition of neat
embeddings.
\begin{definition}
  Consider $\Codim{n}$-manifolds $X$ and $Y$ and fix a Riemannian
  metric on $Y$. A \emph{neat embedding} of $X$ into $Y$ is a smooth
  map $f\co X\to Y$ so that:
  \begin{itemize}
  \item $f^{-1}(Y(v))=X(v)$ for all $v\in\{0,1\}^n$.
  \item $f|_{X(v)}\co X(v)\to Y(v)$ is an embedding for each
    $v\in \{0,1\}^n$.
  \item For all $w<v\in\{0,1\}^n$, $f(X(v))$ is perpendicular to
    $Y(w)$ with respect to the Riemannian metric on $Y$, and in
    particular is transverse to $Y(w)$.
  \end{itemize}
\end{definition}

\subsection{Flow categories}\label{sec:flow-cat}
Next we recall some notions
about flow categories, from Cohen-Jones-Segal~\cite{CJS-gauge-floerhomotopy} (see
also~\cite[Section 3]{RS-khovanov}), partly to fix terminology for
this paper.

\begin{definition}\label{def:flow-cat}
  A \emph{flow category} $\FlowCat$ is a topological category whose
  objects $\Ob(\FlowCat)$ form a discrete space, equipped with a
  grading function $\gr\from\Ob(\FlowCat)\to\Z$, and whose morphism spaces
  satisfy the following conditions:
  \begin{enumerate}[label=(FC-\arabic*),ref=(FC-\arabic*)]
  \item\label{item:FC-id} For any $x\in\Ob(\FlowCat)$,
    $\Hom(x,x)=\{\Id\}$.
  \item\label{item:FC-mflds} For distinct $x,y\in\Ob(\FlowCat)$ with
    $\gr(x)-\gr(y)=k$, $\Hom(x,y)$ is a (possibly empty) compact
    $(k-1)$-dimensional $\Codim{k-1}$-manifold; and %
  \item\label{item:FC-3} The composition maps combine to produce a diffeomorphism of
    $\Codim{k-2}$-manifolds
    \[
    \coprod_{\substack{z\in\Ob(\FlowCat)\setminus\{x,y\}\\\gr(z)-\gr(y)=i}}
    \Hom(z,y)\times\Hom(x,z)\cong\del_i\Hom(x,y).
    \]
  \end{enumerate}
\end{definition}
(In~\cite{CJS-gauge-floerhomotopy}, it is not required that the space
of objects be discrete. In Morse theory, this corresponds to allowing
Morse-Bott functions.)

The identity morphisms in a flow category are somewhat special, and it
is often convenient to ignore them. So, for objects $x,y$ in
$\FlowCat$, the \emph{moduli space from $x$ to $y$}, $\Moduli(x,y)$,
is defined to be $\Hom(x,y)$ if $x\neq y$, and empty if $x=y$. (In
Morse theory, this corresponds to the moduli space of non-constant
downwards gradient flows from $x$ to $y$.)

For any flow category $\FlowCat$, let $\Sigma^k\FlowCat$ denote the
flow category obtained by increasing the gradings of each object by
$k$.

\begin{definition}\label{def:flow-cat-neat-embed}
  For each integer $i$, fix an integer $D_i\geq 0$, and let $\TupV{D}$
  denote this sequence. A \emph{neat embedding} of a flow category
  $\FlowCat$ relative to $\TupV{D}$ is a collection $\jmath_{x,y}$ of
  neat embeddings (with the standard Riemannian metric on the target space)
\[
\jmath_{x,y}\from\Moduli(x,y)\into \R^{D_{\gr(y)}}\times\R_+\times\R^{D_{\gr(y)+1}}\times\R_+\times\dots\times\R_+\times\R^{D_{\gr(x)-1}}
\]
of $\Codim{\gr(x)-\gr(y)-1}$-manifolds for all $x,y\in\Ob(\FlowCat)$, subject to the following:
\begin{enumerate}
\item\label{item:CJS-neat-embed} For all integers $i,j$, 
  \[
  \!\!\!\coprod_{\substack{x,y\\\gr(x)=i,\gr(y)=j}}\!\!\!\jmath_{x,y}\from\!\!\!\coprod_{\substack{x,y\\\gr(x)=i,\gr(y)=j}}\!\!\!\Moduli(x,y)\to\R^{D_{j}}\times\R_+\times\dots\times\R_+\times\R^{D_{i-1}}
  \]
  is a neat embedding of $\Codim{i-j-1}$-manifolds.
\item\label{item:CJS-neat-embed-coherent} For all $x,y,z\in\Ob(\FlowCat)$, and all points $(q,p)\in\Moduli(y,z)\times\Moduli(x,y)$
  \[
  \jmath_{x,z}(q\circ p)=(\jmath_{y,z}(q),0,\jmath_{x,y}(p)).
  \]
\end{enumerate}
A \emph{coherent framing} for a neat embedding $\jmath$ is a collection
of framings of the normal bundles $\nu_{\jmath_{x,y}}$ of $\jmath_{x,y}$
for all $x,y\in\Ob(\FlowCat)$, such that for all
$x,y,z\in\Ob(\FlowCat)$, the product framing of
$\nu_{\jmath_{y,z}}\times\nu_{\jmath_{x,y}}$ equals the pullback framing
of $\circ^*\nu_{\jmath_{x,z}}$, where $\circ$ denotes composition.
\end{definition}

\begin{definition}[{\cite[Definition~3.21]{RS-khovanov}}]\label{def:flow-cat-chain-complex}
  A \emph{framed flow category} is a flow category
  $\FlowCat$, along with a coherent framing for some
  neat embedding of $\Cat$ (relative to some $\TupV{D}$).

  For a framed flow category, there is an \emph{associated cochain
    complex} $C^*(\FlowCat)$, defined as follows.  The $n\th$ chain
  group $C^n$ is the $\Z$-module freely generated by the objects of
  $\FlowCat$ of grading $n$. The differential $\diff$ is of degree
  one. For $x,y\in\Ob(\FlowCat)$ with $\gr(x)-\gr(y)=1$, the
  coefficient $\langle\diff y,x\rangle$ of $x$ in $\delta(y)$ is the number of points in
  $\Moduli(x,y)$, counted with sign.  We say a framed flow
  category \emph{refines} its associated chain complex.
\end{definition}

\begin{remark}
  In order to define the associated chain complex, one only needs
  the framing of the $0$-dimensional moduli spaces; and in order to
  check that $\delta^2=0$, one only needs to ensure
  that the framing extends to the $1$-dimensional moduli spaces.
\end{remark}

To a framed flow category, Cohen-Jones-Segal associate a based CW
complex $\Realize{\FlowCat}$ whose cells (except the basepoint)
correspond to the objects of the flow
category~\cite{CJS-gauge-floerhomotopy}. The following formulation of
the Cohen-Jones-Segal construction is described in more detail in
\cite[Definition~3.24]{RS-khovanov}.

\begin{definition}\label{def:CJS-realization}
  Let $\FlowCat$ be a framed flow category with a neat embedding
  $\jmath$ relative to some $\TupV{D}$, and assume all objects of
  $\FlowCat$ have grading in $[B,A]$ for some fixed $A,B\in\Z$. Using
  the framing of $\nu_{\jmath_{x,y}}$, extend $\jmath_{x,y}$ to a map
  \[
  \ol{\jmath}_{x,y}\from\Moduli(x,y)\times[-\delta,\delta]^{D_{\gr(y)}+\dots+D_{\gr(x)-1}}\to\R^{D_{\gr(y)}}\times\R_+\times\dots\times\R_+\times\R^{D_{\gr(x)-1}}.
  \]
  Choose $\delta$ small enough and $T$ large enough so that the map
  $
  \coprod_{x,y\mid \gr(x)=i,\gr(y)=j}\ol{\jmath}_{x,y}
  $
  is an embedding into
  $
  (-T,T)^{D_{j}}\times[0,T)\times\dots\times[0,T)\times(-T,T)^{D_{i-1}}
  $
  for all integers $i,j$. In the based CW complex $\Realize{\FlowCat}$, the
  cell associated to an object $x$ of grading $m$ is
  \[
  \Cell{x}=\prod_{i=B}^{m-1}\Bigl([0,T]\times[-T,T]^{D_i}\Bigr)\times[-\delta,\delta]^{D_m+\dots+D_{A-1}}.
  \]
  For any other object $y$ with $\gr(y)=n<m$, the embedding
  $\ol{\jmath}_{x,y}$ identifies $\Cell{y}\times\Moduli(x,y)$ with the
  following subset of $\bdy\Cell{x}$,
  \[
  \Cell[y]{x}=\prod_{i=B}^{n-1}\Bigl([0,T]\times[-T,T]^{D_i}\Bigr)\times\{0\}\times\image(\ol{\jmath}_{x,y})\times[-\delta,\delta]^{D_m+\dots+D_{A-1}}.
  \]
  The attaching map for $\Cell{x}$ sends
  $\Cell[y]{x}\cong\Cell{y}\times\Moduli(x,y)$ via the projection map
  to $\Cell{y}$, and sends $\bdy\Cell{x}\setminus\bigcup_y\Cell[y]{x}$
  to the basepoint.
\end{definition}

\begin{lemma}
  \cite[Lemma~3.25]{RS-khovanov} \Definition{CJS-realization} defines a CW
  complex, whose reduced cellular cochain complex is isomorphic (after
  shifting the gradings down by $D_B+\dots+D_{A-1}-B$) to the chain complex
  associated to the framed flow category $\FlowCat$ from
  \Definition{flow-cat-chain-complex}. The isomorphism sends cells
  of $\Realize{\FlowCat}$ to the corresponding objects of
  $\FlowCat$. 
\end{lemma}

\begin{definition}
  The \emph{Cohen-Jones-Segal realization of $\FlowCat$} is the formal
  (de)suspension
  $\Sigma^{B-D_B-\dots-D_{A-1}}\Realize{\FlowCat}$.
\end{definition}

\subsection{Permutohedra}\label{sec:permutohedra}
We will be interested in a particular family of flow categories, in
which the moduli spaces are unions of permutohedra. So, we recall some
basic facts about permutohedra.

Before starting, let us fix some notations about polytopes, mostly
following Ziegler~\cite{Zie-top-polytopes}. Let $P\subset\R^n$ be a polytope
or an $\mc{H}$-polyhedron (as described in
\cite[Definition~0.1]{Zie-top-polytopes}). If there is an affine
half-space of $\R^n$ which contains $P$, then the intersection of its
boundary with $P$ is called a \emph{face} of $P$; and if $\dim(P)=d$,
then we declare the entire polytope $P$ to be its unique
$d$-dimensional face.  This gives a CW complex structure on $P$, with
the cells being the nonempty faces. The faces of dimension $0$, $1$, and $(d-1)$
are called vertices, edges, and facets, respectively. A
$d$-dimensional polytope is called \emph{simple} if every vertex is
contained in exactly $d$ facets; simple polytopes are multifaceted
manifolds.

For $\si\in \permu{n}$ a permutation, let
$v_{\si}=(\si^{-1}(1),\dots,\si^{-1}(n))\in\R^n$. The \emph{$(n-1)$-dimensional permutohedron}
$\Permu{n-1}$ is the convex hull in $\R^n$ of the $n!$ points
$v_{\si}$~\cite[Example~0.10]{Zie-top-polytopes}. The permutohedron $\Permu{n-1}$ lies in the
affine subspace
$\mb{A}^{n-1}\defeq\Set{(x_1,\dots,x_n)\in\R^n}{\sum_ix_i=n(n+1)/2}$
of $\R^n$. As its name suggests,
 $\Permu{n-1}$ is $(n-1)$-dimensional, and the $v_\sigma$ are its vertices.

For each non-empty, proper subset $S$ of $\{1,\dots,n\}$ of cardinality
say $k$, let $H_S\sbs\mb{A}^{n-1}\sbs\R^n$ be the half-space
$\Set{(x_1,\dots,x_n)\in\mb{A}^{n-1}}{\sum_{i\in S}x_i\geq
  k(k+1)/2}$. The permutohedron $\Permu{n-1}$ can also be defined as
the intersection of the $2^n-2$ half-spaces $H_S$. In fact, the facets of $\Permu{n-1}$ are exactly the
$F_S\defeq\Permu{n-1}\cap\bdy H_S$.

The facets $F_S$ are identified with products of lower-dimensional
permutohedra:
  \begin{lemma}\label{lem:permu-facet}
    Let $a_1<a_2<\dots<a_{k}$ be the elements in $S$, and let
    $b_1<b_2<\dots<b_{n-k}$ be the elements in $\{1,2,\dots,n\}\sm
    S$. Then the map $f_S\co \R^n\to\R^{k}\times\R^{n-k}$
    \[
    f_S(x_1,\dots,x_n)=((x_{a_1},\dots,x_{a_k}),(x_{b_1}-k,\dots,x_{b_{n-k}}-k))
    \]
    identifies the facet $F_S\subset\R^n$ with
    $\Permu{k-1}\times\Permu{n-1-k}\subset\R^{k}\times\R^{n-k}$.
  \end{lemma}
  \begin{proof}
    It suffices to show that $f_S$ takes the vertices of $F_S\subset
    \Permu{n-1}$ to the vertices of $\Permu{k-1}\times
    \Permu{n-1-k}$. The vertices of $F_S$ are the points
    $(x_1,\dots,x_n)$ so that
    $\{x_{a_1},\dots,x_{a_k}\}=\{1,\dots,k\}$ and
    $\{x_{b_1},\dots,x_{b_{n-k}}\}=\{k+1,\dots,n\}$. It is
    immediate that $f_S$ takes these vertices bijectively to the
    vertices of $\Permu{k-1}\times\Permu{n-1-k}$.
  \end{proof}


  The permutohedron $\Permu{n-1}$ is simple, i.e., each vertex lies in
  exactly $n-1$ facets: $v_{\si}$ lies in the facet $F_{\{\si(1),\dots,\si(k)\}}$
  for each $1\leq k<n$, and no others. Therefore, every $d$-dimensional face belongs to
  exactly $n-1-d$ facets; and the subsets corresponding to those
  facets are nested. Hence, $d$-dimensional faces correspond to sequences of
  $n-1-d$ nested proper, non-empty subsets of $\{1,\dots,n\}$. Further:
  \begin{lemma}\label{lem:permu-face}
    The space $\Permu{n-1}$ can be treated as an
    $\Codim{n-1}$-manifold by declaring
    \[
    \bdy_i\Permu{n-1}=\bigcup_{\substack{S\\\card{S}=i}}F_S
    \]
    for $1\leq i\leq n-1$.
  \end{lemma}
  \begin{proof}
    We must check:
    \begin{enumerate}
    \item\label{item:permu-face-1} Every point $x$ belongs to $c(x)$ facets.
    \item\label{item:permu-face-2} Each $\bdy_i\Permu{n-1}$ is a
      multifacet, that is, a disjoint union of facets.
    \item\label{item:permu-face-3} $\bigcup_i\bdy_i\Permu{n-1}=\bdy\Permu{n-1}$.
    \item\label{item:permu-face-4} For each $i\neq j$,
      $\bdy_i\Permu{n-1}\cap \bdy_j\Permu{n-1}$ is a multifacet of
      $\bdy_i\Permu{n-1}$ (and $\bdy_j\Permu{n-1}$).
    \end{enumerate}
    
    Point~(\ref{item:permu-face-1}) follows from the fact that
    $\Permu{n-1}$ is a simple polyhedron.

    For point~(\ref{item:permu-face-2}), we claim that if $|S|=|T|=i$
    and $S\neq T$ then $F_S\cap F_T=\emptyset$; it follows that
    $\bdy_i\Permu{n-1}$ is the disjoint union of the facets $F_S$
    (with $|S|=i$).  But if $v_{\si}$ is a vertex in
    $F_S\cap F_T$, then 
    \begin{align*}
    &\phantom{\Longleftrightarrow}\sum_{j\in S} \si^{-1}(j)=\sum_{j\in T}
    \si^{-1}(j)=i(i+1)/2\\
    &\Longleftrightarrow\set{\si^{-1}(j)}{j\in
      S}=\set{\si^{-1}(j)}{j\in T}=\{1,\dots,i\}\\
    &\Longleftrightarrow S=T=\{\si(1),\dots,\si(i)\}.
    \end{align*}

    Point~(\ref{item:permu-face-3}) is immediate from the definitions.
    
    For point~(\ref{item:permu-face-4}), suppose that $|S|=i$. After
    identifying $F_S$ with $\Permu{i-1}\times\Permu{n-i-1}$ using
    \Lemma{permu-facet}, we get
    \[
    F_S\cap \bdy_j\Permu{n-1}=
    \begin{cases}
     \Permu{i-1}\times(\bdy_{j-i}\Permu{n-i-1}) & i<j\\
     (\bdy_j\Permu{i-1})\times \Permu{n-i-1}  & i>j.
    \end{cases}
    \]
    Therefore, $F_S\cap \bdy_j\Permu{n-1}$ is a disjoint union of
    facets of $F_S\cong\Permu{i-1}\times\Permu{n-i-1}$.  Since the
    $F_S$ for $|S|=i$ are disjoint, $\bdy_i\Permu{n-1}\cap
    \bdy_j\Permu{n-1}$ is a disjoint union of facets of
    $\bdy_i\Permu{n-1}$ as well.
  \end{proof}
  
  We will also use the following well-known cubical complex structure
  on $\Permu{n-1}$. For any permutation $\si\in\permu{n}$, let
  $C_{\si}$ be the convex hull of the barycenters of all the faces
  that contain the vertex $v_{\si}$.
\begin{lemma}\label{lem:permu-cubical-complex}
  Each $C_{\si}$ is combinatorially equivalent to an $(n-1)$-dimensional cube,
  and these cubes form a cubical complex subdivision of $\Permu{n-1}$.
\end{lemma}
\begin{proof}
  We present the proof from Ovchinnikov~\cite[Section~3]{Ovc-top-weak-order-complexes}, for which he cites Ziegler.
  Consider the following intersection of half-spaces
  \[
  \mathit{FC}_{\si}=\bigcap_{S\mid v_{\si}\in F_S} H_S.
  \]
  The space $\mathit{FC}_{\si}$ is an $(n-1)$-dimensional cone with cone point $v_{\si}$, with $n-1$ facets
  (corresponding to the facets of $\Permu{n-1}$ that contain
  $v_{\si}$), and therefore is a simplicial cone.  (For comparison
  with~\cite[Section~3]{Ovc-top-weak-order-complexes},
  $\mathit{FC}_{\si}$ is the dual of the facet cone of the dual of
  $v_{\si}$ (in the face fan of the dual of $\Permu{n-1}$).)

  Next, consider the vertex cone of $v_{\si}$ (in the normal fan of
  $\Permu{n-1}$), $\mathit{VC}_{\si}$. By definition, the cone point of $\mathit{VC}_{\si}$ is the
  barycenter of $\Permu{n-1}$, and the edges of $\mathit{VC}_{\si}$ are obtained by dropping
  perpendiculars from the barycenter to the facets of $\Permu{n-1}$
  that contain $v_{\si}$. The cone $\mathit{VC}_{\si}$ is an
  $(n-1)$-dimensional cone with $(n-1)$ edges, and therefore,
  $\mathit{VC}_{\si}$ is a simplicial cone. The $d$-dimensional faces
  of $\mathit{VC}_{\si}$ correspond to the $(n-1-d)$-dimensional faces
  of $\Permu{n-1}$ that contain $v_{\si}$. Given corresponding faces
  $f_{\mathit{VC}}$ of $\mathit{VC}_{\si}$ and $f$ of $\Permu{n-1}$,
  $f_{\mathit{VC}}$ is perpendicular to $f$ and passes through the
  barycenter of $f$.

  Therefore, $C_{\si}$ is the intersection of the two simplicial cones
  $\mathit{FC}_{\si}$ and $\mathit{VC}_{\si}$. Since the edges of
  $\mathit{VC}_{\si}$ pass through the interiors of the facets of
  $\mathit{FC}_{\si}$, $\mathit{VC}_{\si}$ and $\mathit{FC}_{\si}$
  intersect transversely, and therefore, $\mathit{VC}_{\si}\cap
  \mathit{FC}_{\si}$ is combinatorially equivalent to a cube.

  The facets of $C_{\si}$ are of two types: the ones contained in
  $\mathit{FC}_{\si}$, which are not identified with the facets of any
  other cube and lie in the boundary of $\Permu{n-1}$; and the ones
  contained in $\mathit{VC}_{\si}$, which are identified with facets
  of other cubes and lie in the interior of $\Permu{n-1}$. Indeed the
  facets of the latter type correspond to the edges $e$ of $\Permu{n-1}$
  that contain $v_{\si}$: the facet corresponding to $e$ is formed by taking the convex hull
  of the barycenters of all the faces of $\Permu{n-1}$ that contain
  $e$. With these identifications, it is clear that these cubes
  $C_{\si}$ come together to form a cubical subdivision of the
  permutohedron.
\end{proof}

\subsection{The cube flow category}\label{sec:cube-flow-cat}
We start by recalling the cube flow category
from~\cite[Definition~4.1]{RS-khovanov}. There
we gave a definition in terms of Morse flows on $[0,1]^{n}$. Here, we
give a more directly combinatorial definition in terms of
permutohedra (see also \Remark{cube-is-cube}):
\begin{definition}\label{def:cube-flow-cat}
  The objects of the cube flow category $\CubeFlowCat{n}$ are the same as
  the objects of the cube category $\CCat{n}$, i.e., tuples
  $u=(u_1,\dots,u_n)\in\{0,1\}^n$. The grading on the objects is
  defined by $\gr(u)=\card{u}=\sum_i u_i$. 
  
  The space $\Moduli(u,v)$ is defined to be empty unless $u> v$. If
  $u>v$ and $\gr(u)-\gr(v)=k>0$ then we define
  $\Moduli(u,v)=\Permu{k-1}$, the $(k-1)$-dimensional permutohedron. 
  Note that, by \Lemma{permu-face},
  $\Moduli(u,v)$ is a $\Codim{k-1}$-manifold.

  The composition map $\Moduli(v,w)\times \Moduli(u,v)\to
  \Moduli(u,w)$ is defined as follows. Assume that $u>v>w$,
  $\gr(u)-\gr(v)=k$ and $\gr(v)-\gr(w)=l$. Suppose that
  $u_{a_1}=\dots=u_{a_{k+l}}=1$ and $w_{a_1}=\dots=w_{a_{k+l}}=0$
  (i.e., the $a_i$ are the coordinates in which $u$ and $w$ differ),
  where $a_1<a_2<\dots<a_{k+l}$. Let $S$ be the subset of
  $\{1,\dots,k+l\}$ such that $v_{a_s}=1$ for $s\in S$. (The set $S$
  has cardinality $l$.)  By \Lemma{permu-facet}, there is a
  corresponding facet $F_S\subset \Permu{k+l-1}$, and $F_S$ is
  identified with $\Permu{l-1}\times\Permu{k-1}=\Moduli(v,w)\times
  \Moduli(u,v)$. The composition map is the corresponding inclusion
  map $\Moduli(v,w)\times\Moduli(u,v)=F_S\into \Moduli(u,w)$.
\end{definition}

\begin{lemma}\label{lem:cube-is-flow-cat} 
  \Definition{cube-flow-cat} defines a flow category.
\end{lemma}
\begin{proof}
  Conditions~\ref{item:FC-id} and~\ref{item:FC-mflds} of
  \Definition{flow-cat} are immediate from the definitions and
  \Lemma{permu-face}. For Condition~\ref{item:FC-3}, it is enough to
  recall from \Lemma{permu-face} that 
  $\bdy_l\Permu{k+l-1}=\bigcup_{\card{S}=l}F_S.$

  Finally, we need to check that this defines a category, or in other
  words, that composition is associative. Towards this end, for any
  $u>v$ with $\gr(u)-\gr(v)=k>0$, it is convenient to treat
  $\Moduli(u,v)=\Permu{k-1}$ as a subset of
  $ 
  \prod_{i\mid u_i>v_i}\R$ 
  instead of $\R^k$,
  where the two ambient spaces are identified by linearly ordering
  the coordinates $a_1<a_2<\dots<a_k$ in which $u$ and $v$
  differ. With this viewpoint, for $u>v>w$, with $\gr(u)-\gr(v)=k$
  and $\gr(v)-\gr(w)=l$, the composition map
  $\Permu{l-1}\times\Permu{k-1}\to\Permu{k+l-1}$ is induced from the map
  \[
    \prod_{i\mid v_i>w_i}\!\!\!\R\times\!\!\!\!\prod_{i\mid u_i>v_i}\!\!\!\R
    \ \stackrel{+(\vec{0},\vec{l})}{\relbar\joinrel\relbar\joinrel\longrightarrow}
    \prod_{i\mid v_i>w_i}\!\!\!\R\times\!\!\!\!\prod_{i\mid
      u_i>v_i}\!\!\!\R\ \stackrel{\cong}{\to}\prod_{i\mid u_i>w_i}\!\!\!\R,
    \]
  where the first map adds $l$ to each of the coordinates of
  $
  \prod_{i\mid u_i>v_i}\R$, and the second map is coordinate-wise
  identification. (See also \Lemma{permu-facet}.) Now, given
  $u>v>w>x$, with $\gr(u)-\gr(v)=k$, $\gr(v)-\gr(w)=l$ and
  $\gr(w)-\gr(x)=m$, the corresponding double compositions are:
  \[
  \xymatrixcolsep{4em}
  \xymatrix{
    \displaystyle\prod_{i\mid w_i>x_i}\!\!\!\RR\times\!\!\!\! \displaystyle\prod_{i\mid v_i>w_i}\!\!\!\RR\times\!\!\!\! \displaystyle\prod_{i\mid u_i>v_i}\!\!\!\RR
    \ar[r]^--{+(\vec{0},\vec{0},\vec{l})}\ar[d]^-{+(\vec{0},\vec{m},\vec{0})} & 
    \displaystyle\prod_{i\mid w_i>x_i}\!\!\!\RR\times\!\!\!\! \displaystyle\prod_{i\mid u_i>w_i}\!\!\!\RR\ar[d]^-{+(\vec{0},\vec{m)}} \\
    \displaystyle\prod_{i\mid v_i>x_i}\!\!\!\RR\times\!\!\!\! \displaystyle\prod_{i\mid u_i>v_i}\!\!\!\RR\ar[r]^-{+(\vec{0},\vec{l}+\vec{m})} & 
    \displaystyle\prod_{i\mid u_i>x_i}\!\!\!\RR,
  }
  \]
  where we have suppressed the reshuffling of factors from the notation.
  So, both compositions are given by
  $+(\vec{0},\vec{m},\vec{l}+\vec{m})$ (and the same reshuffling of
  factors).
\end{proof}

\begin{remark}\label{rem:cube-is-cube}
  We have not shown that the definition of the cube flow category from
  \Definition{cube-flow-cat} agrees with the Morse-theoretic
  definition from~\cite[Definition~4.1]{RS-khovanov}, as doing so
  would seem to require a nontrivial digression about the smooth
  structures on moduli spaces of Morse flows. For the purposes of this
  paper (and future work), note that the combinatorial definition used
  here works just as well for the construction of the
  Khovanov homotopy type in~\cite{RS-khovanov}, and all the results
  stated in \cite{RS-khovanov, RS-steenrod, RS-s-invariant}
  remains true with this combinatorial definition. To be more
  precise, the only statements that involve the moduli spaces on the
  cube flow category coming the Morse flows are
  \cite[Lemmas~4.2--4.3]{RS-khovanov}, which are immediate for the
  combinatorial definition. Therefore, when we talk about the Khovanov
  homotopy type in this paper (and future work), we mean the homotopy
  type defined using the cube flow category from
  \Definition{cube-flow-cat}.
\end{remark}

Before moving on to the definition of cubical flow categories
in \Section{cubical-flow-cat}, we digress a little to study the
cubical complex structure from \Lemma{permu-cubical-complex} on the
permutohedron $\Moduli(u,v)$.
\begin{definition}\label{def:cube-moduli-cubical-complex}
  For $u>v$ in $\{0,1\}^n$, define the space
  \[
  M_{u,v}=\Big(\coprod_{\substack{m\in\{1,\dots,|u|-|v|\}\\ u=u^0>\dots>u^m=v}}[0,1]^{m-1}\Big)/\sim
  \]
  where, for each chain $u=u^0>\dots>u^m=v$, each $1\leq i\leq m-1$,
  and each point $(t_1,\dots,t_{m-2})$ in the cube $[0,1]^{m-2}$
  corresponding to the chain $u^0>\dots>u^{i-1}>u^{i+1}>\dots>u^m$ the
  equivalence relation $\sim$ identifies
  \[
  (t_1,\dots,t_{m-2})_{u=u^0>\dots>u^{i-1}>u^{i+1}>\dots>u^m=v}\sim (t_1,\dots,t_{i-1},1,t_{i},\dots,t_{m-2})_{u=u^0>\dots>u^m=v}.
  \]

  For $u>v>w$ in $\{0,1\}^n$, define a map
  $M_{v,w}\times M_{u,v}\to M_{u,w}$ by:
  \begin{equation*}
    \begin{split}
      ((t_1,\dots,t_{\ell-1})_{v=v^0>\dots>v^{\ell}=w},&(s_1,\dots,s_{m-1})_{u=u^0>\dots>u^m=v})\\
      &\mapsto (s_1,\dots,s_{m-1},0,t_1,\dots,t_{\ell-1})_{u=u^0>\dots>u^m=v=v^0>\dots,v^{\ell}=w}.
    \end{split}
  \end{equation*}
\end{definition}

\begin{figure}
  \centering
  \begin{tikzpicture}[scale=2.5]
    \begin{scope}
      \foreach \i [count=\c from 0] in {60,120,...,360}{
        \coordinate (h\c) at (\i:1cm);}
      
      \node[anchor=south] at (h0) {$(1,2,3)$};
      \node[anchor=south] at (h1) {$(2,1,3)$};
      \node[anchor=east] at (h2) {$(2,3,1)$};
      \node[anchor=north] at (h3) {$(3,2,1)$};
      \node[anchor=north] at (h4) {$(3,1,2)$};
      \node[anchor=west] at (h5) {$(1,3,2)$};
      
      \draw[thick] (h0)--(h1)--(h2)--(h3)--(h4)--(h5)--cycle;
    \end{scope}
    \begin{scope}[xshift=3cm]
      \foreach \i [count=\c from 0] in {60,120,...,420}{
        \coordinate (h\c) at (\i:1cm);}
      
      \coordinate (z) at (0,0);

      \foreach \d [count=\c from 0] in {1,2,...,6}{
        \coordinate (m\c) at ($(h\c)!0.5!(h\d)$);
        \draw[thick] (h\c) -- (h\d);
        \draw[thick] (z) -- (m\c);
      }

      \node[fill=white,inner sep=0,outer sep=0,opacity=0.5,text opacity=1] at (z) {\tiny $(\ )$};
     
    \foreach \i/\a in {0/110,1/010,2/011,3/001,4/101,5/100}{
      \node[fill=white,inner sep=0,outer sep=0,opacity=0.5,text opacity=1] at ($(m\i)!0.5!(z)$) {\tiny $(\a)$};}

      \foreach \i/\a/\b in {0/110/100,1/110/010,2/011/010,3/011/001,4/101/001,5/101/100}{
      \node at ($(h\i)!0.3!(z)$) {\tiny $(\a,\b)$};}

    \end{scope}
  \end{tikzpicture}
  \caption{\textbf{A cubical decomposition of the permutohedron.} Left: The $2$-dimensional permutohedron $\Moduli(111,000)$ with vertices labeled by the corresponding permutations of $(1,2,3)$. Right: The corresponding cubical complex $M_{111,000}$, where we have labeled the cube corresponding to a sequence $111=u^0>\dots>u^m=000$ by $(u^1,\dots,u^{m-1})$.}
  \label{fig:tie-fighter}
\end{figure}

\begin{lemma}\label{lem:cube-moduli-cubical-complex}
  There are homeomorphisms $h_{u,v}\co
  M_{u,v}\stackrel{\cong}{\longrightarrow}\Moduli(u,v)$ 
  from the spaces of \Definition{cube-moduli-cubical-complex} to the
  permutohedra $\Moduli(u,v)$ so that:
  \begin{enumerate}
  \item\label{item:cube-coprod-same} The $h_{u,v}$ identify the cubes
    in the definition of $M_{u,v}$ with the cubes in the cubical
    complex structure on $\Moduli(u,v)$ from
    \Lemma{permu-cubical-complex} which are not entirely contained in
    $\bdy \Moduli(u,v)$.
  \item\label{item:cube-coprod-bdy} The $h_{u,v}$ identify the points in the
    cubes for $M_{u,v}$ where some coordinate is $0$ with the points in
    $\bdy\Moduli(u,v)$.
  \item\label{item:cube-coprod-commute} For any $u>v>w$, the
    following diagram commutes:
    \[
    \xymatrixcolsep{4em}
    \xymatrix{
      M_{v,w}\times M_{u,v}\ar[r]^-{\cong}_-{h_{v,w}\times h_{u,v}}\ar[d]&\Moduli(v,w)\times\Moduli(u,v)\ar[d]\\
      M_{u,w}\ar[r]^-{\cong}_-{h_{u,w}}&\Moduli(u,w).
    }
    \]
    Here, the left vertical
    arrow is the map from \Definition{cube-moduli-cubical-complex},
    and the right vertical arrow is the composition in
    $\CubeFlowCat{n}$.
  \end{enumerate}
\end{lemma}

(See Figure~\ref{fig:tie-fighter}.)

\begin{proof}
  The chain $c=\{u=u^0>\dots >u^m=v\}$ in $\{0,1\}^n$ corresponds to
  the $(|u|-|v|-m)$-dimensional face
  $F_c=\Moduli(u^{m-1},u^m)\times\dots\times\Moduli(u^0,u^1)$ of the
  permutohedron $\Moduli(u,v)$. If we take the convex hull of the
  barycenters of all the faces of the permutohedron that contain
  $F_c$, we get an $(m-1)$-dimensional cube $C_c$ which appears in the
  cubical complex structure from \Lemma{permu-cubical-complex}.
  (If $F_c$ is a vertex of the permutohedron, or equivalently if $c$ is a maximal chain, then the cube $C_c$ is one of the $C_\sigma$ from \Lemma{permu-cubical-complex}.)
  We will identify $C_c$ with the cube
  $[0,1]^{m-1}$ corresponding to $c$ in $M_{u,v}$.

  Let $t_1,\dots,t_{m-1}$ be the coordinates of $[0,1]^{m-1}$. As a
  first step, we identify the vertices of $C_c$ and $[0,1]^{m-1}$. A
  vertex of $C_c$ corresponds to a barycenter of some face containing
  $F_c$, which in turn corresponds to some subchain $c'$ of $c$; the
  corresponding vertex in $[0,1]^{m-1}$ has $t_i=0$ if $u^i\in c'$,
  and has $t_i=1$ otherwise. The identification on the vertices leads
  to our desired identification as follows. Construct the simplicial
  complex subdivision of $C_c$ (respectively, $[0,1]^{m-1}$) by
  joining every face to the barycenter of $\Moduli(u,v)$
  (respectively, the vertex $\vec{1}\in[0,1]^{m-1}$), and extend the
  identification linearly over each simplex.

  It is fairly straightforward to check that such identifications
  induce a cubical complex homeomorphism between the cubical complex
  $M_{u,v}$ and the cubical complex structure on $\Moduli(u,v)$, and
  these homeomorphisms satisfy the conditions of the lemma. Further
  details are left to the reader.
\end{proof}

\subsection{Definition of a cubical flow category}\label{sec:cubical-flow-cat}

\begin{definition}\label{def:cubical}
  A \emph{cubical flow category} is a flow category $\FlowCat$
  equipped with a grading-preserving functor $\Funky\co
  \Sigma^k\FlowCat\to \CubeFlowCat{n}$ for some $k\in\ZZ,n\in\NN$ so
  that for each $x,y\in\Ob(\FlowCat)$, $\Funky\co \Moduli(x,y)\to
  \Moduli(\Funky(x),\Funky(y))$ is a (trivial) covering
  map.
\end{definition}

Thus, if $\Funky\co \FlowCat\to\CubeFlowCat{n}$ is a cubical flow
category and $x,y\in\Ob(\FlowCat)$ then $\Hom(x,y)$ is empty unless
$\Funky(x)\geq \Funky(y)$. If $x=y$, then $\Hom(x,y)=\{\Id\}$, and if
$\Funky(x)=\Funky(y)$ but $x\neq y$ then $\Hom(x,y)$ is empty. If
$\Funky(x)>\Funky(y)$, then the moduli space
$\Moduli_{\FlowCat}(x,y)=\Hom(x,y)$ is a (possibly empty) disjoint
union of permutohedra.

\begin{convention}
  Sometimes we suppress the grading information if it is inessential
  to the discussion, and drop the grading shift $\Sigma^k$ from the
  notation.
\end{convention}

A framing of the cube flow category $\CubeFlowCat{n}$ in the sense
of \Definition{flow-cat-chain-complex} induces a sign assignment $s$
on the cube (see \Section{basic-notation}) as follows: for $u>v$ with $\card{u}-\card{v}=1$,
$s_{u,v}=0$ if the point $\Moduli(u,v)$ is framed positively, and
$s_{u,v}=1$ otherwise. Furthermore, every sign assignment on the cube
is induced from an essentially unique framing of $\CubeFlowCat{n}$;
see~\cite[Section~4.2]{RS-khovanov}. The chain
complex associated to $\CubeFlowCat{n}$, framed according to some sign
assignment $s$, is defined as follows. The chain group is generated by the vertices of the cube $\{0,1\}^n$, and differential is given by
\[
\diff(v)=\sum_{\substack{u>v\\\card{u}-\card{v}=1}}(-1)^{s_{u,v}}u.
\]
This is an acyclic chain complex, and is often referred to as the \emph{cube chain complex}.

Furthermore, if $(\FlowCat,\Funky\from\FlowCat\to\CubeFlowCat{n})$ is
a cubical flow category, then any sign assignment $s$ on the cube
induces a framing of the $0$-dimensional moduli spaces in $\FlowCat$
as well: all the points in $\Moduli_{\FlowCat}(x,y)$ are framed
positively if $s_{\Funky(x),\Funky(y)}=0$, and otherwise, all the points
in $\Moduli_{\FlowCat}(x,y)$ are framed negatively. The pullback of
the (essentially unique) framing on $\CubeFlowCat{n}$ inducing $s$
produces an essentially canonical extension of this framing to the
entire cubical flow category $\FlowCat$. The chain complex associated
to $\FlowCat$ for this framing has the following differential: for
$x,y\in\Ob(\FlowCat)$ with $\gr(x)-\gr(y)=1$, the coefficient of $x$
in $\diff y$ is
\[ 
\langle\diff(y),x\rangle=
\begin{cases}
  (-1)^{s_{\Funky(x),\Funky(y)}}&\text{$\Funky(x)>\Funky(y)$}\\
  0&\text{otherwise.}
\end{cases} 
\]
This chain complex only depends on the sign assignment $s$
and not the entire framing of $\CubeFlowCat{n}$.

Even though the definition of cubical flow categories seems fairly
restrictive, there are many examples:
\begin{example}\label{exam:simplicial-cubical}
  Given any (finite) simplicial complex $S_\bullet$ with vertices
  $v_1,\dots,v_n$, there is a corresponding cubical flow category
  $(\FlowCat,\Funky\co \Sigma\FlowCat\to \CubeFlowCat{n})$, defined as
  follows. The objects of $\FlowCat$ correspond to the simplices of
  $S_\bullet$, which in turn can be viewed as non-empty subsets of
  $\{v_1,\dots,v_n\}$. Given an object in $\FlowCat$, corresponding to
  a subset $T\subseteq \{v_1,\dots,v_n\}$, define $\Funky(T)$ to be
  the vector in $\{0,1\}^n$ whose $i\th$ coordinate is $1$ if and only
  if $v_i\in T$. Note that the map $\Funky$ is injective on
  objects. Let $\FlowCat$ be the full subcategory of $\CubeFlowCat{n}$
  generated by the objects in the image of $\Funky$. The chain complex
  associated to $\FlowCat$ is isomorphic to the simplicial chain
  complex for $S_\bullet$.

  (One can think of the category $\FlowCat$ as coming from choosing a
  Morse function on each simplex in $S_\bullet$ with a unique interior
  maximum and no other interior critical points, and so that these
  Morse functions are compatible under restriction. The moduli spaces
  in $\FlowCat$ are then the corresponding Morse moduli spaces.)
\end{example}

\begin{example}\label{exam:KhFlowCat}
  The Khovanov flow category~\cite[Definition~5.3]{RS-khovanov}
  $\KhFlowCat(K)$ associated to a link diagram $K$ (with $n$ crossings
  $c_1,\dots,c_n$) is, by construction, a cubical flow category.  For
  any $v\in\{0,1\}^n$, the subset of $\Ob(\KhFlowCat(K))$ that maps to
  $v$ are precisely the Khovanov generators over $v$ (as defined
  in \Section{khovanov-basic}):
  \[
  \Funky^{-1}(v)=F(v).
  \]
  For any $u>v\in\{0,1\}^n$ with $|u|-|v|=1$, and any
  $x\in\Funky^{-1}(u)=F(u)$, $y\in\Funky^{-1}(v)=F(v)$, the moduli
  space is
  \[
  \Moduli_{\KhFlowCat(K)}(x,y)=\begin{cases}\{\pt\}&\text{if $x$ appears in $\diffKh(y)$,}\\\emptyset&\text{otherwise.}\end{cases}
  \]
  Therefore, the chain complex associated to $\KhFlowCat(K)$ is isomorphic to the Khovanov chain complex $\KhCx(K)$ (from \Definition{KhCx-diffKh}).
\end{example}

\subsection{Cubical neat embeddings}\label{sec:cubical-neat}
Consider the cube flow category $\CubeFlowCat{n}$, and fix a tuple
$\TupV{d}=(d_0,\dots,d_{n-1})\in\NN^{n}$ and a real number $R>0$. For any $u>v$ in
$\Ob(\CubeFlowCat{n})=\{0,1\}^n$, let
\[
E_{u,v}=\Biggl[\prod_{i=|v|}^{|u|-1}(-R,R)^{d_i}\Biggr]\times
\Moduli_{\CubeFlowCat{n}}(u,v).
\]
Equip $E_{u,v}$ with the Riemannian metric induced from the standard
metric on the Euclidean space after viewing the permutohedron
$\Moduli_{\CubeFlowCat{n}}(u,v)$ as a polyhedron in
$\R^{|u|-|v|}$. For any $u>v>w$ in $\Ob(\CubeFlowCat{n})$, there is a
map $E_{v,w}\times E_{u,v}\to E_{u,w}$ given by:
\begin{align*}
  E_{v,w}\times
  E_{u,v}&=\prod_{i=|w|}^{|v|-1}(-R,R)^{d_i}\times
  \Moduli_{\CubeFlowCat{n}}(v,w)\times
  \prod_{i=|v|}^{|u|-1}(-R,R)^{d_i}\times
  \Moduli_{\CubeFlowCat{n}}(u,v)\\
  &\cong\prod_{i=|w|}^{|u|-1}(-R,R)^{d_i}\times
  \Moduli_{\CubeFlowCat{n}}(v,w)\times
  \Moduli_{\CubeFlowCat{n}}(u,v)\\
  &\stackrel{\Id\times\circ}{\lhook\joinrel\relbar\joinrel\relbar\joinrel\rightarrow} \prod_{i=|w|}^{|u|-1}(-R,R)^{d_i}\times
  \Moduli_{\CubeFlowCat{n}}(u,w).
\end{align*}

\begin{definition}\label{def:cube-cat-neat-embed}
  Fix a cubical flow category
  $(\FlowCat,\Funky\from\Sigma^k\FlowCat\to\CubeFlowCat{n})$.  A
  \emph{cubical neat embedding} $\iota$ of $(\FlowCat,\Funky)$ (or, more
  succinctly, of $\FlowCat$) relative to a tuple
  $\TupV{d}=(d_0,\dots,d_{n-1})\in\NN^n$ consists of neat embeddings
  \[
  \iota_{x,y}\from \Moduli_{\FlowCat}(x,y)\into E_{\Funky(x),\Funky(y)},
  \]
  such that:
  \begin{enumerate}
  \item\label{item:trivial-covering} For each $x,y\in\Ob(\FlowCat)$, the following diagram commutes:
    \[
    \xymatrix{
      \Moduli_{\FlowCat}(x,y)\ar[r]^{\iota_{x,y}}\ar[dr]_{\Funky} & E_{\Funky(x),\Funky(y)}\ar[d]^{\text{projection}}\\
      & \Moduli_{\CubeFlowCat{n}}(\Funky(x),\Funky(y)).
    }
    \]
  \item\label{item:induced-coprod-map} For each $u,v\in\Ob(\CubeFlowCat{n})$, the induced map
    \[
    \!\!\coprod_{\substack{x,y\\\Funky(x)=u,\Funky(y)=v}}\!\!\iota_{x,y}\from
    \!\!\coprod_{\substack{x,y\\\Funky(x)=u,\Funky(y)=v}}\!\!\Moduli_{\FlowCat}(x,y)\to E_{u,v}
    \]
    is a neat embedding.
  \item\label{item:neat-embed-commute} For each $x,y,z\in\Ob(\FlowCat)$, the following diagram commutes:
    \[
    \xymatrix{
      \Moduli_{\FlowCat}(y,z)\times\Moduli_{\FlowCat}(x,y)\ar[r]\ar[d]&
      \Moduli_{\FlowCat}(x,z)\ar[d]\\
      E_{\Funky(y),\Funky(z)}\times E_{\Funky(x),\Funky(y)}\ar[r]& E_{\Funky(x),\Funky(z)}.
    }
    \]
  \end{enumerate}
\end{definition}

In order to construct the cubical realization, we need to
extend these embeddings $\iota_{x,y}$ to maps 
\[
\ol{\iota}_{x,y}\from \Biggl[\prod_{i=|\Funky(y)|}^{|\Funky(x)|-1}[-\ep,\ep]^{d_i}\Biggr]\times\Moduli_{\FlowCat}(x,y)\to E_{\Funky(x),\Funky(y)},
\]
for some $\epsilon>0$,
so that that the analogue of the diagram from
Condition~(\ref{item:neat-embed-commute}) still commutes, and the extension of the
map from Condition~(\ref{item:induced-coprod-map}) is
still an embedding. One way to choose such a family of extensions would be to
coherently frame the normal bundles of the embeddings
$\iota_{x,y}$ (in a similar sense to \Definition{flow-cat-neat-embed})
and then use the construction from
\Definition{CJS-realization}. Instead, we will use the following
explicit extension.

For any $x,y\in\Ob(\FlowCat)$, let $u$ and $v$ denote $\Funky(x)$ and
$\Funky(y)$, respectively, and let $\pi^R_{u,v}$ and $\pi^M_{u,v}$
denote the projections of
$\Biggl[\displaystyle\prod_{i=|v|}^{|u|-1}(-R,R)^{d_i}\Biggr]\times\Moduli_{\CubeFlowCat{n}}(u,v)$
onto the two factors. Given sufficiently small $\ep>0$, extend the
embedding $\iota_{x,y}$ to a map
\begin{equation}\label{eq:extend-iota}
  \begin{split}
    \ol{\iota}_{x,y}\co \Biggl[\prod_{i=|v|}^{|u|-1}[-\ep,\ep]^{d_i}\Biggr]\times\Moduli_{\FlowCat}(x,y)&\to E_{u,v}=\Biggl[\prod_{i=|v|}^{|u|-1}(-R,R)^{d_i}\Biggr]\times
    \Moduli_{\CubeFlowCat{n}}(u,v)\\
    (a,\gamma)&\stackrel{\ol{\iota}_{x,y}}{\mapsto}
    (a+\pi^R_{u,v}\iota_{x,y}(\gamma),\pi^M_{u,v}\iota_{x,y}(\gamma)).
  \end{split}
\end{equation}
The definition of $\ol\iota$ ensures that the analogue of the diagram
from Condition~(\ref{item:neat-embed-commute}) still commutes. By condition~\ref{item:trivial-covering}, $\iota(\Moduli_{\FlowCat}(x,y)$ is transverse to the fibers $\left[\prod_{i=|v|}^{|u|-1}(-R,R)^{d_i}\right]\times\{\gamma\}$. So, for
$\ep$ sufficiently small, the extension of the map from
Condition~(\ref{item:induced-coprod-map}) is still an embedding. We make this a requirement on $\epsilon$:
\begin{convention}\label{conv:eps-small}
  When talking about extensions $\ol{\iota}_{x,y}$ of cubical neat
  embeddings, we will always assume that $\epsilon$ is chosen to be
  small in the sense that the induced map 
    \[
    \!\!\coprod_{\substack{x,y\\\Funky(x)=u,\Funky(y)=v}}\!\!\ol{\iota}_{x,y}\from
    \!\!\coprod_{\substack{x,y\\\Funky(x)=u,\Funky(y)=v}}\Biggl[\prod_{i=|v|}^{|u|-1}[-\ep,\ep]^{d_i}\Biggr]\times\Moduli_{\FlowCat}(x,y)\to E_{u,v}
    \]
    is an embedding.
\end{convention}

\begin{remark}\label{rem:explicit-isotopic-implicit}
  In~\cite{RS-khovanov}, we used a framing of the normal bundles,
  rather than the kind of explicit extension above, to trivialize
  tubular neighborhoods.  When identifying the cubical realization
  with the Cohen-Jones-Segal realization
  (see \Section{cubical-CJS-same}), we will need an isotopy between
  these two trivializations.  

  Let
  $V_0 = \prod_{i=|v|}^{|u|-1}\RR^{d_i}\times \{0\}\subset
  (TE_{u,v})|_{\iota_{x,y}(\Moduli_{\FlowCat}(x,y))}$
  and let $V_1$ 
  be the normal bundle to $\iota_{x,y}\Moduli_{\FlowCat}(x,y)$. 
  %
  Since $\pi^M_{u,v}\circ\iota_{x,y}$ is a covering map, the
  projection $d\pi^R_{u,v}\co V_1\to V_0$ is an isomorphism. For $t\in
  [0,1]$ let $\pi_t\co V_1\to
  (TE_{u,v})|_{\iota_{x,y}(\Moduli_{\FlowCat}(x,y))}$ 
  be $\pi_t = t\Id + (1-t)d\pi^R_{u,v}$
  and let $V_t=\pi_t(V_1)$. The $V_t$ are a 1-parameter family of subbundles
  connecting $V_0$ to $V_1$, and each $V_t$ is a complement to
  $T(\iota_{x,y}\Moduli_{\FlowCat}(x,y))$.

  The bundle $V_0$ is a trivial bundle, and in particular framed.
  Pushing forward this framing by $\pi_t$ gives a framing of each
  $V_t$.  Exponentiating these framings gives a $1$-parameter family
  of extensions $\ol{\iota}_{x,y}^t$ of $\iota_{x,y}$, each satisfying the analogue of
  Condition~(\ref{item:neat-embed-commute}). The framing
  $\ol{\iota}_{x,y}^0$ is our explicit extension $\ol{\iota}_{x,y}$
  from Equation~\eqref{eq:extend-iota}, and $\ol{\iota}_{x,y}^1$ is an
  extension coming from coherently framing the normal bundles of
  $\iota_{x,y}$.  Since each $V_t$ is a complement to
  $T(\iota_{x,y}\Moduli_{\FlowCat}(x,y))$, each $\ol{\iota}_{x,y}^t$
  satisfies the analog of Condition~(\ref{item:induced-coprod-map})
  for some $\ep_t$; compactness allows us to choose a uniform $\ep$
  for which each $\ol{\iota}_{x,y}^t$ satisfies the analog of
  Condition~(\ref{item:induced-coprod-map}). This produces the
  required isotopy between the extension from
  Equation~\eqref{eq:extend-iota} and an extension coming from some
  framing of the normal bundle.
\end{remark}

\subsection{Cubical realization}\label{sec:cubical-realize}
\begin{definition}\label{def:cubical-realize}
  Fix a cubical neat embedding $\iota$ of a cubical flow category
  $(\FlowCat,\Funky\co \Sigma^k\FlowCat\to \CubeFlowCat{n})$, relative
  to a tuple $\TupV{d}=(d_0,\dots,d_{n-1})\in\NN^d$, and fix $\epsilon>0$ satisfying \Convention{eps-small}. We build a CW complex
  $\CRealize{\FlowCat}={\CRealize{\FlowCat}}_{\Funky,\iota}$ from this
  data as follows:
  \begin{itemize}
  \item The CW complex has one cell for each
    $x\in\Ob(\FlowCat)$. Letting $u$ denote $\Funky(x)$, this cell is
    given by
    \[
    \Cell{x} = \prod_{i=0}^{|u|-1}[-R,R]^{d_i}\times\prod_{i=|u|}^{n-1}[-\ep,\ep]^{d_i}\times\Precone\Moduli_{\CubeFlowCat{n}}(u,\vect{0}),
    \]
    where $\Precone\Moduli_{\CubeFlowCat{n}}(u,\vect{0})$ is defined
    to be $[0,1]\times \Moduli_{\CubeFlowCat{n}}(u,\vect{0})$ if
    $u\neq\vec{0}$, or the point $\{0\}$ if $u=\vect{0}$.
  \item For any $x,y\in\Ob(\FlowCat)$ with $\Funky(x)=u>\Funky(y)=v$, the cubical neat
    embedding $\iota$ furnishes an embedding
    \begin{align*}
      \Cell{y}&\times \Moduli_{\FlowCat}(x,y)\\
      &=
      \prod_{i=0}^{|v|-1}[-R,R]^{d_i}\times\prod_{i=|v|}^{n-1}[-\ep,\ep]^{d_i}
      \times \Precone\Moduli_{\CubeFlowCat{n}}(v,\vect{0})\times
      \Moduli_{\FlowCat}(x,y)\\
      &\cong\prod_{i=0}^{|v|-1}[-R,R]^{d_i}\times\prod_{i=|u|}^{n-1}[-\ep,\ep]^{d_i}
      \times \Precone\Moduli_{\CubeFlowCat{n}}(v,\vect{0})\times
      \Bigl(\prod_{i=|v|}^{|u|-1}[-\ep,\ep]^{d_i}\times
      \Moduli_{\FlowCat}(x,y)\Bigr)\\
      &\stackrel{\Id\times\ol{\iota}_{x,y}}{\lhook\joinrel\relbar\joinrel\relbar\joinrel\rightarrow}\prod_{i=0}^{|v|-1}[-R,R]^{d_i}\times\prod_{i=|u|}^{n-1}[-\ep,\ep]^{d_i}\times\Precone\Moduli_{\CubeFlowCat{n}}(v,\vect{0})\times\Bigl(\prod_{i=|v|}^{|u|-1}[-R,R]^{d_i}\times\Moduli_{\CubeFlowCat{n}}(u,v)\Bigr)\\
      &\cong\prod_{i=0}^{|u|-1}[-R,R]^{d_i}\times\prod_{i=|u|}^{n-1}[-\ep,\ep]^{d_i}\times\Precone\Moduli_{\CubeFlowCat{n}}(v,\vect{0})\times\Moduli_{\CubeFlowCat{n}}(u,v)\\
      &\into\prod_{i=0}^{|u|-1}[-R,R]^{d_i}\times\prod_{i=|u|}^{n-1}[-\ep,\ep]^{d_i}\times\bdy(\Precone\Moduli_{\CubeFlowCat{n}}(u,\vect{0}))\\
      &\subset\bdy\Cell{x}.
    \end{align*}
    Here, the second inclusion comes from the composition map if
    $v\neq\vect{0}$, or the inclusion $\{0\}\into[0,1]$ if
    $v=\vect{0}$.  Let $\Cell[y]{x}\subset \bdy\Cell{x}$ denote the
    image of this embedding.
  \item The attaching map for $\Cell{x}$ sends
    $\Cell[y]{x}\cong\Cell{y}\times \Moduli_{\FlowCat}(x,y)$ by the
    projection map to $\Cell{y}$ and sends the complement of
    $\cup_y\Cell[y]{x}$ in $\bdy\Cell{x}$ to the basepoint.
  \end{itemize}
  The \emph{cubical realization} of $(\FlowCat,\Funky)$ is defined to be the formal desuspension
  $\FlowCatSpace{\FlowCat}=\Sigma^{-(k+d_0+\cdots+d_{n-1})}\CRealize{\FlowCat}$.
  (The desuspension ensures that the gradings of the objects in $\FlowCat$ agree
  with the dimensions of the corresponding cells in
  $\FlowCatSpace{\FlowCat}$.)
\end{definition}

\begin{lemma}
  The attaching maps in the cubical realization are well-defined.
\end{lemma}
\begin{proof}
  As in the proof of
  \cite[Lemma~3.25]{RS-khovanov},
  we must show that for any $x,y,z\in\Ob(\FlowCat)$ with $\gr(x)> \gr(y)>\gr(z)$,
  the dashed arrow in the following diagram exists
  such that the diagram commutes.
  \[
  \xymatrix{
    \Cell[z]{x}\cap \Cell[y]{x} \ar@{^{(}->}[r]\ar@{^{(}->}[d] \ar@{-->}[ddrr]  & \Cell[y]{x} \ar[r] & \Cell{y}\\
    \Cell[z]{x} \ar[d]& & \del\Cell{y}\ar@{^{(}->}[u]\\
    \Cell{z} & & \Cell[z]{y} \ar[ll] \ar@{^{(}->}[u]
  }
  \]

  Let $u,v,w$ denote $\Funky(x),\Funky(y),\Funky(z)$, respectively. We
  may assume $u>v>w$; otherwise, it is not hard to verify that
  $\Cell[z]{x}$ and $\Cell[y]{x}$ are disjoint. In a similar vein to
  \Definition{cubical-realize}, let $\Cell[y,z]{x}\subset\bdy\Cell{x}$
  be the image of the following embedding:
  \begin{align*}
    \Cell{z}&\times\Moduli_{\FlowCat}(y,z)\times \Moduli_{\FlowCat}(x,y)\\
    &=
    \prod_{i=0}^{|w|-1}[-R,R]^{d_i}\times\prod_{i=|w|}^{n-1}[-\ep,\ep]^{d_i}
    \times \Precone\Moduli_{\CubeFlowCat{n}}(w,\vect{0})\times
    \Moduli_{\FlowCat}(y,z)\times\Moduli_{\FlowCat}(x,y)\\
    &\cong\prod_{i=0}^{|w|-1}[-R,R]^{d_i}\times\prod_{i=|u|}^{n-1}[-\ep,\ep]^{d_i}
    \times \Precone\Moduli_{\CubeFlowCat{n}}(w,\vect{0})\times
    \Bigl(\prod_{i=|w|}^{|v|-1}[-\ep,\ep]^{d_i}\times
    \Moduli_{\FlowCat}(y,z)\Bigr)\\
    &\qquad\qquad{}\times
    \Bigl(\prod_{i=|v|}^{|u|-1}[-\ep,\ep]^{d_i}\times
    \Moduli_{\FlowCat}(x,y)\Bigr)\\
    &\stackrel{\Id\times\ol{\iota}_{y,z}\times\ol{\iota}_{x,y}}{\lhook\joinrel\relbar\joinrel\relbar\joinrel\relbar\joinrel\relbar\joinrel\relbar\joinrel\rightarrow}\prod_{i=0}^{|w|-1}[-R,R]^{d_i}\times\prod_{i=|u|}^{n-1}[-\ep,\ep]^{d_i}\times\Precone\Moduli_{\CubeFlowCat{n}}(w,\vect{0})\\
    &\qquad\qquad\qquad{}\times\Bigl(\prod_{i=|w|}^{|v|-1}[-R,R]^{d_i}\times\Moduli_{\CubeFlowCat{n}}(v,w)\Bigr)\times\Bigl(\prod_{i=|v|}^{|u|-1}[-R,R]^{d_i}\times\Moduli_{\CubeFlowCat{n}}(u,v)\Bigr)\\
    &\cong\prod_{i=0}^{|u|-1}[-R,R]^{d_i}\times\prod_{i=|u|}^{n-1}[-\ep,\ep]^{d_i}\times\Precone\Moduli_{\CubeFlowCat{n}}(w,\vect{0})\times\Moduli_{\CubeFlowCat{n}}(v,w)\times\Moduli_{\CubeFlowCat{n}}(u,v)\\
    &\into\prod_{i=0}^{|u|-1}[-R,R]^{d_i}\times\prod_{i=|u|}^{n-1}[-\ep,\ep]^{d_i}\times\bdy(\Precone\Moduli_{\CubeFlowCat{n}}(u,\vect{0}))\\
    &\subset\bdy\Cell{x}.
  \end{align*}
  (As in \Definition{cubical-realize}, the second inclusion usually
  comes from the composition map; the only special case is if
  $w=\vect{0}$, when it comes partly from the inclusion
  $\{0\}\into[0,1]$.)

  We claim that $\Cell[y,z]{x}=\Cell[z]{x}\cap\Cell[y]{x}$. The
  direction $\Cell[y,z]{x}\subseteq\Cell[z]{x}\cap\Cell[y]{x}$ is
  immediate since the inclusion
  \[
  \Cell{z}\times\Moduli_{\FlowCat}(y,z)\times
  \Moduli_{\FlowCat}(x,y)\stackrel{\cong}{\longrightarrow}\Cell[y,z]{x}\subset\bdy\Cell{x}
  \]
  factors in both of the following ways (using Condition~(\ref{item:neat-embed-commute}) of \Definition{cube-cat-neat-embed}):
  \begin{align*}
    \Cell{z}\times\Moduli_{\FlowCat}(y,z)\times\Moduli_{\FlowCat}(x,y)&\stackrel{\Id\times\circ}{\relbar\joinrel\longrightarrow}\Cell{z}\times\Moduli_{\FlowCat}(x,z)\stackrel{\cong}{\longrightarrow}\Cell[z]{x}\subset\bdy\Cell{x}\\
    \Cell{z}\times\Moduli_{\FlowCat}(y,z)\times\Moduli_{\FlowCat}(x,y)&\stackrel{\cong}{\longrightarrow}\Cell[z]{y}\times\Moduli_{\FlowCat}(x,y)\into\Cell{y}\times\Moduli_{\FlowCat}(x,y)\stackrel{\cong}{\longrightarrow}\Cell[y]{x}\subset\bdy\Cell{x}.
  \end{align*}

  The other direction requires more work. We
  will abuse notation slightly and identify points with their images under
  the composition map in $\CubeFlowCat{n}$. View any point
  $p\in\Cell[z]{x}\cap\Cell[y]{x}$ as a point $(p_1,p_2,p_3,p_4)$ in
  the following subspace of $\bdy\Cell{x}$:
  \[
  \Bigl(\prod_{i=0}^{|w|-1}[-R,R]^{d_i}\times\prod_{i=|u|}^{n-1}[-\ep,\ep]^{d_i}\Bigr)
  \times\prod_{i=|w|}^{|v|-1}[-R,R]^{d_i} \times\prod_{i=|v|}^{|u|-1}[-R,R]^{d_i}\times\bdy(\Precone\Moduli_{\CubeFlowCat{n}}(u,\vect{0})).
  \]
  For $p$ to lie in $\Cell[z]{x}$, $p_4$ must lie in the
  subspace
  \[
  \Precone\Moduli_{\CubeFlowCat{n}}(w,\vect{0})\times
  \Moduli_{\CubeFlowCat{n}}(u,w);
  \]
  similarly, for $p$ to lie in $\Cell[y]{x}$, $p_4$ must lie in
  the subspace
  \[
  \Precone\Moduli_{\CubeFlowCat{n}}(v,\vect{0})\times
  \Moduli_{\CubeFlowCat{n}}(u,v).
  \]
  Therefore, $p_4$ must lie in the subspace
  $\Precone\Moduli_{\CubeFlowCat{n}}(w,\vect{0})\times\Moduli_{\CubeFlowCat{n}}(v,w)\times
  \Moduli_{\CubeFlowCat{n}}(u,v)$.
  Write
  $p_4$ also in component form as $(p_{4,1},p_{4,2},p_{4,3})$. Since $p$
  lies in $\Cell[y]{x}$, we know $(p_3,p_{4,3})\in\image(\ol{\iota}_{x,y})$, and
  since $p$ lies in $\Cell[z]{x}$, we know
  $(p_2,p_3,p_{4,2},p_{4,3})\in\image(\ol{\iota}_{x,z})$. Moreover, since $\iota$ is a
  cubical neat embedding (\Definition{cube-cat-neat-embed}),
  \[
  \image(\ol{\iota}_{x,z})\cap\Bigl(\prod_{i=|w|}^{|v|-1}[-R,R]^{d_i}\times\Moduli_{\CubeFlowCat{n}}(v,w)\times\prod_{i=|v|}^{|u|-1}[-R,R]^{d_i}\times
  \Moduli_{\CubeFlowCat{n}}(u,v)\Bigr)=\image\Bigl(\coprod_{y'\mid\Funky(y')=v}\ol{\iota}_{y',z}
  \times\ol{\iota}_{x,y'}\Bigr),
  \]
  and therefore, there exists $y'$ with $\Funky(y')=v$ such that
  $(p_3,p_{4,3})\in\image(\ol{\iota}_{x,y'})$ and
  $(p_2,p_{4,2})\in\image(\ol{\iota}_{y',z})$. Condition~(\ref{item:induced-coprod-map})
  from \Definition{cube-cat-neat-embed} (but with $\ol{\iota}$ instead
  of $\iota$) ensures that $y'=y$, and then it is straightforward to
  see that $p$ lies in $\Cell[y,z]{x}$. This completes the proof that
  $\Cell[y,z]{x}=\Cell[z]{x}\cap\Cell[y]{x}$.

  Define the dashed arrow from
  $\Cell[z]{x}\cap\Cell[y]{x}=\Cell[y,z]{x}\cong\Cell{z}\times\Moduli_{\FlowCat}(y,z)\times
  \Moduli_{\FlowCat}(x,y)$ to
  $\Cell[z]{y}\cong\Cell{z}\times\Moduli_{\FlowCat}(y,z)$ to be the
  projection map. From the definition of $\Cell[y,z]{x}$, it is easy
  to verify that the resulting diagram commutes.
\end{proof}


\begin{proposition}\label{prop:cubical-realize-stabilize}
  Up to stable homotopy equivalence, the cubical realization is
  independent of the choice of tuple $\TupV{d}$. More precisely, let
  $\iota$ be a cubical neat embedding relative to
  $\TupV{d}=(d_0,\dots,d_{n-1})$ and let $\CRealize{\FlowCat}$ be the
  cubical realization corresponding to $\iota$.  Fix a
  tuple $\TupV{d'}=(d'_0,\dots,d'_{n-1})$ with $d'_i\geq d_i$ for all
  $i$. There is an induced cubical neat embedding of $\FlowCat$
  relative to $\TupV{d'}$, gotten by identifying the space $E_{u,v}$
  for $\iota$ with the subspace
  \[
  \prod_{i=|v|}^{|u|-1}(-R,R)^{d_i}\times\{0\}^{d'_i-d_i}\times\Moduli_{\CubeFlowCat{n}}(u,v)
  \]
  of the space $E'_{u,v}$ for $\iota'$.  Let
  $\CRealize{\FlowCat}^{\prime}$ be the cubical
  realization corresponding to $\iota'$ and let
  $N=\card{\TupV{d'}}-\card{\TupV{d}}=\sum_{i=0}^{n-1}d'_i-\sum_{i=0}^{n-1}d_i$. Then
  there is a homotopy equivalence
  \[
  \Sigma^{N}{\CRealize{\FlowCat}}\simeq {\CRealize{\FlowCat}}^{\prime},
  \]
  taking cells to the corresponding cells.
\end{proposition}
\begin{proof}
  The proof is the same as Case~(3) in the proof of
  \cite[Lemma~3.27]{RS-khovanov}.
\end{proof}

One can also show, by following the proofs of
\cite[Lemmas~3.25--3.27]{RS-khovanov}, that the stable homotopy type
of ${\CRealize{\FlowCat}}_{\iota}$ is independent of the cubical
neat embedding $\iota$ and the parameters $R$ and $\ep$. Since this
result also follows from \Theorem{cubical-CJS-realization}, we do not
give further details.

We conclude this section by returning to our simplicial complex example:
\begin{proposition}
  Let $S_\bullet$ be a simplicial complex with $n$ vertices and let
  $(\FlowCat,\Funky\co \Sigma\FlowCat\to\CubeFlowCat{n})$ be the
  corresponding cubical flow category, as in
  \Example{simplicial-cubical}. Let ${\Realize{S_\bullet}}_+$ denote
  the disjoint union of the geometric realization of $S_\bullet$ and a
  basepoint.  Then there is a stable homotopy equivalence
  \[
  \FlowCatSpace{\FlowCat}\simeq {\Realize{S_\bullet}}_+.
  \]
  Moreover, the stable homotopy equivalence may be chosen so
  that it induces an isomorphism between the simplicial cochain
  complex of $S_\bullet$ and the reduced cellular cochain complex of
  $\FlowCatSpace{\FlowCat}$ (with the CW complex structure from
  \Definition{cubical-realize}).
\end{proposition}
We leave the proof as a (somewhat involved) exercise to the reader.

\subsection{Cubical realization agrees with the Cohen-Jones-Segal realization}\label{sec:cubical-CJS-same}

The main aim of this subsection is to prove the following.
\begin{theorem}\label{thm:cubical-CJS-realization}
  Let $(\FlowCat,\Funky\from\FlowCat\to\CubeFlowCat{n})$ be a cubical
  flow category. For any cubical neat embedding $\iota$ relative to
  any tuple $\TupV{d}=(d_0,\dots,d_{n-1})$, and parameters $R,\ep$,
  the cubical realization $\FlowCatSpace{\FlowCat}$ (from
  \Definition{cubical-realize}) is stably homotopy equivalent to the
  Cohen-Jones-Segal realization (from \Section{flow-cat}) of
  $\FlowCat$ with framing induced from some framing of
  $\CubeFlowCat{n}$. Furthermore, the stable homotopy equivalence
  sends cells to corresponding cells by degree $\pm 1$ maps.
\end{theorem}

\begin{proof}
  Using \Proposition{cubical-realize-stabilize}, we may assume that
  the $d_i$'s are sufficiently large so that there is a neat embedding
  $\jmath$ of the cube flow category $\CubeFlowCat{n}$ relative to
  $\TupV{D}=(\ldots,0,\ldots,0,d_0,\ldots,\allowbreak d_{n-1},0,\ldots,0\ldots)$
  in the sense of \Definition{flow-cat-neat-embed}. Fix some framing
  of the cube flow category $\CubeFlowCat{n}$, and construct the
  extension $\ol{\jmath}$ as in \Definition{CJS-realization}. Let
  $\delta$ and $T$ be the corresponding parameters. After scaling if
  necessary, we may assume $\delta=R$, and after increasing $T$ if
  necessary, we may assume $T\geq 1$.

  Note that $\ol{\jmath}\circ\iota$ is a neat embedding (in the
  sense of \Definition{flow-cat-neat-embed}) of our flow category
  $\FlowCat$ relative to $\TupV{D}$; it has an extension
  $\ol{\jmath}\circ\ol{\iota}$ (again, as in
  \Definition{CJS-realization}) with respect to the parameters $\ep$ and $T$,
  and $\ol{\jmath}\circ \ol{\iota}$ is isotopic to the extension coming from the framing of
  $\FlowCat$ induced from the framing of $\CubeFlowCat{n}$; see also
  \Remark{explicit-isotopic-implicit}.

  Let $\CRealize{\FlowCat}$ be the CW complex constructed from the
  cubical neat embedding $\iota$ and its extension $\ol{\iota}$ with
  respect to the parameters $\ep$ and $R$; and let $\Realize{\FlowCat}$ be
  the CW complex constructed from the neat embedding
  $\ol{\jmath}\circ\iota$ and its extension
  $\ol{\jmath}\circ\ol{\iota}$ with respect to the parameters
  $\ep$ and $T$. Cells of both $\CRealize{\FlowCat}$ and $\Realize{\FlowCat}$ correspond to objects of 
  $\FlowCat$, and hence to each other. 
  We will produce a quotient map from $\Realize{\FlowCat}$ to
  $\CRealize{\FlowCat}$ which will send each cell via a degree $\pm 1$
  map to the corresponding cell. It follows that the quotient map is a
  stable homotopy equivalence.

  For any $x\in\Ob(\FlowCat)$ with
  $\Funky(x)=u\in\Ob(\CubeFlowCat{n})$, the cell associated to $x$ in
  $\CRealize{\FlowCat}$ is
  \[
  \Cell{x}'=
  \begin{cases}
    \prod_{i=0}^{|u|-1}[-R,R]^{d_i}\times\prod_{i=|u|}^{n-1}[-\ep,\ep]^{d_i}\times[0,1]\times\Moduli_{\CubeFlowCat{n}}(u,\vect{0})&\text{if $u\neq\vec{0}$,}\\
    \prod_{i=0}^{n-1}[-\ep,\ep]^{d_i}\times\{0\}&\text{if $u=\vec{0}$.}
  \end{cases}
  \]
  while the cell associated to $x$ in $\Realize{\FlowCat}$ is
  \[
  \Cell{x}=\prod_{i=0}^{|u|-1}[-T,T]^{d_i}\times\prod_{i=|u|}^{n-1}[-\ep,\ep]^{d_i}\times\prod_{i=0}^{|u|-1}[0,T].
  \]
  Let $[\Cell{x}]$ (respectively $[\Cell{x}']$) denote the image of
  $\Cell{x}$ (respectively $\Cell{x}'$) in $\Realize{\FlowCat}$
  (respectively $\CRealize{\FlowCat}$). We will define a map
  $\Cell{x}\to[\Cell{x}']$ and check that this induces a well-defined
  map $[\Cell{x}]\to[\Cell{x}']$.

  If $u=\vec{0}$, identify
  $\Cell{x}'\cong\prod_{i=0}^{n-1}[-\ep,\ep]^{d_i}\cong\Cell{x}$. If
  $u\neq \vec{0}$, the embedding
  \begin{align*}
  \ol{\jmath}_{u,\vec{0}}\from\Moduli_{\CubeFlowCat{n}}(u,\vec{0})\times\prod_{i=0}^{|u|-1}[-R,R]^{d_i}&\into
  (-T,T)^{d_0}\times[0,T)\times\dots\times[0,T)\times(-T,T)^{d_{|u|-1}}\\
  &\cong\prod_{i=1}^{|u|-1}[0,T)\times\prod_{i=0}^{|u|-1}(-T,T)^{d_i}
  \end{align*}
  induces the following codimension-zero embedding of $\Cell{x}'$ into
  $\Cell{x}$:
  \begin{align*}
    \Cell{x}'&=\prod_{i=0}^{|u|-1}[-R,R]^{d_i}\times\prod_{i=|u|}^{n-1}[-\ep,\ep]^{d_i}\times[0,1]\times\Moduli_{\CubeFlowCat{n}}(u,\vect{0})\\
    &\cong\Moduli_{\CubeFlowCat{n}}(u,\vect{0})\times\prod_{i=0}^{|u|-1}[-R,R]^{d_i}\times\prod_{i=|u|}^{n-1}[-\ep,\ep]^{d_i}\times[0,1]\\
    &\stackrel{\ol{\jmath}_{u,\vec{0}}\times\Id}{\lhook\joinrel\relbar\joinrel\relbar\joinrel\relbar\joinrel\rightarrow} \prod_{i=1}^{|u|-1}[0,T]\times\prod_{i=0}^{|u|-1}[-T,T]^{d_i}\times\prod_{i=|u|}^{n-1}[-\ep,\ep]^{d_i}\times[0,1]\\
    &\cong\prod_{i=0}^{|u|-1}[-T,T]^{d_i}\times\prod_{i=|u|}^{n-1}[-\ep,\ep]^{d_i}\times[0,1]\times\prod_{i=1}^{|u|-1}[0,T]\\
    &\into\prod_{i=0}^{|u|-1}[-T,T]^{d_i}\times\prod_{i=|u|}^{n-1}[-\ep,\ep]^{d_i}\times\prod_{i=0}^{|u|-1}[0,T]\\
    &=\Cell{x}.
  \end{align*}
  (Here, the $\cong$ arrows correspond to the obvious reshuffling of
  the factors, and the second inclusion is induced from the inclusion
  $[0,1]\into[0,T]$.)  In either case, map $\Cell{x}$ to $[\Cell{x}']$
  by identifying the image of this embedding with $\Cell{x}'$, and
  quotienting everything else to the basepoint.
  To see that this gives a
  well-defined, continuous map from the CW complex $\Realize{\FlowCat}$ to the CW
  complex $\CRealize{\FlowCat}$, we need to check that for any other
  $y\in\Ob(\FlowCat)$ with $\Funky(y)=v<u$, the following commutes:
  \[
  \xymatrix{
    \Cell{y}'\times\Moduli_{\FlowCat}(x,y)\ar@{^(->}[r]\ar@{^(->}[d]&  \Cell{y}\times\Moduli_{\FlowCat}(x,y)\ar@{^(->}[d]\\
    \Cell{x}'\ar@{^(->}[r]&\Cell{x}.  
  }
  \]
  (The horizontal arrows are induced from the inclusions defined
  above. The right vertical arrow is the inclusion defined in
  \Definition{CJS-realization} for $\ol{\jmath}\circ\ol{\iota}$, while
  the left vertical arrow is the inclusion defined in
  \Definition{cubical-realize} for $\ol{\iota}$.)

  When $v=\vec{0}$, after removing the constant factor of
  $\prod_{i=|u|}^{n-1}[-\ep,\ep]^{d_i}$ and doing some consistent reshuffling,
  the diagram is
  \[
  \xymatrix{
    \{0\}\times\displaystyle\prod_{i=0}^{|u|-1}[-\epsilon,\epsilon]^{d_i}\times \Moduli_{\FlowCat}(x,y)\ar@{^(->}[rr]^-{\cong}\ar@{^(->}[d]^-{(\kappa_1,\ol{\iota}_{x,y})}&& \displaystyle\prod_{i=0}^{|u|-1}[-\epsilon,\epsilon]^{d_i}\times\Moduli_{\FlowCat}(x,y)\ar@{^(->}[d]^-{(\kappa^0_3,\Id)\circ\ol{\jmath}_{u,\vec{0}}\circ\ol{\iota}_{x,y}}\\
    [0,1]\times\displaystyle\prod_{i=0}^{|u|-1}[-R,R]^{d_i}\times\Moduli_{\CubeFlowCat{n}}(u,\vect{0})\ar@{^(->}[rr]^-{(\kappa_2,\ol{\jmath}_{u,\vec{0}})}&&\displaystyle\prod_{i=0}^{|u|-1}[0,T]\times\displaystyle\prod_{i=0}^{|u|-1}[-T,T]^{d_i},  
  }
  \]
  where $\kappa_1$ is the inclusions $\{0\}\into[0,1]$, $\kappa_2$ is
  the inclusion $[0,1]\into[0,T]$, and $\kappa^\ell_3$ (for
  $0\leq\ell<|u|$) is the inclusion
  \[
  \prod_{i=\ell+1}^{|u|-1}[0,T]\cong \{0\}\times
  \prod_{i=\ell+1}^{|u|-1}[0,T]\stackrel{(\kappa_2\circ\kappa_1,\Id)}{\lhook\joinrel\relbar\joinrel\relbar\joinrel\relbar\joinrel\rightarrow}
  \prod_{i=\ell}^{|u|-1}[0,T].\] Therefore, the diagram commutes.

  When $v\neq\vec{0}$, once again after removing the constant factor
  of $\prod_{i=|u|}^{n-1}[-\ep,\ep]^{d_i}$, and some consistent
  reshuffling, the diagram factors as
  \[
  \xymatrix@C=10ex{
    \left[{\begin{split}[0,1]&\times\Moduli_{\CubeFlowCat{n}}(v,\vec{0})\times\displaystyle\prod_{i=0}^{|v|-1}[-R,R]^{d_i}\\&\times\displaystyle\prod_{i=|v|}^{|u|-1}[-\ep,\ep]^{d_i}\times\Moduli_{\FlowCat}(x,y)\end{split}}\right]
    \ar@{^(->}[r]^-{(\kappa_2,\ol{\jmath}_{v,\vec{0}},\Id)}\ar@{^(->}[d]^-{(\Id,\ol{\iota}_{x,y})}
    & 
    \left[{\begin{split}\displaystyle&\prod_{i=0}^{|v|-1}[0,T]\times\displaystyle\prod_{i=0}^{|v|-1}[-T,T]^{d_i}\\\times&\displaystyle\prod_{i=|v|}^{|u|-1}[-\ep,\ep]^{d_i}\times\Moduli_{\FlowCat}(x,y)\end{split}}\right]
    \ar@{^(->}[d]^-{(\Id,\ol{\jmath}_{u,\vec{0}}\circ\ol{\iota}_{x,y})}\\
    \left[{\begin{split}[0,1]&\times\Moduli_{\CubeFlowCat{n}}(v,\vec{0})\times\displaystyle\prod_{i=0}^{|u|-1}[-R,R]^{d_i}\\&\times\Moduli_{\CubeFlowCat{n}}(u,v)\end{split}}\right]
    \ar@{^(->}[r]^-{(\kappa_2,\ol{\jmath}_{v,\vec{0}},\ol{\jmath}_{u,v})}\ar@{^(->}[d]^-{(\Id,\circ)\circ\rho}& 
    \left[\displaystyle\prod_{i=0}^{|v|-1}[0,T]\times\displaystyle\prod_{i=0}^{|u|-1}[-T,T]^{d_i} \times\displaystyle\prod_{i=|v|+1}^{|u|-1}[0,T]\right]
    \ar@{^(->}[d]^-{\rho\circ(\Id,\kappa_3^{|v|})}\\
    \left[[0,1]\times\displaystyle\prod_{i=0}^{|u|-1}[-R,R]^{d_i}\times\Moduli_{\CubeFlowCat{n}}(u,\vect{0})\right]
    \ar@{^(->}[r]^-{(\kappa_2,\ol{\jmath}_{u,\vec{0}})}&
    \left[\displaystyle\prod_{i=0}^{|u|-1}[0,T]\times\displaystyle\prod_{i=0}^{|u|-1}[-T,T]^{d_i}\right],  
  }
  \]
  where $\rho$ throughout denotes some (further) consistent
  reshuffling of factors, and $\kappa_2,\kappa_3^\ell$ from
  before. The top square commutes automatically; the bottom square
  commutes since Condition~(\ref{item:CJS-neat-embed-coherent}) of
  \Definition{flow-cat-neat-embed} holds for $\ol{\jmath}$ as well
  (because the extension was defined via coherent framings of the
  normal bundles of $\jmath$).

  Therefore, we get a well-defined map from $\Realize{\FlowCat}$ to
  $\CRealize{\FlowCat}$ which sends the cell $[\Cell{x}]$ to the
  corresponding cell $[\Cell{x}']$ by a degree $\pm 1$ map. It is easy
  to check that the grading shifts match up, and therefore, we get the
  required stable homotopy equivalence.
\end{proof}

\section{Functors from the cube to the Burnside category and their realizations}\label{sec:realize-functor}
In this section we give a reformulation of cubical flow categories, as
$2$-functors from the cube category to the Burnside category. We then give
a choice-free way of realizing such a functor, in terms of a
thickening construction. (A smaller but choice-dependent way to
realize such a functor is given in \Section{smaller-cube}.)  The proof
that this realization agrees with the cubical realization is deferred
until \Section{cubical-to-cubes}.  The main results of this
section can be summarized in two theorems:

\begin{theorem}\label{thm:cubical-is-cube-fun}
  The data of a cubical flow category is equivalent to the data of a
  strictly unitary, lax $2$-functor from the cube category to the
  Burnside category.
\end{theorem}
This is stated more precisely and proved as
Lemmas~\ref{lem:flow-to-Burn} and~\ref{lem:Burn-gives-flow-cat}.

\begin{theorem}\label{thm:realize-cube-fun}
  Given a functor $F$ from the cube category to the Burnside category
  there is a canonically associated CW spectrum $\Realize{F}$, the
  realization of $F$.
\end{theorem}
This is stated precisely as \Construction{realize-functor} and
\Lemma{iso-realize-same}. In \Section{cubical-to-cubes}, we show that
this construction agrees with the cubical realization (a combination
of Theorems~\ref{thm:smaller-diag}
and~\ref{thm:box-equals-cubical}). A feature of the realization
procedure introduced in this section is that it behaves well with
respect to products (\Proposition{realize-product}).

We start by reviewing the Burnside (2-)category and the data required
to define a functor from the cube to it, in
Section~\ref{sec:Burnside}. We then recall some properties of homotopy
colimits, in Section~\ref{sec:colimit}, before turning to new material
in Section~\ref{sec:cubical-to-burnside}, where we formulate precisely
and prove Theorem~\ref{thm:cubical-is-cube-fun}. The thickening and
realization procedure are given in Sections~\ref{sec:thicken}
and~\ref{sec:subsec-realize}, while invariance under natural isomorphisms of $2$-functors is proved
in \Section{thick-invariance}.
\Section{products} discusses the behavior of this realization
procedure under certain kinds of products.

\subsection{The Burnside category and 2-functors to it}\label{sec:Burnside}
Given sets $X$ and $Y$, a \emph{correspondence} (or
\emph{span}) from $X$ to $Y$ is a set $A$ and maps $s\co A\to X$
and $t\co A\to Y$ (for \emph{source} and \emph{target}). Given a
correspondence $(A,s_A,t_A)$ from $X$ to $Y$ and $(B,s_B,t_B)$ from
$Y$ to $Z$ the \emph{composition} $(B,s_B,t_B)\circ (A,s_A,t_A)$ is
the correspondence $(C,s,t)$ from $X$ to $Z$ given by
\[
C=B\times_Y A=\{(b,a)\in B\times A\mid t(a)=s(b)\}\qquad
s(b,a)=s_A(a)\qquad t(b,a)=t_B(b).
\]
Given correspondences $(A,s_A,t_A)$ and $(B,s_B,t_B)$ from $X$ to $Y$,
a \emph{morphism of correspondences} from $(A,s_A,t_A)$ to
$(B,s_B,t_B)$ is a bijection of sets $f\co A\to B$ which commutes with
the source and target maps, i.e., so that $s_A=s_B\circ f$ and
$t_A=t_B\circ f$. Composition (of set maps) makes the set of
correspondences from $X$ to $Y$ into a category. Further, composition
of correspondences makes $($Sets$,$ Correspondences$,$ Morphisms of
correspondences$)$ into a weak 2-category (bicategory in the language
of~\cite{Benabou-other-bicategories}). By the \emph{Burnside category}
we mean the sub-2-category of finite sets and finite
correspondences. We denote the Burnside category by
$\BurnsideCat$. (More typically, one defines the Burnside category of
a group $G$ in terms of $G$-sets and $G$-equivariant correspondences;
for us, $G$ is the trivial group.)

\begin{remark}
  Some authors refer to our Burnside category as the \emph{Burnside
    $2$-category}, and refer to the $1$-category with objects sets and morphisms
  isomorphism classes of correspondences as the Burnside category.
\end{remark}

We will typically drop the maps $s$ and $t$ from the notation,
referring simply to a correspondence $A$ from $X$ to $Y$.

As mentioned above, the Burnside category is a weak $2$-category: the identity and
associativity axioms only hold up to natural isomorphism. That is,
given a set $X$, the \emph{identity correspondence} of $X$ is simply
the set $X$ itself, with the identity map as source and target
maps. Given another correspondence $A$ from $W$ to $X$ there is a
natural isomorphism
\[
\lambda\co X\times_X A\stackrel{\cong}{\longrightarrow} A
\]
defined by $\lambda(x,a)=a$.  Similarly, given a correspondence $B$
from $X$ to $Y$ there is a natural isomorphism
\[
\rho\co B\times_X X\stackrel{\cong}{\longrightarrow} B
\]
defined by $\rho(b,x)=b$.  Finally, given correspondences $A$
from $W$ to $X$, $B$ from $X$ to $Y$ and $C$ from $Y$ to $Z$ there is
a natural isomorphism
\[
\alpha\co (C\times_Y B)\times_X A\to C\times_Y(B\times_XA)
\]
defined by $\alpha((c,b),a)=(c,(b,a))$.

This distinction between weak and strict $2$-categories may seem superficial
here, but the distinction between weak and strict
$2$-functors, to which we turn next, will be crucial.

%
\begin{definition}\label{def:2-functor}
  Given (weak) $2$-categories $\Cat$ and $\Dat$, a \emph{lax 2-functor} $F\co
  \Cat\to \Dat$ consists of the following data:
  \begin{itemize}
  \item For each object $x\in\Ob(\Cat)$ an object $F(x)\in\Ob(\Dat)$.
  \item For each pair of objects $x,y\in\Ob(\Cat)$ a functor
    $F_{x,y}\co \Hom(x,y)\to \Hom(F(x),F(y))$.
  \item For each object $x\in\Ob(\Cat)$ a 2-morphism $F_{\Id_x}\co
    \Id_{F(x)}\to F_{x,x}(\Id_x)$.
  \item For any three objects $x,y,z\in\Ob(\Cat)$ a natural transformation
    \[
    F_{x,y,z}(\cdot,\cdot)\co F_{y,z}(\cdot)\circ_1 F_{x,y}(\cdot)\to F_{x,z}(\cdot\circ_1\cdot).
    \]
    (Here, both $F_{y,z}(\cdot)\circ_1 F_{x,y}(\cdot)$ and $F_{x,z}(\cdot\circ_1\cdot)$ are functors
    $\Hom(y,z)\times \Hom(x,y)\to \Hom(F(x),F(z))$; $\circ_1$ denotes the
    $1$-composition in $\Cat$ or $\Dat$, not the composition of functors.)
  \end{itemize}
  These data must satisfy the following compatibility conditions:
  \begin{enumerate}[label=(Fn-\arabic*)]
  \item\label{item:str-unit} For each pair of objects $x,y\in
    \Ob(\Cat)$, the following diagrams commute:
    \[
    \xymatrixcolsep{5em}
    \xymatrix{
      \Id_{F(y)}\circ_1 F_{x,y}(\cdot)\ar[r]^-\lambda\ar[d]_{F_{\Id_y}\circ_1 \Id} & F_{x,y}(\cdot) \\
      F_{y,y}(\Id_y)\circ_1 F_{x,y}(\cdot) \ar[r]_-{F_{x,y,y}} &
      F_{x,y}(\Id_y\circ_1 \cdot)\ar[u]_{F_{x,y}(\lambda)}
      }
      \qquad\qquad
    \xymatrix{
      F_{x,y}(\cdot)\circ_1 \Id_{F(x)}\ar[r]^-\rho\ar[d]_{\Id\circ_1 F_{\Id_x}} & F_{x,y}(\cdot) \\
      F_{x,y}(\cdot) \circ_1 F_{x,x}(\Id_x) \ar[r]_-{F_{x,x,y}} &
      F_{x,y}(\cdot \circ_1 \Id_x).\ar[u]_{F_{x,y}(\rho)}
      }
    \]
    (Here, $\lambda$ is the natural isomorphism from $\Id_y\circ_1
    \cdot$ to the identity functor of $\Cat$ or
    $\Dat$, as appropriate, and $\rho$ is the natural isomorphism
    from $\cdot\circ_1 \Id_x$ to the identity functor of $\Cat$ or $\Dat$.)
  \item\label{item:str-assoc} For each quadruple of objects
    $x,y,z,w\in\Ob(\Cat)$, the following diagram commutes:
    \[
    \xymatrix{
      (F_{y,z}(\cdot)\circ_1 F_{x,y}(\cdot))\circ_1 F_{w,x}(\cdot) \ar[r]^{\alpha} \ar[d]_{F_{x,y,z}(\cdot,\cdot)\circ_1\cdot}& F_{y,z}(\cdot)\circ_1(F_{x,y}(\cdot)\circ_1 F_{w,x}(\cdot))\ar[d]^{\cdot\circ_1 F_{w,x,y}(\cdot,\cdot)}\\
      F_{x,z}(\cdot\circ_1\cdot)\circ_1 F_{w,x}(\cdot) \ar[d]_{F_{w,x,z}(\cdot\circ_1\cdot,\cdot)} & F_{y,z}(\cdot)\circ_1 F_{w,y}(\cdot\circ_1\cdot)\ar[d]^{F_{w,y,z}(\cdot,\cdot\circ_1\cdot)}\\
      F_{w,z}((\cdot\circ_1\cdot)\circ_1\cdot)\ar[r]_{F_{w,z}(\alpha)} & F_{w,z}(\cdot\circ_1(\cdot\circ_1\cdot)).
    }
    \]
    (Here, $\alpha$ is the natural isomorphism from
    $((\cdot\circ_1\cdot)\circ_1\cdot)$ to
    $(\cdot\circ_1(\cdot\circ_1\cdot))$, the two different orders of
    triple-compositions, associated to $\Cat$ or $\Dat$, as
    appropriate.)
  \end{enumerate}
\end{definition}
(See, for instance,~\cite[Definition 4.1 and Remark
4.2]{Benabou-other-bicategories}, where lax $2$-functors are called
homomorphisms.)

We will often drop the subscript from $F$: given an element
$f\in\Hom_\Cat(x,y)$ and a lax 2-functor $F\co \Cat\to\Dat$ we will
often write $F(f)$ for $F_{x,y}(f)$.

\begin{definition}\cite[Remark 4.2]{Benabou-other-bicategories}
  We call a lax $2$-functor $F\co \Cat\to\Dat$ \emph{strictly unitary} if
  for all objects $x\in\Ob(\Cat)$, $F_{x,x}(\Id_x)=\Id_{F(x)}$ and
  $F_{\Id_x}$ is the identity 2-morphism.
\end{definition}

As mentioned above, we will be mainly interested in $2$-functors from
$\CCat{n}$ to $\BurnsideCat$, which moreover will be strictly
unitary. In this case, \Definition{2-functor} simplifies substantially:
\begin{lemma}\label{lem:2func-is}
  A strictly unitary, lax $2$-functor $F\co \CCat{n}\to \BurnsideCat$ is
  determined by the following data:
  \begin{itemize}
  \item For each object $v\in\Ob(\CCat{n})=\{0,1\}^n$ the set $X_v=F(v)$.
  \item For each pair of objects $v,w\in\Ob(\CCat{n})$ such that
    $v>w$, a correspondence $A_{v,w}=F(\cmorph{v}{w})$ from $X_v$ to $X_w$.
  \item For each triple of objects $u,v,w\in \Ob(\CCat{n})$ such that
    $u>v>w$, a bijection $F_{u,v,w}\co A_{v,w}\times_{X_v}A_{u,v}\to
    A_{u,w}$.
  \end{itemize}

  These data satisfy the compatibility condition:
  \begin{enumerate}[label=(CF-\arabic*)]
  \item\label{item:CuFunc} For any $u,v,w,x\in\Ob(\CCat{n})$ with $u>v>w>x$, the
    following diagram commutes:
    \[
    \xymatrix@C=15ex{
      A_{w,x}\times_{X_w}A_{v,w}\times_{X_v} A_{u,v} \ar[r]^-{\Id\times F_{u,v,w}}\ar[d]_{F_{v,w,x}\times\Id} & A_{w,x}\times_{X_w} A_{u,w}\ar[d]^{F_{u,w,x}}\\
      A_{v,x}\times_{X_v} A_{u,v}\ar[r]_{F_{u,v,x}} & A_{u,x}.
    }
    \]
    (We have suppressed some non-confusing parentheses.)
  \end{enumerate}

  Moreover, any collection of data satisfying this compatibility
  condition determines a strictly unitary, lax $2$-functor $F\co \CCat{n}\to
  \BurnsideCat$.
\end{lemma}
\begin{proof}
  Since $F$ is strictly unitary, we have $F(\cmorph{v}{v})=\Id_{X_v}$
  (which is $X_v$, viewed as a correspondence from itself to
  itself). If $v\not\geq w$ then $\Hom(v,w)=\emptyset$, so $F(\cmorph{v}{w})$
  is the unique functor from the empty category. Thus, the $F(\cmorph{v}{w})$
  are entirely specified by the correspondences $A_{v,w}=F_{v,w}(\cmorph{v}{w})$
  with $v>w$. Next, the source $\Hom(v,w)\times \Hom(u,v)$ of 
  $F_{u,v,w}$ is nonempty if and only
  if $u\geq v\geq w$, in which case $\Hom(v,w)\times \Hom(u,v)$ 
  consists of the single element $(\cmorph{v}{w},\cmorph{u}{v})$. Since
  $\CCat{n}$ is a strict $2$-category and $F$ is strictly unitary,
  Condition~\ref{item:str-unit} is equivalent to the statement that
  $F_{v,v,w}\co A_{v,w}\times_{X_v} X_{v}\to A_{v,w}$ is the canonical
  isomorphism $\rho$, and $F_{v,w,w}\co X_w\times_{X_w}A_{v,w}\to
  A_{v,w}$ is the canonical isomorphism $\lambda$. So, $F$ is
  determined by the specified data. Condition~\ref{item:str-assoc} is
  equivalent to Condition~\ref{item:CuFunc}; in
  Condition~\ref{item:CuFunc} we have abused notation to identify the
  two sides of the top row of Condition~\ref{item:str-assoc}, and the
  bottom arrow in Condition~\ref{item:str-assoc} is an equality
  because $\CCat{n}$ is a strict $2$-category and $F$ is strictly
  unitary. The result follows.
\end{proof}

\begin{lemma}\label{lem:characterize-functor}
  Up to natural isomorphism, a strictly unitary, lax $2$-functor $F\co
  \CCat{n}\to \BurnsideCat$ is determined by the sets $F(v)$, the
  correspondences $F(\cmorph{v}{w})$ with $v>w$ and $|v|-|w|=1$, and
  the maps $F^{-1}_{u,v',w}\circ F_{u,v,w}\co F(\cmorph{v}{w})\circ
  F(\cmorph{u}{v})\to F(\cmorph{v'}{w})\circ F(\cmorph{u}{v'})$ with $u>v,v'>w$ and
  $|u|-|w|=2$.
\end{lemma}
\begin{proof}
  This follows from the observations that:
  \begin{itemize}
  \item Any morphism in $\CCat{n}$ is a composition of maps associated
    to edges, so, $F(\cmorph{v}{w})$ is determined for all $v\geq w$.
  \item Any two (directed) edge paths from $v$ to $w$ in $\CCat{n}$
    are related by a sequence of swaps across 2-dimensional faces, so
    $F_{u,v,w}$ is determined for all $u\geq v\geq w$.
  \end{itemize}
  Further details are left to the reader.
\end{proof}

\subsection{Homotopy colimits and homotopy coherent diagrams}\label{sec:colimit}
The realization of a $2$-functor $\CCat{n}\to\BurnsideCat$ in this section and
Section~\ref{sec:smaller-cube} will involve an iterated mapping cone, which can
be described as a homotopy colimit. So, we review briefly the notion of homotopy
colimits here.

Given a diagram $F\co\Dat\to\BSpaces$ of based topological spaces, the
\emph{homotopy colimit} of $F$, $\hocolim F=\hocolim_\Dat F$, is
another based topological space. Similarly, if $F\co\Dat\to\Spectra$
is a diagram of spectra then
we can again form the homotopy colimit of $F$, $\hocolim F$, which is
a spectrum. We will give a construction in a slightly more general
setting presently, but first we note some key properties of the
homotopy colimit, all of which hold for both diagrams of spaces and
diagrams of spectra:

\begin{enumerate}[label=(ho-\arabic*)]
\item\label{item:hocolim-nat-trans} Suppose that $F,G\co\Cat\to\BSpaces$ are diagrams and $\eta\co
  F\to G$ is a natural transformation. Then $\eta$ induces a map
  $\hocolim\eta\co \hocolim F\to \hocolim G$. If $\eta(c)$ is a stable
  homotopy equivalence for each $c\in\Ob(\Cat)$ then $\hocolim\eta$ is
  a stable homotopy equivalence as well.
\item\label{item:hocolim-sum} Suppose that $F,G\co \Cat\to\BSpaces$
  are diagrams and $F\vee G\co\Cat\to\BSpaces$ is the diagram obtained
  by taking their wedge sum, $(F\vee G)(v)=F(v)\vee G(v)$. Then
  the natural map $\hocolim{F\vee G} \to \hocolim{F}\vee\hocolim{G}$
  is an equivalence.
\item\label{item:hocolim-prod} Suppose that $F\co \Cat\to\BSpaces$ and
  $G\co\Dat\to\BSpaces$. Then there is an induced functor $F\smas G\co
  \Cat\times\Dat\to \BSpaces$, with $(F\smas G)(v\times w)=F(v)\smas
  G(w)$, and a natural weak equivalence $(\hocolim
  F)\smas(\hocolim G) \to \hocolim(F\smas G)$.

\item\label{item:hocolim-cofinal} Let $G\co \Cat\to\Dat$ be a map of
  diagrams (i.e., a functor between small categories). Given
  $d\in\Ob(\Dat)$, the \emph{undercategory} of $d$ has objects
  $\{(c,f)\mid c\in\Cat, f\co d\to G(c)\}$, and
  $\Hom((c,f),(c',f'))=\{g\co c\to c'\mid f'=G(g)\circ f\}$.  Let
  $d\downarrow G$ denote the undercategory of $d$. The functor $G$ is
  called \emph{homotopy cofinal} if for each $d\in\Ob(\Dat)$,
  $d\downarrow G$ has contractible nerve.

  Now, let $F\co \Dat\to\BSpaces$ be a diagram. Then
  there is an induced functor $F\circ G\co \Cat\to\BSpaces$.
  Suppose that $G$ is homotopy cofinal. Then
  \[
  \hocolim F\circ G\simeq \hocolim F.
  \]
  (In the case of homotopy limits, this is~\cite[Cofinality Theorem
  XI,9.2]{BK-top-book}. The homotopy colimit can be
  characterized by knowing that the mapping space $\Map(\hocolim F,
  Z)$ is equivalent to the homotopy limit of the diagram of spaces
  $\Map(F, Z)$.)
\end{enumerate}

In \Section{smaller-cube}, we will need to talk about homotopy
colimits of diagrams which are only homotopy-commutative, but where
the homotopies are part of the data, and are coherent up to higher
homotopies (also part of the data). The rest of this subsection
focuses on one formulation of this generalization.

We start with an appropriate notion of diagrams:
\begin{definition}\label{def:coherent-diag}
  \cite[Definition 2.3]{Vogt-top-hocolim}
  A \emph{homotopy coherent diagram in $\BSpaces$} consists of:
  \begin{itemize}
  \item A small category $\Cat$.
  \item For each $x\in\Ob(\Cat)$ a space $F(x)\in\BSpaces$.
  \item For each $n\geq 1$ and each sequence 
    \[
    x_0\stackrel{f_1}{\longrightarrow} x_1 \stackrel{f_2}{\longrightarrow}
    \cdots\stackrel{f_n}{\longrightarrow} x_n
    \]
    of composable morphisms in $\Cat$ a continuous map 
    \[
    F(f_n,\dots,f_1)\co [0,1]^{n-1}\times F(x_0)\to F(x_n)
    \]
    with $F(f_n,\dots,f_1)([0,1]^{n-1}\times\{*\})=*$, the basepoint
    in $F(x_n)$. 
  \end{itemize}
  Letting $(t_1,\dots,t_{n-1})$ denote points in $[0,1]^{n-1}$, these
  maps $F$ are required to satisfy the conditions:
  \newlength{\tmpwidth}
  \setlength{\tmpwidth}{\widthof{$f_n$}}
  \begin{multline}\label{eq:coherent-diag}
  F(f_n,\dots,f_1)(t_{1},\dots,t_{n-1})\\
  =
  \begin{cases}
    F(f_n,\dots,f_2)(t_{2},\dots,t_{n-1}) & \makebox[.8\tmpwidth]{$f_1$}=\Id\\
    F(f_n,\dots,f_{i+1},f_{i-1},\dots,f_1)(t_{1},\dots,t_{i-1}\cdot t_{i},\dots,t_{n-1}) & \makebox[.8\tmpwidth]{$f_i$}=\Id,\ 1<i<n\\
    F(f_{n-1},\dots,f_1)(t_{1},\dots,t_{n-2}) & \makebox[.8\tmpwidth]{$f_n$}=\Id\\
    [F(f_n,\dots,f_{i+1})(t_{i+1},\dots,t_{n-1})]\circ [F(f_i,\dots,f_1)(t_{1},\dots,t_{i-1})]& \makebox[.8\tmpwidth]{$t_i$}=0\\
    F(f_n,\dots,f_{i+1}\circ f_i,\dots,f_1)(t_{1},\dots,t_{i-1},t_{i+1},\dots,t_{n-1}) & \makebox[.8\tmpwidth]{$t_i$}=1.
  \end{cases}
  \end{multline}
\end{definition}
(A homotopy coherent diagram is what Vogt~\cite{Vogt-top-hocolim}
calls an $h\Cat$-diagram. We are restricting to the case that his
topological category $\Cat$ is discrete.)

We will abuse notation and denote a homotopy coherent diagram as above
by $F\co\Cat\to\BSpaces$. It will be clear from context when we mean
a commutative diagram or a homotopy coherent diagram.

\begin{example}\label{exam:commutative-strong-homotopy}
  Any commutative diagram $F\co\Cat\to\BSpaces$ can be viewed as a
  homotopy coherent diagram by defining
  \[
  F(f_n,\dots,f_1)(t_{1},\dots,t_{n-1})=F(f_n\circ\cdots\circ f_1).
  \]
\end{example}

\begin{remark}\label{rem:simplicial-cat}
  Associated to an ordinary category $\Cat$, there is a simplicially enriched
  category ${\mathfrak C}[\Cat]$, introduced by
  Leitch~\cite{Leitch-commutative-cube} and further developed by many
  authors
  (e.g. \cite{Cordier-coherent-diagrams,DK-simplicial-localization,Lurie-HTT})
  such that homotopy coherent diagrams $\Cat\to\BSpaces$ are the same
  (up to the replacement of the continuous function
  $t_{i-1} \cdot t_i$ by an equivalent piecewise linear one) as
  simplicial functors from ${\mathfrak C}[\Cat]$ to spaces.
  Lemma~\ref{lem:cube-moduli-cubical-complex} essentially proves that
  the cube flow category $\CubeFlowCat{n}$ is the topological category
  obtained from ${\mathfrak C}[\CCat{n}]$ by taking the realization of
  each morphism set. So, homotopy coherent $\CCat{n}$-diagrams give
  functors out of $\CubeFlowCat{n}$.
\end{remark}

\begin{definition}\cite[Paragraph (5.10)]{Vogt-top-hocolim}\label{def:hocolim}
  Given a homotopy coherent diagram $F\co\Cat\to\BSpaces$, the
  \emph{homotopy colimit} of $F$ is defined by
  \begin{equation}\label{eq:hocolim2}
  \hocolim F = \{*\}\amalg \coprod_{n\geq 0}
  \coprod_{x_0\stackrel{f_1}{\longrightarrow}\cdots\stackrel{f_n}{\longrightarrow}x_n}
  [0,1]^{n}\times F(x_0)/\sim,
  \end{equation}
  where the second coproduct is over $n$-tuples of
  composable morphisms in $\Cat$ and the case $n=0$ corresponds to the
  objects $x_0\in\Ob(\Cat)$. Letting $(t_1,\dots,t_n)$ denote points in
  $[0,1]^n$ and $p$ a point in $F(x_0)$, the equivalence relation $\sim$ is given by
  \setlength{\tmpwidth}{\widthof{$t_n$}}
  \begin{multline*}
  (f_n,\dots,f_1;t_1,\dots,t_n;p)
  \sim
  \begin{cases}
    (f_n,\dots,f_2;t_2,\dots,t_n;p) &  \makebox[\tmpwidth]{$f_1$}=\Id\\
    (f_n,\dots,f_{i+1},f_{i-1},\dots,f_1;t_1,\dots,t_{i-1}t_{i},\dots,t_n;p) & \makebox[\tmpwidth]{$f_i$}=\Id,\ i>1\\
    (f_n,\dots,f_{i+1};t_{i+1},\dots,t_{n};F(f_i,\dots,f_1)(t_{1},\dots,t_{i-1},p)) & \makebox[\tmpwidth]{$t_i$}=0\\
    (f_n,\dots,f_{i+1}\circ f_i,\dots,f_1;t_1,\dots,t_{i-1},t_{i+1},\dots,t_n;p) & \makebox[\tmpwidth]{$t_i$}=1,\ i<n\\
    (f_{n-1},\dots,f_1;t_{1},\dots,t_{n-1};p) & \makebox[\tmpwidth]{$t_n$}=1\\
    * & \makebox[\tmpwidth]{$p$}=*,
  \end{cases}
  \end{multline*}
  where $*$ denotes the basepoint.
\end{definition}

\begin{observation}\label{obs:hocolim-no-id}
  The first three cases in the compatibility
  condition~\eqref{eq:coherent-diag} for a homotopy coherent diagram
  imply that $F$ is determined by its values on sequences of
  non-identity morphisms (and on objects). If one restricts to only
  non-identity morphisms, however, the compatibility condition for $F$
  becomes more complicated. In the special case that $\Cat$ has no
  isomorphisms except for identity maps, however, the compatibility
  condition for sequences $(f_n,\dots,f_1)$ of non-identity morphisms
  is simply the last two cases of \Formula{coherent-diag}.
  
  Similarly, the first two relations in the definition of $\sim$ mean
  that we can write
  \[
  \hocolim F = \{*\}\amalg \coprod_{n\geq 0}
  \coprod_{\substack{x_0\stackrel{f_1}{\longrightarrow}\cdots\stackrel{f_n}{\longrightarrow}x_n\\
    \forall i\in\{1,\dots,n\},\ f_i\neq\Id}}
  [0,1]^{n}\times F(x_0)/\sim',
  \]
  for some equivalence relation $\sim'$, the difference being that we
  consider only non-identity morphisms when $n>0$. The equivalence
  relation $\sim'$ is more complicated than $\sim$. In the special
  case that $\Cat$ has no isomorphisms except for identity maps,
  $\sim'$ is simply given by the last four cases of the definition of
  $\sim$. 

  In the homotopy coherent diagrams and homotopy colimits considered
  in this paper, the categories $\Cat$ will have no non-identity
  isomorphisms, and so we will work with these smaller formulations.
\end{observation}

There is a notion of a morphism between homotopy coherent diagrams
($h$-morphisms~\cite[Definition 2.7]{Vogt-top-hocolim}) $F,G\co
\Cat\to \BSpaces$, which relaxes the notion of a morphism (natural
transformation) between diagrams. In particular, a morphism $F\to G$ of
homotopy coherent diagrams includes the data of maps $F(v)\to G(v)$
for each $v\in\Ob(\Cat)$; we call these maps the \emph{underlying
  maps} of the morphism. There is also the notion of a 
(simplicial) homotopy of 
morphisms~\cite[Definition 2.7]{Vogt-top-hocolim}, and hence 
the notion of a homotopy equivalence of
homotopy coherent diagrams. A special case is that any morphism
of homotopy coherent diagrams whose underlying maps are homotopy
equivalences is a homotopy equivalence of 
diagrams~\cite[Proposition 4.6]{Vogt-top-hocolim}. Further:
\begin{proposition}\cite[Theorem 5.12]{Vogt-top-hocolim}\label{prop:hequiv-diags-hocolim}
  If $F,G\co\Cat\to \BSpaces$ are homotopy equivalent diagrams then
  $\hocolim F\simeq \hocolim G$.
\end{proposition}

There is also a rectification result, that any homotopy coherent
diagram can be made coherent:
\begin{proposition}\cite[Theorem 5.6]{Vogt-top-hocolim}
  Given any homotopy coherent diagram $F\co \Cat\to \BSpaces$ there is
  an honest diagram $G\co\Cat\to\BSpaces$ which is homotopy equivalent
  to $F$. 
\end{proposition}

Finally:
\begin{proposition}\cite[Section 8]{Vogt-top-hocolim}
  If $F\co\Cat\to\BSpaces$ is an honest diagram then the homotopy
  colimits of $F$, viewed as an honest diagram and as a homotopy
  coherent diagram, are homotopy equivalent.
\end{proposition}

\begin{corollary}
  Properties~\ref{item:hocolim-nat-trans}--\ref{item:hocolim-cofinal},
  or their obvious analogues, hold for homotopy colimits of homotopy
  coherent diagrams.
\end{corollary}

\subsection{Cubical flow categories are functors from the cube to the Burnside category}\label{sec:cubical-to-burnside}
\begin{construction}\label{construction:Flow-to-Burn}
  Fix a cubical flow category $\Funky\co
  \FlowCat\to\CubeFlowCat{n}$. We will construct a strictly unitary, lax
  2-functor $F\co \CCat{n}\to\BurnsideCat$, as follows. By
  \Lemma{2func-is}, it suffices to define the sets $X_v$ ($v\in
  \{0,1\}^n$), correspondences $A_{v,w}$ ($v>w\in \{0,1\}^n$) and isomorphisms of correspondences
  $F_{u,v,w}$ ($u>v>w\in\{0,1\}^n$). We do so as follows:
  \begin{itemize}
  \item Given $v\in\{0,1\}^n$ define $F(v)=\Funky^{-1}(v)$.
  \item Given $v>w$, define $A_{v,w}$ to be the set of path components
    of 
    \[
    \coprod_{\substack{x\in \Funky^{-1}(v)\\ y\in \Funky^{-1}(w)}}\Hom(x,y).
    \]
    We will write $\pi_0(X)$ for the set of path components of
    $X$. Then the source (respectively target) map $s\co A_{v,w}\to
    X_v$ (respectively $t\co A_{v,w}\to X_w$) is defined by
    $s(\pi_0(\Hom(x,y)))=x$ (respectively $t(\pi_0(\Hom(x,y)))=y$).
  \item Given $u>v>w$ the composition map in $\FlowCat$ induces a continuous map
    \[
    \circ \co \Bigl(\coprod_{\substack{y\in \Funky^{-1}(v)\\ z\in
        \Funky^{-1}(w)}} \Hom(y,z)\Bigr)\times_{\Funky^{-1}(v)}\Bigl(\coprod_{\substack{x\in \Funky^{-1}(u)\\ y\in
        \Funky^{-1}(v)}} \Hom(x,y)\Bigr)\to 
    \coprod_{\substack{x\in \Funky^{-1}(u)\\ z\in
        \Funky^{-1}(w)}} \Hom(x,z).
    \]
    Taking path components gives a map $A_{v,w}\times_{X_v}A_{u,v}\to
    A_{u,w}$, which we define to be $F_{u,v,w}.$
  \end{itemize}
\end{construction}

\begin{lemma}\label{lem:flow-to-Burn}
  \Construction{Flow-to-Burn} defines a strictly unitary, lax $2$-functor.
\end{lemma}
\begin{proof}
  By \Lemma{2func-is}, we only need to check the compatibility
  Condition~\ref{item:CuFunc}, which is immediate from associativity
  of composition in $\FlowCat$.
\end{proof}

\begin{construction}\label{construction:Burn-to-Flow}
  Fix a strictly unitary, lax $2$-functor $F\co \CCat{n}\to\BurnsideCat$. We will construct a
  cubical flow category $\Funky\co\FlowCat\to\CubeFlowCat{n}$, as
  follows.
  \begin{itemize}
  \item $\Ob(\FlowCat)=\amalg_{v\in \{0,1\}^n}F(v)$. The functor
    $\Funky$ sends an object $x\in F(v)$ to $v$.
  \item For any object $x$, $\Hom(x,x)$ consists of the identity morphism.
  \item Given objects $x$ and $y$, with $v=\Funky(x)>\Funky(y)=w$,
    consider the set 
    \[
    B_{x,y}=s^{-1}(x)\cap t^{-1}(y)\subset A_{v,w}=F(\cmorph{v}{w}).
    \]
    Define $\Hom(x,y)=B_{x,y}\times\Moduli_{\CubeFlowCat{n}}(v,w)$. The
    map $\Funky\co \Hom(x,y)\to \Hom(\Funky(x),\Funky(y))$ is
    projection to the permutohedron $\Moduli_{\CubeFlowCat{n}}(v,w)$.
  \item Given objects $x$, $y$, $z$ with
    $\Funky(x)>\Funky(y)>\Funky(z)$ define the composition map
    $\Hom(y,z)\times\Hom(x,y)\to \Hom(x,z)$ as follows. Let
    $u=\Funky(x)$, $v=\Funky(y)$, $w=\Funky(z)$.  The $2$-functor
    includes a map $F_{u,v,w}\co A_{v,w}\times_{F(v)} A_{u,v}\to
    A_{u,w}$.  The composition map in $\CubeFlowCat{n}$ gives a map
    $\circ\co \Moduli_{\CubeFlowCat{n}}(v,w)\times
    \Moduli_{\CubeFlowCat{n}}(u,v)\to
    \Moduli_{\CubeFlowCat{n}}(u,w)$. Define the composition map in
    $\FlowCat$ to be
    \[
    F_{u,v,w}\times \circ\co (B_{y,z}\times
    \Moduli_{\CubeFlowCat{n}}(v,w))\times (B_{x,y}\times
    \Moduli_{\CubeFlowCat{n}}(u,v))\to (B_{x,z}\times
    \Moduli_{\CubeFlowCat{n}}(u,w)).
    \]
    (That is, we apply $F_{u,v,w}$ to the $B$ factors and $\circ$ to the
    $\Moduli$ factors.)
  \end{itemize}
\end{construction}

\begin{lemma}\label{lem:Burn-gives-flow-cat}
  \Construction{Burn-to-Flow} defines a cubical flow category.
\end{lemma}
\begin{proof}
  The proof is similar to the proof of \Lemma{cube-is-flow-cat}, and
  is left to the reader.
\end{proof}

It is straightforward to verify that
Constructions~\ref{construction:Flow-to-Burn}
and~\ref{construction:Burn-to-Flow} are inverse to each other, in a
sense which does not seem worth spelling out precisely.

\begin{example}\label{exam:KhFunc}
  Applying \Construction{Flow-to-Burn} to the Khovanov flow
  category from \Example{KhFlowCat} yields a functor
  $\KhFunc\from\CCat{n}\to\BurnsideCat$. For any $v\in\{0,1\}^n$, the
  set $\KhFunc(v)$ consists of the Khovanov generators over
  $v$, denoted $F(v)$ in \Section{khovanov-basic}. For
  $u>v\in\{0,1\}^n$ with $|u|-|v|=1$, and for
  $x\in\KhFunc(u),y\in\KhFunc(v)$, the set
  \[
  B_{x,y}=s^{-1}(x)\cap t^{-1}(y)\subseteq A_{u,v}=\KhFunc(\cmorph{u}{v})
  \]
  consists of one element if $x$ appears in $\diffKh(y)$ (see
  \Definition{KhCx-diffKh}), and is empty otherwise. The maps
  $F_{u,v,w}$ when $|u|-|w|=2$ are defined using the \emph{ladybug
    matching}~\cite[Section~5.4]{RS-khovanov}; see
  also \Section{two-functors-same}.
\end{example}

\begin{remark}\label{rem:no-Lee}
  In \Section{khovanov-basic}, we discussed a generalization of the
  Khovanov theory that works over the ring $\Z[h,t]$: a functor from
  $(\CCat{n})^{\op}$ to the category of (graded)
  $\Z[h,t]$-modules. Setting $h=t=0$ recovers Khovanov's original
  functor to $\ZZ\hyModCat$~\cite{Kho-kh-categorification},
  setting $(h,t)=(0,1)$ gives the theory studied by
  Lee~\cite{Lee-kh-endomorphism}, and setting $(h,t)=(1,0)$ gives a
  theory introduced by Bar-Natan~\cite{Bar-kh-tangle-cob}.

  \begin{figure}
    \centering
    \includegraphics{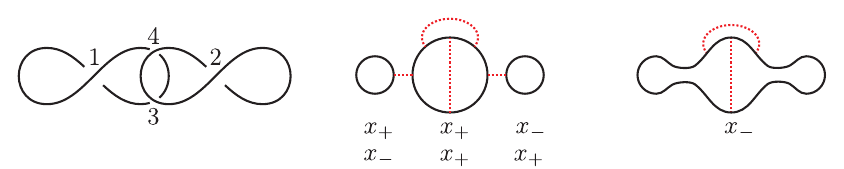}
    \caption{\textbf{An example showing that the Lee complex does not
        come from a functor $\CCat{n}\to\BurnsideCat$ that extends
        $\KhFunc$.} Left: A particular diagram for the two-component
      unlink, and an ordering of its crossings. Center: The
      corresponding resolution configuration (the
      $\vec{0}$-resolution with dashed lines recording the crossings)
      and two labelings of this resolution. Both labelings are in the
      image of $F(\cmorph{1100}{0000})$ of the labeling $x_-$ of the circle
      in the $(1,1,0,0)$-resolution (right). Further, the two labelings give
      incompatible restrictions on the map
      $F_{(1111,1110,1100)}^{-1}\circ F_{(1111,1101,1100)}$ associated
      to the subcube $(1,1,*,*)$.}
    \label{fig:no-Lee}
  \end{figure}

  There is a functor $(\BurnsideCat)^{\op}\to\ZZ\hyModCat$ given as
  follows: to a set $X$, associate the abelian group $\Z\langle
  X\rangle$ freely generated by the elements of $X$; to a correspondence
  $(A,s,t)$ from $X$ to $Y$, associate the following map $\Z\langle
  Y\rangle\to\Z\langle X\rangle$:
  \begin{equation}\label{eq:burnside-to-abelian}
  y\mapsto \sum_{x\in X}\#\card{\Set{a\in A}{s(a)=x,t(a)=y}}x.
  \end{equation}
  \Example{KhFunc} lifts Khovanov's functor
  $(\CCat{n})^{\op}\to\ZZ\hyModCat$ to the functor
  $\KhFunc\from\CCat{n}\to\BurnsideCat$. It is natural to ask whether
  any other specializations of $h$ and $t$ comes from a strictly
  unitary, lax $2$-functor $\CCat{n}\to\BurnsideCat$. Any candidate must
  have $h=0$ and $t\in\NN$, since the coefficients in
  Equation~\eqref{eq:burnside-to-abelian} need to be positive and
  integral. The special case $(h,t)=(0,1)$ (i.e., Lee's theory) does
  not come from any functor $\CCat{n}\to\BurnsideCat$ for arbitrary
  link diagrams extending the functor $\KhFunc$, as can be seen by
  considering the diagram in \Figure{no-Lee}.  The question of whether
  there is such an extension for $h=0$, $t>1$ is, as far as we know,
  open.
\end{remark}

\subsection{The thickened diagram}\label{sec:thicken}
Fix a small category $\Dat$, which we regard as a strict 2-category
whose only 2-morphisms are identity maps. Fix also a strictly unitary,
lax $2$-functor $F\co \Dat\to \BurnsideCat$, i.e., a $\Dat$-diagram in
$\BurnsideCat$. In this section, we will associate to $(\Dat,F)$ a new
(1-)category $\thic{\Dat}$ and, for each $k\geq1$, an (honest) functor
$\thicf[k]{F}\co \thic{\Dat}\to \BSpaces$, the category of based topological spaces. There will also be natural
transformations $\Sigma\circ \thicf[k]{F}\to \thicf[k+1]{F}$ (where
$\Sigma$ denotes suspension), so that we get an induced diagram
$\thicf{F}\co \thic{\Dat}\to\Spectra$, the category of symmetric spectra. To realize a functor from the
cube category to the Burnside category we will apply this construction
and then take an iterated mapping cone; see \Section{subsec-realize}.

We start by defining $\thic{\Dat}$:
\begin{definition}
  Let $\Dat$ be a small category. The \emph{thickening} of $\Dat$ is
  the small category $\thic{\Dat}$ defined as follows:
  \begin{itemize}
  \item The objects of $\thic{\Dat}$ are composable pairs of morphisms
    $u\stackrel{f}{\longrightarrow} v\stackrel{g}{\longrightarrow} w$
    in $\Dat$.
  \item The morphisms in $\thic{\Dat}$ are commutative diagrams: given
    composable pairs
    $u\stackrel{f}{\longrightarrow} v\stackrel{g}{\longrightarrow} w$
    and
    $u'\stackrel{f'}{\longrightarrow} v'\stackrel{g'}{\longrightarrow}
    w'$,
    \[
      \Hom((f,g),(f',g'))=\Bigl\{ (\alpha\co u\to u',\beta\from v'\to
      v,\gamma\from w\to w')\mid \mathcenter{ \xymatrix{
          u\ar[r]^f\ar[d]^\alpha & v\ar[r]^g & w\ar[d]^\gamma\\
          u'\ar[r]^{f'} & v'\ar[u]_{\beta}\ar[r]^{g'} & w' }} \text{
        commutes} \Bigr\}.
    \]
    (Note the direction of the middle vertical arrow.)
  \item Composition of morphisms is given by stacking diagrams
    vertically:
    $(\alpha',\beta',\gamma')\circ
    (\alpha,\beta,\gamma)=(\alpha'\circ\alpha,\beta\circ\beta',\gamma'\circ\gamma)$.
  \end{itemize}
\end{definition}

\begin{example}\label{exam:thicken-cat}
  For $\Dat=\CCat{1}$, $\thic{\Dat}$ has four objects: $1\to 1\to 1$,
  $1\to 1\to 0$, $1\to 0\to 0$ and $0\to0\to0$. There are unique
  morphisms
  \[
  (1\to 1\to 1)\longrightarrow(1\to 1\to 0)\longleftarrow (1\to 0\to 0)\longrightarrow (0\to0\to0).
  \]
  (Again, note the direction of the middle arrow.)
\end{example}

The thickening operation respects products:
\begin{lemma}\label{lem:thick-cat-prod}
  Given small categories $\Cat$ and $\Dat$, there is an isomorphism
  $q\co
  \thic{\Cat}\times\thic{\Dat}\stackrel{\cong}{\longrightarrow}\thic{\Cat\times\Dat}$
  given on objects by
  \[
    q\bigl( [u_C\stackrel{f_C}{\to} v_C \stackrel{g_C}{\to} w_C]\times 
    [u_D\stackrel{f_D}{\to} v_D \stackrel{g_D}{\to} w_D]\bigr) =
    [u_C\times u_D\stackrel{f_C\times f_D}{\longrightarrow} v_C\times
    v_D \stackrel{g_C\times g_D}{\longrightarrow}w_C\times w_D].
  \]
\end{lemma}
\begin{proof}
  This is immediate from the definitions.
\end{proof}

Next we define $\thicf[k]{F}$. 
\begin{definition}\label{def:thic-k-F}
  Let $\Dat$ be a small category, which we regard as a $2$-category
  with only identity $2$-morphisms and let $F\co\Dat\to \BurnsideCat$ be a
  strictly unitary, lax $2$-functor. Given an integer $k\geq 0$ we define a new functor
  $\thicf[k]{F}\co \thic{\Dat}\to \BSpaces$ (of $1$-categories) as follows.
  
  On objects, let
  \[
    \thicf[k]{F}(u\stackrel{f}{\longrightarrow}
    v\stackrel{g}{\longrightarrow} w) = \bigvee_{a\in
      F(f)}\prod_{\substack{b\in F(g)\\ s(b)=t(a)}}S^k.
  \]
  (Here, $s\co F(g)\to F(v)$ and $t\co F(f)\to F(v)$ are the source and
  target maps of the correspondences $F(g)$ and $F(f)$, respectively.)

  To define $\thicf[k]{F}$ on morphisms fix a commutative diagram
  \begin{equation}\label{eq:thic-morphism}
    \mathcenter{
      \xymatrix{
        u\ar[r]^f\ar[d]^\alpha & v\ar[r]^g & w\ar[d]^\gamma\\
        u'\ar[r]^{f'} & v'\ar[u]_{\beta}\ar[r]^{g'} & w'.
      }
    }
  \end{equation}
  We must construct a map
  \begin{equation}\label{eq:thic-F-source-targ}
    \bigvee_{a\in F(f)}\prod_{\substack{b\in F(g)\\ s(b)=t(a)}}S^k \to
    \bigvee_{a'\in F(f')}\prod_{\substack{b'\in F(g')\\
        s(b')=t(a')}}S^k.
  \end{equation}
  It suffices to construct this map one $a$ at a time, so fix
  $a\in F(f)$. The maps $F_{u,u',v'}$ and $F_{u,v',v}$ induce a
  bijection
  \[
    F(f)\cong F(\beta)\times_{F(v')}F(f')\times_{F(u')}F(\alpha);
  \]
  let $(y,a',x)\in F(\beta)\times_{F(v')}F(f')\times_{F(u')}F(\alpha)$
  be the triple corresponding to $a$. The map $\thicf[k]{F}$ will send
  $\prod_{\substack{b\in F(g)\\ s(b)=t(a)}}S^k$, the summand
  corresponding to $a$, to $\prod_{\substack{b'\in F(g')\\
      s(b')=t(a')}}S^k$, the summand corresponding to $a'$.

  Next, the maps $F_{v',v,w}$ and $F_{v',w,w'}$ induce a bijection
  \[
    F(g')\cong F(\gamma)\times_{F(w)}F(g)\times_{F(v)}F(\beta).
  \]
  Consider the map
  \begin{equation}\label{eq:diag-map}
    \prod_{\substack{b\in F(g)\\ s(b)=t(a)}}S^k
    \stackrel{\prod_b\Delta_b}{\longrightarrow}\prod_{\substack{b\in
        F(g)\\ s(b)=t(a)}}\prod_{\substack{b'=(z,\tilde{b},\tilde{y})\in F(g')\\ \tilde{b}=b\\
        \tilde{y}=y}}S^k\cong\prod_{\substack{b'\in F(g')\\ b'=(z,b,y)\\
        s(b)=t(a)}}S^k,
  \end{equation}
  where $\Delta_b$ is the diagonal map $S^k\to \prod_{b'}S^k$.  Notice
  that $\{b'\in F(g')\mid b'=(z,b,y),\ s(b)=t(a)\}$ is a subset of
  $\{b'\in F(g')\mid s(b')=t(a')\}$ since $s(b')=s(y)=t(a')$. We can
  extend the map~\eqref{eq:diag-map} to a map
  \[
    \prod_{\substack{b\in F(g)\\ s(b)=t(a)}}S^k\longrightarrow
    \prod_{\substack{b'\in F(g')\\ s(b')=t(a')}}S^k
  \]
  by mapping to the basepoint in the remaining factors. This is the
  desired map.

  (It can be helpful to think of this argument diagrammatically; see
  \Figure{build-thic-func}.)
\end{definition}

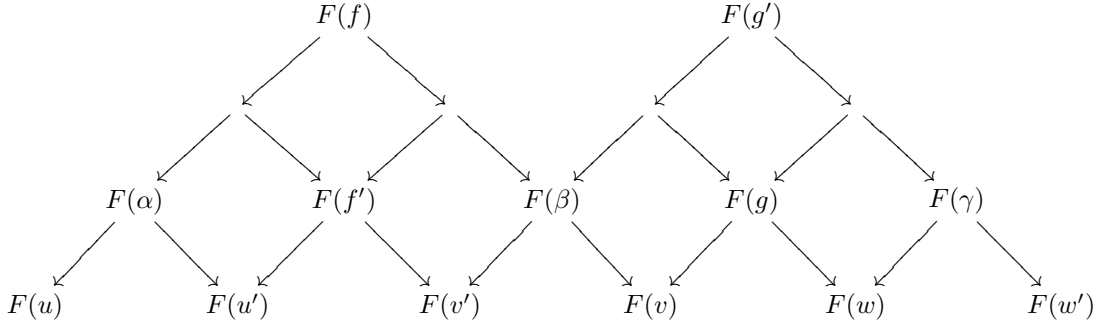
\begin{figure}
  \centering
  \[
  \xymatrix@C=1em{
    & & & F(f)\ar[dl]\ar[dr] & & & & F(g')\ar[dl]\ar[dr] & & \\
    & & \ar[dr]\ar[dl] & & \ar[dl]\ar[dr] & & \ar[dr]\ar[dl] & & \ar[dl]\ar[dr]\\
    & F(\alpha)\ar[dl]\ar[dr] & & F(f')\ar[dl]\ar[dr] & & F(\beta)\ar[dl]\ar[dr] & & F(g)\ar[dl]\ar[dr] & & F(\gamma)\ar[dl]\ar[dr] & \\
    F(u) & & F(u') & & F(v') & & F(v) & & F(w) & & F(w') 
  }
  \]
  \caption{\textbf{Constructing the functor from the thickened
      diagram.} In the first row of arrows, leftwards arrows are
    source maps, and rightwards arrows are target maps. All squares
    are fiber products. To define the map~\eqref{eq:diag-map}, fix
    $a\in F(f)$ and $b\in F(g)$. The element $a$ specifies $a'\in
    F(f')$ and $y\in F(\beta)$. We then consider the elements
    $b'\in F(g')$ which map down to $y$ and $b$. The diagonal map
    $\Delta_b$ corresponds to this set. Any such $b'$ has source equal
    to the target of $a'$.}
  \label{fig:build-thic-func}
\end{figure}

\begin{lemma}
  Definition~\ref{def:thic-k-F} defines a functor $\thicf[k]{F}$ whose
  values have natural actions of the symmetric group $\permu{k}$.
\end{lemma}
We omit the proof, which is straightforward, albeit a bit elaborate.

\begin{example}\label{exam:thicken-functor}
  Consider the functor $F\co \CCat{1}\to \BurnsideCat$ given by
  $F(1)=\{x,y\}$, $F(0)=\{z,w\}$, and $F(\cmorph{1}{0})=\{a,b,c,d\}$ with
  $s(a)=x$, $s(b)=s(c)=s(d)=y$, $t(a)=t(b)=t(c)=z$,
  $t(d)=w$. Graphically, $F$ is given by:
  \[
  \xymatrix{
    1: & x \ar[d]^{a}&& y\ar[dll]_{b,c}\ar[d]^d\\
    0: & z && w.
  }
  \]
  Recall the thickening $\thic{\CCat{1}}$ from
  \Example{thicken-cat}. The induced diagram $\thicf[k]{F}\co
  \thic{\CCat{1}}\to \BSpaces$ is given by
  \[
  S^k_{x,x}\vee S^k_{y,y}\stackrel{\Id\vee\Delta}\longrightarrow
  (S^k_{x,a})\vee(S^k_{y,b}\times S^k_{y,c}\times S^k_{y,d})\longleftarrow
  S^k_{a,z}\vee S^k_{b,z}\vee S^k_{c,z}\vee S^k_{d,w}\longrightarrow 
  S^k_{z,z}\vee S^k_{w,w}.
  \]
  (Here, for instance, the sphere $S^k_{x,a}$ corresponds to the pair
  $(x,a)\in F(\cmorph{1}{1})\times F(\cmorph{1}{0})$ over the object $1\to 1\to
  0$.)  We claim that the second map is the inclusion of the
  $k$-skeleton. To see this, note that it corresponds to the diagram
  \[
  \xymatrix{
    1\ar[r]^{\cmorph{1}{0}}\ar[d]^{\cmorph{1}{1}} & 0\ar[r]^{\cmorph{0}{0}} & 0\ar[d]^{\cmorph{0}{0}}\\
    1 \ar[r]^{\cmorph{1}{1}} & 1\ar[r]^{\cmorph{1}{0}}\ar[u]_{\cmorph{1}{0}} & 0.
  }
  \]
  The map decomposes along wedge sums. For $b$,
  for instance, we get the map $S^k_{b,z}\to (S^k_{y,b}\times
  S^k_{y,c}\times S^k_{y,d})$ as follows:
  \[
  \prod_{\{z\}}S^k=\left[\prod_{\substack{p\in F(\cmorph{0}{0})\\
        s(p)=t(b)}}S^k\right]\stackrel{\prod\Delta}{\longrightarrow}
  \left[\prod_{\substack{q'\in F(\cmorph{1}{0})\\q'=(r,p,b)\in
        F(\cmorph{0}{0})\times F(\cmorph{0}{0})\times F(\cmorph{1}{0})\\
        s(p)=t(b)}}\hspace{-1em}S^k\right]=\prod_{\{b\}}S^k \into
  \prod_{\{b,c,d\}}S^k.
  \]
  Similarly, the third map sends $S^k_{a,z}$, $S^k_{b,z}$ and
  $S^k_{c,z}$ by the identity map to $S^k_{z,z}$ and
  $S^k_{d,w}$ by the identity map to $S^k_{w,w}$.
\end{example}

Finally, we show that the $\thicf[k]{F}$ induce a functor
$\thicf{F}\co \thic{\Dat}\to \Spectra$.  Given a based space $X$,
identify $S^n \smas X$ as
$[0,1]^n\times X/\bigl(\bdy[0,1]^n \times X\cup
[0,1]^n\times\{*\}\bigr)$, where $*$ is the basepoint in $X$.  Given
$n\geq 0$ and based spaces $X_i$, there is a canonical map
\begin{align*}
  \sigma^n\co S^n \smas \prod_{i} X_i&\to \prod_{i}S^n \smas X_i,\\
  (y,x_1,\dots,x_n)&\mapsto((y,x_1),\dots,(y,x_n)).
\end{align*}
Given $k,n\geq 0$, define a map
\begin{equation}
  \label{eq:suspend-thicf}
  \eta(u\stackrel{f}{\to}
  v\stackrel{g}{\to} w)\co S^n \smas\thicf[k]{F}(u\stackrel{f}{\longrightarrow}
  v\stackrel{g}{\longrightarrow} w)
  \to \thicf[n+k]{F}(u\stackrel{f}{\longrightarrow}
  v\stackrel{g}{\longrightarrow} w)
\end{equation}
as the composition
\begin{equation}\label{eq:stab}
  \begin{split}
  S^n\smas \thicf[k]{F}(u\stackrel{f}{\longrightarrow}
v\stackrel{g}{\longrightarrow} w)=
S^n \smas \Bigl(\bigvee_{a\in F(f)}\prod_{\substack{b\in F(g)\\ s(b)=t(a)}}S^k\Bigr)
\cong&\bigvee_{a\in F(f)}S^n \smas\prod_{\substack{b\in F(g)\\ s(b)=t(a)}}S^k\\
\stackrel{\bigvee_{a\in F(f)}\sigma^n}{\longrightarrow} &\bigvee_{a\in F(f)}\prod_{\substack{b\in F(g)\\ s(b)=t(a)}}S^{n+k}
=\thicf[k+n]{F}(u\stackrel{f}{\longrightarrow}
v\stackrel{g}{\longrightarrow} w).
\end{split}
\end{equation}

\begin{lemma}\label{lem:suspend-thicf}
  The maps $\eta$ are $\permu{n} \times \permu{k}$-equivariant and
  form a natural transformation, in the sense that for each
  $u\stackrel{f}{\longrightarrow} v\stackrel{g}{\longrightarrow} w$
  the square
  \[
    \xymatrix{
      S^n\smas \thicf[k]{F}(u\stackrel{f}{\longrightarrow}
      v\stackrel{g}{\longrightarrow} w)\ar[rrrr]^-{S^n\smas\thicf[k]{F}
        \left(\begin{smallmatrix}
            u & \to & v &\to & w\\
            \downarrow & & \uparrow & & \downarrow\\
            u' & \to & v' &\to&w'
          \end{smallmatrix}\right)
      }\ar[d]_{\eta(u\stackrel{f}{\to} v\stackrel{g}{\to} w)} & & & &
      S^n\smas \thicf[k]{F}(u'\stackrel{f'}{\longrightarrow}
      v'\stackrel{g'}{\longrightarrow} w') \ar[d]^{\eta(u'\stackrel{f'}{\to} v'\stackrel{g'}{\to} w')}\\
      \thicf[k+n]{F}(u\stackrel{f}{\longrightarrow}
      v\stackrel{g}{\longrightarrow} w)\ar[rrrr]_-{\thicf[k+n]{F}      \left(\begin{smallmatrix}
            u & \to & v &\to & w\\
            \downarrow & & \uparrow & & \downarrow\\
            u' & \to & v' &\to&w'
          \end{smallmatrix}\right)
      } & & & &\thicf[k+n]{F}(u'\stackrel{f'}{\longrightarrow}
      v'\stackrel{g'}{\longrightarrow} w')
    }
  \]
  commutes. Further, each of the maps
  $\eta(u\stackrel{f}{\to} v\stackrel{g}{\to} w)$ induces an
  isomorphism on $\pi_i$ for $0\leq i\leq 2k-2$.
\end{lemma}
\begin{proof}
  The proof is straightforward, and is left to the reader.
\end{proof}

\begin{corollary}\label{cor:thicf-F-spectra}
  The functors $\thicf[k]{F}$ and natural transformations $\eta$ assemble to
  define a diagram of symmetric spectra $\thic{F}\co \thic{\Dat}\to \Spectra$.
\end{corollary}

\begin{remark}
  Lemma~\ref{lem:suspend-thicf} amounts to a verification that we can
  more concisely express $\thic{F}$ within the category of spectra
  itself by the formula
  \[
  \thic{F}(u\stackrel{f}{\longrightarrow} v\stackrel{g}{\longrightarrow} w)
  = \bigvee_{a\in F(f)}\prod_{\substack{b\in F(g)\\ s(b)=t(a)}}\SphereS.
  \]
  The wedge and product play the roles of coproduct and product in
  this category.
\end{remark}

\begin{remark}\label{rem:more-thickenings}
  The construction of the thickened functor
  $\thicf[k]{F}\co \thic{\Dat}\to \BSpaces$ from a functor
  $F\co\Dat\to\BurnsideCat$ would work with any based space $X_k$ in place of
  the sphere $S^k$: the construction only uses properties of the wedge sum and
  product operations. Further, if the spaces $X_k$ form a sequential spectrum
  then there is a corresponding thickened diagram in spectra. However, this
  construction is scarcely more general than the one we gave: in spectra, the
  diagram coming from $(X_k)$ is naturally equivalent to the smash product of
  the spectrum $X$ and the diagram obtained from the sphere spectrum
  $(S^k)$.
\end{remark}

\subsection{The realization}\label{sec:subsec-realize}

\begin{construction}
  \label{construction:realize-functor}
  Let $\thic{\CCat{n}}_+$ be the category obtained from $\thic{\CCat{n}}$ by adding
  a new object $*$ and a unique morphism $(u\to v\to w)\to *$ from
  each vertex $(u\to v\to w)$ of $\thic{\CCat{n}}$ with $w\neq \vec{0}$. 

  Given a functor $F\co \CCat{n}\to\BurnsideCat$, define
  $\thicf{F}^+\co \thic{\CCat{n}}_+\to\Spectra$ by
  $\thicf{F}^+|_{\thic{\CCat{n}}}=\thicf{F}$ and
  $\thicf{F}^+(*)=\{\pt\}$, a single point.

  Let $\Realize{F}$ be the homotopy colimit of $\thicf{F}^+$. We call
  $\Realize{F}$ the \emph{realization of $F$}.
\end{construction}

\begin{example}
  Continuing with \Example{thicken-functor}, we have
  \[
  \Realize{F}=\hocolim\left(
    \mathcenter{\xymatrix{
      \ar[d]\SphereS_{x,x}\vee \SphereS_{y,y}\ar[r]&
      (\SphereS_{x,a})\vee(\SphereS_{y,b}\times \SphereS_{y,c}\times \SphereS_{y,d}) &
      \SphereS_{a,z}\vee \SphereS_{b,z}\vee \SphereS_{c,z}\vee \SphereS_{d,w}\ar[l]\ar[r]& 
      \SphereS_{z,z}\vee \SphereS_{w,w}\\
      \{\pt\} & & & 
    }}
  \right),
  \]
  where $\SphereS$ denotes the sphere spectrum.
\end{example}

Instead of $\thic{\CCat{n}}_+$, it will be convenient, sometimes, to
work with the larger enlargement
$\thic{\CCat{n}}_{\othplus}=\left(\thic{\CCat{1}}_+\right)^{n}$ of
$\thic{\CCat{n}}$.  Given a functor $F\co \CCat{n}\to\BurnsideCat$
extend $\thicf{F}$ to a functor $\thicf{F}^{\othplus}\co
\thic{\CCat{n}}_{\othplus}\to\Spectra$ by setting
$\thicf{F}^{\othplus}|_{\thic{\CCat{n}}}=\thicf{F}$ and
$\thicf{F}^{\othplus}(d)=\{\pt\}$ if $d\not\in\Ob(\thic{\CCat{n}})$.

\begin{lemma}\label{lem:bigger-thic-plus}
  For any functor $F\co\CCat{n}\to\BurnsideCat$ there is a stable homotopy equivalence
  $
  \hocolim \thicf{F}^+\simeq \hocolim\thicf{F}^{\othplus}.
  $
\end{lemma}
\begin{proof}
  Consider the functor $G\co
  \thic{\CCat{n}}_{\othplus}\to\thic{\CCat{n}}_+$ which is the
  identity on $\thic{\CCat{n}}$ and sends all objects not in
  $\thic{\CCat{n}}$ to $*$. We claim that $G$ is homotopy cofinal. To
  see this, we divide the computation of the undercategories into three cases:
  \begin{enumerate}[label=(\alph*)]
  \item\label{item:over-of-star} The undercategory $*\downarrow G$ of
    $*$ is the full subcategory of $\thic{\CCat{n}}_{\othplus}$
    spanned by the objects not in $\thic{\CCat{n}}$. The nerve of this
    category is homeomorphic to
    $\{\vec{x}=(x_1,\dots,x_n)\in[0,1]^n\mid \exists i \text{ such
      that }x_i=0\}$, which is contractible.
  \item For any object $d=(u\to v\to \vec{0})$ of $\thic{\CCat{n}}$,
    the undercategory $d\downarrow G$ of $d$ is the full subcategory of
    $\thic{\CCat{n}}$ of objects $d'$ for which there is a map $d\to
    d'$. The object $d$ is an initial object for this subcategory, so
    the nerve is contractible.
  \item Fix an object $d=(u\to v\to w)$ of $\thic{\CCat{n}}$ with
    $w\neq\vec{0}$.  The undercategory $d\downarrow G$ of $d$ is the
    full subcategory of $\thic{\CCat{n}}_{\othplus}$ spanned by the
    objects $d'$ not in $\thic{\CCat{n}}$ and the objects $d'$ in
    $\thic{\CCat{n}}$ for which there is a map $d\to d'$. Let $\Dat$
    be the full subcategory of $\thic{\CCat{n}}_{\othplus}$ spanned by
    the objects $d'$ not in $\thic{\CCat{n}}$ and let $\Eat$ be the
    full subcategory of $\thic{\CCat{n}}_{\othplus}$ consisting of
    objects $d'$ for which there is a morphism $d\to d'$. The
    categories $\Dat$ and $\Eat$ are each downwards closed in
    $d\downarrow G$, i.e., there are no morphisms out of $\Dat$ or
    $\Eat$. Thus, the nerve of $d\downarrow G$ is the union of the
    nerves of $\Dat$ and $\Eat$, glued along the nerve of
    $\Dat\cap\Eat$. We already saw in part~\ref{item:over-of-star}
    that the nerve of $\Dat$ is contractible. The category $\Eat$ has
    an initial object, $d$, and hence the nerve of $\Eat$ is
    contractible.
    Finally, the realization of the category
    $\Dat\cap\Eat$ is similar to the realization of $\Dat$ (i.e., a
    union of coordinate hyperplanes), and so contractible. It follows
    that the realization of $d\downarrow G$ is contractible.
  \end{enumerate}

  Thus, the functor $G$ is homotopy cofinal.  It is immediate from the
  definitions that $\thicf{F}^{\othplus}=\thicf{F}^+\circ G$, so the
  result follows from property~\ref{item:hocolim-cofinal} of homotopy
  colimits (\Section{colimit}).%
\end{proof}

\subsection{An invariance property of the realization}
\label{sec:thick-invariance}
\begin{lemma}\label{lem:iso-realize-same}
  If $F,G\co \CCat{n}\to\BurnsideCat$ are naturally isomorphic
  $2$-functors then the realizations of $F$ and $G$ are stably
  homotopy equivalent.
\end{lemma}
\begin{proof}
  A natural isomorphism $T$ from $F$ to $G$ specifies:
  \begin{itemize}
  \item A bijection $T_v\co F(v)\to G(v)$ for each $v\in\{0,1\}^n$, and
  \item A bijection $T_{v,w}\co F(\cmorph{v}{w})\to G(\cmorph{v}{w})$ for each $v>w\in\{0,1\}^n$
  \end{itemize}
  satisfying the conditions that
  \[
  \xymatrix{
    & F(\cmorph{v}{w})\ar[dl]^{s}\ar[dr]^{t}\ar[rrr]^{T_{v,w}} & & & G(\cmorph{v}{w})\ar[dl]^{s}\ar[dr]^{t} & \\
    F(v)\ar@/_2pc/[rrr]_{T_v} & & F(w)\ar@/_2pc/[rrr]_{T_w} & G(v) & & G(w)
  }
  \]
  and
  \[
  \xymatrix{
    F(\cmorph{v}{w})\times_{F(v)}F(\cmorph{u}{v})\ar[r]^-{F_{u,v,w}}
    \ar[d]^{T_{v,w}\times T_{u,v}} & F(\cmorph{u}{w})\ar[d]^{T_{u,w}}\\
    G(\cmorph{v}{w})\times_{G(v)}G(\cmorph{u}{v}) \ar[r]_-{G_{u,v,w}}& G(\cmorph{u}{w})
  }
  \]
  commute.
  (See~\cite[Section I,2.4]{Gray-other-categories} and note that an
  isomorphism in $\BurnsideCat$ between two sets induces a bijection
  between them.)
  Given a natural transformation $T$, there is a corresponding map of
  diagrams defined as follows. Given $u>v>w$ we want a map 
  \[
  \bigvee_{a\in F(\cmorph{u}{v})}\prod_{\substack{b\in F(\cmorph{v}{w})\\s(b)=t(a)}}\SphereS
  \to
  \bigvee_{a\in G(\cmorph{u}{v})}\prod_{\substack{b\in G(\cmorph{v}{w})\\s(b)=t(a)}}\SphereS.
  \]
  This map sends the wedge summand corresponding to $a\in F(\cmorph{u}{v})$ to the wedge
  summand corresponding to $T_{u,v}(a)$, and sends the factor
  corresponding to $b\in F(\cmorph{v}{w})$ to the factor corresponding to
  $T_{v,w}(b)$. It is straightforward to verify that this gives a map
  of diagrams, which is by definition an
  isomorphism. It follows that the map induces a stable homotopy
  equivalence of homotopy colimits. (In fact, if we work with diagrams
  of $S^k$'s, the map would be a homeomorphism.)
\end{proof}

\subsection{Products and realization}\label{sec:products}
In the language of flow categories, the product is a rather
complicated object. In the language of functors from the cube to the
Burnside category, however, the product is quite simple:

\begin{definition}\label{def:product}
  There is a map $\times\co \BurnsideCat\times\BurnsideCat\to\BurnsideCat$ which
  sends a pair of objects $(X,Y)$ to their direct product $X\times Y$, a pair of
  correspondences $((A,s_A,t_A),(B,s_B,t_B))$ to
  $(A\times B,s_A\times s_B,t_A\times t_B)$, and a pair of isomorphisms of
  correspondences $(F\co A\to A',G\co B\to B')$ to the isomorphism
  $F\times G\co A\times B\to A'\times B'$.
  Given functors $F\co \CCat{m}\to \BurnsideCat$ and $G\co
  \CCat{n}\to\BurnsideCat$, we define the \emph{product} $F\times G$ of $F$ and
  $G$, to be the composition
  \[
    \CCat{m+n}\stackrel{\cong}{\longrightarrow} \CCat{m}\times\CCat{n}\stackrel{(F,G)}{\longrightarrow}\BurnsideCat\times\BurnsideCat\stackrel{\times}{\longrightarrow}\BurnsideCat.
  \]

  Explicitly, $(F\times G)\co \CCat{m+n}\to \BurnsideCat$ is given as follows. Identify $\{0,1\}^m\times\{0,1\}^n$ with $\{0,1\}^{m+n}$. Then:
  \begin{itemize}
  \item For $(u_1,u_2)\in \{0,1\}^m\times\{0,1\}^n=\{0,1\}^{m+n}$
    define 
    \[
    (F\times G)(u_1,u_2)=F(u_1)\times G(u_2).
    \]
  \item For $(u_1,u_2)>(v_1,v_2)\in\{0,1\}^{m+n}$ define $(F\times
    G)(\cmorph{(u_1,u_2)}{(v_1,v_2)})$ to be the correspondence
    \[
    \xymatrix{
      & F(\cmorph{u_1}{v_1})\times G(\cmorph{u_2}{v_2}) \ar[dl]_{s_F\times
        s_G}\ar[dr]^{t_F\times t_G}& \\
    F(u_1)\times G(u_2) & & F(v_1)\times G(v_2).
    }
    \]
  \item For $(u_1,u_2)>(v_1,v_2)>(w_1,w_2)\in \{0,1\}^{m+n}$ define
    $(F\times G)_{(u_1,u_2),(v_1,v_2),(w_1,w_2)}$
    by 
    \[
(F\times
    G)_{(u_1,u_2),(v_1,v_2),(w_1,w_2)}(x,y)=(F_{u_1,v_1,w_1}(x),
    G_{v_1,v_2,v_3}(y)).
    \]
  \end{itemize}
\end{definition}

\begin{lemma}
  \Definition{product} specifies a strictly unitary, lax $2$-functor.
\end{lemma}
\begin{proof}
  This is immediate.
\end{proof}

Note that smash products distribute across wedge sums. Moreover, while
$X\smas(Y\times Z)$ is not homotopy equivalent to $(X\smas Y)\times
(X\smas Z)$, there is a natural map $X\smas(Y\times Z)\to (X\smas Y)\times
(X\smas Z)$ defined by $(x,y,z)\mapsto ((x,y),(x,z))$. This
generalizes to a map 
\[
p\co \bigl(\prod_{a\in A}X_a\bigr)\smas\bigl(\prod_{b\in B}Y_b\bigr)\to
\prod_{(a,b)\in A\times B}X_a\smas Y_b.
\]
The map $p$ is natural in both factors.

Given functors $F\co \Cat\to\BSpaces$ and $G\co \Dat\to\BSpaces$ we
can take the smash product of $F$ and $G$ to obtain a functor $F\smas
G\co \Cat\times\Dat\to \BSpaces$ (cf.~\ref{item:hocolim-prod}). At a vertex $(c,d)$ of
$\Cat\times\Dat$, $(F\smas G)(c,d)=F(c)\smas G(d)$.

\begin{lemma}\label{lem:thicken-product}
  The thickening construction from \Section{thicken} respects
  products, in the following sense. Fix functors $F\co \CCat{m}\to
  \BurnsideCat$ and $G\co \CCat{n}\to\BurnsideCat$. Let $F\times G\co
  \CCat{m}\times\CCat{n}\to \BurnsideCat$ be as in \Definition{product}.
  Let $q\co \thic{\Cat}\times\thic{\Dat}\to \thic{\Cat\times\Dat}$ be the isomorphism from Lemma~\ref{lem:thick-cat-prod}.
  Then, the map $p$ induces a natural transformation from
  $\thicf[k]{F}\smas\thicf[l]{G}$ to $\thicf[k+l]{(F\times G)}\circ q$
  so that on vertices the natural transformation is a weak homotopy
  equivalence up to dimension $k+l+\min(k,l)-1$.  These
  natural transformations respect the spectrum structure, and so
  induce maps of diagrams $\thicf{F}\smas\thicf{G}\to \thicf{(F\times
    G)}\circ q$
  so that the map at each vertex is a stable homotopy equivalence.
\end{lemma}
\begin{proof}
  This is straightforward from the definitions. To illustrate, we
  describe the natural transformation at the level of vertices. At
  a vertex $[u_C\stackrel{f_C}{\to} v_C \stackrel{g_C}{\to} w_C]\times 
  [u_D\stackrel{f_D}{\to} v_D \stackrel{g_D}{\to} w_D]$ we have
  \[
  \thicf{F}(u_C\stackrel{f_C}{\to} v_C \stackrel{g_C}{\to} w_C)\smas
  \thicf{G}(u_D\stackrel{f_D}{\to} v_D \stackrel{g_D}{\to} w_D)
  =
  \bigl(\bigvee_{a_C\in F(f_C)}\prod_{\substack{b_C\in F(g_C)\\ s(b_C)=t(a_C)}}S^k\bigr)
  \smas
  \bigl(\bigvee_{a_D\in G(f_D)}\prod_{\substack{b_D\in G(g_D)\\ s(b_D)=t(a_D)}}S^l\bigr)  
  \]
  while
  \[
  \thicf[k+l]{(F\times G)}\bigl(u_C\times u_D\stackrel{f_C\times f_D}{\longrightarrow} v_C\times
  v_D \stackrel{g_C\times g_D}{\longrightarrow}w_C\times w_D\bigr)
  =
  \bigvee_{\substack{a_C\in F(f_C)\\a_D\in G(f_D)}}
  \prod_{\substack{b_C\in F(g_C)\\b_D\in G(g_D)\\ s(b_C)=t(a_C)\\ s(b_D)=t(a_D)}}S^{k+l}.
  \]
  The map $p$ sends the first to the second in an obvious way, and is
  an equivalence up to dimension $k+l+\min(k,l)-1$. Moreover, this map
  is $\permu{k}\times\permu{l}$-equivariant and respects the
  structure maps of the spectrum.
\end{proof}

\begin{proposition}\label{prop:realize-product}
  Given functors $F\co \CCat{n}\to\BurnsideCat$ and $G\co
  \CCat{m}\to\BurnsideCat$, we have $\Realize{F\times G}\simeq
  \Realize{F}\smas\Realize{G}$.
\end{proposition}
\begin{proof}
  Let $\simeq$ denote stable homotopy equivalence.  By \Lemma{bigger-thic-plus},
  \begin{align*}
    \Realize{F}&\simeq \hocolim \thicf{F}^\othplus & \Realize{G}&\simeq \hocolim\thicf{G}^\othplus & \Realize{F\times G}&\simeq\hocolim(\thicf{F\times G})^\othplus.
  \end{align*}
  From the definitions, it follows that the natural
  transformation from $\thicf{F}\smas\thicf{G}$ to $\thicf{(F\times
    G)}$ of \Lemma{thicken-product} extends to a natural
  transformation from $\thicf{F}^\othplus \smas \thicf{G}^\othplus$ to
  $(\thicf{F\times G})^\othplus$, which is a stable homotopy
  equivalence on objects.  By point~\ref{item:hocolim-nat-trans}
  in \Section{colimit},
  \[
  \hocolim (\thicf{F\times G})^\othplus\simeq \hocolim (\thicf{F}^\othplus \smas \thicf{G}^\othplus).
  \]
  By point~\ref{item:hocolim-prod}
  in \Section{colimit},
\[
    \hocolim \left(\thicf{F}^\othplus\smas\thicf{G}^\othplus\right) \simeq \left(\hocolim \thicf{F}^\othplus\right)\smas\left(\hocolim\thicf{G}^\othplus\right).
\]
  The result follows.
\end{proof}

An even simpler operation is disjoint union:
\begin{definition}\label{def:disjoint-union}
  There is a map $\amalg\co \BurnsideCat\times\BurnsideCat\to\BurnsideCat$ which
  sends a pair of objects $(X,Y)$ to their disjoint union $X\amalg Y$, a pair of
  correspondences $((A,s_A,t_A),(B,s_B,t_B))$ to
  $(A\amalg B,s_A\amalg s_B,t_A\amalg t_B)$, and a pair of isomorphisms of
  correspondences $(F\co A\to A',G\co B\to B')$ to the isomorphism
  $F\amalg G\co A\amalg B\to A'\amalg B'$.
  Given functors $F,G\co \CCat{n}\to \BurnsideCat$, we define the \emph{disjoint union} $F\amalg G$ of $F$ and
  $G$, to be the composition
  \[
    \CCat{n}\stackrel{\Delta}{\longrightarrow} \CCat{n}\times\CCat{n}\stackrel{(F,G)}{\longrightarrow}\BurnsideCat\times\BurnsideCat\stackrel{\amalg}{\longrightarrow}\BurnsideCat.
  \]
  where the first arrow is the diagonal map.

  Explicitly, $(F\amalg G)\co \CCat{n}\to \BurnsideCat$ is given as follows:
  \begin{itemize}
  \item For $v\in\{0,1\}^n$, $(F\amalg G)(v)=F(v)\amalg G(v)$.
  \item For $v>w\in\{0,1\}^n$, 
    $(F\amalg G)_{v,w}(\cmorph{v}{w})$ is defined to be the correspondence
    \[
    \xymatrix{ & F(\cmorph{v}{w})\amalg G(\cmorph{v}{w})\ar[dl]_{s_F\amalg
        s_G}\ar[dr]^{t_F\amalg t_G}& \\
      F(v)\amalg G(v) & & F(w)\amalg G(w).  }
    \]
  \item For $u>v>w\in\{0,1\}^n$, $(F\amalg G)_{u,v,w}$ is defined by
    the commutative diagram
    \[
    \xymatrix{
    [F(\cmorph{v}{w})\amalg G(\cmorph{v}{w})]\times_{F(v)\amalg
      G(v)}[F(\cmorph{u}{v})\amalg G(\cmorph{u}{v})] \ar[rr]^-{(F\amalg G)_{u,v,w}}
    \ar[d]_\cong & & F(\cmorph{u}{w})\amalg G(\cmorph{u}{w}),\\
    [F(\cmorph{v}{w})\times_{F(v)}F(\cmorph{u}{v})]\amalg
    [G(\cmorph{v}{w})\times_{G(v)}G(\cmorph{u}{v})]\ar[urr]_{F_{u,v,w}\amalg
      G_{u,v,w}} 
    }
    \]
    where the vertical arrow is the obvious bijection.
  \end{itemize}
\end{definition}

\begin{lemma}
  \Definition{disjoint-union} specifies a strictly unitary, lax $2$-functor.
\end{lemma}
\begin{proof}
  This is immediate from the definitions.
\end{proof}

Given diagrams $F,G\co \Cat\to \BSpaces$ there is an induced diagram
$F\vee G\co \Cat\to\BSpaces$ with $(F\vee G)(v)=F(v)\vee G(v)$ and
$(F\vee G)(f)=F(f)\vee G(f)$.

\begin{lemma}\label{lem:thicken-disjoint-union}
  The thickening construction respects disjoint unions in the sense
  that given $F,G\co \Cat\to \BurnsideCat$, $\thicf[k]{(F\amalg G)}\cong\thicf[k]{F}\vee \thicf[k]{G}$.
\end{lemma}
\begin{proof}
  Again, this is immediate from the definitions.
\end{proof}

\begin{proposition}\label{prop:realize-disjoint-union}
  Given functors $F,G\co \CCat{n}\to\BurnsideCat$, $\Realize{F\amalg
    G}\simeq \Realize{F}\vee\Realize{G}$.
\end{proposition}
\begin{proof}
  \Lemma{thicken-disjoint-union} extends immediately to the statement
  that $\thicf{F\amalg G}^+\cong\thicf{F}^+\vee \thicf{G}^+$, and homotopy
  colimit commutes with wedge sum (point~\ref{item:hocolim-sum}
  in \Section{colimit}).
\end{proof}

\section{Building a smaller cube from little box maps}\label{sec:smaller-cube}
In this section we show that the realization of a functor $\CCat{n}\to\BurnsideCat$ can be understood in terms of a smaller diagram.

\begin{definition}
  Let $\CCat{n}_+$ be the result of adding a single object, which we
  will denote $*$, to $\CCat{n}$ and declaring that
  \begin{align*}
  \Hom(v,*)&=
  \begin{cases}
    \text{one element $\cmorph{v}{*}$} & v\neq \vec{0}\\
    \emptyset & v=\vec{0}
  \end{cases} & 
  \Hom(*,v)&=
  \begin{cases}
    \{\Id\} & v=*\\
    \emptyset & v\neq *.
  \end{cases}
  \end{align*}  
\end{definition}
In other words, $\CCat{n}_+$ is the result of replacing the terminal object in
$\CCat{n}$ by two copies of itself, with no maps between them.

The goal of this section is to prove:
\begin{theorem}\label{thm:smaller-diag}
  Fix a strictly unitary, lax $2$-functor $F\co \CCat{n}\to\BurnsideCat$. Then there is a
  homotopy coherent diagram $\SpDiag{F}\co \CCat{n}_+\to\Spectra$ so that:
  \begin{enumerate}
  \item\label{item:verts-to-spheres} For each $v\in\Ob(\CCat{n})$, $\SpDiag{F}(v)=\bigvee_{a\in F(v)}\SphereS$.
  \item\label{item:star-to-pt} $\SpDiag{F}(*)$ is a single point.
  \item\label{item:same-hocolim} $\hocolim{\SpDiag{F}}\simeq \Realize{F}$.
  \end{enumerate}
\end{theorem}

The diagram $\SpDiag{F}$ is a special case of a more general
construction, which we give in \Section{box-refinement}. We specialize to
$\SpDiag{F}$ in \Section{coherent-box-maps}, and prove that this
homotopy colimit is equivalent to $\Realize{F}$
in \Section{small-is-big}. Before turning to these constructions, we
introduce a particular class of maps between spheres
in \Section{box-maps}.

\subsection{Box maps}\label{sec:box-maps}

\begin{definition}
  By a \emph{box in $\RR^k$} we mean a subset of the form
  $B=[a_1,b_1]\times\cdots\times[a_k,b_k]$ for some $a_1,\dots,b_k$
  with $a_i<b_i$. Given a box $B$ in $\RR^k$, a \emph{sub-box} $B'$ of
  $B$ is a box $B'$ in $\RR^k$ so that $B'\subset B$.
\end{definition}

\begin{figure}
  \centering
  \includegraphics{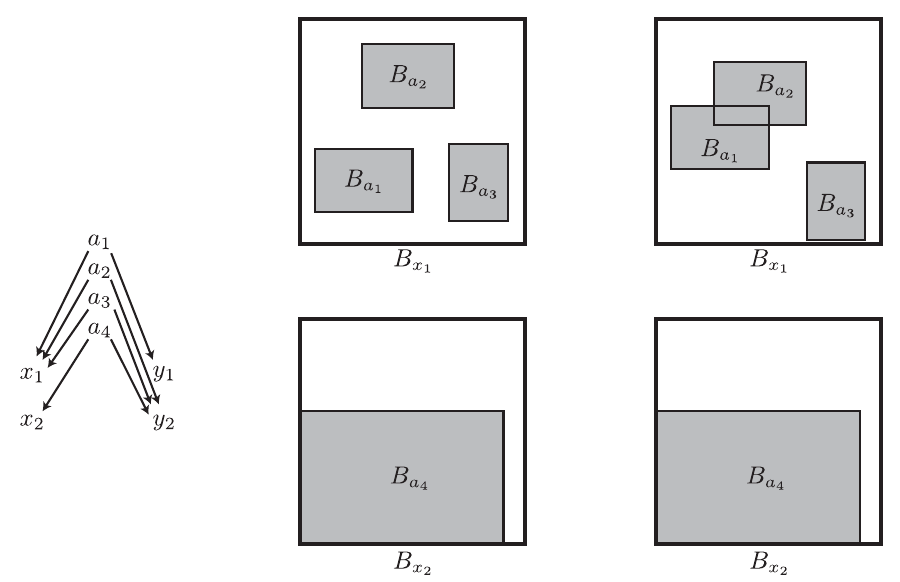}
  \caption{\textbf{Spaces of boxes.} Left: A correspondence $A$ from a
    $2$-element set $X$ to another $2$-element set $Y$. Center: A
    point in $E(\{B_x\},s)$. Right: A point in $F(\{B_x\},s,t)$ which is
    not in $E(\{B_x\},s)$. The boxes $B_{a_i}$ are shaded.}
  \label{fig:box-space}
\end{figure}

\begin{definition}
  Given a box $B$ in $\RR^k$, let $E(B,\ell)$ be the space of ordered
  $\ell$-tuples of disjoint boxes in $B$, topologized as a subspace of
  $(\RR^{2k})^\ell$.

  Given a correspondence $(A,s\co A\to X, t\co A\to Y)$, consider a
  collection $\{B_x\mid x\in X\}$ of distinct copies of the standard
  unit box. 
  Let $F(\{B_x\},s,t)$ be the space of collections
  $\{B_a\subset B_{s(a)}\mid a\in A\}$ of sub-boxes labeled by
  elements of $A$, satisfying the condition that
  $B_a\cap B_{a'}=\emptyset$ if $a\neq a'$ and $s(a)=s(a')$ and
  $t(a)=t(a')$. See \Figure{box-space}.

  Let $E(\{B_x\},s)\subset F(\{B_x\},s,t)$ be the space of collections
  $\{B_a\subset B_{s(a)}\mid a\in A\}$ satisfying the stronger condition that
  $B_a\cap B_{a'}=\emptyset$ if $a\neq a'$ and $s(a)=s(a')$.
  Again, see \Figure{box-space}.
\end{definition}

\begin{example}
  If $A$ is an $\ell$-element correspondence and $X$ has one element
  $x$ then $E(\{B_x\},s)=E(B_x,\ell)$.
\end{example}

\begin{definition}
  Fix once and for all an identification $S^k=[0,1]^k/\bdy$. When we
  view $S^k$ as a pointed space, we assume that this identification
  identifies the basepoint and $\bdy [0,1]^k$. For any box $B$ in
  $\RR^k$, there is a standard homeomorphism $B\to [0,1]^k$, given by
  $(x_1,\dots,x_k)\mapsto(\frac{x_1-a_1}{b_1-a_1},\dots,
  \frac{x_k-a_k}{b_k-a_k})$, an hence an induced identification
  $S^k\cong B/\bdy B$.

  Given $\{B_1,\dots,B_\ell\}\in E(B,\ell)$ the associated \emph{basic
    box map} is the composition
  \begin{equation}\label{eq:box-map-simple}
    S^k = B/\bdy B \to B / \bigl(B\setminus (\interior{B_1}\cup\cdots\cup\interior{B_\ell})\bigr) 
    =\bigvee_{a=1}^\ell B_a/\bdy B_a =\bigvee_{a=1}^\ell S^k\to S^k,
  \end{equation}
  where the last map is the identity on each summand (so the
  composition has degree $\ell$). This defines a continuous map
  $E(B,\ell)\to \Map(S^k,S^k)$.
\end{definition}

\begin{remark}
  The map $E(B,\ell)\to \Map(S^k,S^k)$ is the composition of the
  coaction of the little $k$-cubes operad with the fold map.
\end{remark}

\begin{definition}
  Given a correspondence $A$ from $X$ to $Y$ and a collection of boxes
  $e=\{B_a\subset B_{s(a)}\mid a\in A\}\in E(\{B_x\},s)$, there is an
  induced map
  \begin{equation}\label{eq:box-map-wedge-to-wedge}
    \Phi(e,A)\co \bigvee_{x\in X}S^k \to \bigvee_{y\in Y}S^k.
  \end{equation}
  defined by
  \begin{equation}\label{eq:auxiliary-box-map}
    \Phi_{e,A}|_{S^k_x}\co S^k_x = B_x/\bdy B_x \to B_x / \bigl(B_x\setminus (\bigcup_{\substack{a\in A\\s(a)=x}}\interior{B_a})\bigr) 
    =\bigvee_{\substack{a\in A\\s(a)=x}} B_a/\bdy B_a
    =\bigvee_{\substack{a\in A\\ s(a)=x}} S^k_a\to \bigvee_{y\in Y}S^k_y
  \end{equation}
  where the last map sends $S^k_a$ by the identity map to
  $S^k_{t(a)}$. This defines a continuous map
  $\Phi(\cdot,A)\co E(\{B_x\},s)\to \Map(\bigvee_{x\in
    X}S^k,\bigvee_{y\in Y}S^k)$.  We call a map of the form
  $\Phi(e,A)$ for some $e\in E(\{B_x\},s)$ a \emph{disjoint box
    map}, and say it \emph{refines the correspondence $(A,s,t)$}.
\end{definition}

\begin{example}
  The basic box maps are exactly the disjoint box maps where $X$ and
  $Y$ each consists of a single element.
\end{example}

\begin{definition}
  Given a correspondence $A$ from $X$ to $Y$ and a collection of boxes
  $e=\{B_a\subset B_{s(a)}\mid a\in A\}\in F(\{B_x\},s,t)$, there is an
  induced map
  \begin{equation}\label{eq:box-map-wedge-to-prod}
    \bigvee_{x\in X}S^k\to \prod_{y\in Y}S^k
  \end{equation}
  whose component sending $S^k_x$ to $S^k_y$ is the basic box map
  associated to $(B_x,\{B_a\mid s(a)=x,t(a)=y\})$. We call a map
  arising this way an \emph{overlapping box map}, and say it
  \emph{refines the correspondence $(A,s,t)$}.

  By a \emph{box map} we mean either a disjoint box map or an
  overlapping box map.
\end{definition}

\begin{example}\label{exam:disjoint-overlap-box}
  Given $e\in E(\{B_x\},s)\subset F(\{B_x\},s,t)$, the overlapping box
  map associated to $e$ is obtained by composing the disjoint box map
  $\Phi(e,A)$ with the standard inclusion
  $\bigvee_{y\in Y}S^k\into \prod_{y\in Y}S^k$.
\end{example}

\begin{remark}
  We have not defined, and will not need, box maps from products, just
  box maps to products.
\end{remark}

\begin{example}\label{exam:diag-is-box}
  The diagonal map $S^k\to \prod_{i=1}^m S^k$ is a box map: take
  $X=\{x\}$ to have a single element, $A$ and $Y$ to have $m$ elements
  each, $t$ to be a bijection, and each box $B_a$ to be all of $B_x$.
\end{example}

\begin{lemma}\label{lem:box-maps-compisition}
  A composition of box maps is a box map.
\end{lemma}
\begin{proof}
  Given composable box maps $F$ and $G$, the preimages under $F$ of
  the boxes for $G$ are the boxes for $G\circ F$.
\end{proof}

The value to us of these constructions comes from the following:
\begin{lemma}\label{lem:box-maps-contractible}
  If $B$ is $k$-dimensional then the space $E(B,\ell)$ is
  $(k-2)$-connected. More generally, $E(\{B_x\},s)$ and
  $F(\{B_x\},s,t)$ are $(k-2)$-connected.
\end{lemma}
\begin{proof}
  The space $E(B,\ell)$ is homotopy equivalent to the ordered
  configuration space of $\ell$ points in the interior of $B$, i.e., $B^\ell\setminus
  \Delta$, where $\Delta$ is the fat diagonal. Since $\Delta$ is a finite
  union of smooth submanifolds of codimension $k$, the result
  for $E(B,\ell)$ follows. Next, $E(\{B_x\},s)\cong \prod_{x\in
    X}E([0,1]^k,|s^{-1}(x)|)$ is a product of $(k-2)$-connected spaces, and
  hence is $(k-2)$-connected. Finally, $F(\{B_x\},s,t)\cong
  \prod_{(x,y)\in X\times Y}E([0,1]^k,|s^{-1}(x)\cap t^{-1}(y)|)$,
  and so again is $(k-2)$-connected.
\end{proof}

Informally, \Lemma{box-maps-contractible} says that the space of box
maps is highly connected. (For this statement to be correct, we must
think of a box map as not just the map but also the labeling of the
boxes.)  Indeed, we will show next that the space of box maps is also
highly connected in relative terms.

\begin{definition}
  For any set map $s\from A\to X$, define $E^\circ(\{B_x\},s)$ to be
  the subspace of $E(\{B_x\},s)$ where we require the box $B_a$ to lie
  in the interior of the box $B_{s(a)}$ for all $a\in A$.
\end{definition}

If $A_0\subset A$ is a subset and $s_0\from A_0\to X$ is the
restriction of $s$ then there is a map
$E^\circ(\{B_x\},s)\to E^\circ(\{B_x\},s_0)$ gotten by forgetting the
boxes labeled by $A\setminus A_0$.

\begin{lemma}\label{lem:box-maps-rel-contractible}
  Fix a set map $s\co A\to X$ and a subset $A_0\subset A$. Let $s_0=s|_{A_0}$.
  If the boxes $\{B_x\}$ are $k$-dimensional, then for any $i\leq
  k-1$ and any commutative diagram
  \[
  \xymatrix{S^{i-1}\ar[r]\ar@{^(->}[d]& E^\circ(\{B_x\},s)\ar[d]\\
    D^i\ar[r]& E^\circ(\{B_x\},s_0),}
  \]
  there exists a lift $D^i\to E^\circ(\{B_x\},s)$ making the diagram commute. 
\end{lemma}

\begin{proof}
  By induction, we may assume that $A\setminus A_0$ consists of a
  single element, say $a_1$, and let $x_1=s(a_1)$ and
  $\ell=|s_0^{-1}(x_1)|$.  The projection $\pi\from
  E^\circ(\{B_x\},s)\to E^\circ(\{B_x\},s_0)$ is seen to be a
  fiber bundle by the following argument.

  Let $\pi_\bullet\from E_\bullet\to E^\circ(\{B_x\},s_0)$ be the
  fiber bundle where the fiber over a point is the complement of the
  $\ell$ boxes in the interior of $B_{x_1}$. It is easy to see that
  $\pi_\bullet$ is a fiber bundle as follows. For $z\in
  E^\circ(\{B_x\},s_0)$, fix a minimal triangulation of the complement
  of the $\ell$ boxes in $z$. This triangulation persists in some
  small neighborhood $U$ of $z$ in $E^\circ(\{B_x\},s_0)$, and induces
  a PL homeomorphism between $\pi_\bullet^{-1}(z)$ and
  $\pi_\bullet^{-1}(z')$ for any $z'\in U$.

  Now for any $z\in E^\circ(\{B_x\},s_0)$, construct a coordinate
  chart on $\pi^{-1}(z)$ by the following variables: the center $C$ of
  the box $B_{a_1}$ viewed as a point in $\pi_\bullet^{-1}(z)$; the
  `aspect ratio' $R$ of the box $B_{a_1}$, presented as a
  $(k-1)$-tuple of ratios of the $k$ side-lengths of $B_{a_1}$; and a
  proportion $P\in(0,1)$ of the volume of the box $B_{a_1}$ relative to
  the volume of the largest box with the same center and same aspect
  ratio that lies in $B_{x_1}$ in the complement of the interiors of
  other $\ell$ boxes. This identifies $\pi^{-1}(z)$ with
  $\pi_\bullet^{-1}(z)\times(0,\infty)^{k-1}\times(0,1)$, and the
  identification holds for small open sets around $z$. But
  $\pi_\bullet$ is a fiber bundle, and therefore, so is $\pi$.


  The fiber over each point is homeomorphic to the space of boxes in
  the complement of $\ell$ disjoint boxes in the interior of
  $B_{x_1}$. The fiber, being homotopy equivalent to the complement of
  $\ell$ points in $\R^k$, is $(k-2)$-connected, so the statement
  follows.
\end{proof}

\subsection{Refining diagrams via box maps}\label{sec:box-refinement}
\begin{definition}\label{def:spatial-refinement}
  Fix a small category $\Dat$ and a strictly unitary, lax $2$-functor
  $F\co \Dat\to \BurnsideCat$ (i.e., a $\Dat$-diagram in $\BurnsideCat$). A \emph{$k$-dimensional spatial
    refinement} of $F$ is a homotopy coherent diagram
  $\SpRefine[k]{F}\co \Dat\to \BSpaces$ so that
  \begin{itemize}
  \item For any $u\in\Ob(\Dat)$, $\SpRefine[k]{F}(u)=\bigvee_{x\in
      F(u)}S^k=\bigl(\amalg_{x\in F(u)}B_x\bigr)/\bdy$;
  \item For any $u,v\in\Ob(\Dat)$ and $f\co u\to v$,
    $\SpRefine[k]{F}(f)$ is a disjoint box map which refines the
    correspondence $F(f)$ from $F(u)$ to $F(v)$
    (see \Section{box-maps}); and, more generally,
  \item For any sequence of morphisms
    \[
    u_0\stackrel{f_1}{\longrightarrow} u_1 \stackrel{f_2}{\longrightarrow}
    \cdots\stackrel{f_n}{\longrightarrow} u_n
    \]
    in $\Dat$,
    \[
      \SpRefine[k]{F}(f_n,\dots,f_1)\co [0,1]^{n-1}\to\Map(\bigvee_{x\in F(u_0)}S^k,\bigvee_{x\in F(u_n)}S^k)
    \]
    is a family of box maps induced by a map $[0,1]^{n-1}\to
    E(\{B_x\}_{x\in F(u_0)},s_{F(f_n\circ\cdots\circ f_0)})$
    (refining the correspondence 
    \[
    F(f_n\circ\cdots\circ f_1)\cong F(f_n)\times_{F(u_{n-1})}\cdots\times_{F(u_1)}F(f_1)    
    \]
    from $F(u_0)$ to $F(u_n)$).    
  \end{itemize}
\end{definition}

\begin{proposition}\label{prop:box-refinement-exist-unique}
  Let $\Dat$ be a small category in which every sequence of composable
  non-identity morphisms has length at most $n$. Fix a $\Dat$-diagram
  $F$ in $\BurnsideCat$.
  \begin{enumerate}
  \item\label{item:spatial-exists} If $k\geq n$ then there is a $k$-dimensional spatial refinement of $F$.
  \item\label{item:spatial-unique} If $k\geq n+1$ then any two $k$-dimensional spatial refinements of
    $F$ are homotopic (as homotopy coherent diagrams).
  \item\label{item:spatial-suspend} If $\SpRefine[k]{F}$ is a $k$-dimensional spatial refinement
    of $F$ then the result of suspending each $\SpRefine[k]{F}(u)$ and
    $\SpRefine[k]{F}(f_n,\dots,f_1)(\vec{t})$ gives a
    $(k+1)$-dimensional spatial refinement of $F$.
  \end{enumerate}
\end{proposition}
\begin{proof}
  We start with point~(\ref{item:spatial-exists}).  Given
  $u\in\Ob(\Dat)$, define $\SpRefine[k]{F}(u)=\bigvee_{x\in F(u)}S^k$;
  write the $S^k$ summand corresponding to $x$ as $B_x/\bdy$, where
  $B_x$ is a box in $\RR^k$ (e.g., $B_x=[0,1]^k$). 
  Next, by \Observation{hocolim-no-id} it suffices to consider only non-identity morphisms.
  For each non-identity morphism
  $f\co u\to v$ in $\Dat$ choose a disjoint box map which refines the
  correspondence $F(f)$. Let $e_{f}\in E(\{B_x\mid x\in
  F(u)\},s_{F(f)})$ be the collection of little boxes corresponding to $F(f)$.

  We have now defined $\SpRefine[k]{F}$ on vertices and arrows. The
  diagram does not commute, so it remains to define the coherence
  homotopies associated to sequences of composable morphisms. We will
  build these inductively. As a warm up, we spell out the first
  case carefully before proceeding to the general case.

  Fix a composable pair of morphisms
  $u\stackrel{f}{\longrightarrow}v\stackrel{g}{\longrightarrow} w$ in
  $\Dat$.  There are two points in $E(\{B_x\mid x\in
  F(u)\},s_{F(g\circ f)})$ associated to $(g,f)$. One is the point
  $e_{g\circ f}$. The other is defined as follows. The point $e_{g}$
  corresponds to a collection of boxes $\mathcal{B}_g$ in $\{B_y\mid
  y\in F(v)\}$, labeled by elements of $F(g)$. The inverse image
  $\Phi(e_{f},F(f))^{-1}(\mathcal{B}_g)$ of these boxes is a
  collection of boxes in $\{B_x\mid x\in F(u)\}$. The boxes
  $\Phi(e_{f},F(f))^{-1}(\mathcal{B}_g)$ inherit a labeling by
  elements of $F(g)\times_{F(v)}F(f) \cong_{F_{u,v,w}}F(g\circ f)$. 
  This labeling makes $\Phi(e_{f},F(f))^{-1}(\mathcal{B}_g)$ into a second
  point in $E(\{B_x\mid x\in F(u)\},s_{F(g\circ f)})$, which by abuse
  of notation we will call $e_{g}\circ e_{f}$ (see also \Lemma{box-maps-compisition}).
  
  By \Lemma{box-maps-contractible}, since $k\geq 2$ (or else we would
  not have a composable pair $(g,f)$), the space $E(\{B_x\mid x\in
  F(u)\},s_{F(g\circ f)})$ is connected, so we can find a path from
  $e_{g\circ f}$ to $e_{g}\circ e_{f}$. Fix such a path, and call it
  $e_{g,f}\co [0,1]\to E(\{B_x\mid x\in F(u)\},s_{F(g\circ f)})$. Then
  $e_{g,f}$ defines a homotopy $\Phi(e_{g,f},F(g\circ f))$ from
   $\Phi(e_{g},F(g))\circ\Phi(e_{f},F(f))$ to $\Phi(e_{g\circ f},F(g\circ f))$.

  More generally, suppose that for any sequence
  $v_0\stackrel{f_1}{\longrightarrow}\cdots\stackrel{f_\ell}{\longrightarrow}v_{\ell}$
  of non-identity morphisms we have chosen a map
  $e_{f_\ell,\dots,f_1}\co [0,1]^{\ell-1}\to E(\{B_x\mid x\in
  F(v_0)\},s_{F(f_\ell\circ\cdots\circ f_1)})$, and these maps are
  compatible in the following sense.  Let $(t_1,\dots,t_{\ell-1})$ be
  the coordinates on $[0,1]^{\ell-1}$.  Then for any $1\leq i\leq
  \ell-1$ we require that:
  \begin{equation}\label{eq:bm-sp-compat}
    \begin{split}
      e_{f_\ell,\dots,f_1}(t_1,\dots,t_{i-1},0,t_{i+1},\dots,t_{\ell-1})
      &=e_{f_\ell,\dots,f_i}(t_{i+1},\dots,t_{\ell-1})\circ e_{f_{i-1},\dots,f_{1}}(t_1,\dots,t_{i-1}) \\
      e_{f_\ell,\dots,f_1}(t_1,\dots,t_{i-1},1,t_{i+1},\dots,t_{\ell-1})
      &=e_{f_\ell,\dots,f_i\circ f_{i-1},\dots,f_1}(t_1,\dots,t_{i-1},t_{i+1},\dots,t_{\ell-1}).
    \end{split}
  \end{equation}
  Then given
  $v_0\stackrel{f_1}{\longrightarrow}\cdots\stackrel{f_{\ell+1}}{\longrightarrow}v_{\ell+1}$
  there is a map $S^{\ell-1}=\bdy([0,1]^{\ell})\to E(\{B_x\mid x\in
  F(v_0)\},s_{F(f_{\ell+1}\circ\cdots\circ f_1)})$ defined by
  \begin{equation}\label{eq:map-to-extend}
    \begin{split}
      (t_1,\dots,t_{i-1},0,t_{i+1},\dots,t_{\ell})
      &\mapsto e_{f_{\ell+1},\dots,f_{i+1}}(t_{i+1},\dots,t_{\ell})\circ e_{f_i,\dots,f_{1}}(t_1,\dots,t_{i-1}) \\
      (t_1,\dots,t_{i-1},1,t_{i+1},\dots,t_{\ell})
      &\mapsto e_{f_{\ell+1},\dots,f_{i+1}\circ f_i,\dots,f_{1}}(t_1,\dots,t_{i-1},t_{i+1},\dots,t_{\ell}).
    \end{split}
  \end{equation}
  The inductive hypothesis implies that this map is continuous. Since
  $k\geq \ell+1$, by \Lemma{box-maps-contractible}, the space
  $E(\{B_x\mid x\in F(v_0)\},s_{F(f_{\ell+1}\circ\cdots\circ f_1)})$
  is $(\ell-1)$-connected, so the map~\eqref{eq:map-to-extend} extends
  to a map $[0,1]^{\ell}\to E(\{B_x\mid x\in
  F(v_0)\},s_{F(f_{\ell+1}\circ\cdots\circ f_1)})$. Define
  $e_{f_{\ell+1},\dots,f_1}$ to be any such extension.
  
  Now, \Equation{box-map-wedge-to-wedge} gives a map 
  \[
  \Phi(e_{v_0,\dots,v_{\ell+1}},F(f_{\ell+1}\circ\cdots\circ f_1))\co [0,1]^\ell\times \bigvee_{x\in F(v_0)}S^k\to \bigvee_{x\in F(v_{\ell+1})}S^k.
  \]
  It follows from the compatibility conditions~\eqref{eq:bm-sp-compat} that these maps define a
  homotopy coherent diagram.

  Next, for point~(\ref{item:spatial-unique}), fix spatial refinements
  $\SpRefine[k]{F}$ and $\SpRefine[k]{F}'$ of $F$.  Consider the
  category $\CCat{1}\times \Dat$. It suffices to define a homotopy
  coherent diagram $G\co \CCat{1}\times\Dat\to \BSpaces$ so that
  $G|_{\{0\}\times\Dat}=\SpRefine[k]{F}$,
  $G|_{\{1\}\times\Dat}=\SpRefine[k]{F}'$, and for any
  $u\in\Ob(\Dat)$, $G(\cmorph{1}{0}\times\Id_u)$ is a homotopy
  equivalence~\cite[Proposition 4.6]{Vogt-top-hocolim}.  To define
  $G$, note that $G|_{\{0\}\times\Dat}$ and $G|_{\{1\}\times\Dat}$ are
  already specified.
  Let $G(\cmorph{1}{0}\times\Id_u)$ be the identity map. More
  generally, define (somewhat arbitrarily) $G(\cmorph{1}{0}\times
  g)=\SpRefine[k]{F}(g)$. It follows from the fact that both
  $\SpRefine[k]{F}$ and $\SpRefine[k]{F}'$ refine $F$ that the
  resulting diagram $G$ is homotopy commutative. Extend $G$ to a
  homotopy coherent diagram inductively, as in the proof of
  point~(\ref{item:spatial-exists}).
 
  Finally, point~(\ref{item:spatial-suspend}) is immediate from the definitions.
\end{proof}

\subsection{A coherent cube of box maps}\label{sec:coherent-box-maps}
\begin{definition}
  Given a strictly unitary, lax $2$-functor $F\co \CCat{n}\to\BurnsideCat$
  let $\SpRefine[k]{F}\co \CCat{n}\to\BSpaces$ be a spatial refinement
  of $F$. Let $\SpDiag[k]{F}\co \CCat{n}_+\to\BSpaces$ be the diagram obtained
  from $\SpRefine[k]{F}$ by defining $\SpDiag[k]{F}(*)$ to be a single
  point. Let $\SpDiag{F}$ be the diagram obtained from $\SpDiag[k]{F}$
  by replacing each vertex $\SpDiag[k]{F}(u)$ with $\Sigma^{-k}\bigl(\Sigma^\infty \SpDiag[k]{F}(u)\bigr)$, the $k$-fold formal desuspension of its suspension spectrum.
\end{definition}

\begin{corollary}\label{cor:box-realize-spectrum}
  Up to stable homotopy equivalence, the spectrum $\hocolim\SpDiag{F}$
  depends only on the functor $F$. In fact, for any $k>n$, the
  homotopy type of $\hocolim\SpDiag[k]{F}$ is independent of the
  choices in its construction.
\end{corollary}
\begin{proof}
  This is immediate from \Proposition{box-refinement-exist-unique},
  together with the fact that the homotopy colimits of homotopy equivalent
  homotopy coherent diagrams are homotopy equivalent (\cite[Theorem
  5.12]{Vogt-top-hocolim}, quoted as
  \Proposition{hequiv-diags-hocolim}).
\end{proof}

As in \Section{subsec-realize}, we can also work with a larger enlargement 
$\CCat{n}_{\othplus}=(*\leftarrow 1\rightarrow 0)^{\times n}$ of $\CCat{n}$. 
Extend $\SpRefine[k]{F}$ to a functor $\SpDiagOth[k]{F}\co\CCat{n}_{\othplus}\to\BSpaces$ by
setting $\SpDiagOth[k]{F}|_{\CCat{n}}=\SpRefine[k]{F}$ and 
$\SpDiagOth[k]{F}(v)=\{\pt\}$ if $v$ is an object which is not in $\CCat{n}$, 
i.e., if some coordinate of $v$ is $*$. 

\begin{lemma}\label{lem:bigger-plus}
  For any functor $F\co\CCat{n}\to\BurnsideCat$ and any spatial
  refinement $\SpRefine[k]{F}$ of $F$ there is a stable homotopy
  equivalence
  $
  \hocolim \SpDiag[k]{F}\simeq \hocolim\SpDiagOth[k]{F}.
  $
\end{lemma}
\begin{proof}
  The proof is similar to but easier than the proof of
  \Lemma{bigger-thic-plus}, and is left to the reader.
\end{proof}

\subsection{The realizations of the small cube and big cube agree}\label{sec:small-is-big}
Before proving \Theorem{smaller-diag}, we introduce an auxiliary
category, the arrow category of $\CCat{n}$, and study its relationship
with $\CCat{n}$ and $\thic{\CCat{n}}$.
\begin{definition}
  Given a small category $\Cat$, \emph{the arrow category of $\Cat$},
  which we denote $\ArrowCat(\Cat)$ has
  $\Ob(\ArrowCat(\Cat))=\bigcup_{u,v\in\Ob(\Cat)}\Hom(u,v)$ the set of
  morphisms in $\Cat$. Given objects $f\co u\to v$ and $g\co w\to x$
  in the arrow category, $\Hom(f,g)$ consists of pairs
  $(\alpha\co u\to w,\beta\co v\to x)$ so that
  \begin{equation}\label{eq:arrow-in-arrow}
    \mathcenter{
    \xymatrix{
      u\ar[r]^f\ar[d]_\alpha & v\ar[d]^\beta\\
      w\ar[r]_g & x
    }}
  \end{equation}
  commutes. Maps compose in the obvious way:
  $(\gamma,\delta)\circ(\alpha,\beta)=(\gamma\circ\alpha,\delta\circ\beta)$.
  
  There is a functor $A\co \Cat\to\ArrowCat(\Cat)$ defined by
  \begin{align*}
    A(u)&=\Id_{u} \\
    A(f\co u\to v)
    &=\mathcenter{\xymatrix{
      u \ar[r]^{\Id_u} \ar[d]^{f} & u \ar[d]^{f}\\
    v\ar[r]^{\Id_v} & v.}}
  \end{align*}
  There is also a functor $B\co \thic{\Cat}\to \ArrowCat(\Cat)$
  defined by
  \begin{align*}
    B(u\stackrel{f}{\longrightarrow} v\stackrel{g}{\longrightarrow}w)&=g\circ f\\
    B\left(  \mathcenter{
    \xymatrix{
    u\ar[r]^f\ar[d]^\alpha & v\ar[r]^g & w\ar[d]^\gamma\\
    u'\ar[r]^{f'} & v'\ar[u]_{\beta}\ar[r]^{g'} & w'}}\right)
     &= 
       \mathcenter{
       \xymatrix{
       u\ar[r]^{g\circ f}\ar[d]_\alpha & w\ar[d]^\gamma\\
    u'\ar[r]_{g'\circ f'} & w'.}}
  \end{align*}
\end{definition}

In the special case of the cube category,
\[
  \ArrowCat(\CCat{1})=\left( \cmorph{1}{1}\longrightarrow
    \cmorph{1}{0}\longrightarrow \cmorph{0}{0}\right) \qquad
  \ArrowCat(\CCat{n})=\left(\ArrowCat(\CCat{1})\right)^{n}.
\]
We will need a version with extra objects added, analogous to
$\CCat{n}_{\othplus}$:
\begin{definition}
  Let
  \begin{align*}
    \ArrowCat(\CCat{1})_{\othplus}&=\left( *\longleftarrow
      \cmorph{1}{1}\longrightarrow \cmorph{1}{0}\longrightarrow
      \cmorph{0}{0}\right)\\
    \ArrowCat(\CCat{n})_{\othplus}&=\left(\ArrowCat(\CCat{1})_{\othplus}\right)^{n}.
  \end{align*}
  The functors $A$ and $B$ have extensions
  \begin{align*}
    A_\othplus&\co \CCat{n}_{\othplus}\to \ArrowCat(\CCat{n})_{\othplus}\\
  B_\othplus&\co \thic{\CCat{n}}_{\othplus}\to \ArrowCat(\CCat{n})_{\othplus}.
\end{align*}
These are products of the $1$-dimensional case, which is given by:
\[
\xymatrix{
  {*} \ar@{-->}[d] & 1 \ar[l]\ar[rr] \ar@{-->}[d]& &0\ar@{-->}[d] \\
  {*} & \cmorph{1}{1}\ar[l]\ar[r] & \cmorph{1}{0}\ar[r] & \cmorph{0}{0}
}\qquad
\xymatrix{
  {*} \ar@{.>}[d] & 111 \ar[l]\ar[r] \ar@{.>}[d]& 110\ar@{.>}[dr] & & 100\ar[ll]\ar[r]\ar@{.>}[dl] &000\ar@{.>}[d] \\
  {*} & \cmorph{1}{1}\ar[l]\ar[rr] &  & \cmorph{1}{0}\ar[rr] & & \cmorph{0}{0} 
}
\]
The dashed arrows denote $A_\othplus$, and the dotted arrows denote $B_\othplus$.
\end{definition}
                                    
To relate various diagrams, we will need to know $A_\othplus$ and
$B_\othplus$ are homotopy cofinal:
\begin{lemma}\label{lem:CCat-to-Ar-cofinal}
  The functor $A_\othplus$ is homotopy cofinal.
\end{lemma}
\begin{proof}
  Recall from point~\ref{item:hocolim-cofinal} in \Section{colimit}
  that homotopy cofinality means that each undercategory $d\downarrow
  A_{\othplus}$ has contractible nerve.  Since taking
  undercategories
  commutes with taking products, it suffices to
  verify the one-dimensional case. This verification is
  straightforward, and is left to the reader.
\end{proof}

\begin{lemma}\label{lem:thic-to-Ar-cofinal}
  The functor $B_\othplus$ is homotopy cofinal.  
\end{lemma}
\begin{proof}
  As in the proof of \Lemma{CCat-to-Ar-cofinal}, it suffices to verify
  the 1-dimensional case, which is straightforward.
\end{proof}

Next we see that for any small category $\Cat$, any functor $F\co
\Cat\to\BurnsideCat$ lifts to a functor $\vec{F}\co \ArrowCat(\Cat)\to
\BurnsideCat$. 
\begin{definition}
  Given a functor $F\co \Cat\to\BurnsideCat$, define a functor $\vec{F}\co \ArrowCat(\Cat)\to
  \BurnsideCat$ as follows.
  For $f\in\Ob(\ArrowCat(\Cat))$, $F(f)$ is a correspondence, and in
  particular a set; define $\vec{F}(f)=F(f)$. This defines $\vec{F}$
  on $\Ob(\ArrowCat(\Cat))$. Given a diagram as in
  Formula~\eqref{eq:arrow-in-arrow}, define
  $\vec{F}(\alpha,\beta)=F(g\circ \alpha)$ (which is exactly the same
  as $F(\beta\circ f)$, but merely in bijection with
  $F(g)\circ F(\alpha)$ and $F(\beta)\circ F(f)$). The source and
  target maps are given by
  \[
    \xymatrix{
      F(\beta)\times_{F(v)}F(f)\ar[d]&\vec{F}(\alpha,\beta)\ar[r]^-{F_{u,w,x}^{-1}}\ar[l]_-{F_{u,v,x}^{-1}}&F(g)\times_{F(w)}F(\alpha)\ar[d]\\\mathllap{\vec{F}(f)=}F(f)&&
      F(g)\mathrlap{=\vec{F}(g).}  }
  \]
  For any pair of composable morphisms
  \[
    \xymatrix{
      u\ar[r]^f\ar[d]_\alpha & v\ar[d]^\beta\\
      w\ar[r]^g\ar[d]_\gamma & x\ar[d]^\delta\\
      y \ar[r]_h & z, }
  \]
  $\vec{F}_{f,g,h}$ should specify an isomorphism from
  $\vec{F}(\gamma,\delta)\times_{F(g)}\vec{F}(\alpha,\beta)=
  F(\delta\circ g)\times_{F(g)}F(g\circ \alpha)$ to
  $F(\delta\circ g\circ \alpha)=F(\gamma\circ
  \alpha,\delta\circ\beta)$. Define this isomorphism
  to be the composition of the isomorphisms
  \begin{align*}
    F(\delta\circ g)\times_{F(g)}F(g\circ \alpha)
    &\cong
      F(\delta)\times_{F(x)}F(g)\times_{F(g)}F(g)\times_{F(w)}F(\alpha)\\
    &\cong
      F(\delta)\times_{F(x)}F(g)\times_{F(w)}F(\alpha)\\
    &\cong F(\delta\circ g\circ \alpha).
  \end{align*}
\end{definition}

\begin{lemma}\label{lem:lift-to-arrow}
  These maps make $\vec{F}$ into a strictly unitary, lax
  $2$-functor, and $\vec{F}\circ A=F$.
\end{lemma}
We leave the proof as an exercise to the reader.

\begin{lemma}\label{lem:len-in-arr}
  In $\ArrowCat(\CCat{n})_\othplus$, any sequence of composable,
  non-identity morphisms has length at most $2n$.
\end{lemma}
\begin{proof}
  This is immediate from the definitions.
\end{proof}

\begin{corollary}\label{cor:spatial-refine-arr}
  If $k\geq 2n$ then any strictly unitary, lax $2$-functor $\vec{F}\co
  \ArrowCat(\CCat{n})\to\BurnsideCat$ admits a $k$-dimensional spatial
  refinement $\wtv{F}\co
  \ArrowCat(\CCat{n})\to\BSpaces$.
\end{corollary}
\begin{proof}
  This is immediate from \Lemma{len-in-arr} and
  \Proposition{box-refinement-exist-unique}.
\end{proof}

\begin{lemma}\label{lem:box-to-thic}
  Fix $k>2n$.
  Given $F\co \CCat{n}\to\BurnsideCat$, let $\wtv{F}$ be a
  $k$-dimensional spatial refinement of $\vec{F}$ and consider the homotopy coherent
  diagram $\wtv{F}\circ B\co \thic{\CCat{n}}\to\BSpaces$. 
  There is a morphism of homotopy coherent diagrams
  $G_k\co \wtv{F}\circ B\to \thicf[k]{F}$ so 
  that on each object the underlying map induces an isomorphism on 
  $H_i$ for $i\leq 2k-1$.
\end{lemma}
\begin{proof}
  Recall that a morphism from $\wtv{F}\circ B$ to
  $\thicf[k]{F}$ is a diagram over $\CCat{1}\times \thic{\CCat{n}}$
  whose restriction to $\{1\}\times\thic{\CCat{n}}$ is
  $\wtv{F}\circ B$ and whose restriction to
  $\{0\}\times\thic{\CCat{n}}$ is $\thicf[k]{F}$. We will build such a
  diagram inductively, using box maps from wedges to products.

  On $\{0\}\times\thic{\CCat{n}}$ and $\{1\}\times\thic{\CCat{n}}$,
  $G$ is already specified. Notice that for each object
  $(u\stackrel{\cmorph{u}{v}}{\longrightarrow}v
  \stackrel{\cmorph{v}{w}}{\longrightarrow}
  w)\in\Ob(\thic{\CCat{n}})$, the space $(\wtv{F}\circ B)(u\to v\to
  w)=\bigvee_{a\in F(\cmorph{v}{w}\circ \cmorph{u}{v})}S^k$ is the
  $k$-skeleton of $\thicf[k]{F}(u\to v\to w)=\bigvee_{a\in
    F(\cmorph{u}{v})}\prod_{\substack{b\in F(\cmorph{v}{w}),\
      s(b)=t(a)}}S^k$.  For each arrow of the form
  $\cmorph{1}{0}\times \Id_{u\to v\to w}$, define
  $G(\cmorph{1}{0}\times \Id_{u\to v\to w})$ to be the inclusion of
  the $k$-skeleton (which is an isomorphism on $H_i$ for $i\leq
  2k-1$).  More generally, given a morphism $(\alpha,\beta,\gamma)$ as
  in Formula~(\ref{eq:thic-morphism}), define $G(\cmorph{1}{0}\times
  (\alpha,\beta,\gamma))$ to be the composition
  $\thicf[k]{F}(\alpha,\beta,\gamma)\circ
  G(\cmorph{1}{0}\times\Id_{u\to v\to w})$.  (Factoring in the
  other order would work just as well, though it would give a
  different map.)

  The result is a homotopy commutative diagram $G$, so that the
  restriction to ${\{0\}\times \thic{\CCat{n}}}$ is commutative and
  the restriction to ${\{1\}\times \thic{\CCat{n}}}$ is homotopy
  coherent. By construction, the maps on the $1$-side are disjoint box
  maps. The maps associated to the edges $\cmorph{1}{0}\times\Id_{u\to v\to w}$ from the $1$-side to the
  $0$-side are wedge sums of overlapping box maps:  the inclusion
  \[
    \bigvee_{x\in F(\cmorph{v}{w}\circ \cmorph{u}{v})}S^k=
    \bigvee_{a\in F(\cmorph{u}{v})}\bigvee_{\substack{b\in F(\cmorph{v}{w})\\
        s(b)=t(a)}}S^k
    \into \bigvee_{a\in F(\cmorph{u}{v})}\prod_{\substack{b\in F(\cmorph{v}{w})\\
        s(b)=t(a)}}S^k
  \]
  restricts to
  $\bigvee_{\substack{b\in F(\cmorph{v}{w}),\ s(b)=t(a)}}S^k$ as the
  overlapping box map associated to the correspondence
  $(\{b\in F(\cmorph{v}{w})\mid s(b)=t(a)\},\Id,\Id)$ from
  $\{b\in F(\cmorph{v}{w})\mid s(b)=t(a)\}$ to itself where each
  sub-box is the whole box;
  cf.~\Example{disjoint-overlap-box}.
  More succinctly, the map
  $S^k_{(a,b)}\to \prod_{b'\in F(\phi_{v,w}),\ s(b')=t(a)}S^k$ is the
  box map corresponding to the $1$-element correspondence from
  $\{(a,b)\}$ to $b\in \{b'\in F(\phi_{v,w}),\ s(b')=t(a)\}$, where the
  sub-box $B_{(a,b)}\subset B_b$ is equal to the whole box $B_b$.

  To extend the diagram constructed so far to an entire homotopy
  coherent diagram, we need to define families of maps
  \[
    \bigvee_{x\in F(\cmorph{u}{w})}S^k \to 
    \bigvee_{a\in F(\cmorph{u'}{v'})}\prod_{\substack{b\in F(\cmorph{v'}{w'})\\
        s(b)=t(a)}}S^k
  \]
  corresponding to sequences of morphism starting at $(1,u\to v\to w)$ and
  ending at $(0,u'\to v'\to w')$. Observe that for there to be such a
  sequence we must have $u\geq u'$ and $w\geq w'$.  We choose these
  extensions inductively, maintaining the following restrictions:
  \begin{enumerate}[leftmargin=*,label=(X-\arabic*)]
  \item\label{item:X1} Fix a morphism $(u\to v\to w)\to (u'\to v'\to w')$ in
    $\thic{\CCat{n}}$. Let $x\in F(\cmorph{u}{w})$. Decompose $x$ as
    \[
      (x_1,x_2,x_3,x_4)\in F(\cmorph{u}{u'})\times_{F(u')}F(\cmorph{u'}{v'})\times_{F(v')}F(\cmorph{v'}{v})\times_{F(v)}F(\cmorph{v}{w}).
    \]
    Then $G((u\to v\to w)\to (u'\to v'\to w'))$ sends the sphere
    associated to $x$ to the wedge summand of 
    \[
      \bigvee_{a\in F(\cmorph{u'}{v'})}\prod_{\substack{b\in F(\cmorph{v'}{w'})\\
          s(b)=t(a)}}S^k
    \]
    corresponding to $x_2$.
  \item\label{item:X2} The map
    \[
      h_x\co S^k\to \prod_{\substack{b\in F(\cmorph{v'}{w'})\\
          s(b)=t(x_2)}}S^k
    \]
    associated to $x$ in the previous condition is an overlapping box map.
  \end{enumerate}
  
  For the inductive step, we need to show that:
  \begin{enumerate}[label=(Y-\arabic*)]
  \item\label{item:Y1} If we post-compose a map satisfying
    Properties~\ref{item:X1}--\ref{item:X2} with the map in
    Formula~\eqref{eq:thic-F-source-targ} we obtain another map
    satisfying~\ref{item:X1}--\ref{item:X2}, and
  \item\label{item:Y2} If we pre-compose a map satisfying
    Properties~\ref{item:X1}--\ref{item:X2} with one of the box maps
    used to define $G$ on the $1$-side we obtain another map
    satisfying~\ref{item:X1}--\ref{item:X2}.
  \end{enumerate}
  Then, Lemma~\ref{lem:box-maps-contractible} implies we can find the
  desired family of maps to continue the induction.
  
  Statement~\ref{item:Y2} is clear. For
  Statement~\ref{item:Y1}, let $h_x$ be as in Condition~\ref{item:X2},
  for $x\in F(\cmorph{u}{w})$.  Given a map
  $g\co (u'\to v'\to w')\to (u''\to v''\to w'')$, the map
  $\thicf[k]{F}(g)$ sends the wedge summand corresponding to
  $x_2\in F(\cmorph{u'}{v'})$ to the summand corresponding to
  $x_2'$ where $x_2$ corresponds to
  $(x_1',x_2',x_3')\in
  F(\cmorph{u'}{u''})\times_{F(u'')}F(\cmorph{u''}{v''})\times_{F(v'')}F(\cmorph{v''}{v'})$. Since
  $x$ decomposes as $(x_1,x_1',x_2',x_3',x_3,x_4)$,
  Property~\ref{item:X1} is satisfied. Further, the map
  \[
    \thicf[k]{F}(g)\co \prod_{\substack{b\in F(\cmorph{v'}{w'})\\
        s(b)=t(x_2)}}S^k\to
    \prod_{\substack{b'\in F(\cmorph{v''}{w''})\\
        s(b')=t(x_2')}}S^k
  \]
  sends $S^k_b$ to $\prod_{b'=(x_3',b,y)}S^k_{b'}=\prod_{y\mid s(y)=t(b)}S^k_{(x_3',b,y)}$ by the diagonal
  map. If we replace each box for $h$ labeled by $c$ with
  $|\{y\in F(\cmorph{w'}{w''})\mid s(y)=t(b)\}|$ boxes, labeled by the
  corresponding $b'=(x_3',b,y)$, then the resulting box map is
  $\thicf[k]{F}(g)\circ h$. In particular, we have verified that
  $\thicf[k]{F}(g)\circ h$ is a box map, as desired.

  (The following diagram may be helpful in keeping track of the source
  and targets of the various elements in the previous paragraph:
  \[
    \begin{tikzpicture}
      \node at (0,0) (u) {$u$};
      \node at (2,0) (v) {$v$};
      \node at (4,0) (w) {$w$};
      \node at (0,-1) (up) {$u'$};
      \node at (2,-1) (vp) {$v'$};
      \node at (4,-1) (wp) {$w'$};
      \node at (0,-2) (upp) {$u''$};
      \node at (2,-2) (vpp) {$v''$};
      \node at (4,-2) (wpp) {$w''$};
      \draw[->] (u) to (v);
      \draw[->] (v) to node[below]{$x_4$} (w);
      \draw[->] (up) to node[below]{$x_2$} (vp);
      \draw[->] (vp) to node[below]{$b$} (wp);
      \draw[->] (upp) to node[below]{$x_2'$} (vpp);
      \draw[->] (vpp) to node[below]{$b'$} (wpp);
      \draw[->] (u) to node[left]{$x_1$} (up);
      \draw[->] (up) to node[left]{$x_1'$} (upp);
      \draw[->] (vp) to node[right]{$x_3$} (v);
      \draw[->] (vpp) to node[right]{$x_3'$} (vp);
      \draw[->] (w) to (wp);
      \draw[->] (wp) to node[right]{$y$} (wpp);
      \draw[->, bend left=20] (u) to node[above]{$x$} (w);
    \end{tikzpicture}
  \]
  The letter labeling each arrow $\to$ is the element of $F(\to)$ under consideration.)
  
  As noted above, Lemma~\ref{lem:box-maps-contractible} and induction
  now complete the proof.
  %
  %
\end{proof}

\begin{corollary}\label{cor:box-to-thic}
  There is a stable homotopy equivalence 
  $\hocolim (\wtv{F}\circ B)^\othplus \simeq \hocolim\thicf{F}^\othplus$.
\end{corollary}
\begin{proof}
  The morphism of diagrams $G_k$ from \Lemma{box-to-thic} extends uniquely
  to a morphism of thickened diagrams
  $G^\othplus_k\co (\wtv{F}\circ B)^\othplus\to \thicf[k]{F}^\othplus$. 
  Further, the diagram $\wtv{F}$ and the morphisms $G_k^\othplus$
  can be chosen so that $G_{k+1}^\othplus$ is the suspension of $G_k^\othplus$.
  It follows that there is an induced map of diagrams of spectra $G^\othplus$, 
  and the underlying maps of $G^\othplus$ are equivalences.
\end{proof}

\begin{proof}[Proof of \Theorem{smaller-diag}]
  Let $\vec{F}\co \ArrowCat(\CCat{n})\to \BurnsideCat$ be the functor
  from \Lemma{lift-to-arrow}. By \Corollary{spatial-refine-arr} there
  is a spatial refinement $\wtv{F}$ of $\vec{F}$. The composition
  $\SpRefine{F}=\wtv{F}\circ A$ is a spatial refinement of $F$. We
  will show that the corresponding diagram $\SpDiag{F}\co
  \CCat{n}_+\to\Spectra$ satisfies the conditions of the
  theorem. Indeed, all of the conditions except that the homotopy
  colimit is $\Realize{F}$ are immediate. We compute the homotopy
  colimit.

  By \Lemma{bigger-plus}, 
  \[
  \hocolim_{\CCat{n}_+}\SpDiag{F}\simeq
  \hocolim_{\CCat{n}_\othplus}\SpRefine{F}^\othplus.
  \]
  By \Lemma{CCat-to-Ar-cofinal}, \Lemma{thic-to-Ar-cofinal}, and
  Property~\ref{item:hocolim-cofinal} of homotopy colimits,
  \[
  \hocolim_{\CCat{n}_\othplus}\SpRefine{F}^\othplus \simeq
  \hocolim_{\ArrowCat(\CCat{n})_\othplus} \wtvdag{F}\simeq
  \hocolim_{\thic{\CCat{n}}_\othplus}\wtvdag{F}\circ B^\othplus.
  \]
  By \Corollary{box-to-thic}, there is a homotopy equivalence 
  \[
  \hocolim_{\thic{\CCat{n}}_\othplus}\wtvdag{F}\circ B^\othplus
  \simeq \hocolim_{\thic{\CCat{n}}_\othplus}\thicf{F}^\othplus.
  \]
  By \Lemma{bigger-thic-plus}, 
  \[
  \hocolim_{\thic{\CCat{n}}_\othplus}\thicf{F}^\othplus\simeq
  \hocolim_{\thic{\CCat{n}}_+}\thicf{F}^+=\Realize{F},
  \]
  proving the result.
\end{proof}

\section{A CW complex structure on the realization of the small cube}
\label{sec:cubical-to-cubes}

In this section, we prove that the realization in terms of little
cubes (\Section{smaller-cube}) is stably homotopy equivalent to the
cubical realization (Section~\ref{sec:cubical}). We start by studying the cell structure on the
little cubes realization:

\begin{proposition}\label{prop:cwcomplex-box-hocolim}
  Let $F\co \CCat{n}\to\BurnsideCat$ be a strictly unitary, lax $2$-functor and
  $\SpRefine[k]{F}\co \CCat{n}\to\BSpaces$ a spatial refinement of $F$
  (\Definition{spatial-refinement}).  Then the homotopy colimit of
  $\SpDiag[k]{F}$ carries a CW complex structure whose cells except the
  basepoint correspond to the elements of the set
  $\coprod_{u\in\{0,1\}^n}F(u)$. Further, the equivalences
  $\Sigma\hocolim \SpDiag[k]{F}\simeq \hocolim \SpDiag[k+1]{F}$ can be chosen to
  be cellular, so $\hocolim\SpDiag{F}$ inherits the structure of a CW spectrum
  (in the sense of~\cite[Section III.2]{Adams-top-book}).
\end{proposition}

\begin{proof}
  Per \Observation{hocolim-no-id}, when taking the homotopy colimit we
  may (and will) consider only chains of non-identity arrows.

  For $u\in\Ob(\CCat{n})$ and $x\in F(u)$, let $B_x$ be the box that is
  associated to $x$ during the construction of $\SpDiag[k]{F}$; that
  is, $B_x/\bdy B_x$ is the $S^k$-summand corresponding to $x$ in
  $\SpDiag[k]{F}(u)=\bigvee_{x\in F(u)}S^k$. Following \Definition{hocolim} (and with $\sim$ denoting the same equivalence relation), we can
  write the homotopy colimit as
  \begin{align*}
    \hocolim\SpDiag[k]{F}
  &=\Biggl(\{*\}\amalg \coprod_{u\in\{0,1\}^n}\Bigl(\coprod_{m\geq 0}
  \coprod_{\substack{u=u^0\stackrel{f_1}{\longrightarrow}\cdots\stackrel{f_m}{\longrightarrow}u^m\\u^i\in\{0,1\}^n\cup\{*\}\\f_i\neq\Id}}
  [0,1]^{m}\Bigr)\times\Bigl(\bigvee_{x\in F(u)}B_x/\bdy B_x\Bigr)\Biggr)/\sim\\
  &=\Biggl(\Bigl[\{*\}\amalg \coprod_{u\in\{0,1\}^n}\Bigl(\coprod_{m\geq 0}
  \coprod_{\substack{u=u^0\stackrel{f_1}{\longrightarrow}\cdots\stackrel{f_m}{\longrightarrow}u^m\\u^i\in\{0,1\}^n\cup\{*\}\\f_i\neq\Id}}
  [0,1]^{m}\Bigr)/\sim_1\Bigr]\times\Bigl(\coprod_{x\in F(u)}B_x\Bigr)\Biggr)/\sim_2.
  \end{align*}
  where we have broken up the identification $\sim$ into a two-step
  identification $\sim_1$ and $\sim_2$, defined as follows:
  \begin{align*}
  (f_m,\dots,f_1;&t_1,\dots,t_m)\\
  &\sim_1
  \begin{cases}
    (f_m,\dots,f_{i+1}\circ f_i,\dots,f_1;t_1,\dots,t_{i-1},t_{i+1},\dots,t_m) & t_i=1,\ i<m\\
    (f_{m-1},\dots,f_1;t_{1},\dots,t_{m-1}) & t_m=1
  \end{cases}\\
  (f_m,\dots,f_1;&t_1,\dots,t_m;y)\\
  &\sim_2
  \begin{cases}
    (f_m,\dots,f_{i+1};t_{i+1},\dots,t_{m};\SpDiag[k]{F}(f_i,\dots,f_1)(t_{1},\dots,t_{i-1})(y)) & t_i=0\\
    * & y\in\bdy B_x.
  \end{cases}
  \end{align*}

  Now fix $u\in\{0,1\}^n$, and let us study the cubical complex
  \begin{equation}\label{eq:cubical-cx-mu}
  M_u\defeq\Bigl(\coprod_{m\geq 0}
  \coprod_{\substack{u=u^0\stackrel{f_1}{\longrightarrow}\cdots\stackrel{f_m}{\longrightarrow}u^m\\u^i\in\{0,1\}^n\cup\{*\}\\f_i\neq\Id}}
  [0,1]^{m}\Bigr)/\sim_1.
  \end{equation}
  If $u=\vec{0}$, then $M_u$ is a single point, which we write as
  $\{0\}$ for reasons that will soon be apparent. When $u\neq\vec{0}$,
  we divide the chains
  $u=u^0\stackrel{f_1}{\longrightarrow}\cdots\stackrel{f_m}{\longrightarrow}u^m$
  into two types: the ones ending at $\vec{0}$ or $*$, and the ones
  ending at neither. In the first case, when $u^m\in\{\vec{0},*\}$,
  the facet $[0,1]^{m-1}\times\{1\}$ is identified with the cube
  $[0,1]^{m-1}$ coming from the sub-chain 
  $u=u^0\stackrel{f_1}{\longrightarrow}\cdots
  \stackrel{f_{m-1}}{\longrightarrow}u^{m-1}$. Therefore, we can write
  \begin{align*}
    M_u&=\Biggl(\Bigl(\coprod_{m\geq 1}
    \coprod_{\substack{u=u^0\stackrel{f_1}{\longrightarrow}\cdots\stackrel{f_m}{\longrightarrow}\vec{0}\\f_i\neq\Id}}
    [0,1]^{m-1}\times[0,1]\Bigr)\amalg\Bigl(\coprod_{m\geq 1}
    \coprod_{\substack{u=u^0\stackrel{f_1}{\longrightarrow}\cdots\stackrel{f_m}{\longrightarrow}\ast\\f_i\neq\Id}}
    [0,1]^{m-1}\times[0,1]\Bigr)\\
    &\qquad\qquad\amalg\Bigl(\coprod_{m\geq 1}
    \coprod_{\substack{u=u^0>\cdots > u^{m-1}\\u^i\in\{0,1\}^n\setminus\{\vec{0}\}}}
    [0,1]^{m-1}\Bigr)\Biggr)/\sim_1\\
    &\cong\Biggl(\Bigl(\coprod_{m\geq 1}
    \coprod_{\substack{u=u^0\stackrel{f_1}{\longrightarrow}\cdots\stackrel{f_m}{\longrightarrow}\vec{0}\\f_i\neq\Id}}
    [0,1]^{m-1}\times[0,1]\Bigr)\amalg\Bigl(\coprod_{m\geq 1}
    \coprod_{\substack{u=u^0\stackrel{f_1}{\longrightarrow}\cdots\stackrel{f_m}{\longrightarrow}\ast\\f_i\neq\Id}}
    [0,1]^{m-1}\times[1,2]\Bigr)\\
    &\qquad\qquad\amalg\Bigl(\coprod_{m\geq 1}
    \coprod_{\substack{u=u^0>\cdots > u^{m-1}\\u^i\in\{0,1\}^n\setminus\{\vec{0}\}}}
    [0,1]^{m-1}\Bigr)\Biggr)/\sim_1\\
    &=\Bigl(\coprod_{m\geq 1}
    \coprod_{\substack{u=u^0>\cdots > u^{m-1}\\u^i\in\{0,1\}^n\setminus\{\vec{0}\}}}
    [0,1]^{m-1}\Bigr)/\sim_1\times[0,2]
  \end{align*}
  where the second identification is via the linear map $[0,1]\to[1,2]$
  that sends $0$ to $2$ and $1$ to $1$. 
  This quotient space is just $M_{u,\vec{0}}\times [0,2]$, where
  $M_{u,\vec{0}}\cong \Moduli_{\CubeFlowCat{n}}(u,\vec{0})$ is the
  cubical complex from \Definition{cube-moduli-cubical-complex}.

  Therefore, we can write
  \[
  \hocolim\SpDiag[k]{F}=
    \Bigl(\{*\}\amalg \displaystyle\Big[\coprod_{u\in\{0,1\}^n\setminus\{\vec{0}\}}\displaystyle\coprod_{x\in F(u)}\Moduli_{\CubeFlowCat{n}}(u,\vec{0})\times[0,2]\times B_x\Big]
    \amalg\Big[\displaystyle\coprod_{x\in F(\vec{0})}\{0\}\times B_x\Big]\Bigr)/\sim_2.
  \]
  This gives the required CW complex on the homotopy
  colimit, where the cell corresponding to $x$ is
  \[
  \Cell{x}=\begin{cases}
    \Moduli_{\CubeFlowCat{n}}(u,\vec{0})\times[0,2]\times B_x&\text{if $u\neq\vec{0}$}\\
    \{0\}\times B_x&\text{if $u=\vec{0}$.}
  \end{cases}
  \]
  The identification $\sim_2$ glues the boundary $\bdy\Cell{x}$
  to lower-dimensional cells. Specifically, everything over $\bdy B_x$
  is identified with the basepoint $*$. If $u\neq\vec{0}$, everything
  over $\{2\}\subset\bdy([0,2])$ is identified with $*$ as well, and
  everything over $\{0\}\subset\bdy([0,2])$ is identified with cells
  corresponding to $u=\vec{0}$. Finally, points on
  $\bdy\Moduli_{\CubeFlowCat{n}}(u,\vec{0})$ correspond to points on
  some cube $[0,1]^l\subset\Moduli_{\CubeFlowCat{n}}(u,\vec{0})$ where
  some coordinate is $0$ (\Lemma{cube-moduli-cubical-complex}~(\ref{item:cube-coprod-bdy})),
  and such points are identified by $\sim_2$ to points of
  $\Moduli_{\CubeFlowCat{n}}(u',\vec{0})$ for some $u>u'>\vec{0}$.

  The fact that the suspension maps are cellular is trivial if we
  choose compatible spatial refinements. Specifically, take the boxes used to
  define $\SpDiag[k+1]{F}$ to be $[0,1]$ times the boxes used to define
  $\SpDiag[k]{F}$.
\end{proof}

It is not hard to show that if $\FlowCat$ is the cubical flow
category corresponding to $F$ (from
\Construction{Burn-to-Flow}) then its associated
chain complex (from \Definition{flow-cat-chain-complex}) is isomorphic
to the reduced cellular cochain complex of the above CW complex, via
an isomorphism sending the objects of $\FlowCat$ to the corresponding
cells in the CW complex. We will not prove this now, since it
follows from \Theorem{box-equals-cubical}.

\begin{theorem}\label{thm:box-equals-cubical}
  Let $(\FlowCat,\Funky\from\FlowCat\to\CubeFlowCat{n})$ be a cubical
  flow category, and let $F\from\CCat{n}\to\BurnsideCat$ be the
  corresponding functor (\Construction{Flow-to-Burn}). Then the
  cubical realization of $\FlowCat$ (\Definition{cubical-realize}) is stably homotopy equivalent to the
  realization of $F$ as the homotopy colimit of the homotopy coherent
  diagram $\SpDiag{F}\from\CCat{n}_+\to\Spectra$ from
  \Theorem{smaller-diag}; and the homotopy equivalence sends the
  cells in the CW complex structure on the homotopy colimit from
  \Proposition{cwcomplex-box-hocolim} to the corresponding cells in
  the cubical realization of $\FlowCat$ via maps of degree $\pm 1$.
\end{theorem}

\begin{proof}
  Fix a cubical neat embedding $\iota$ of $\FlowCat$ relative to
  $\TupV{d}=(d_0,\dots,d_{n-1})$ and let $k=\sum_i d_i$.
  Let $\ep$ and $R$ be the parameters from \Section{cubical-neat}, and let
  \[
  \ol{\iota}_{x,y}\from \prod_{i=|\Funky(y)|}^{|\Funky(x)|-1}[-\ep,\ep]^{d_i}\times\Moduli_{\FlowCat}(x,y)\to\Big[\prod_{i=|\Funky(y)|}^{|\Funky(x)|-1}[-R,R]^{d_i}\Big]\times\Moduli_{\CubeFlowCat{n}}(\Funky(x),\Funky(y))
  \]
  be the extension of $\iota$ from \Formula{extend-iota}.  Recall that
  given $u\in\Ob(\CCat{n})$ and $x\in F(u)$ the cubical realization 
  $\CRealize{\FlowCat}$ has a corresponding cell
  \[
  \Cell{x}=
  \begin{cases}
    \prod_{i=0}^{|u|-1}[-R,R]^{d_i}\times\prod_{i=|u|}^{n-1}[-\ep,\ep]^{d_i}\times[0,1]\times\Moduli_{\CubeFlowCat{n}}(u,\vect{0})&\text{if $u\neq\vec{0}$,}\\
    \prod_{i=0}^{n-1}[-\ep,\ep]^{d_i}\times\{0\}&\text{if $u=\vec{0}$.}
  \end{cases}
  \]

  The strategy of the proof is to use the cubical neat embedding to
  build a particular spatial refinement $\SpRefine[k]{F}$ of $F$ and
  construct a map from $\hocolim \SpDiag[k]{F}$, with the CW structure
  from Proposition~\ref{prop:cwcomplex-box-hocolim}, to the cubical
  realization $\CRealize{\FlowCat}$ of the cubical flow category
  $\FlowCat$ associated to $F$, that sends cells to cells by degree $\pm1$ maps,
  and hence is a stable homotopy equivalence.

  The diagram $\SpRefine[k]{F}$ is defined as
  follows. The box associated to $x\in F(u)$ is
  \[
  B_x=\prod_{i=0}^{|u|-1}[-R,R]^{d_i}\times\prod_{i=|u|}^{n-1}[-\ep,\ep]^{d_i}.
  \]
  Next, consider a sequence of composable non-identity morphisms
  $u=u^0\stackrel{f_1}{\longrightarrow}\cdots\stackrel{f_m}{\longrightarrow}
  u^m=v$ in $\CCat{n}$.  We will define
  \[
  \SpDiag[k]{F}(f_m,\dots,f_1)=\Phi(e_{t_1,\dots,t_{m-1}},F(\cmorph{u}{v}))\from
  [0,1]^{m-1}\times\SpDiag[k]{F}(u)\to \SpDiag[k]{F}(v)
  \]
  for an appropriate 
  family of boxes $e_{t_1,\dots,t_{m-1}}\co [0,1]^{m-1}\to E(\{B_x\mid x\in F(u)\},s_{F(\cmorph{u}{v})})$.
  In other words, if for $\gamma\in F(\cmorph{u}{v})$ we write
  $B_{\gamma}=\prod_{i=0}^{|v|-1}[-R,R]^{d_i}\times\prod_{i=|v|}^{n-1}[-\ep,\ep]^{d_i}$
  then $e_{t_1,\dots,t_{m-1}}$ is a $[0,1]^{m-1}$-parameter family of
  embeddings (as disjoint sub-boxes)
  \[
  \coprod_{\gamma\mid s(\gamma)=x}B_{\gamma}\into B_{x},\qquad \forall
  x\in F(u).
  \]

  To define $e_{t_1,\dots,t_{m-1}}$, fix $\gamma\in F(\cmorph{u}{v})$ with
  $s(\gamma)=x$, and let $y=t(\gamma)$. Let $\upsilon_\gamma$ be the section of
  the covering map
  $\Moduli_{\FlowCat}(x,y)\to\Moduli_{\CubeFlowCat{n}}(u,v)$ whose
  image is the path component corresponding to $\gamma$. Consider the
  map
  \begin{align*}
    \Moduli_{\CubeFlowCat{n}}(u,v)\times B_{\gamma}&=\Moduli_{\CubeFlowCat{n}}(u,v)\times\prod_{i=0}^{|v|-1}[-R,R]^{d_i}\times\prod_{i=|v|}^{n-1}[-\ep,\ep]^{d_i}\\
    &\stackrel{(\upsilon_{\gamma},\Id)}{\lhook\joinrel\relbar\joinrel\relbar\joinrel\relbar\joinrel\rightarrow}\Moduli_{\FlowCat}(x,y)\times\prod_{i=0}^{|v|-1}[-R,R]^{d_i}\times\prod_{i=|v|}^{n-1}[-\ep,\ep]^{d_i}\\
    &\cong\prod_{i=0}^{|v|-1}[-R,R]^{d_i}\times\Big(\prod_{i=|v|}^{|u|-1}[-\ep,\ep]^{d_i}\times\Moduli_{\FlowCat}(x,y)\Big)\times\prod_{i=|u|}^{n-1}[-\ep,\ep]^{d_i}\\
    &\stackrel{(\Id,\ol{\iota}_{x,y},\Id)}{\lhook\joinrel\relbar\joinrel\relbar\joinrel\relbar\joinrel\longrightarrow}\prod_{i=0}^{|v|-1}[-R,R]^{d_i}\times\Big(\prod_{i=|v|}^{|u|-1}[-R,R]^{d_i}\times\Moduli_{\CubeFlowCat{n}}(u,v)\Big)\times\prod_{i=|u|}^{n-1}[-\ep,\ep]^{d_i}\\
    &\stackrel{(\Id,\pi^R,\Id)}{\relbar\joinrel\relbar\joinrel\relbar\joinrel\relbar\joinrel\onto}\prod_{i=0}^{|v|-1}[-R,R]^{d_i}\times\prod_{i=|v|}^{|u|-1}[-R,R]^{d_i}\times\prod_{i=|u|}^{n-1}[-\ep,\ep]^{d_i}\\
    &\cong\prod_{i=0}^{|u|-1}[-R,R]^{d_i}\times\prod_{i=|u|}^{n-1}[-\ep,\ep]^{d_i}=B_x
  \end{align*}
  and the induced map
  \begin{equation}\label{eq:primitive-box-hocoherent-diag}
  \Moduli_{\CubeFlowCat{n}}(u,v)\times \coprod_{\substack{\gamma\in
    F(\cmorph{u}{v})\\s(\gamma)=x}} B_{\gamma}\to B_x.
  \end{equation}
  It follows from the definition of cubical neat embeddings and the formula for $\ol{\iota}_{x,y}$ that
  for any point $\pt\in\Moduli_{\CubeFlowCat{n}}(u,v)$, the
  restriction 
  $\{\pt\}\times\coprod_{\gamma\in F(\cmorph{u}{v})\mid s(\gamma)=x} B_{\gamma}\to B_x$
  is an inclusion of
  disjoint sub-boxes. Therefore, we may view the map
  from Equation~\eqref{eq:primitive-box-hocoherent-diag} as a
  $\Moduli_{\CubeFlowCat{n}}(u,v)$-parameter family of sub-boxes
  $\coprod_{\gamma\mid s(\gamma)=x} B_{\gamma}\subset B_x$.

  The chain $u=u^0>\dots >u^m=v$ corresponds to some
  cube $[0,1]^{m-1}$ in the cubical complex $M_{u,v}$ from
  \Definition{cube-moduli-cubical-complex}, which via
  \Lemma{cube-moduli-cubical-complex} is identified with some cube
  $[0,1]^{m-1}\subset\Moduli_{\CubeFlowCat{n}}(u,v)$. Restrict
  the map from Formula~\eqref{eq:primitive-box-hocoherent-diag} to
  $[0,1]^{m-1}\times\coprod_{\gamma\mid s(\gamma)=x} B_{\gamma}$ to
  obtain the required $[0,1]^{m-1}$-parameter family of sub-boxes
  $\coprod_{\gamma\mid s(\gamma)=x} B_{\gamma}\subset B_x$.

  We check that
  $\SpRefine[k]{F}$ is indeed a homotopy coherent diagram. For any
  sequence of composable non-identity morphisms
  $u=u^0\stackrel{f_1}{\longrightarrow}\cdots\stackrel{f_m}{\longrightarrow}
  u^m=v$ in $\CCat{n}$, we need to show that
  \begin{align*}
    \SpRefine[k]{F}(f_m,\dots,f_1)(t_{1},\dots,t_{i-1},&1,t_{i+1},\dots,t_{m-1})\\
    &=\SpRefine[k]{F}(f_m,\dots,f_{i+1}\circ f_i,\dots,f_1)(t_{1},\dots,t_{i-1},t_{i+1}\dots,t_{m-1})\\
    \SpRefine[k]{F}(f_m,\dots,f_1)(t_{1},\dots,t_{i-1},&0,t_{i+1},\dots,t_{m-1})\\
    &=\SpRefine[k]{F}(f_m,\dots,f_{i+1})(t_{i+1},\dots,t_{m-1})\circ\SpRefine[k]{F}(f_i,\dots,f_1)(t_{1},\dots,t_{i-1}).
  \end{align*}
  The first equation is immediate from
  \Definition{cube-moduli-cubical-complex} since the facet of
  $[0,1]^{m-1}$ which has $t_i=1$ is identified with the cube
  $[0,1]^{m-2}$ coming from the sequence of composable morphisms
  $u=u^0\stackrel{f_1}{\longrightarrow}\cdots\stackrel{f_{i-1}}{\longrightarrow}
  u^{i-1}\stackrel{f_{i+1}\circ
    f_i}{\longrightarrow}u^{i+1}\stackrel{f_{i+2}}{\longrightarrow}\cdots\stackrel{f_m}{\longrightarrow}u^m=v$.
  The second equation follows from
  \Lemma{cube-moduli-cubical-complex}~(\ref{item:cube-coprod-commute})
  since the facet of $[0,1]^{m-1}$ that has $t_i=0$ lies in the facet
  $\Moduli_{\CubeFlowCat{n}}(u^i,v)\times
  \Moduli_{\CubeFlowCat{n}}(u,u^i)$ of
  $\Moduli_{\CubeFlowCat{n}}(u,v)$ and is identified with the product
  $[0,1]^{m-i-1}\times[0,1]^{i-1}$, coming from the sequences
  $u^i\stackrel{f_{i+1}}{\longrightarrow}\cdots\stackrel{f_m}{\longrightarrow}
  u^m=v$ and
  $u^0\stackrel{f_1}{\longrightarrow}\cdots\stackrel{f_i}{\longrightarrow}
  u^i$, respectively. Since the maps $\SpDiag[k]{F}$ were defined via
  cubical neat embeddings that satisfied
  \Definition{cube-cat-neat-embed}~(\ref{item:neat-embed-commute}),
  the second equation holds.

  Finally, we construct the desired cellular map from $\hocolim
  \SpDiag[k]{F}$ to $\CRealize{\FlowCat}$.

  For any $u\in\Ob(\CCat{n})$ and any $x\in F(u)$, the cell associated
  to $x$ in $\hocolim \SpDiag[k]{F}$ is
  \begin{align*}
  \Cell{x}'&=\begin{cases}
    \Moduli_{\CubeFlowCat{n}}(u,\vec{0})\times[0,2]\times B_x&\text{if $u\neq\vec{0}$}\\
    \{0\}\times B_x&\text{if $u=\vec{0}$,}
  \end{cases}\\
  \shortintertext{while the cell associated to $x$ in $\CRealize{\FlowCat}$ is}
  \Cell{x}&=\begin{cases}
    \Moduli_{\CubeFlowCat{n}}(u,\vec{0})\times[0,1]\times B_x&\text{if $u\neq\vec{0}$}\\
    \{0\}\times B_x&\text{if $u=\vec{0}$.}
  \end{cases}
  \end{align*}
  Map $\Cell{x}'$ to $\Cell{x}$ by the quotient map $[0,2]\to
  [0,2]/[1,2]\cong [0,1]$, and the identity map on all other factors. 
  This map certainly has degree $\pm 1$ on each cell.
  To check
  that it produces a well-defined map on CW complexes, we must
   check that it commutes with the attaching maps. 
   Everything over $\bdy B_x$ was quotiented to the basepoint on
  either side. If $u=\vec{0}$, this is the only identification on
  either side. If $u\neq \vec{0}$, everything over $\{2\}\subset
  [0,2]$ was quotiented to the basepoint for $\Cell{x}'$, while
  everything over $\{1\}\subset [0,1]$ was quotiented to the basepoint
  for $\Cell{x}$. Therefore, we only need to consider $u\neq \vec{0}$
  and concentrate on the attaching maps on the portion of the boundary
  of $\Cell{x}$ (respectively, $\Cell{x}'$) that lives over $\bdy
  \Moduli_{\CubeFlowCat{n}}(u,\vec{0})$ or $\{0\}\subset\bdy([0,1])$
  (respectively, $\{0\}\subset\bdy([0,2])$).

  Consider the subcomplex 
  \[
  \wt{M}_u\defeq\Bigl(
  \coprod_{\substack{u=u^0>\dots>u^m\\u^i\in\{0,1\}^n}}
  [0,1]^{m}\Bigr)/\sim_1\cong \Moduli_{\CubeFlowCat{n}}(u,\vec{0})\times[0,1]
  \]
  of the cubical complex $M_u\cong
  \Moduli_{\CubeFlowCat{n}}(u,\vec{0})\times[0,2]$ from
  \Equation{cubical-cx-mu} in the proof of
  \Proposition{cwcomplex-box-hocolim}.  We are interested in the part
  of $\bdy\Cell{x}'$ that lives over the following subset $N_u$ of
  $\bdy\wt{M}_u$:
  \begin{equation*}
    \begin{split}
      N_u&\defeq\Big(\bdy\Moduli_{\CubeFlowCat{n}}(u,\vec{0})\times[0,1]\Big)\cup
      \Big(\Moduli_{\CubeFlowCat{n}}(u,\vec{0})\times\{0\}\Big)\\
      &\phantom{:}=\Bigl(
      \coprod_{u=u^0>\dots>u^m}\coprod_{i=1}^{m}
      [0,1]^{i-1}\times\{0\}\times[0,1]^{m-i}\Bigr)/\sim_1.
    \end{split}
  \end{equation*}
  Let $p\in\bdy\Cell{x}'$ be some point living over $N_u$, and write
  $p=(p_1,p_2)$, where $p_1\in N_u$ and $p_2\in B_x$. Assume $p_1$
  lies in the cube $[0,1]^m$ corresponding to some chain
  $u=u^0>\dots>u^m$. Let $p_{1,1},\dots, p_{1,m}$ be the
  coordinates of $p_1$ as a point in the cube, and assume
  $p_{1,\ell}=0$. 
  Let $\phi$ denote the restriction of $\SpRefine[k]{F}(\cmorph{u^{\ell-1}}{u_\ell},\dots, \cmorph{u^0}{u^1})$
  to $[0,1]^{\ell-1}\times (B_x/\bdy B_x)$. 
  Under the CW complex attaching map (denoted $\sim_2$ in the
  proof of \Proposition{cwcomplex-box-hocolim}), $p$ is attached to
  the point 
  \[
  \bigl((p_{1,\ell+1},\dots,p_{1,m}),\phi((p_{1,1},\dots,p_{1,\ell-1}),p_2)\bigr)\in
  [0,1]^{m-\ell}\times \bigvee_{y\in F(u^\ell)}(B_{y}/\bdy B_{y}),
  \]
  where $[0,1]^{m-\ell}$ is the cube in $\wt{M}_{u^\ell}$
  corresponding to the chain $u^\ell>\dots>u^m$ . The map $\phi$ is
  constructed as a $[0,1]^{m-\ell}$-parameter family of box maps.
  Therefore, the attaching map glues $p$ to the basepoint $*$ unless
  $p_2$ lies in the interior of one of the sub-boxes at the point
  $(p_{1,1},\dots,p_{1,\ell-1})$ in the family; and if $p_2$ lies in
  the interior of some box $B_{\gamma_0}$ then $p$ is glued to the
  point
  \[
  q\defeq ((p_{1,\ell+1},\dots,p_{1,m}),q_2)\in
  [0,1]^{m-\ell}\times B_{y_0}\subset \wt{M}_{u^\ell}\times B_{y_0}\subset \Cell{y_0}'
  \]
  where $y_0=t(\gamma_0)$ and $q_2\defeq
  \phi((p_{1,1},\dots,p_{1,\ell-1}),p_2)\in B_{y_0}$.

  We need to check that $p$, now viewed as a point in $\bdy\Cell{x}$,
  is also glued to $q$, now viewed as a point in $\Cell{y_0}$, in the
  CW complex $\CRealize{\FlowCat}$. In the construction of
  $\CRealize{\FlowCat}$ (\Definition{cubical-realize}), we extended the
  cubical neat embedding
  \[
  \iota_{x,y_0}\from\Moduli_{\FlowCat}(x,y_0)\into
  \Moduli_{\CubeFlowCat{n}}(u^\ell,\vec{0})\times \prod_{i=|u^\ell|}^{|u|-1}(-R,R)^{d_i}
  \]
  to an embedding
  \[
  \ol{\iota}_{x,y_0}\from\Moduli_{\FlowCat}(x,y_0)\times\prod_{i=|u^\ell|}^{|u|-1}[-\ep,\ep]^{d_i}\into
  \Moduli_{\CubeFlowCat{n}}(u^\ell,\vec{0})\times \prod_{i=|u^\ell|}^{|u|-1}[-R,R]^{d_i},
  \]
  and used it to define embeddings $\imath\from
  \Moduli_{\FlowCat}(x,y_0)\times B_{y_0}\into
  \Moduli_{\CubeFlowCat{n}}(u,u^\ell)\times B_x$ and
  $\jmath\from\Moduli_{\FlowCat}(x,y_0)\times\Cell{y_0}\into
  \bdy\Cell{x}$.  (Note that
  $\Precone\Moduli_{\CubeFlowCat{n}}(u^\ell,\vec{0})=\wt{M}_{u^\ell}$.)
  
  For any path component $\gamma$ of $\Moduli_{\FlowCat}(x,y_0)$, let
  $\imath_{\gamma}$ and $\jmath_{\gamma}$ denote the restrictions
  $\imath|_{\gamma\times B_{y_0}}$ and
  $\jmath|_{\gamma\times\Cell{y_0}}$; and as before, let
  $\upsilon_{\gamma}$ denote the section of
  $\Moduli_{\FlowCat}(x,y_0)\to\Moduli_{\CubeFlowCat{n}}(u,u^\ell)$
  whose image is $\gamma$.  Since $\iota_{x,y_0}$ satisfies
  \Definition{cube-cat-neat-embed}~(\ref{item:trivial-covering}), and
  its extension $\ol{\iota}_{x,y_0}$ was defined via
  \Equation{extend-iota}, there exists a map $\mu_{\gamma}\from
  B_{y_0}\to \Moduli_{\CubeFlowCat{n}}(u,u^\ell)\times B_x$, so that
  $\imath_{\gamma}(\upsilon_{\gamma}(a),b)=(a,\mu_{\gamma}(a,b))$, for all $(a,b)\in
  \Moduli_{\CubeFlowCat{n}}(u,u^\ell)\times B_{y_0}$.

  Let $q_1=(p_{1,\ell+1},\dots,p_{1,m})\in[0,1]^{m-\ell}$ and
  $q'=(p_{1,1},\dots,p_{1,\ell-1})\in[0,1]^{\ell-1}$; treat $q'$ as a
  point in $\Moduli_{\CubeFlowCat{n}}(u,u^{\ell})$, after viewing the
  cube $[0,1]^{\ell-1}$ as the cube corresponding to the chain
  $u=u^0>\dots>u^\ell$ in the cubical complex structure on
  $\Moduli_{\CubeFlowCat{n}}(u,u^\ell)$ from
  \Lemma{cube-moduli-cubical-complex}.  Let $\kappa$ be the inclusion
  map $\Moduli_{\CubeFlowCat{n}}(u,u^\ell)\times\wt{M}_{u^\ell}\times
  B_x\into \Cell{x}$.  Since
  $\SpRefine[k]{F}(\cmorph{u^{\ell-1}}{u_\ell},\dots, \cmorph{u^0}{u^1})$ (used to
  define $\phi$) is defined via
  \Equation{primitive-box-hocoherent-diag} using the cubical neat
  embeddings $\iota_{x,y_0}$ (which are used to define the maps
  $\imath_{\gamma}$, and consequently, $\mu_{\gamma}$), $p_2$ is in
  the interior of the box $B_{\gamma_0}$, and $\phi(q',p_2)=q_2$, it
  follows that $\mu_{\gamma_0}(q',q_2)=p_2$. Therefore,
  \begin{align*}
  \jmath_{\gamma_0}(\upsilon_{\gamma}(q'),q)&=\jmath_{\gamma_0}(\upsilon_{\gamma}(q'),q_1,q_2)=\kappa(q',q_1,\mu_{\gamma_0}(q',q_2))\\
  &=\kappa(q',q_1,p_2)=\kappa((p_{1,1},\dots,p_{1,\ell-1}),(p_{1,\ell+1},\dots,p_{1,m}),p_2)\\
  &=((p_{1,1},\dots,p_{1,\ell-1},0,p_{1,\ell+1},\dots,p_{1,m}),p_2)=(p_1,p_2)=p.
  \end{align*}
  The last equation is justified by the fact that the
  cubical complex structures on $\Moduli_{\CubeFlowCat{n}}(u,u^\ell)$
  and $\wt{M}_{u^\ell}$ respect the product structure on facets of
  $\wt{M}_u=\Moduli_{\CubeFlowCat{n}}(u,\vec{0})\times[0,1]$
  (\Lemma{cube-moduli-cubical-complex}~(\ref{item:cube-coprod-commute})).
  Therefore, $p$ is glued to $q$ in the CW complex
  $\CRealize{\FlowCat}$ as well.
\end{proof}

\section{The Khovanov homotopy type}\label{sec:Kh-htpy}
We pause briefly to review where we stand. 
We have introduced a special kind of flow categories,
cubical flow categories (\Section{cubical}), and shown that the data
of a cubical flow category is equivalent to a strictly unitary, lax
$2$-functor from the cube $\CCat{n}$ to the Burnside $2$-category
(\Section{realize-functor}). Given a cubical flow category (or functor
from the cube to the Burnside category) we have four ways of realizing
the functor as a spectrum:
\begin{itemize}
\item The original Cohen-Jones-Segal realization (\Section{flow-cat}).
\item The cubical realization, a modification of the Cohen-Jones-Segal
  construction taking into account the map to the cube
  (\Section{cubical-realize}).
\item Thickening the diagram, producing a canonical diagram in
  spectra, and taking the iterated mapping cone
  (\Section{realize-functor}).
\item Using the ``little box'' construction to produce a homotopy
  coherent cube in spectra, and then taking the iterated mapping cone
  (\Section{coherent-box-maps}).
\end{itemize}
Moreover,
Theorems~\ref{thm:cubical-CJS-realization},~\ref{thm:smaller-diag},
and~\ref{thm:box-equals-cubical} together imply that, up to stable
homotopy equivalence, these realizations all agree.

For the rest of the paper, we turn to a particular cubical flow
category: the Khovanov flow category constructed in~\cite{RS-khovanov}
(see also \Section{khovanov-basic} and
Examples~\ref{exam:KhFlowCat} and~\ref{exam:KhFunc}).
\begin{definition}
  Given an oriented link diagram $K$, let $\KhFunc(K)\co\CCat{n}\to\BurnsideCat$ be the
  Khovanov functor constructed in \Example{KhFunc}. Let $\Realize{\KhFunc(K)}$
  be the
  result of applying \Construction{realize-functor} to
  $\KhFunc(K)$ and let
  $\KhSpace(K)=\Sigma^{-n_-}\Realize{\KhFunc(K)}$, where $n_-$ is the
  number of negative crossings in $K$.
\end{definition}
Theorems~\ref{thm:cubical-CJS-realization},~\ref{thm:smaller-diag},
and~\ref{thm:box-equals-cubical} imply that, up to stable homotopy equivalence,
$\KhSpace(K)$ is the same as applying any of the other three realization
constructions to $\KhFunc(K)$.  A more direct description of $\KhFunc(K)$ was
given by Hu-Kriz-Kriz~\cite{HKK-Kh-htpy}; see \Section{two-functors-same}.

Similar constructions can be carried out for the reduced Khovanov flow category:
\begin{definition}
  Let $\rKhFunc(K)\co\CCat{n}\to\BurnsideCat$ be the functor corresponding to
  the reduced Khovanov flow category from~\cite[Section 8]{RS-khovanov}, via
  \Construction{Flow-to-Burn}.  Let $\rKhSpace(K)$ be the corresponding spectrum,
  obtained by applying \Construction{realize-functor} to
  $\rKhFunc(K)$ and desuspending $n_-$ times.
\end{definition}

\section{Relationship with Hu-Kriz-Kriz}\label{sec:HKK}
The  goal of this section is to prove:
\begin{theorem}
  \label{thm:agree-with-HKK}
  Fix a link diagram $K$. Let $\HKK(K)$ be the homotopy type associated
  to $K$ by Hu-Kriz-Kriz~\cite[Theorem 5.4]{HKK-Kh-htpy}. Then $\KhSpace(K)$ is
  stably homotopy equivalent to $\HKK(K)$.
\end{theorem}

The Hu-Kriz-Kriz construction has four steps:
\begin{enumerate}
\item Construct a $2$-functor $\CCat{n}\to\BurnsideCat$. (The category
  $\CCat{n}$ is denoted $I^n$ in~\cite{HKK-Kh-htpy}, and
  the category $\BurnsideCat$ is denoted $\mathcal{S}_2$.)
\item Use the Elmendorf-Mandell machine~\cite{EM-top-machine} to turn
  the $2$-functor $\CCat{n}\to\BurnsideCat$ into an $A_\infty$-functor
  $B_2(\CCat{n})'\to \Spectra$, where $B_2(\CCat{n})'$ is an
  auxiliary category which we review below.
\item Use Elmendorf-Mandell's rectification result~\cite[Theorem 1.4]{EM-top-machine}
  to lift this composition to a strict functor $\CCat{n}\to\Spectra$.
\item Expand $\CCat{n}$ to a category $\mathcal{I}$, analogous to the
  expansion of $\CCat{n}$ to $\CCat{n}_\othplus$, extend the functor
  $\CCat{n}\to\Spectra$ to a functor $\mathcal{I}\to\Spectra$, and
  take the homotopy colimit.
\end{enumerate}
We will prove that the two constructions agree, step-by-step.

\subsection{The functors from the cube to the Burnside category agree}\label{sec:two-functors-same}
Hu-Kriz-Kriz's construction of the Khovanov homotopy type uses the
following auxiliary category:
\begin{definition}
  The \emph{$(1+1)$-dimensional embedded cobordism category}
  $\Cob_\emb^{1+1}$ is the $2$-category defined as follows. The
  objects of $\Cob_\emb^{1+1}$ are oriented $1$-manifolds $C$ embedded
  in $S^2$
  along with a 2-coloring of the components of $S^2\setminus C$, by
  the colors ``white'' and ``black'', so that if $B(S^2\setminus C)$
  denotes the closure of the black region, then $C$ is the oriented
  boundary of $B(S^2\setminus C)$. The morphisms from $C_1$ to $C_0$
  are oriented cobordisms $\Sigma$ embedded in $[0,1]\times S^2$
  satisfying $\Sigma\cap(\{i\}\times S^2)=\{i\}\times C_i$ for
  $i\in\{0,1\}$, along with a 2-coloring of
  $([0,1]\times S^2)\setminus \Sigma$, so that if
  $B(([0,1]\times S^2)\setminus \Sigma)$ denotes the closure of the
  black region, then $\Sigma$ is oriented as the boundary of
  $B(([0,1]\times S^2)\setminus \Sigma)$. The $2$-morphisms are
  isotopy classes of isotopies of cobordisms relative boundary.
\end{definition}

The Hu-Kriz-Kriz functor $\CCat{n}\to\BurnsideCat$ is constructed in
two steps~\cite[Section 5]{HKK-Kh-htpy}, which we give as
Constructions~\ref{construction:diag-to-cob} and~\ref{construction:cob-to-Burn}.
\begin{construction}\label{construction:diag-to-cob}
  \cite[Section 4.3]{HKK-Kh-htpy} Given a link diagram $L$ in $S^2$
  with $n$ crossings $c_1,\dots,c_n$, along with a checkerboard
  coloring of the link diagram, we construct a lax $2$-functor from
  $\CCat{n}$ to $\Cob_\emb^{1+1}$. This functor was partially
  described in \Section{khovanov-basic}: to $v\in\{0,1\}^n$, associate
  the complete resolution $\mc{P}(v)$ which is a collection of
  disjoint circles in $S^2$. The checkerboard coloring for $L$ induces
  a 2-coloring of the complement of these circles; orient the circles
  as the boundary of the black region. To $u>v\in\{0,1\}^n$, associate
  the embedded cobordism $\Sigma\subset [0,1]\times S^2$, which is a
  product cobordism outside a neighborhood of the crossings where $u$
  and $v$ differ, and has a saddle for each such crossing. We declare
  the cobordism to be running from
  $\mc{P}(u)=\Sigma\cap(\{1\}\times S^2)$ to
  $\mc{P}(v)=\Sigma\cap(\{0\}\times S^2)$.  Up to isotopy, $\Sigma$ is
  independent of the order of the saddles, and in fact the isotopies
  changing the order of saddles are themselves well-defined up to
  isotopy. Thus, this construction gives a lax $2$-functor
  $\CCat{n}\to \Cob_{\emb}^{1+1}$.
\end{construction}

\begin{construction}\label{construction:cob-to-Burn}
  \cite[Section 3.4]{HKK-Kh-htpy} There is a lax $2$-functor
  $\mathcal{L}\co \Cob_\emb^{1+1}\to \BurnsideCat$ defined as
  follows. On objects, $\mathcal{L}$ sends a $1$-manifold
  $C\subset S^2$ to the set of all possible labelings of the
  components of $C$ by elements of $\{x_+,x_-\}$:
  $\mathcal{L}(C)=\prod_{C_i\in\pi_0(C)}\{x_+,x_-\}$.  The value of
  $\mathcal{L}$ on morphisms is more complicated. For any embedded
  cobordism $\Sigma$ and any connected component $\Sigma_0$ of
  $\Sigma$, consider the 2-coloring of
  $([0,1]\times S^2)\setminus \Sigma_0$ that agrees with the given
  2-coloring of $([0,1]\times S^2)\setminus\Sigma$ near $\Sigma_0$,
  and let $B(([0,1]\times S^2)\setminus\Sigma_0))$ denote the closure
  of the black region. Observe that
  \begin{align*}
    H_1(([0,1]\times S^2)\setminus\Sigma_0)/H_1((\{0,1\}\times S^2)\setminus\bdy\Sigma_0)&\cong \Z^{2g(\Sigma_0)}\\
    H_1(B(([0,1]\times S^2)\setminus\Sigma_0))/H_1(B((\{0,1\}\times S^2)\setminus\bdy\Sigma_0))&\cong \Z^{g(\Sigma_0)}.
  \end{align*}
  A \emph{valid labeling} of a cobordism
  $\Sigma\subset [0,1]\times S^2$ consists of:
  \begin{itemize}
  \item
    A labeling of each boundary component of $\Sigma$ by $x_+$ or
    $x_-$, and
  \item
    A labeling of each genus $1$ component $\Sigma_0$ of $\Sigma$ by
    $\alpha$ or $-\alpha$, where $\{\pm\alpha\}$ are the generators of
    $H_1(B(([0,1]\times S^2)\setminus\Sigma_0))/H_1(B((\{0,1\}\times
    S^2)\setminus\bdy\Sigma_0))\cong \Z$.
  \end{itemize}
  so that:
  \begin{itemize}
  \item Each connected component of $\Sigma$ has genus $0$ or $1$.
  \item For each genus $0$ connected component of $\Sigma$, the number
    of boundary components in $\{0\}\times S^2$ labeled $x_-$ plus the
    number of boundary components in $\{1\}\times S^2$ labeled $x_+$
    is $1$.
  \item For each genus $1$ connected component of $\Sigma$, all
    boundary components in $\{0\}\times S^2$ are labeled $x_+$ and all
    boundary components in $\{1\}\times S^2$ are labeled $x_-$.
  \end{itemize}
  (See~\cite[Formula (12)]{HKK-Kh-htpy}.) Define $\mathcal{L}(\Sigma)$
  to be the set of valid labelings of $\Sigma$. The source and target
  maps of $\mathcal{L}(\Sigma)$ send a labeling of $\Sigma$ to the
  induced labeling of the boundary components.

  For $\Sigma_0$ and $\Sigma_1$ composable cobordisms in $[0,\frac{1}{2}]\times S^2$ and $[\frac{1}{2},1]\times S^2$, respectively, the composition $2$-isomorphism
  $\mathcal{L}(\Sigma_0)\circ
  \mathcal{L}(\Sigma_1)\stackrel{\cong}{\longrightarrow}\mathcal{L}(\Sigma_0\circ\Sigma_1)$ 
  is obvious except for the situation when one gets a 
  cobordism $\Sigma\subset [0,1]\times S^2$ with a genus $1$ component by stacking two genus $0$
  cobordisms. The composition map
  decomposes as a product over connected components of $\Sigma$, so for simplicity
  assume that $\Sigma$ is connected. Given valid labelings on
  $\Sigma_0$ and $\Sigma_1$ that agree on
  $(\bdy\Sigma_0)\cap(\bdy\Sigma_1)\subset\{\frac{1}{2}\}\times S^2$,
  we want to construct a valid labeling on $\Sigma$, which
  amounts to labeling $\Sigma$ by $\alpha$ or $-\alpha$.  It follows
  from the labeling conditions that there is a unique component $C$ of
  $(\bdy\Sigma_0)\cap(\bdy\Sigma_1)$ that is non-separating in
  $\Sigma$ and is labeled $x_+$. Orient $C$ as the boundary of the
  black region, and let $C_b$ and $C_w$ be the push-offs of $C$ into
  the black and the white regions, respectively. One of $C_b$ and
  $C_w$ is a generator of
  $H_1(([0,1]\times S^2)\setminus\Sigma)/H_1((\{0,1\}\times
  S^2)\setminus\bdy\Sigma)\cong \ZZ^2$ and the other one is zero.  If
  $C_b$ is the generator, label $\Sigma$ by $[C]$.  If $C_w$ is the
  generator, let $D$ be a curve on $\Sigma$, oriented so that the
  algebraic intersection number $D\cdot C=1$, and label $\Sigma$ by
  $[D]$. It is clear that these composition $2$-isomorphisms satisfy
  the coherence condition~\ref{item:str-assoc}.
\end{construction}

\begin{remark}
  It is not hard to see that choosing the other checkerboard coloring
  on the link diagram yields a naturally isomorphic functor
  $\CCat{n}\to\BurnsideCat$. On the other hand, one could have
  required $D\cdot C$ to be $-1$ instead of $1$; this would have
  produced a different functor. This global choice is essentially the
  choice of ladybug matching from \cite[Section~5.4]{RS-khovanov}.
\end{remark}

\begin{lemma}\label{lem:HKK-same-Burnside}
  Hu-Kriz-Kriz's $2$-functor
  $\HKKfun\from\CCat{n}\to\BurnsideCat$~\cite{HKK-Kh-htpy} is
  naturally isomorphic to the $2$-functor
  $\KhFunc$ constructed in \Example{KhFunc} by applying
  \Construction{Flow-to-Burn} to the Khovanov flow category
  from~\cite{RS-khovanov}.
\end{lemma}

\begin{proof}
  The functors $\HKKfun$ and $\KhFunc$ are identical on
  objects.  By \Lemma{characterize-functor}, we only need to show that
  they agree on the edges and that the composition $2$-isomorphisms for
  the functors agree on the $2$-dimensional faces of the cube.

  For $u>v\in\{0,1\}^n$ with $|u|-|v|=1$, the corresponding embedded
  cobordism is a merge or split. Let $F$ denote either
  $\HKKfun$ or $\KhFunc$. It is straightforward from the definitions
  that for any $x\in F(u)$ and $y\in F(v)$, $s^{-1}(x)\cap
  t^{-1}(y)\subset F(\cmorph{u}{v})$ is empty if $x$ does not appear in the
  Khovanov differential of $y$, and consists of one point otherwise,
  so in particular there is a unique bijection $\HKKfun(\cmorph{u}{v})\cong
  \KhFunc(\cmorph{u}{v})$.

  Now consider $u>w\in\{0,1\}^n$ with $|u|-|w|=2$, and let $v,v'$ be
  the two intermediate vertices. The composition $2$-isomorphisms for
  $\HKKfun$ and $\KhFunc$
  produce bijections between $F(\cmorph{v}{w})\circ F(\cmorph{u}{v})$ and
  $F(\cmorph{v'}{w})\circ F(\cmorph{u}{v'})$. We want to show that these two
  bijections are the same.

  Fix $x\in F(u)$ and $z\in F(w)$, and consider $s^{-1}(x)\cap
  t^{-1}(z)\subset F(\cmorph{u}{w})$.  This set can have $0$, $1$, or $2$
  elements. The only nontrivial case to check is when the set has $2$
  elements, which occurs precisely when $x$ and $z$ are related by a
  ladybug configuration (\Figure{ladybug-matching}).

  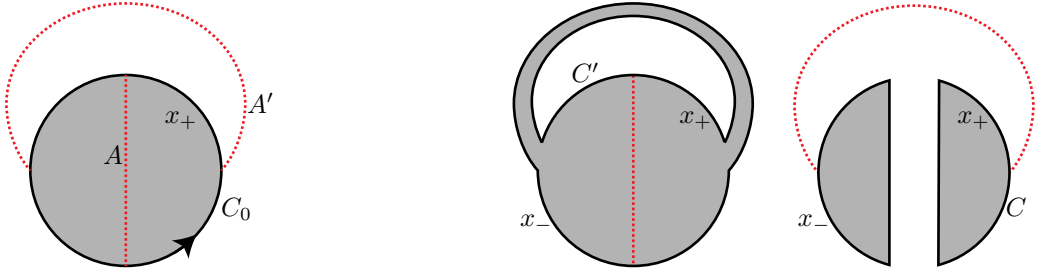
\begin{figure}
    \begin{overpic}{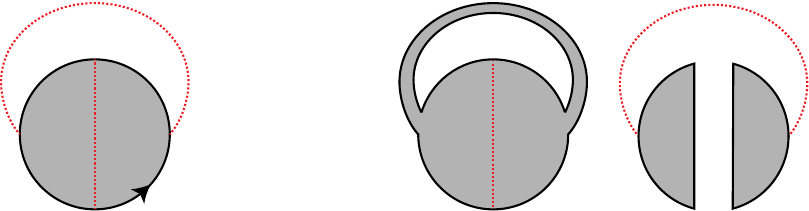}
      \put(15.5,14) {$x_+$}
      \put(9.5,10){$A$}
      \put(23.5,15){$A'$}
      \put(50,4) {$x_-$}
      \put(65.5,14) {$x_+$}
      \put(77,4) {$x_-$}
      \put(92.5,14) {$x_+$}
      \put(21,5){$C_0$}
      \put(97.2,5){$C$}
      \put(55,18){$C'$}
    \end{overpic}
    \caption{\textbf{The ladybug matching.} Left: The ladybug
      configuration. Right: Two intermediate configurations, obtained
      by attaching embedded $1$-handles along the two arcs. Black
      regions are indicated by dark
      shading.}\label{fig:ladybug-matching}
  \end{figure}

  Therefore, assume $x$ and $z$ are related by a ladybug
  configuration, and without loss of generality assume that the
  corresponding embedded genus $1$ cobordism $\Sigma$ is
  connected. The composition $2$-isomorphisms for $F$ are unchanged
  under isotopy in $S^2$: this is \cite[Lemma~5.8]{RS-khovanov} for
  $\KhFunc$, and is immediate from the definition for
  $\HKKfun$. Therefore, we may further assume that the ladybug
  configuration is as shown in \Figure{ladybug-matching}, in the
  following sense. The circle $C_0$ is the complete resolution at $w$,
  and it is labeled $x_+$ by $z$; and the circle $C_1$, the complete
  resolution at $u$, is obtained by attaching embedded $1$-handles
  along the arcs $A$ and $A'$, and $C_1$ is labeled $x_-$ by $x$. The
  embedded cobordism $\Sigma\subset[0,1]\times S^2$ connects $C_1$ to
  $C_0$, that is, $\Sigma\cap(\{i\}\times S^2)=\{i\}\times
  C_i$. Further, the black region contains the arc $A$.  Consider the
  labelings of the circles at the complete resolutions at $v$ and $v'$
  shown in \Figure{ladybug-matching}; the corresponding generators
  $y\in F(v)$ and $y'\in F(v')$ are matched by the ladybug matching
  for $\KhFunc$ \cite[Figure~5.1b]{RS-khovanov}.

  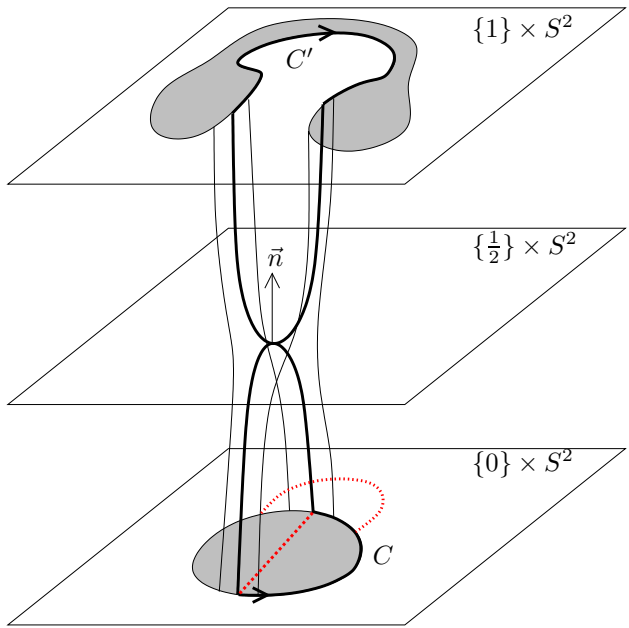
\begin{figure}
    \centering
    \begin{overpic}[width=0.5\textwidth]{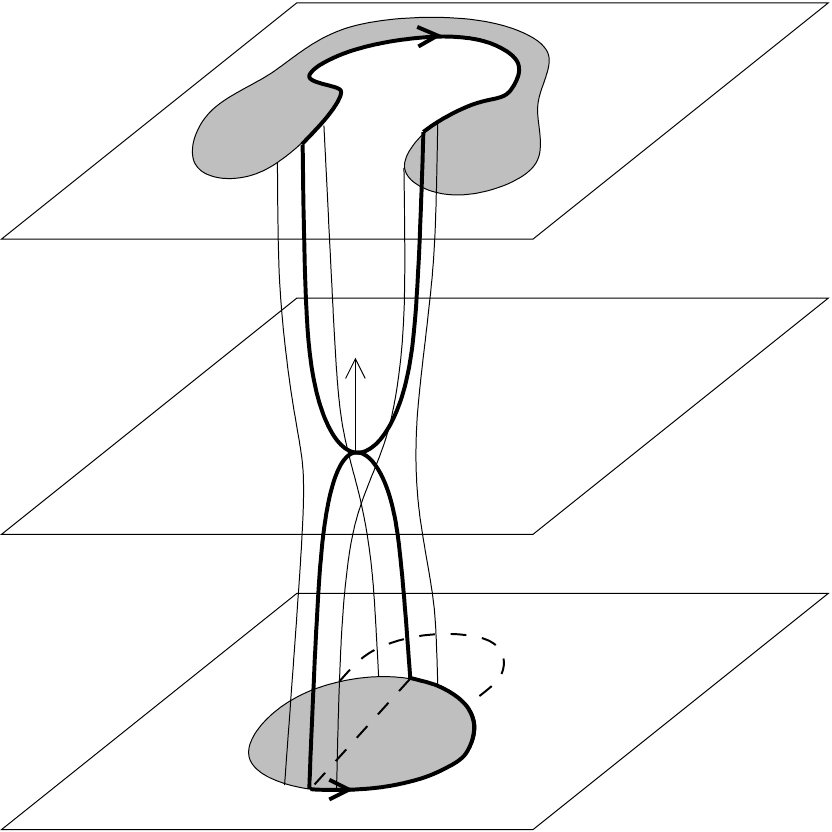}
      \put(75,25){$\{0\}\times S^2$}
      \put(75,60){$\{\frac{1}{2}\}\times S^2$}
      \put(75,95){$\{1\}\times S^2$}
      \put(59,10){$C$}
      \put(45,90){$C'$}
      \put(42,58){$\vec{n}$}
    \end{overpic}
    \caption{\textbf{The embedded cobordism $\Sigma$.} The cobordism
      connects $C_0$ at the bottom to $C_1$ at the top. The portion of
      the cobordism near the saddle point $p$, the oriented curves $C$
      and $C'$, and the normal vector $\vec{n}$ are
      shown.}\label{fig:embedded-cob}
  \end{figure}

  To show that $y$ and $y'$ are also matched by $\HKKfun$, let $C$
  (respectively, $C'$) be the circle that is labeled $x_+$ by $y$
  (respectively, $y'$); therefore, $C$ (respectively, $C'$) is a
  homology generator of the white (respectively, black) component of
  the complement of $\Sigma$. Isotope the cobordism $\Sigma$ inside
  $[0,1]\times S^2$ relative boundary so that both the saddles occur
  at $\{\frac{1}{2}\}\times S^2$. Isotope the curves $C$ and $C'$ on
  $\Sigma$ so that they intersect transversally at one point, $C$ lies
  in $[0,\frac{1}{2}]\times S^2$, and $C'$ lies in
  $[\frac{1}{2},1]\times S^2$. Therefore, $C$ and $C'$ intersect at
  the saddle point $p$ corresponding to the arc $A$. The portion of
  the cobordism $\Sigma$ near $p$ is shown in
  \Figure{embedded-cob}. Let $\vec{n}$ be the normal vector to
  $\Sigma$ at $p$ pointing away from the black region, and let
  $\vec{t}$ and $\vec{t}'$ be the tangent vectors to $C$ and $C'$ at
  $p$. It is clear from \Figure{embedded-cob} that
  $[\vec{n},\vec{t'},\vec{t}]$ constitute a positive basis;
  consequently in $\Sigma$, oriented as the boundary of the black
  region, the intersection number $C'\cdot C=1$. Therefore, for both
  $y$ and $y'$, the surface $\Sigma$ is labeled by $[C']$, and thus
  $y$ and $y'$ are matched by $\HKKfun$.
\end{proof}

\subsection{Iterated mapping cones}\label{sec:mapping-cone}
The easiest part of the identification is to see that the
Hu-Kriz-Kriz notion of iterated mapping cone agrees with
ours. Specifically, to take the iterated mapping cone of
their functor $K(F_{\HKKa})\co\CCat{n}\to\Spectra$ 
(see Sections~\ref{sec:two-functors-same} and~\ref{sec:elmendorff-mandell-machine}) 
they enlarge $\CCat{n}$ slightly to a category
$\mathcal{I}$, and extend $F_{\HKKa}$ to a functor $\wt{F}_{\HKKa}\co
\mathcal{I}\to \Spectra$ by declaring that the new vertices in
$\mathcal{I}$ map to $\{*\}$, a one-point space~\cite[Section
5.2]{HKK-Kh-htpy}.  We observe that their enlargement is the same as
$\CCat{n}_\othplus$:
\begin{lemma}\label{lem:HKK-extension}
  There is an isomorphism $\mathcal{I}\cong \CCat{n}_\othplus$ which
  commutes with the inclusions of $\CCat{n}$:
  \[
  \xymatrix{
    \mathcal{I}\ar[rr]^\cong & & \CCat{n}_\othplus\\
    & \CCat{n}\ar@{_(->}[ul] \ar@{^(->}[ur]& 
  }
  \]
\end{lemma}
\begin{proof}
  The category $\mathcal{I}$ has, as objects, pairs $(J\subseteq
  \{1,\dots,n\},\phi\co J\to\{0,1\})$, and there is a morphism
  $\phi\to \psi$ (which is unique) if and only if $\phi$ is a
  restriction of $\psi$. The cube $\CCat{n}$ sits in $\mathcal{I}$ as
  the subcategory $\{(J,\phi)\mid 0\not\in\image(\phi)\}$.  Recall
  that $\CCat{n}_{\othplus}=(\CCat{1}_{+})^n$, and
  $\Ob(\CCat{1}_+)=\{0,1,*\}$. Given an object
  $o=(v_1,\dots,v_n)\in\CCat{n}_{\othplus}$, define $J=\{i\mid
  v_i\in\{0,*\}\}$ and $\phi(i)=
  \begin{cases}
    0 & v_i=*\\
    1 & v_i=0
  \end{cases}
  $. With this dictionary, the rest of the verification is straightforward.
\end{proof}

\subsection{Another kind of homotopy coherent diagram}
\label{sec:B2}
Hu-Kriz-Kriz use a slightly different notion from Vogt of homotopy
coherent diagrams (\Section{colimit}), which is defined in two
steps:
\begin{definition}\label{def:cat-prime}
  Given a small category $\Cat$, let $\Cat'$ be the $2$-category with
  the same objects as $\Cat$,
  \[
    \Hom_{\Cat'}(x,y)=\coprod_{x=x_0,x_1,\dots,x_n=y}\Hom_\Cat(x_{n-1},x_n)\times\Hom_\Cat(x_{n-2},x_{n-1})\times\cdots\times\Hom_\Cat(x_0,x_1)
  \]
  the set of finite sequences of composable morphisms starting at $x$
  and ending at $y$, and a unique $2$-morphism from $(f_n,\dots,f_1)$
  to $(g_m,\dots,g_1)$ whenever
  $f_n\circ\cdots\circ f_1=g_m\circ\cdots\circ g_1$
  (compare~\cite[Section 4.1]{HKK-Kh-htpy}). Composition is given by
  concatenation of sequences. There is a projection $\Cat'\to\Cat$,
  where we view $\Cat$ as a $2$-category with only identity
  2-morphisms, which sends $(f_n,\dots,f_1)$ to
  $f_n\circ\cdots\circ f_1$.
\end{definition}

\begin{definition}\label{def:B2}
  Given a $2$-category $\Dat$, we can form a simplicially enriched category
  $B_2(\Dat)$ by replacing each $\Hom$ category in $\Dat$ by the
  realization of its nerve.
\end{definition}

Combining Definitions~\ref{def:cat-prime} and~\ref{def:B2},
the analogue of a homotopy coherent diagram
in~\cite{HKK-Kh-htpy} is a (simplicially enriched) functor $B_2(\Cat')\to
\Spectra$, or more generally an $A_\infty$ functor $B_2(\Cat')\to
\Spectra$. Note that there is a projection $\Pi\co B_2(\Cat')\to \Cat$
induced by the projection $\Cat'\to\Cat$ (and the triviality that
$B_2(\Cat)=\Cat$). So, given a functor $F\co\Cat\to \Spectra$ (say)
there is an induced (simplicially enriched) functor $B_2(F') = F \circ \Pi\co
B_2(\Cat')\to\Spectra$.

\begin{lemma}\label{lem:B2-contractible}
  For any small category $\Cat$, the projection map
  $\Pi\co B_2(\Cat')\to\Cat$ is a homotopy equivalence on each $\Hom$
  space.
\end{lemma}
\begin{proof}
  The category $\Hom_{\Cat'}(x,y)$ decomposes as a disjoint union of subcategories
  \[
  \Hom_{\Cat'}(x,y)=\coprod_{f\in \Hom(x,y)}\{(f_n,\dots,f_1)\mid
  f_n\circ \cdots\circ f_1=f\},
  \]
  and for each of these subcategories, every object is initial.
\end{proof}

The notion of homotopy colimits extends easily to functors
from simplicially enriched categories: in Formula~(\ref{eq:hocolim2}), say, one
replaces the disjoint union over sequences of composable morphisms
with the disjoint union of products
$\coprod_{x_0,\dots,x_n}\Hom(x_{n-1},x_n)\times\cdots\times
\Hom(x_0,x_1)\times[0,1]^n\times F(x_0)$, and quotient by the same
equivalence relation from \Definition{hocolim}. The properties
of homotopy colimits stated in Section~\ref{sec:colimit} extend
without change to diagrams from simplicially enriched categories.

\begin{lemma}\label{lem:top-hocolim-is-hocolim}
  Let $F\co\Cat\to\Spectra$ be a diagram from a small category
  $\Cat$. Then there is a homotopy equivalence $\hocolim_\Cat F\simeq
  \hocolim_{B_2(\Cat')} B_2(F')$.  
\end{lemma}
\begin{proof}
  By definition $B_2(F')=F\circ \Pi$. By \Lemma{B2-contractible}, the
  projection $\Pi$ is a quasi-equivalence, and in particular homotopy
  cofinal, so the result follows from
  Property~\ref{item:hocolim-cofinal} of homotopy colimits.
\end{proof}

\subsection{The Elmendorff-Mandell
  machine}\label{sec:elmendorff-mandell-machine}
A permutative category is a strictly unital, strictly associative,
symmetric monoidal category (see, e.g.,~\cite[Definition
3.1]{EM-top-machine}). Let $\PermuCat$ denote the (strict)
$2$-category of permutative categories~\cite[Definition
3.2]{EM-top-machine}.  The Elmendorff-Mandell
machine~\cite{EM-top-machine} is a functor
\[
  K\co B_2(\PermuCat)\to \Spectra
\]
from ($B_2$ of) the category of permutative categories to spectra.
(Constructions of this type go back to
Segal~\cite{Segal-top-categories}.) Rather than explain how the
machine works, we list the properties we will need:
\begin{enumerate}[label=(EM-\arabic*),ref=(EM-\arabic*)]
\item\label{item:EM-object-gives-point} Given a permutative category
  $\Cat$ and an object $x$ in $\Cat$ there is an induced map $K(x)\co
  \SphereS\to K(\Cat)$. Further, this is natural in the sense that
  given a functor of permutative categories $F\co \Cat\to\Dat$ the
  following diagram commutes:
  \[
  \xymatrix{
    \SphereS\ar[d]_{K(x)}\ar[dr]^{K(F(x))}\\
    K(\Cat) \ar[r]_{K(F)}& K(\Dat).
  }
  \]
\item\label{item:EM-prod-coprod} The Cartesian product is both the categorical product and
  coproduct in the category of permutative categories.
\item\label{item:EM-preserve-prods} Given permutative categories
  $\Cat$ and $\Dat$, $K(\Cat\times\Dat)=K(\Cat)\times K(\Dat)$. 
\item\label{item:EM-sets-to-sphere} The category $\Sets$ of finite
  sets, with disjoint union, is equivalent to a permutative category
  \cite{Isbell-coherent-algebras}; to keep the exposition clear we
  will continue to use the name $\Sets$ for this category. The map
  $\SphereS\to K(\Sets)$ induced by a $1$-element set and
  property~\ref{item:EM-object-gives-point} is a stable homotopy
  equivalence.  (This is a version of the Barratt-Priddy-Quillen
  theorem.)
\item\label{item:EM-2-func} Given a $2$-category $\Cat$ in which all 2-morphisms are
  isomorphisms and a strict $2$-functor
  $F\co\Cat\to\PermuCat$, there is an induced (simplicially enriched) functor
  $K(B_2(F))\co B_2(\Cat)\to \Spectra$.
\end{enumerate}

\begin{construction}\label{construction:Burn-to-permu}
  Given a set $X$, we can consider the category $\prod_{x\in X}\Sets$.
  Given a correspondence $(C,s,t)$ from $X$ to $Y$, there is an
  induced functor $\prod_{x\in X}\Sets\to \prod_{y\in Y}\Sets$ which
  sends
  \[
    (A_x)_{x\in X}\mapsto \bigl(\bigcup_{x\in X} (s^{-1}(x)\cap
    t^{-1}(y))\times A_x\bigr)_{y\in Y}.
  \]
  (Note that the union operation in the above formula is actually a
  disjoint union.)  An isomorphism between correspondences $(C,s,t)$
  and $(C',s',t')$ can be viewed as simply a relabeling of the
  elements of $C$; and this relabeling induces a natural isomorphism
  between the two functors.
\end{construction}

\begin{lemma}
  \Construction{Burn-to-permu} defines a lax $2$-functor $\BurnsideCat \to \PermuCat$.
\end{lemma}
\begin{proof}[Sketch of proof]
  Verify that the category $\prod_{x\in X}\Sets$ is naturally
  isomorphic to the category $\Sets/X$, and that under this
  identification the functor induced by $C$ is naturally isomorphic to
  the functor $A \mapsto C\times_X A$.
\end{proof}

\begin{remark}
  Using properties~\ref{item:EM-preserve-prods}
  and~\ref{item:EM-sets-to-sphere},
  $K(\prod_{x\in X}\Sets)\simeq\prod_{x\in X}K(\Sets)\simeq
  \prod_{x\in X}\SphereS$. With respect to this decomposition,
  however, the map $\prod_{x\in X}\SphereS\to \prod_{y\in Y}\SphereS$
  induced by a correspondence $C\co X\to Y$ is not geometrically
  obvious.
\end{remark}

\begin{construction}
  Given a category $\Cat$, viewed as a $2$-category with only identity
  $2$-morphisms, and a strictly unitary, lax $2$-functor
  $F\co \Cat\to\BurnsideCat$ there is an induced strict $2$-functor
  $F'\co \Cat'\to\BurnsideCat$, where $\Cat'$ is as in
  \Definition{cat-prime}. Composing with the
  $\prod_{x\in(-)}\Sets$ construction gives a strict $2$-functor
  $\Cat'\to \PermuCat$, which we will still denote $F'$, with
  $F'(u)=\prod_{x\in F(u)}\Sets$. Finally, by~\ref{item:EM-2-func}, the $K$-theory
  functor gives a functor $K(B_2(F'))\co B_2(\Cat')\to \Spectra$.
\end{construction}
Hu-Kriz-Kriz apply this construction to the Khovanov functor
$F_{\HKKa}\co \CCat{n}\to\BurnsideCat$, and then take the iterated
mapping cone as in \Section{mapping-cone} to obtain their Khovanov
stable homotopy type.

\begin{figure}
  \centering
  \[
    \begin{tikzpicture}
      \node at (0,0) (HKK) {$\hocolim K(B_2(F'))^\othplus$};
      \node at (6,0) (vecFBp) {$\hocolim K(B_2((\vec{F}\circ B)'))^\othplus$};
      \node at (12,0) (Gpp) {$\hocolim K(B_2(G'_p))^\othplus$};      
      \node at (12,-1.25) (Gp) {$\hocolim K(G_p)^\othplus$};      
      \node at (6,-1.25) (G) {$\hocolim G^\othplus$};      
      \node at (0,-1.25) (ours) {$\hocolim\thicf{F}^\othplus$};      
      \draw[<->] (HKK) to node[above]{\ref{item:F-FAr}} (vecFBp);
      \draw[<->] (vecFBp) to node[above]{\ref{item:big-step}} (Gpp);
      \draw[<->] (Gpp) to node[right]{\ref{item:G-B2G}} (Gp);
      \draw[<->] (Gp) to node[above]{\ref{item:G-Gp}} (G);
      \draw[<->] (G) to node[above]{\ref{item:thic-F-thic-G}} (ours);
    \end{tikzpicture}
  \]
  \caption{\textbf{The steps in the proof of
      Proposition~\ref{prop:HKK-real-agree}.} Each arrow is labeled by
    the step showing it is an equivalence.}
    \label{fig:prop-HKK-outline}
\end{figure}

\begin{proposition}\label{prop:HKK-real-agree}
  Given a strictly unitary, lax $2$-functor $F\co\CCat{n}\to\BurnsideCat$,
  there is a stable homotopy equivalence between the Hu-Kriz-Kriz
  realization $\hocolim K(B_2(F'))^\othplus$ and the realization $\hocolim
  \thicf{F}^\othplus$ from Section~\ref{sec:subsec-realize}.
\end{proposition}
\begin{proof}
  To identify the Hu-Kriz-Kriz construction and the thickening
  construction from \Section{realize-functor} we use a sequence of
  intermediate diagrams, somewhat in the spirit
  of \Section{small-is-big}. (See \Figure{prop-HKK-outline} for an
  overview.)
\begin{enumerate}[leftmargin=*,label=(\arabic*)]
\item\label{item:F-FAr} By \Lemma{lift-to-arrow}, we can lift the functor $F\co
  \CCat{n}\to\BurnsideCat$ to a functor $\vec{F}\co
  \ArrowCat(\CCat{n})\to\BurnsideCat$, so that $F=\vec{F}\circ
  A$. Composing with the composition map $B\co
  \thic{\CCat{n}}\to\ArrowCat(\CCat{n})$ gives a functor $\vec{F}\circ
  B\co \thic{\CCat{n}}\to\BurnsideCat$.

  There is a homotopy equivalence $\hocolim
  K(B_2(F'))^{\othplus}\simeq \hocolim K(B_2((\vec{F}\circ
  B)'))^\othplus$; the argument is similar to the one in
  Section~\ref{sec:small-is-big}.
  %
  In the commutative diagram 
  \[
    \xymatrix{
    B_2({\CCat{n}}')_\othplus\ar[r]\ar[d]_\simeq &
    B_2(\ArrowCat(\CCat{n})')_\othplus\ar[d]^{\simeq} &
    B_2((\thic{\CCat{n}})')_\othplus\ar[l]\ar[d]^{\simeq}\\
    \CCat{n}_{\othplus}\ar[r]_-{A_\othplus} &
    \ArrowCat(\CCat{n})_{\othplus} & \thic{\CCat{n}}_\othplus\ar[l]^-{B_\othplus}, 
  }
  \]
  the vertical arrows are quasi-equivalences by
  \Lemma{B2-contractible}, and 
  $A_\othplus$ and $B_\othplus$ are homotopy cofinal by
  \Lemma{CCat-to-Ar-cofinal} and \Lemma{thic-to-Ar-cofinal}, so
  the functors
  in the top row are also homotopy cofinal. Hence, by
  Property~\ref{item:hocolim-cofinal} of homotopy colimits
  $\hocolim K(B_2(F'))^\othplus\simeq \hocolim
  K(B_2(\vec{F}'))^\othplus\simeq \hocolim K(B_2((\vec{F}\circ
  B)'))^\othplus$.
  %
\item\label{item:thic-F-thic-G} Recall that $\thicf{F}$ sends an object
  $(u\stackrel{f}{\longrightarrow} v\stackrel{g}{\longrightarrow} w)$ to
  $\bigvee_{a\in F(f)}\prod_{b\in F(g),\ s(b)=t(a)}\SphereS$. Further,
  the definition of $\thicf{F}$ uses only the universal properties of
  product and coproduct, so for any spectrum $X$ we could
  define a functor $\thicf[X]{F}\co \thic{\CCat{n}}\to\Spectra$ with
  $\thicf[X]{F}(u\stackrel{f}{\longrightarrow}
  v\stackrel{g}{\longrightarrow} w) = \bigvee_{a\in F(f)}\prod_{b\in
    F(g),\ s(b)=t(a)}X$ (cf.~\Remark{more-thickenings}). Moreover, this thickening procedure is
  clearly natural in $X$. Taking the special case $X=K(\Sets)$ and
  using Property~\ref{item:EM-sets-to-sphere} gives a diagram $G\co
  \thic{\CCat{n}}\to\Spectra$ with
  \[
  G(u\stackrel{f}{\longrightarrow} v\stackrel{g}{\longrightarrow} w)=
  \bigvee_{a\in F(f)}\prod_{\substack{b\in F(g)\\s(b)=t(a)}}K(\Sets),
  \]
  and a natural transformation of diagrams $\thicf{F}\to G$ which is
  an equivalence on objects. In particular, $\hocolim
  \thicf{F}^\othplus \simeq \hocolim G^\othplus$.
  
\item\label{item:G-Gp} There is also a functor $G_p\co
  \thic{\CCat{n}}\to\PermuCat$ defined similarly to $\thicf{F}$, using
  the product and coproduct on $\PermuCat$ in place of the product and
  wedge sum of spaces, and using $\Sets$ in place of $\SphereS$. That
  is, 
  \[
  G_p(u\stackrel{f}{\longrightarrow} v\stackrel{g}{\longrightarrow}
  w)= \coprod_{a\in F(f)}\prod_{\substack{b\in F(g)\\s(b)=t(a)}}\Sets.
  \]
  Recall from~\ref{item:EM-prod-coprod} that finite products and coproducts of permutative categories
  are given by the Cartesian product. Thus, using
  Property~\ref{item:EM-preserve-prods},
  \[
    K(G_p)(u\stackrel{f}{\longrightarrow} v\stackrel{g}{\longrightarrow} w) \simeq \prod_{a\in F(f)}\prod_{\substack{b\in F(g)\\s(b)=t(a)}}K(\Sets).
  \]
  (Note that since $\thic{\CCat{n}}$ has only identity 2-morphisms,
  $B_2(\thic{\CCat{n}})$ can be identified with $\thic{\CCat{n}}$, and
  therefore, we have written $K(G_p)$ instead of $K(B_2(G_p))$.)
    
  Using the stable homotopy equivalence between wedge sum and product
  we get a natural transformation of diagrams from $G$ to $K(G_p)$
  which is an equivalence on objects. In particular, $\hocolim
  G^\othplus \simeq \hocolim K(G_p)^\othplus$.
\item\label{item:G-B2G} By \Lemma{top-hocolim-is-hocolim}, there is a
  homotopy equivalence $\hocolim K(G_p)^\othplus\simeq \hocolim
  B_2(K(G_p)')^\othplus$.  However, since $G_p$ is a strict functor from
  a $1$-category to $\PermuCat$,
  $B_2(K(G_p)')=K(G_p)\circ\Pi=K(B_2(G_p'))$, 
  where $\Pi\from
  B_2((\thic{\CCat{n}})')\to\thic{\CCat{n}}$ is the projection;
  therefore, $\hocolim K(G_p)^\othplus\simeq \hocolim
  K(B_2(G_p'))^\othplus$.
\item\label{item:big-step} To finish the identification, 
  there are isomorphisms $H$ from $G_p'$ to $(\vec{F}\circ
  B)'$. On the objects, the natural transformation will induce the
  equivalence from
  \[
  G_p'(u\stackrel{f}{\longrightarrow}
  v\stackrel{g}{\longrightarrow} w)=
  G_p(u\stackrel{f}{\longrightarrow} v\stackrel{g}{\longrightarrow}
  w)=
  \coprod_{a\in F(f)}\prod_{\substack{b\in F(g)\\
         s(b)=t(a)}}\Sets
  \]
  to
  \[
  (\vec{F}\circ B)'(u\stackrel{f}{\longrightarrow}
  v\stackrel{g}{\longrightarrow} w)=(\vec{F}\circ
  B)(u\stackrel{f}{\longrightarrow} v\stackrel{g}{\longrightarrow}
  w)=\prod_{c \in F(g\circ f)}\Sets
  \]
  which comes from the natural bijection $F(g \circ f) \to F(g)
  \times_{F(v)} F(f)$ between their indexing sets, and the
  identification of coproducts and products in $\PermuCat$ with
  Cartesian products. Because Cartesian product of sets is not
  strictly associative, and because of the 2-categorical nature of the identification between
  co-products and products in $\PermuCat$,
  this does not strictly define a strict natural
  isomorphism $G_p' \to (\vec{F}\circ B)'$. Instead, it defines a
  pseudonatural equivalence: for a map $\varphi\co
  (u\stackrel{f}{\longrightarrow} v\stackrel{g}{\longrightarrow} w)
  \to (u'\stackrel{f'}{\longrightarrow} v'\stackrel{g'}{\longrightarrow}
  w')$ in $\thic{\CCat{n}}$, there is a natural isomorphism of
  functors
  \[
  (\vec{F}\circ B)'(\varphi) \circ H(u\stackrel{f}{\longrightarrow}
  v\stackrel{g}{\longrightarrow} w) \cong
  H(u'\stackrel{f'}{\longrightarrow} v'\stackrel{g'}{\longrightarrow}
  w') \circ G_p'(\varphi).
  \]
  which respects composition in $\thic{\CCat{n}}$.

  We define $\mathscr{E}$ to be the category with objects $1$ and $0$
  and a unique morphism between any pair of objects. Let $\Dat
  =\mathscr{E} \times \thic{\CCat{n}}_\othplus$, and let $\Dat_0$ and $\Dat_1$
  be the full subcategories of $\Dat$ spanned by the objects of
  $\{0\}\times\thic{\CCat{n}}_\othplus$ and $\{1\}\times\thic{\CCat{n}}_\othplus$,
  respectively.  Note that the projection $\Dat \to \thic{\CCat{n}}_\othplus$
  is an equivalence, and restricts to isomorphisms
  from both $\Dat_0$ and $\Dat_1$ to $\thic{\CCat{n}}_\othplus$.

  The pseudonatural equivalence above defines a lax 2-functor, also
  denoted $H$, from $\Dat$ to $\PermuCat$ whose restriction to $\{1\}
  \times \thic{\CCat{n}}$ is $G_p'$ and whose restriction to $\{0\}
  \times \thic{\CCat{n}}$ is $(\vec{F} \circ B)'$.
  Therefore, $K(B_2(H))$ restricts to 
  $K(B_2(G_p'))^\othplus$ and $K(B_2(\vec{F} \circ B)')^\othplus$,
  and since both inclusions $\Dat_0\to\Dat$ and $\Dat_1\to\Dat$ are
  homotopy cofinal,
  \[
  \hocolim K(B_2((\vec{F}\circ B)'))^\othplus \simeq \hocolim K(B_2(H))\simeq \hocolim
  K(B_2(G_p'))^\othplus.
  \]
\end{enumerate}

Putting these all together (cf.~\Figure{prop-HKK-outline}), in
conjunction with \Lemma{HKK-extension}, the Hu-Kriz-Kriz realization
$\hocolim K(B_2(F'))^\othplus$ agrees with our thickening construction
$\hocolim \thicf{F}^\othplus$.  
\end{proof}

\subsection{Proof that the Khovanov homotopy types agree}\label{sec:HKK-is-LLS}

\begin{proof}[Proof of \Theorem{agree-with-HKK}]
  \Lemma{HKK-same-Burnside} identifies the $2$-functors
  $\CCat{n}\to\BurnsideCat$ used in this paper
  and~\cite{HKK-Kh-htpy}. Proposition~\ref{prop:HKK-real-agree}
  identifies Hu-Kriz-Kriz's realization of this functor with ours.
\end{proof}

\section{Khovanov homotopy type of a disjoint union and connected sum}\label{sec:Kh-disj-union}
In this section we prove
Theorems~\ref{thm:disjoint-union},~\ref{thm:connect-sum},
and~\ref{thm:unred-con-sum}.  For the first two theorems, we merely
need to show that the functor associated to a disjoint union
(respectively connect sum) of links is the product of the functors of
the individual links:
\begin{proposition}\label{prop:Kh-func-prod}
  Let $L_1$ and $L_2$ be link diagrams, and let $L_1\amalg L_2$ be
  their disjoint union. Order the crossings in $L_1\amalg L_2$ so that
  all of the crossings in $L_1$ come before all of the crossings in
  $L_2$. Then
  \begin{align*}
  \KhFunc^j(L_1\amalg L_2)&\cong\coprod_{j_1+j_2=j}\KhFunc^{j_1}(L_1)\times
  \KhFunc^{j_2}(L_2),\\
  \shortintertext{where $\times$ denotes the product of functors
    (\Definition{product}), $\coprod$ denotes the disjoint union of functors
  (\Definition{disjoint-union}), and $\cong$ denotes natural
  isomorphism of 2-functors. If we
    fix a basepoint on $L_1$ then}
  \rKhFunc^j(L_1\amalg
  L_2)&\cong\coprod_{j_1+j_2=j}\rKhFunc^{j_1}(L_1)\times
  \KhFunc^{j_2}(L_2).\\
  \shortintertext{If we fix basepoints on $L_1$ and $L_2$ and let $L_1\#L_2$ denote the connected sum (at the basepoints) then}
  \rKhFunc^j(L_1\#L_2)&\cong\coprod_{j_1+j_2=j}\rKhFunc^{j_1}(L_1)\times \rKhFunc^{j_2}(L_2).
  \end{align*}
  \end{proposition}

\begin{proof}
  We will prove the first statement; the proofs of the other two are
  similar. Let $n_i$ be the number of crossings in $L_i$. To keep
  notation simple, write $F=\KhFunc^j(L_1\amalg L_2)$, $X_v=F(v)$,
  $A_{v,w}=F(\cmorph{v}{w})$, $G=\coprod_{j_1+j_2=j}\KhFunc^{j_1}(L_1)\times
  \KhFunc^{j_2}(L_2)$, $Y_v=G(v)$ and $B_{v,w}=G(\cmorph{v}{w})$.

  By \Lemma{characterize-functor}, it suffices to construct bijections
  $\phi_v\co X_v\stackrel{\cong}{\longrightarrow} Y_v$ and
  $\psi_{v,w}\co A_{v,w}\stackrel{\cong}{\longrightarrow} B_{v,w}$ for
  all $v>w$ with $|v|-|w|=1$ so that $\psi_{v,w}$ respects the source
  and target maps and for any $u> w$ with $|u|-|w|=2$, the following
  diagram commutes:
  \begin{equation}\label{eq:disj-un-diag}
  \xymatrix{
    A_{v,w}\times_{X_v}A_{u,v} \ar[rr]^{\psi_{v,w}\times
      \psi_{u,v}}\ar[d]^{F_{u,v,w}}
   & & B_{v,w}\times_{Y_v}B_{u,v}\ar[d]^{G_{u,v,w}}\\
   A_{u,w} & & B_{u,w}\\
   A_{v',w}\times_{X_{v'}}A_{u,v'} \ar[rr]^{\psi_{v',w}\times
     \psi_{u,v'}} \ar[u]_{F_{u,v',w}} & & B_{v',w}\times_{Y_{v'}}B_{u,v'}.\ar[u]_{G_{u,v',w}}
  }
  \end{equation}
  Here, $v$ and $v'$ are the two vertices so that $u>v,v'> w$. Note
  that all arrows in this diagram are isomorphisms.

  The map $\phi_v$ is the canonical identification between Khovanov
  generators for $L_1\amalg L_2$ and pairs of a Khovanov generator for
  $L_1$ and a Khovanov generator for $L_2$. There is a unique map
  $\psi_{v,w}\co A_{v,w}\to B_{v,w}$ for $v>w$ with $|v|-|w|=1$ which
  commutes with the source and target maps, because:
  \begin{enumerate}
  \item\label{item:empty-or-one} Given $x_v\in X_v$ and $x_w\in X_w$,
    $s^{-1}(x_v)\cap t^{-1}(x_w)\subset A_{v,w}$ is either empty, if
    $x_v$ does not occur in the Khovanov differential of $x_w$, or
    consists of a single point, if $x_v$ does occur in the Khovanov
    differential of $x_w$. Similar statements hold for $Y_{v}$ and
    $B_{v,w}$. It follows that if $\psi_{v,w}$ exists then it is
    unique.
  \item\label{item:is-in-fact-chain-map} The canonical identification
    of Khovanov generators does, in fact, give a chain map. So, by the
    observations in the previous point, the map $\psi_{v,w}$ does
    exist.
  \end{enumerate}
  Except in one case, the same argument shows that the
  diagram~\eqref{eq:disj-un-diag} commutes: typically, for each
  $x_u\in X_u$ and $x_w\in X_w$ (with $u>w$ and $|u|-|w|=2$),
  $s^{-1}(x_u)\cap t^{-1}(x_w)\subset A_{u,w}$ is either empty or has
  a single element. The exceptional case is the case of a ladybug
  configuration, as in~\cite[Section~5.4]{RS-khovanov} (see also
  \Figure{ladybug-matching}). In the ladybug case, either both
  crossings under consideration lie in $L_1$ or both crossings lie in
  $L_2$, from which it follows easily that the diagram commutes. (This
  is immediate for the present case when we are considering the
  disjoint union $L_1\amalg L_2$; the connect-sum case $L_1\# L_2$ is
  also fairly obvious.)
\end{proof}

We are now ready to prove Theorems~\ref{thm:disjoint-union}
and~\ref{thm:connect-sum}, which we recall for the reader's convenience:

\begin{reptheorem}{thm:disjoint-union}
  Let $L_1$ and $L_2$ be links, and $L_1\amalg L_2$ their disjoint union. Then
  \begin{equation}\tag{\ref*{eq:Kh-disjoint}}
  \KhSpace^j(L_1\amalg L_2)\simeq \bigvee_{j_1+j_2=j}\KhSpace^{j_1}(L_1)\smas\KhSpace^{j_2}(L_2).
  \end{equation}
  Moreover, if we fix a basepoint $p$ in $L_1$, not at a crossing, and
  consider the corresponding basepoint for $L_1\amalg L_2$, then
  \begin{equation}\tag{\ref*{eq:rKh-disjoint}}
  \rKhSpace^j(L_1\amalg L_2)\simeq \bigvee_{j_1+j_2=j}\rKhSpace^{j_1}(L_1)\smas\KhSpace^{j_2}(L_2).
  \end{equation}
\end{reptheorem}

\begin{reptheorem}{thm:connect-sum}
  Let $L_1$ and $L_2$ be based links and $L_1\# L_2$ the connected sum
  of $L_1$ and $L_2$, where we take the connected sum near the
  basepoints. Then
  \begin{equation}\tag{\ref*{eq:rKh-conn-sum}}
  \rKhSpace^j(L_1\# L_2)\simeq \bigvee_{j_1+j_2=j}\rKhSpace^{j_1}(L_1)\smas\rKhSpace^{j_2}(L_2).
  \end{equation}
\end{reptheorem}

\begin{proof}[Proof of Theorems~\ref{thm:disjoint-union} and~\ref{thm:connect-sum}]
  We will prove Formula~(\ref{eq:Kh-disjoint}); the proofs of
  Formulas~(\ref{eq:rKh-disjoint}) and~(\ref{eq:rKh-conn-sum}) are
  essentially the same. 

  Fix a diagram for $L_1\amalg L_2$ so that there are no crossings
  between $L_1$ and $L_2$. Order the crossings in $L_1\amalg L_2$ so
  that all of the crossings in $L_1$ come before all of the crossings
  in $L_2$. By \Proposition{Kh-func-prod},
  \[
    \KhFunc^j(L_1\amalg L_2)\cong\coprod_{j_1+j_2=j}\KhFunc^{j_1}(L_1)\times
    \KhFunc^{j_2}(L_2).
  \]
  By \Lemma{iso-realize-same}, naturally isomorphic functors have
  stably homotopy equivalent realizations.
  By Propositions~\ref{prop:realize-product}
  and~\ref{prop:realize-disjoint-union}, the realization of $\amalg_{j_1+j_2=j}\KhFunc^{j_1}(L_1)\times
  \KhFunc^{j_2}(L_2)$ is 
  $
  \bigvee_{j_1+j_2=j}\KhSpace^{j_1}(K_1)\smas\KhSpace^{j_2}(L_2).
  $
\end{proof}

These results quickly imply \Corollary{Kh-squares}, which we also recall:
\begin{repcorollary}{cor:Kh-squares}
  For any $n$ there exists a link $L_n$ so that the operation
  \[
  \Sq^n\co \Kh^{i,j}(L_n)\to \Kh^{i+n,j}(L_n)
  \]
  is non-zero, for some $i,j\in\ZZ$. Similarly, there exists a knot $K_n$ so that the operation
  \[
  \Sq^n\co \rKh^{i,j}(K_n)\to \rKh^{i+n,j}(K_n)
  \]
  is non-zero, for some $i,j\in\ZZ$. Further, for this knot, the operation 
  \[
  \Sq^n\co \rKh^{i,j}(K_n)\to \rKh^{i+n,j}(K_n)
  \]
  is also non-zero for some $i,j\in\ZZ$.
\end{repcorollary}
\begin{proof}
  For the first statement, consider the disjoint union $L_n$ of $n$
  copies of the left-handed trefoil $T$. It follows
  from~\cite[Proposition~9.2]{RS-khovanov} that
  \[
  \KhSpace(T)\simeq\Sigma^{-3}\SphereS\vee \Sigma^{-2}\SphereS\vee
  \SphereS\vee \SphereS\vee \Sigma^{-4}\mathbb{R}\mathrm{P}^2
  \]
  (compare~\cite[Example~9.4]{RS-khovanov}). So,
  by \Theorem{disjoint-union},
  \[
  \KhSpace(L_n)\simeq \Sigma^{-4n}(\overbrace{\mathbb{R}\mathrm{P}^2\smas\cdots\smas\mathbb{R}\mathrm{P}^2}^{n\text{ copies}})\vee Y
  \]
  for some space $Y$. It follows from the Cartan formula that 
  \[
  \Sq^n\co H^n\bigl((\mathbb{R}\mathrm{P}^2)^{\smas n};\ZZ/2\ZZ)\to 
  H^{2n}\bigl((\mathbb{R}\mathrm{P}^2)^{\smas n};\ZZ/2\ZZ)
  \]
  is non-trivial. This proves the first part of the result.

  For the second part of the result, let $K$ be the knot
  $15^n_{41127}$ and let $K_n$ be the connect sum of $n$ copies of
  $K$. According to the calculation in~\cite[Figure
  6]{Shu-kh-patterns}, $\rKh^{-2,0}(K;\ZZ)\cong\ZZ$,
  $\rKh^{0,0}(K;\ZZ)=\ZZ/2\ZZ$, and $\rKh^{i,0}(K;\ZZ)=0$ for $i\neq
  -2,0$. In particular, there is a class $\alpha\in
  \rKh^{-1,0}(K;\ZZ/2\ZZ)\cong \ZZ/2\ZZ$ so that $\Sq^1(\alpha)$ is
  non-zero and $\Sq^i(\alpha)=0$ for $i>1$. So, it follows from
  \Theorem{connect-sum} and the Cartan formula that for the class
  $\beta=\alpha\smas\cdots\smas\alpha\in \rKh^{-n,0}(K_n)$,
  $\Sq^n(\beta)$ is non-trivial.

  Finally, we argue that $\Sq^n$ is nonvanishing on $\Kh^{*,*}(K_n)$,
  as well. To see this, recall from~\cite[Theorem 3]{RS-khovanov} that
  there is a cofibration sequence
  \[
    \rKhSpace^{j-1}(K_n)\to \KhSpace^j(K_n)\to\rKhSpace^{j+1}(K_n)
  \]
  inducing the long exact sequence
  \[
    \cdots\to \rKh^{i,j+1}(K_n)\to \Kh^{i,j}(K_n)\to \rKh^{i,j-1}(K_n)\to\rKh^{i+1,j+1}(K_n)\to\cdots.
  \]
  Further, for coefficients in $\ZZ/2\ZZ$,
  \[
    \dim(\Kh^{i,j}(K_n;\ZZ/2\ZZ))=\dim(\rKh^{i,j-1}(K_n;\ZZ/2\ZZ))+\dim(\rKh^{i,j+1}(K_n;\ZZ/2\ZZ))
  \]
  \cite[Proposition 1.7]{OSzR-kh-oddkhovanov}, so the map
  $\Kh^{i,j}(K_n;\ZZ/2\ZZ)\to \rKh^{i,j-1}(K_n;\ZZ/2\ZZ)$ is
  surjective. In particular, there is a class $\gamma\in
  \Kh^{-n,1}(K_n;\ZZ/2\ZZ)$ which maps to $\beta\in
  \rKh^{-n,0}(K_n;\ZZ/2\ZZ)$. By naturality of $\Sq^n$, it follows that
  $\Sq^n(\gamma)$ maps to $\Sq^n(\beta)\neq 0$, so $\Sq^n(\gamma)\neq
  0$, as desired.
  %
\end{proof}

Finally, we turn to the unreduced Khovanov homology of a connected
sum.
\begin{definition}
  Consider the Khovanov homotopy type associated to the unknot,
  $\KhSpace(U)=\SphereS\vee\SphereS$, which is the suspension spectrum
  of $S^0\vee S^0=\{\ast, p_-, p_+\}$. The spectrum
  $\KA{1}\defeq \KhSpace(U)$ has a product
  $\mu\co \KA{1}\smas \KA{1}\to \KA{1}$ induced by the map of spaces
  $(S^0\vee S^0)\smas(S^0\vee S^0)\to (S^0\vee S^0)$ given by
  \[
    p_-\smas p_-\mapsto p_-\qquad p_+\smas p_-\mapsto p_+ \qquad
    p_-\smas p_+\mapsto p_+\qquad p_+\smas p_+\mapsto *.
  \]
\end{definition}

\begin{lemma}
  The operation $\mu$ makes $\KA{1}$ into a ring spectrum.
\end{lemma}
\begin{proof}
  This is immediate from the definitions.
\end{proof}

\begin{remark}
  The map on reduced cohomology induced by $\mu$ is the split map
  $\Kh(U)\to \Kh(U)\otimes \Kh(U)$. (The generators of $\Kh(U)$ are
  $x_-$ corresponding to $p_-$ and $x_+$ corresponding to $p_+$.)
\end{remark}

\begin{remark}
  The notation $\KA{1}$ is chosen to be reminiscent of the first of
  Khovanov's arc algebras $H^n$.
\end{remark}


Next, fix a link diagram $K$ with $n$ crossings and a basepoint
$p\in K$.  We make $\KhSpace(K)$ into a module spectrum over
$\KA{1}$. We will use the box map realization
from \Section{smaller-cube}, applied to the functor
$F_{\Kh}\from\CCat{n}\to\BurnsideCat$, using a particular choice of
spatial refinement, which we need to specify. The resulting CW complex
is described in \Proposition{cwcomplex-box-hocolim}, producing a finite CW
spectrum $\KhSpace(K)$. Before specifying the spatial refinement, we
recall the definitions of the reduced Khovanov functors
$F_{\pm{\Kh}}\from\CCat{n}\to\BurnsideCat$, to fix notation:

\begin{notation}
  For $u\in\{0,1\}^n$, define $F_{+\Kh}(u)$ (respectively,
  $F_{-\Kh}(u)$) to be the subset of $F_{\Kh}(u)$ where the circle in
  the complete resolution $\mc{P}(u)$ containing the basepoint is
  labeled $x_+$ (respectively, $x_-$). For $u>v\in\{0,1\}^n$, define
  the correspondence from $F_{+\Kh}(u)$ to $F_{+\Kh}(v)$
  (respectively, from $F_{-\Kh}(u)$ to $F_{-\Kh}(v)$) to be the subset
  $s^{-1}(F_{+\Kh}(u))\cap t^{-1}(F_{+\Kh}(v))$ (respectively,
  $s^{-1}(F_{-\Kh}(u))\cap t^{-1}(F_{-\Kh}(v))$) of the correspondence
  from $F_{\Kh}(u)$ to $F_{\Kh}(v)$. It is straightforward from the
  definition of the Khovanov differential (\Section{khovanov-basic})
  that this produces well-defined functors
  $F_{\pm{\Kh}}\from\CCat{n}\to\BurnsideCat$.  Furthermore, the map
  from $F_{+\Kh}(u)$ to $F_{-\Kh}(u)$ which relabels the pointed
  circle in $\mc{P}(u)$ from $x_+$ to $x_-$ induces a natural isomorphism
  from $F_{+\Kh}$ to $F_{-\Kh}$; we often write $F_{\wt{\Kh}}$ to
  denote either functor.  Finally, for any $u\in\{0,1\}^n$,
  $F_{\Kh}(u)=F_{+\Kh}(u)\amalg F_{-\Kh}(u)$; and for any
  $u>v\in\{0,1\}^n$, the correspondence $F_{\Kh}(\cmorph{u}{v})$ from
  $F_{\Kh}(u)$ to $F_{\Kh}(v)$ is the disjoint union of the
  correspondence $F_{+\Kh}(\cmorph{u}{v})$ from $F_{+\Kh}(u)$ to
  $F_{+\Kh}(v)$, the correspondence $F_{-\Kh}(\cmorph{u}{v})$ from
  $F_{-\Kh}(u)$ to $F_{-\Kh}(v)$, and some correspondence from
  $F_{-\Kh}(u)$ to $F_{+\Kh}(v)$.
\end{notation}

\begin{definition}
  A spatial refinement $\wt{F}_{\Kh}$ of $F_{\Kh}$ induces spatial
  refinements $\wt{F}_{+\Kh}$ of $F_{+\Kh}$ and $\wt{F}_{-\Kh}$ of
  $F_{-\Kh}$. We call $\wt{F}_{\Kh}$ a \emph{pointed spatial
    refinement} if, with respect to the natural isomorphism between
  $F_{+\Kh}$ and $F_{-\Kh}$, the boxes used to define $\wt{F}_{+\Kh}$
  and $\wt{F}_{-\Kh}$ are identical. In this case, there is an induced
  natural isomorphism between $\wt{F}_{+\Kh}$ and $\wt{F}_{-\Kh}$. For
  pointed spatial refinements, the CW complexes
  $\hocolim(\wt{F}^+_{+\Kh})$ and $\hocolim(\wt{F}^+_{-\Kh})$ are
  canonically isomorphic.
\end{definition}

\begin{lemma}
  Every pointed link diagram $(K,p)$ admits a pointed spatial
  refinement $\wt{F}_{\Kh}$.
\end{lemma}

\begin{proof}
We construct a spatial refinement $\wt{F}_{\Kh}$ of $F_{\Kh}$ in
several steps.  First construct a spatial refinement
$\wt{F}_{-\Kh}$ of $F_{-\Kh}$ with the additional
restriction that the box maps come from the subspaces
$E^\circ(\{B_x\},s)$ of $E(\{B_x\},s)$, that is, the sub-boxes are
contained in the interiors of the bigger boxes. Then use the natural
isomorphism between $F_{+\Kh}$ and $F_{-\Kh}$ to get a
spatial refinement $\wt{F}_{+\Kh}$ of $F_{+\Kh}$. Finally,
extend $\wt{F}_{+\Kh}$ and $\wt{F}_{-\Kh}$ to construct a
spatial refinement $\wt{F}_{\Kh}$ of $F_{\Kh}$, following the
inductive argument in the proof of
\Proposition{box-refinement-exist-unique}~(\ref{item:spatial-exists}).
For the induction step, fix a length-$\ell$ sequence
$v_0\to\dots\to v_{\ell}$ of non-identity morphisms in $\CCat{n}$. 
There is a correspondence $F_{\Kh}(\cmorph{v_0}{v_\ell})$ and a subset
$F_{+\Kh}(\cmorph{v_0}{v_\ell})\amalg
F_{-\Kh}(\cmorph{v_0}{v_\ell})$; let $s$ be the source map of the
correspondence $F_{\Kh}(\cmorph{v_0}{v_\ell})$ and $s'$ the restriction of
$s$ to $F_{+\Kh}(\cmorph{v_0}{v_\ell})\amalg
F_{-\Kh}(\cmorph{v_0}{v_\ell})$. Induction and $\wt{F}_{+\Kh}$ and
$\wt{F}_{-\Kh}$ give a diagram
\[
\xymatrix{
  \bdy([0,1]^{\ell-1})\ar[r] \ar[d] & E^\circ(\{B_x\},s)\ar[d]\\
  [0,1]^{\ell-1} \ar[r]& E^\circ(\{B_x\},s'),
}
\]
where the right-hand vertical map forgets the boxes labeled by
elements of $F_{\Kh}(\cmorph{v_0}{v_\ell})\setminus
(F_{+\Kh}(\cmorph{v_0}{v_\ell})\amalg
F_{-\Kh}(\cmorph{v_0}{v_\ell}))$. The inductive step is to construct a
lift $[0,1]^{\ell-1}\to E^\circ(\{B_x\},s)$ making the diagram
commute.
\Lemma{box-maps-rel-contractible} guarantees the existence of such a
lift. Thus, induction implies that $F_{\Kh}$ has a spatial refinement
extending $\wt{F}_{+\Kh}$ and $\wt{F}_{-\Kh}$.
\end{proof}

\begin{construction}
  Given a pointed spatial refinement, define a map
  $\Psi\from\hocolim(\wt{F}^+_{\Kh})\to\hocolim(\wt{F}^+_{\Kh})$ as
  follows. Notice that $\hocolim(\wt{F}^+_{+\Kh})$ is a subcomplex of
  $\hocolim(\wt{F}^+_{\Kh})$ and $\hocolim(\wt{F}^+_{-\Kh})$ is the
  corresponding quotient complex. Define $\Psi$
  to be the composition
  \[
    \hocolim(\wt{F}^+_{\Kh})\onto\hocolim(\wt{F}^+_{-\Kh})\stackrel{\cong}{\longrightarrow}
    \hocolim(\wt{F}^+_{+\Kh})\into\hocolim(\wt{F}^+_{\Kh}),
  \]
  where the first map is the quotient map, the second map is the
  canonical isomorphism, and the third map is the subcomplex
  inclusion. Note that $\Psi$ is a cellular map. The induced map on
  $\KhCx(K)$, the reduced cellular cochain complex of $\hocolim(\wt{F}^+_{\Kh})$,
  sends generators that label the pointed circle by $x_-$ to zero and
  on the rest of the Khovanov generators relabels the pointed circle
  from $x_+$ to $x_-$.
\end{construction}

Now we are ready to define the $\KA{1}$-module structure on $\KhSpace(K)$. 
\begin{definition}\label{def:mod-struct}
  Define a map $\KhSpace(K)\smas\KA{1}\to\KhSpace(K)$ induced the following map of spaces:
  \begin{equation}\label{eq:mod-struct}
    \hocolim(\wt{F}^+_{\Kh})\smas\{\ast, p_-,
    p_+\}=(\hocolim(\wt{F}^+_{\Kh})\times\{p_-\})\vee(\hocolim(\wt{F}^+_{\Kh})\times\{p_+\}) \to
    \hocolim(\wt{F}^+_{\Kh}).
  \end{equation}
  On the first summand, the map
  $\hocolim(\wt{F}^+_{\Kh})\times\{p_-\}\to \hocolim(\wt{F}^+_{\Kh})$
  is projection to the first factor. On the second summand, the
  map is projection to the first factor composed with the map
  $\Psi\from\hocolim(\wt{F}^+_{\Kh})\to\hocolim(\wt{F}^+_{\Kh})$ defined above.
\end{definition}

\begin{lemma}
  \Definition{mod-struct} makes $\KhSpace(K)$ into a module spectrum
  over $\KA{1}$.
\end{lemma}
\begin{proof}
  This follows from the fact that $\Psi\circ\Psi$ sends all of
  $\hocolim(\wt{F}^+_{\Kh})$ to the basepoint.
\end{proof}

The ring spectrum $\KA{1}$ is commutative, so we can view the action
of $\KA{1}$ on $\KhSpace(K)$ as either a left or a right action.

Note that the induced map on the reduced cellular cochain complexes
associated to the map~\eqref{eq:mod-struct} is the split map
$\KhCx(K)\to \KhCx(K)\otimes \Kh(U)$.

\begin{proposition}\label{prop:quasi-isom-module-spectrum}
  Up to weak equivalence of $\KA{1}$-module spectra,
  $\KhSpace(K)$ is an invariant of pointed links. That is, if $(K,p)$
  and $(K',p')$ are pointed link diagrams representing isotopic
  pointed links, then there exist $\KA{1}$-module spectra
  $\KhSpace(K)=X_0,X_1,\dots, X_{\ell-1},X_\ell=\KhSpace(K')$, and for
  any adjacent pair $X_i,X_{i+1}$, either a map $X_i\to X_{i+1}$ or a
  map $X_{i+1}\to X_i$, which is both an $\KA{1}$-module map and a
  weak equivalence.
\end{proposition}

\begin{proof}
  First, observe that the $\KA{1}$-module structure is independent
  of the choice of box maps; the proof is essentially the same as
  \Proposition{box-refinement-exist-unique}~(\ref{item:spatial-unique}),
  but using \Lemma{box-maps-rel-contractible} instead of
  \Lemma{box-maps-contractible}.


  Next we show that, up to weak equivalence, this $\KA{1}$-module
  spectrum is invariant under Reidemeister moves. By a
  standard argument~\cite[Section~3]{Kho-kh-patterns}, we only
  need to consider Reidemeister moves that do not cross the marked
  point $p$. We follow the framework from
  \cite[Section~6]{RS-khovanov}. Let $K_0$ and $K_1$ (with $n_0$ and
  $n_1$ crossings respectively) be pointed link diagrams
  related by any of the three Reidemeister moves of
  \cite[Figure~6.1]{RS-khovanov}, and assume $n_0<n_1$. The usual
  proof of invariance of Khovanov homology shows that $\KhCx(K_0)$ can be identified with a
  subquotient complex of $\KhCx(K_1)$, inducing a (two-step) zig-zag
  of isomorphisms connecting $\Kh(K_0)$ and $\Kh(K_1)$
  (see also the proofs of
  \cite[Propositions~6.2--6.4]{RS-khovanov}). Indeed, there is a
  particular vertex $w\in\{0,1\}^{n_1-n_0}$, so that for every
  $u\in\{0,1\}^{n_0}$, $F_{\Kh}(K_0)(u)$ is identified with a certain
  subset $S_u\subseteq F_{\Kh}(K_1)((u,w))$, and for every
  $u>v\in\{0,1\}^{n_0}$, the correspondence $F_{\Kh}(K_0)(\cmorph{u}{v})$ is
  identified with the subset $s^{-1}(S_u)\cap t^{-1}(S_v)\subseteq
  F_{\Kh}(K_1)(\cmorph{(u,w)}{(v,w)})$. Furthermore, these identifications
  identify $F_{+\Kh}(K_0)(u)$ with $S_u\cap
  F_{+\Kh}(K_1)((u,w))$ and consequently,
  $F_{-\Kh}(K_0)(u)$ with $S_u\cap F_{-\Kh}(K_1)((u,w))$.

  Construct the $\KA{1}$-module spectrum $\KhSpace(K_1)$ using some
  pointed spatial refinement $\wt{F}_\Kh(K_1)$ for $K_1$. Restricting
  to the subsets $S_u$ and the correspondences between them, we get
  a pointed spatial refinement $\wt{F}_\Kh(K_0)$ for $K_0$, which we
  use to construct the $\KA{1}$-module spectrum $\KhSpace(K_0)$. With
  the CW complex structures from \Proposition{cwcomplex-box-hocolim},
  $\hocolim(\wt{F}^+_{\Kh}(K_0))$ can be identified with a subquotient
  complex of $\hocolim(\wt{F}^+_{\Kh}(K_1))$, leading to a two-step
  zig-zag of maps connecting $\KhSpace(K_1)$ and $\KhSpace(K_2)$.
  Since the Reidemeister moves do not cross the basepoint, it is immediate
  from the definitions of these sub- and quotient complexes that the
  maps are $\KA{1}$-equivariant.  Since they also induce
  isomorphisms on homology, they are weak equivalences.
\end{proof}

For the rest of this section, fix a link diagram for
$K_1\amalg K_2$, which is a disjoint union of link diagrams for $K_1$
and $K_2$, with $n_1$ and $n_2$ crossings respectively, and fix
basepoints $p_i$ on $K_i$ so that the two basepoints are next to one
another (i.e., on the boundary of the same component of $S^2\setminus
(K_1\amalg K_2)$).

Recall:
\begin{definition}
  The \emph{(derived) tensor product} of the module spectra $\KhSpace(K_1)$
  and $\KhSpace(K_2)$ is the homotopy colimit of the diagram
  \begin{equation}\label{eq:space-dtp}
    \xymatrix@C=3ex{
      \KhSpace(K_1)\smas \KhSpace(K_2) &
      \KhSpace(K_1)\smas \KA{1}\smas \KhSpace(K_2)
      \ar@<0.5ex>[l]\ar@<-0.5ex>[l] &
      \KhSpace(K_1)\smas \KA{1}\smas \KA{1}\smas \KhSpace(K_2)
      \ar@<1ex>[l]\ar[l]\ar@<-1ex>[l]
      &\cdots
      \ar@<1.5ex>[l]\ar@<0.5ex>[l]\ar@<-0.5ex>[l]\ar@<-1.5ex>[l]
    }
  \end{equation}
  where the maps are all possible ways of applying $\mu$ to a pair of
  consecutive factors. To be more precise, let $\DeltaInj$ be the
  category with one object $\bul{n}=\{0,\dots,n-1\}$ for each positive
  integer $n$ and $\Hom(\bul{m},\bul{n})$ the set of order-preserving
  injections $\{0,\dots,m-1\}\to\{0,\dots,n-1\}$; for $n>0$ and
  $0\leq i\leq n$, let
  $f_{\bul{n},i}\in\Hom_{\DeltaInj}(\bul{n},\bul{n+1})$ be the
  morphism $\bul{n}\to\bul{n+1}$ whose image is
  $\bul{n+1}\setminus\{i\}$.  (The category $\DeltaInj$ is the
  subcategory of the simplex category generated by the face maps, and
  the $f_{\bul{n},i}$ are the face maps themselves.)  Then the
  diagram~\eqref{eq:space-dtp} can be treated as a (strict) functor
  $F_{\otimes}$ from $\DeltaInj^\op$ to $\CWSpectra$, the category of
  CW spectra. On objects,
  $F_{\otimes}(\bul{n})=\KhSpace(K_1)\smas(\bigwedge_{i=1}^{n-1}\KA{1})\smas\KhSpace(K_2)$.
  On morphisms, $F_\otimes(f^\op_{\bul{n},i})$ is the map
  $\KhSpace(K_1)\smas(\bigwedge_{i=1}^{n}\KA{1})\smas\KhSpace(K_2)\to
  \KhSpace(K_1)\smas(\bigwedge_{i=1}^{n-1}\KA{1})\smas\KhSpace(K_2)$
  gotten by applying $\mu$ to the $(i+1)\th$ pair of consecutive
  factors.  Let
  $\KhSpace(K_1)\otimes_{\KA{1}}\KhSpace(K_2)\defeq\hocolim(F_{\otimes})$
  denote the derived tensor product of $\KhSpace(K_1)$ and
  $\KhSpace(K_2)$.
\end{definition}


\begin{theorem}\label{thm:unred-con-sum}
  There is a stable homotopy equivalence 
  $\KhSpace(K_1\#K_2)\simeq \KhSpace(K_1)\otimes_{\KA{1}}\KhSpace(K_2)$.
\end{theorem}

The proof of \Theorem{unred-con-sum} involves three components. First,
we show how the functor $F_{\Kh}(K_1\# K_2)$ is determined by the
functors $F_{\Kh}(K_1)$ and $F_{\Kh}(K_2)$; this is
Lemmas~\ref{lem:du2cs-isom} and~\ref{lem:du2cs-quotient}, which take a
little work but are purely combinatorial. Second, in
Lemmas~\ref{lem:dtp-cw-structure} and~\ref{lem:coderived}, we prove that
\Theorem{unred-con-sum} holds at the level of cellular cochains, for
an appropriate CW complex structure on
$\KhSpace(K_1)\otimes_{\KA{1}}\KhSpace(K_2)$. This is essentially
immediate from Segal's construction of homotopy colimits and the
connected sum theorem for the Khovanov chain complex. Third, using the
description of $F_{\Kh}(K_1\# K_2)$ in terms of $F_{\Kh}(K_1)$ and
$F_{\Kh}(K_2)$ and carefully chosen spatial refinements, we produce a
(strict) map from the diagram~(\ref{eq:space-dtp}) to
$\KhSpace(K_1\#K_2)$, inducing the desired map of cellular
cochains. This argument is
Lemmas~\ref{lem:disj-union-to-connect-sum-spatial}
and~\ref{lem:extend-space-dtp}. From these three steps,
\Theorem{unred-con-sum} follows easily.

We begin by reconstructing the functor $F_{\Kh}(K_1\# K_2)$ from the
functor $F_{\Kh}(K_1\amalg
K_2)\from\CCat{n_1+n_2}\to\BurnsideCat$.
\begin{definition}
  For $a,b\in\{\ast,+,-\}$, let
  $F_{ab\Kh}(K_1\amalg K_2)\from\CCat{n_1+n_2}\to\BurnsideCat$ denote
  the functor where we only consider the Khovanov generators that
  label the circle containing $p_1$ by $x_a$ if $a\in\{+,-\}$ and
  label the circle containing $p_2$ by $x_b$ if $b\in\{+,-\}$, and we
  restrict the correspondences correspondingly. (If $a$ or $b$ is $*$,
  we make no restriction on the label of the corresponding circle.)
\end{definition}
\begin{lemma}\label{lem:du2cs-isom} For $v\in\{0,1\}^{n_1+n_2}$, 
  the map from $F_{+-\Kh}(K_1\amalg K_2)(v)$ to $F_{-+\Kh}(K_1\amalg
  K_2)(v)$ that interchanges the labelings of the two pointed circles
  in $\mc{P}(v)$ induces an isomorphism from $F_{+-\Kh}(K_1\amalg
  K_2)$ to $F_{-+\Kh}(K_1\amalg K_2)$.
\end{lemma}

\begin{proof}
  The isomorphism from \Proposition{Kh-func-prod} identifies either
  functor to $F_{\wt{\Kh}}(K_1)\times F_{\wt{\Kh}}(K_2)$. The given
  map is the composition $F_{+-\Kh}(K_1\amalg K_2)\cong
  F_{\wt{\Kh}}(K_1)\times F_{\wt{\Kh}}(K_2)\cong F_{-+\Kh}(K_1\amalg
  K_2)$.
\end{proof}

\begin{notation}
  Let $F_{\centernot{++}\Kh}(K_1\amalg K_2)$ denote the functor
  $\CCat{n_1+n_2}\to\BurnsideCat$ where we only consider the Khovanov
  generators that label at least one of the two pointed circles by
  $x_-$, and we restrict the correspondences correspondingly. (The
  notation $F_{\centernot{++}\Kh}$ is the mnemonic ``not $++$''.) That
  is, for all $u\in\{0,1\}^{n_1+n_2}$,
  $F_{\centernot{++}\Kh}(K_1\amalg K_2)=F_{--\Kh}(K_1\amalg
  K_2)(u)\amalg F_{+-\Kh}(K_1\amalg K_2)(u)\amalg F_{-+\Kh}(K_1\amalg
  K_2)(u)$; and for all $u>v\in\{0,1\}^{n_1+n_2}$, the correspondence
  $F_{\centernot{++}\Kh}(K_1\amalg K_2)(\cmorph{u}{v})$ is the
  disjoint union of the correspondences
  $F_{--\Kh}(K_1\amalg K_2)(\cmorph{u}{v})$,
  $F_{+-\Kh}(K_1\amalg K_2)(\cmorph{u}{v})$,
  $F_{-+\Kh}(K_1\amalg K_2)(\cmorph{u}{v})$, some correspondence $P_{u,v}$
  from $F_{--\Kh}(K_1\amalg K_2)(u)$ to $F_{+-\Kh}(K_1\amalg K_2)(v)$,
  and some correspondence $Q_{u,v}$ from $F_{--\Kh}(K_1\amalg K_2)(u)$ to
  $F_{-+\Kh}(K_1\amalg K_2)(v)$. Let $F$ be the functor obtained from
  $F_{\centernot{++}\Kh}(K_1\amalg K_2)$ by identifying
  $F_{+-\Kh}(K_1\amalg K_2)$ and $F_{-+\Kh}(K_1\amalg K_2)$ by the
  isomorphism from \Lemma{du2cs-isom}. That is, for all
  $u\in\{0,1\}^{n_1+n_2}$,
  \[
    F(u)=\big(F_{--\Kh}(u)\amalg F_{+-\Kh}(u)\amalg
    F_{-+\Kh}(u)\big)/(F_{+-\Kh}(u)=F_{-+\Kh}(u));
  \]
  and for all $u>v\in\{0,1\}^{n_1+n_2}$,
  \[
    F(\cmorph{u}{v})=\big(F_{--\Kh}(\cmorph{u}{v})\amalg
    F_{+-\Kh}(\cmorph{u}{v})\amalg F_{-+\Kh}(\cmorph{u}{v})\amalg
    P_{u,v}\amalg Q_{u,v}\big)/(F_{+-\Kh}(\cmorph{u}{v})= F_{-+\Kh}(\cmorph{u}{v})).
  \]
\end{notation}

\begin{lemma}\label{lem:du2cs-quotient} The functor $F$ constructed
  above is isomorphic to $F_{\Kh}(K_1\# K_2)$ via the following map:
  for all $u\in\{0,1\}^{n_1+n_2}$, the isomorphism sends $x\in F(u)$
  to $y\in F_{\Kh}(K_1\# K_2)(u)$ where $y$ labels the connect-sum
  circle by $x_-$ if and only if $x$ labels both the pointed circles
  by $x_-$, and $x$ and $y$ label all the circles that are disjoint
  from the connect-sum region identically.
\end{lemma}

\begin{proof}
  The proof is similar to the proof of \Proposition{Kh-func-prod}. To
  keep the notations similar, let $X_v=F(v)$, $A_{u,v}=F(\cmorph{u}{v})$,
  $G=F_{\Kh}(K_1\# K_2)$, $Y_v=G(v)$, and $B_{u,v}=G(\cmorph{u}{v})$. The
  bijections $\phi_v\from X_v\stackrel{\cong}{\longrightarrow} Y_v$
  are already provided to us. To construct bijections $\psi_{u,v}\from
  A_{u,v}\stackrel{\cong}{\longrightarrow} B_{u,v}$, for all $u>v$
  with $|u|-|v|=1$, we need to check the
  conditions~(\ref{item:empty-or-one})
  and~(\ref{item:is-in-fact-chain-map}) of the proof of
  \Proposition{Kh-func-prod}. 

  For any $u\in\{0,1\}^{n_1+n_2}$ and any $z_u\in F_{\centernot{++}\Kh}(K_1\amalg
  K_2)(u)$, let $\pi(z_u)$ denote the image of $z_u$ in $X_u$; and for any
  $x_u\in X_u$, let $\iota^1(x_u)$ (respectively, $\iota^2(x_u)$)
  denote the preimage of $x_u$ in $F_{*-\Kh}(K_1\amalg K_2)(u)$ (respectively,
  $F_{-*\Kh}(K_1\amalg K_2)(u)$). Then for any $u>v$ with $|u|-|v|=1$, $z_u\in
  F_{\centernot{++}\Kh}(K_1\amalg K_2)(u)$, and $x_v\in X_v$, one of the
  two subsets $s^{-1}(z_u)\cap t^{-1}(\iota^1(x_v))\subseteq
  F_{*-\Kh}(K_1\amalg K_2)(\cmorph{u}{v})$ and $s^{-1}(z_u)\cap
  t^{-1}(\iota^2(x_v))\subseteq F_{-*\Kh}(K_1\amalg K_2)(\cmorph{u}{v})$ is
  empty, and the other one is canonically identified with the subset
  $s^{-1}(\pi(z_u))\cap t^{-1}(x_v)\subseteq A_{u,v}$. This
  follows from the fact that the correspondences in
  $F_{\centernot{++}\Kh}(K_1\amalg K_2)$ preserve two quantum
  gradings, the one coming from $K_1$ and the one coming from $K_2$;
  however, the double quantum gradings of $\iota^1(x_v)$ and
  $\iota^2(x_v)$ are different, and therefore, at least one of
  $s^{-1}(z_u)\cap t^{-1}(\iota^1(x_v))$ and $s^{-1}(z_u)\cap
  t^{-1}(\iota^2(x_v))$ is empty.

  From this observation, condition~(\ref{item:empty-or-one}) is
  immediate. Condition~(\ref{item:is-in-fact-chain-map}) follows from
  additionally noting that the composition $F_{\Kh}(K_1\amalg K_2) \to
  X \stackrel{\phi}{\longrightarrow} Y$ induces the cobordism map
  $\KhCx(K_1\# K_2)\to\KhCx(K_1\amalg K_2)$ associated to splitting at
  the connect-sum region, which is a chain map.

  Finally, we need to check that diagram~\eqref{eq:disj-un-diag} commutes
  for all $u>w$ with $|u|-|w|=2$. Using the
  observation in the previous paragraph, this follows from the same arguments as in the proof of
  \Proposition{Kh-func-prod}.
\end{proof}

Next we observe that \Theorem{unred-con-sum} holds at the level of
cellular cochain complexes.

\begin{lemma}\label{lem:dtp-cw-structure}
  There exists a CW complex structure on
  $\KhSpace(K_1)\otimes_{\KA{1}}\KhSpace(K_2)$ so that the reduced
  cellular cochain complex is the following chain complex
  \begin{equation}\label{eq:derived-cotensor}
    \KhCx(K_1)\otimes \KhCx(K_2) \to
    \KhCx(K_1)\otimes \Kh(U)\otimes \KhCx(K_2)\to\KhCx(K_1)\otimes \Kh(U)\otimes\Kh(U)\otimes \KhCx(K_2)\to\cdots
  \end{equation}
  with the differential given by
  \begin{align*}
    d(x_0\otimes\cdots\otimes x_n)
    &=
      \diff(x_0)\otimes\cdots\otimes x_{i}\otimes\cdots\otimes x_n
      +x_0\otimes\cdots\otimes x_{i}\otimes\cdots\otimes \diff(x_n)\\
    &\qquad{}+ \sum_{i=0}^n(-1)^{i+n}x_0\otimes\cdots\otimes
    x_{i-1}\otimes S(x_i)\otimes x_{i+1}\otimes\cdots\otimes x_n,
  \end{align*}
  where $S$ denotes either the Khovanov Frobenius algebra
  comultiplication map $\Kh(U)\to\Kh(U)\otimes\Kh(U)$ or the cobordism
  map $\KhCx(K_1)\to \KhCx(K_1)\otimes \Kh(U)$ (respectively,
  $\KhCx(K_2)\to \Kh(U)\otimes \KhCx(K_2)$) for splitting off a
  trivial unknot at $p_1$ (respectively, $p_2$), and $\diff$ denotes
  the Khovanov differential on $\KhCx(K_1)$ and $\KhCx(K_2)$.
\end{lemma}

\begin{proof}
  If $F\from\DeltaInj^\op\to\CW$ is a functor,
  Segal~\cite[Appendix~A]{Segal-top-categories} defines its geometric
  realization as
  \[
    \|F\|=\Bigl(\{\ast\}\amalg\coprod_{n=1}^\infty\Delta^{n-1} \times F(\bul{n})\Bigr)/\sim
  \]
  with $(f_*(\zeta),a)\sim(\zeta,F(f^\op)(a))$ for all
  $\zeta\in\Delta^{m-1}$, $a\in F(\bul{n})$, and
  $f\in\Hom_{\DeltaInj}(\bul{m},\bul{n})$, and
  $\Delta^{n-1}\times\{\ast\}\sim\ast$ for all $n$. (The map
  $f_*\from\Delta^{m-1}\to\Delta^{n-1}$ is the face inclusion
  corresponding to $f$.) If $\Delta^{n-1}_+$ denotes the disjoint
  union of $\Delta^{n-1}$ and a basepoint, then we can rewrite
  \[
    \|F\|=\Bigl(\bigvee_{n}\Delta_+^{n-1} \smas F(\bul{n})\Bigr)/\sim
  \]
  with just the relation $(f_*(\zeta),a)\sim(\zeta,F(f^\op)(a))$ for
  all $\zeta\in\Delta^{m-1}$, $a\in F(\bul{n})$, and
  $f\in\Hom_{\DeltaInj}(\bul{m},\bul{n})$.

  Let $F_{\otimes,\ell}\from\DeltaInj^{\op}\to\CW$ be the functor
  obtained from $F_{\otimes}\from\DeltaInj^{\op}\to \CWSpectra$ by
  looking at the $\ell\th$ spaces in the spectra.
  Define the geometric realization $\|F_{\otimes}\|$ as a
  CW spectrum whose $\ell\th$ space is
  \(
  \|F_{\otimes}\|_\ell=\|F_{\otimes,\ell}\|
  \)
  with the structure map
  \[
    \Sigma\|F_{\otimes}\|_\ell\cong\Sigma\Bigl(\bigvee_{n}\Delta_+^{n-1} \smas F_{\otimes,\ell}(\bul{n})\Bigr)/\!\!\sim\,=\Bigl(\bigvee_{n}\Delta_+^{n-1} \smas \Sigma F_{\otimes,\ell}(\bul{n})\Bigr)/\!\!\sim\,\longrightarrow\Bigl(\bigvee_{n}\Delta_+^{n-1} \smas F_{\otimes,\ell+1}(\bul{n})\Bigr)/\!\!\sim\,=\|F_{\otimes}\|_{\ell+1},
  \]
  where the middle arrow is induced by the structure maps
  $\Sigma(F_\otimes(n))_\ell\to(F_\otimes(n))_{\ell+1}$.
  Equipped with the natural CW complex structure
  (\cite[Proposition~A.1(i)]{Segal-top-categories}), the reduced
  cellular cochain complex of $\|F_\otimes\|$ is easily seen to be the one from
  Formula~\eqref{eq:derived-cotensor}. 

  To identify $\|F_{\otimes}\|$ with $\hocolim(F_{\otimes})$, use the
  construction of homotopy colimits of the strict functor
  $F_{\otimes}$ via simplices, instead of cubes; Vogt~\cite[Corollary~8.5]{Vogt-top-hocolim} shows the two
  definitions agree.  Using the
  simplicial model for the homotopy colimit, it is well known that
  $\hocolim(F_{\otimes})$ is the barycentric subdivision of
  $\|F_{\otimes}\|$. (See also
  \cite[Proposition~A.3]{Segal-top-categories} where
  $|\mathrm{simp}(\cdot)|$ plays the role of this space.)
  %
\end{proof}

\begin{lemma}\label{lem:coderived}
  The cobordism map $S\from \KhCx(K_1\#K_2)\to \KhCx(K_1)\otimes
  \KhCx(K_2)$ associated to splitting at the connected sum region
  induces a quasi-isomorphsim 
  \[
  \xymatrix@C=3ex{
    \KhCx(K_1\#K_2)\ar[d]^S\\
    \KhCx(K_1)\otimes \KhCx(K_2) \ar[r]& \KhCx(K_1)\otimes
    \Kh(U)\otimes \KhCx(K_2)\ar[r]&\KhCx(K_1)\otimes
    \Kh(U)\otimes\Kh(U)\otimes \KhCx(K_2)\ar[r]&\cdots }
  \]
  from $\KhCx(K_1\# K_2)$ to the chain complex from
  Formula~\eqref{eq:derived-cotensor}.
\end{lemma}
\begin{proof}
  The Khovanov complex $\KhCx(K_1\#K_2)$ is the cotensor product of
  $\KhCx(K_1)$ and $\KhCx(K_2)$ as comodules over
  $\Kh(U)$~\cite[Lemma 10.5]{RS-khovanov}, while
  Formula~\eqref{eq:derived-cotensor} is the derived cotensor product
  of $\KhCx(K_1)$ and $\KhCx(K_2)$. Thus, the statement presumably
  follows from the fact that $\KhCx(K_1)$ and $\KhCx(K_2)$ are
  co-flat. Rather than going down this rabbit hole, dualize the
  complex~\eqref{eq:derived-cotensor} over $\ZZ$, which exchanges the
  split map $S$ and the merge map $m$, the (derived) cotensor product and
  the (derived) tensor product, and $\KhCx(K)$ and $\KhCx(m(K))$. (The
  last assertion is \cite[Proposition 32]{Kho-kh-categorification}.)
  The result then follows from the fact that $\KhCx(m(K_i))$ is free
  as a $\Kh(U)$-module and Khovanov's connected sum
  theorem~\cite[Proposition 3.3]{Kho-kh-patterns}.
\end{proof}

We turn to the third part of the argument, constructing compatible
spatial refinements for $F_{\Kh}(K_1\# K_2)$ and $F_{\Kh}(K_1\amalg
U\amalg\cdots\amalg U\amalg K_2)$.

\begin{lemma}\label{lem:disj-union-to-connect-sum-spatial}
  Consider any spatial refinement
  $\wt{F}_{\centernot{++}\Kh}(K_1\amalg K_2)$ of
  $F_{\centernot{++}\Kh}(K_1\amalg K_2)$ whose induced spatial
  refinements $\wt{F}_{-+\Kh}(K_1\amalg K_2)$ and
  $\wt{F}_{+-\Kh}(K_1\amalg K_2)$ of
  $F_{-+\Kh}(K_1\amalg K_2)\cong F_{+-\Kh}(K_1\amalg K_2)$ agree. Then, identifying
  $\wt{F}_{-+\Kh}(K_1\amalg K_2)$ and $\wt{F}_{+-\Kh}(K_1\amalg K_2)$
  produces a spatial refinement $\wt{F}_{\Kh}(K_1\# K_2)$ of
  $F_{\Kh}(K_1\# K_2)$.
\end{lemma}

\begin{proof}
  It is immediate from the definitions that identifying
  $\wt{F}_{-+\Kh}(K_1\amalg K_2)$ and $\wt{F}_{+-\Kh}(K_1\amalg K_2)$
  produces a spatial refinement of the functor $F$ above. The
  isomorphism from \Lemma{du2cs-quotient} then produces the spatial
  refinement $\wt{F}_{\Kh}(K_1\# K_2)$ of $F_{\Kh}(K_1\# K_2)$.
\end{proof}

Let $\olDeltaInj=\DeltaInj\cup\bul{0}$ be the category obtained by
adding an object $\bul{0}=\emptyset$ to $\DeltaInj$ and a unique
morphism $\bul{0}\to\bul{n}$ for each $n$; let $f_{\bul{0},0}$ denote
the unique morphism from $\bul{0}$ to $\bul{1}$. We will extend
Diagram~\eqref{eq:space-dtp} to construct a functor
$\olDeltaInj^{\op}\to\CWSpectra$.

\begin{lemma}\label{lem:extend-space-dtp}
There exists a functor $\ol{F}_{\times}\from \olDeltaInj^{\op}\to\CWSpectra$
satisfying the following:
\begin{enumerate}
\item\label{item:objects-map-correctly} $\ol{F}_{\times}(\bul{0})=\KhSpace(K_1\# K_2)$ and
  $\ol{F}_{\times}(\bul{n})=\KhSpace(K_1\amalg
  K_2)\smas(\bigwedge_{i=1}^{n-1}\KA{1})$ for all $n>0$.
\item\label{item:is-extension-upto-nat-trans} Let $F_{\times}$ denote the restriction
  $\ol{F}_{\times}|_{\DeltaInj^{\op}}$. Then there is a natural
  transformation $\eta$ from the functor $F_{\otimes}$ of
  diagram~\eqref{eq:space-dtp} to $F_{\times}$, so that for all $n>0$,
  $\eta_{\bul{n}}\from F_{\otimes}(\bul{n})\to F_{\times}(\bul{n})$
  sends each cell in
  $\KhSpace(K_1)\smas(\bigwedge_{i=1}^{n-1}\KA{1})\smas\KhSpace(K_2)$
  to the corresponding cell in $\KhSpace(K_1\amalg
  K_2)\smas(\bigwedge_{i=1}^{n-1}\KA{1})$ by a degree one map.
\item\label{item:morphisms-map-correctly}
  $\ol{F}_{\times}(f^\op_{\bul{0},0})$ is a map $\KhSpace(K_1\amalg
  K_2)\to \KhSpace(K_1\# K_2)$ so that the induced map on reduced
  cellular cochains is the cobordism map
  \[
  S\from \KhCx(K_1\# K_2)\to \KhCx(K_1\amalg K_2)
  \]
  induced by splitting at the connected sum region.
\end{enumerate}
\end{lemma}

\begin{proof}
  During the construction of $\ol{F}_\times$, we will use spatial
  refinements for $K_1\amalg K_2$ that are pointed spatial refinements
  with respect to both $p_1$ and $p_2$. We will call such spatial
  refinements \emph{doubly pointed spatial refinements}.  Doubly
  pointed spatial refinements are spatial refinements that agree on
  $F_{\ast+\Kh}(K_1\amalg K_2)$ and $F_{\ast-\Kh}(K_1\amalg K_2)$, and
  also on $F_{+\ast\Kh}(K_1\amalg K_2)$ and $F_{-\ast\Kh}(K_1\amalg
  K_2)$ (and, therefore, agree on $F_{++\Kh}(K_1\amalg K_2)$,
  $F_{+-\Kh}(K_1\amalg K_2)$, $F_{-+\Kh}(K_1\amalg K_2)$, and
  $F_{--\Kh}(K_1\amalg K_2)$).
  The CW spectrum $\KhSpace(K_1\amalg K_2)$ constructed using any such
  doubly pointed spatial refinement can be viewed as a strict bimodule
  over $\KA{1}$, with the two actions coming from the two basepoints
  $p_1$ and $p_2$. Therefore, we can construct a strict functor
  $G\from\DeltaInj^\op\to\CW$ by declaring
  $G(\bul{n})=\KhSpace(K_1\amalg
  K_2)\smas(\bigwedge_{i=1}^{n-1}\KA{1})$ and by defining the map
  $G(f^\op_{\bul{n},i})\from \KhSpace(K_1\amalg
  K_2)\smas(\bigwedge_{i=1}^{n}\KA{1})\to \KhSpace(K_1\amalg
  K_2)\smas(\bigwedge_{i=1}^{n-1}\KA{1})$ to be the ring
  multiplication map applied to the $i\th$ pair of consecutive
  $\KA{1}$-factors if $0<i<n$, and the bimodule map coming from $p_1$
  (respectively, $p_2$) using the first (respectively, last) $\KA{1}$
  factor if $i=0$ (respectively, $n$). 


  Now we are in a position to construct $F_{\times}$. The construction
  proceeds in several stages.
  \begin{enumerate}[leftmargin=*,label=(F-\arabic*),ref=(F-\arabic*)]
  \item We start with pointed spatial refinements $\wt{F}_{\Kh}(K_1)$
    and $\wt{F}_{\Kh}(K_2)$ of $F_{\Kh}(K_1)$ and $F_{\Kh}(K_2)$
    (using $k_1$-dimensional and $k_2$-dimensional boxes with
    $k_1+k_2=k$). Since $F_{\Kh}(K_1\amalg K_2)=F_{\Kh}(K_1)\times
    F_{\Kh}(K_2)$ (\Proposition{Kh-func-prod}),
    $\wt{F}_{\Kh}(K_1\amalg K_2)\defeq
    \wt{F}_{\Kh}(K_1)\smas\wt{F}_{\Kh}(K_2)$ (cf.~\ref{item:hocolim-prod}) is a doubly pointed
    spatial refinement for $K_1\amalg K_2$. 
  \item\label{item:l1b-l2} Define
    $F_{\times}\from\DeltaInj^\op\to\CWSpectra$ to be the functor
    associated to this doubly pointed spatial refinement. This
    automatically satisfies the second part of
    \Lemma{extend-space-dtp}~(\ref{item:objects-map-correctly}). To
    relate $F_{\otimes}$ and $F_{\times}$, observe that for all $n>0$,
    $F_{\times}(\bul{n})=F_{\times}(\bul{1})\smas(\bigwedge_{i=1}^{n-1}\KA{1})$
    and $F_{\otimes}(\bul{n})$ is canonically isomorphic to
    $F_{\otimes}(\bul{1})\smas(\bigwedge_{i=1}^{n-1}\KA{1})$
    (preserving the order of the $\KA{1}$-factors); therefore, it is
    enough to relate $F_{\times}(\bul{1})$, which is a formal
    desuspension of the suspension spectrum of
    $\hocolim\big((\wt{F}_{\Kh}(K_1)\smas\wt{F}_{\Kh}(K_2))^+\big)$,
    and $F_{\otimes}(\bul{1})$, which is a formal desuspension of the
    suspension spectrum of
    $\hocolim\big((\wt{F}_{\Kh}(K_1))^+\big)\smas
    \hocolim\big((\wt{F}_{\Kh}(K_2))^+\big)$. However, the former is
    easily seen to be a quotient of the latter, with the quotient map
    sending each cell by a degree one map to the corresponding
    cell. This proves
    \Lemma{extend-space-dtp}~(\ref{item:is-extension-upto-nat-trans}).
  \item The doubly pointed spatial refinement $\wt{F}_{\Kh}(K_1\amalg
    K_2)$ induces a spatial refinement
    $\wt{F}_{\centernot{++}\Kh}(K_1\amalg K_2)$ of
    $F_{\centernot{++}\Kh}(K_1\amalg K_2)$; and its induced spatial
    refinements $\wt{F}_{-+\Kh}(K_1\amalg K_2)$ and
    $\wt{F}_{+-\Kh}(K_1\amalg K_2)$ of $F_{-+\Kh}(K_1\amalg K_2)$ and
    $F_{+-\Kh}(K_1\amalg K_2)$ agree (with $F_{-+\Kh}(K_1\amalg K_2)$
    and $F_{+-\Kh}(K_1\amalg K_2)$ identified by
    \Lemma{du2cs-isom}). Therefore, by
    \Lemma{disj-union-to-connect-sum-spatial}, we get a (pointed)
    spatial refinement $\wt{F}_{\Kh}(K_1\# K_2)$ for $K_1\# K_2$ (with
    the basepoint chosen on either of the two strands near the connect
    sum region).
  \item\label{item:l1a} We use $\wt{F}_{\Kh}(K_1\# K_2)$ to
    construct the CW spectrum $\ol{F}_{\times}(\bul{0})=\KhSpace(K_1\#
    K_2)$; this automatically satisfies the first part of
    \Lemma{extend-space-dtp}~(\ref{item:objects-map-correctly}).
  \item\label{item:l3} The space
    $\hocolim\big((\wt{F}_{\centernot{++}\Kh}(K_1\amalg K_2))^+\big)$
    is a quotient complex of
    $\hocolim\big((\wt{F}_{\Kh}(K_1\amalg K_2))^+\big)$; it has two
    subcomplexes $\hocolim\big((\wt{F}_{-+\Kh}(K_1\amalg K_2))^+\big)$
    and $\hocolim\big((\wt{F}_{+-\Kh}(K_1\amalg K_2))^+\big)$ that
    have a canonical isomorphism between them; and the coequalizer is
    canonically identified with the space
    $\hocolim\big((\wt{F}_{\Kh}(K_1\# K_2))^+\big)$. That is, we have
    a diagram
    \[
    \xymatrix{ &\mathclap{\hocolim\big((\wt{F}_{-+\Kh}(K_1\amalg
      K_2))^+\big)\cong\hocolim\big((\wt{F}_{+-\Kh}(K_1\amalg
      K_2))^+\big)}\\
      \hocolim\big((\wt{F}_{\Kh}(K_1\amalg K_2))^+\big) \ar@{->>}[r]&\hocolim\big((\wt{F}_{\centernot{++}\Kh}(K_1\amalg K_2))^+\big) \ar@{->>}[d] \ar@{<-_)}[u]+<7ex,-2ex> \ar@{<-_)}[u]+<-7ex,-3ex>\\
      &\hocolim\big((\wt{F}_{\Kh}(K_1\#
      K_2))^+\big).}
    \]
    Let
    $\ol{F}_{\times}(f^\op_{\bul{0},0})\from\ol{F}_{\times}(\bul{1})=\KhSpace(K_1\amalg
    K_2)\to\KhSpace(K_1\# K_2)=\ol{F}_{\times}(\bul{0})$ be the map
    induced from the composition
    $\hocolim\big((\wt{F}_{\Kh}(K_1\amalg
    K_2))^+\big)\to\hocolim\big((\wt{F}_{\Kh}(K_1\#
    K_2))^+\big)$. This map is a cellular map sending the cells in
    $\KhSpace(K_1\amalg K_2)$ to the corresponding cells in
    $\KhSpace(K_1\# K_2)$ by degree one maps (with the correspondence
    described in \Lemma{du2cs-quotient}). Therefore, this map
    satisfies
    \Lemma{extend-space-dtp}~(\ref{item:morphisms-map-correctly}).
  \item\label{item:bul-2-to-0} We have to define the map
    $\ol{F}_{\times}(\bul{2})\to\ol{F}_{\times}(\bul{0})$ to be both
    $\ol{F}_{\times}(f^\op_{\bul{0},0})\circ\ol{F}_{\times}(f^\op_{\bul{1},0})$
    and
    $\ol{F}_{\times}(f^\op_{\bul{0},0})\circ\ol{F}_{\times}(f^\op_{\bul{1},1})$;
    so we merely need to check that the latter two maps agree. The two
    maps $\ol{F}_{\times}(f^\op_{\bul{1},0})$ and
    $\ol{F}_{\times}(f^\op_{\bul{1},1})$ are both induced from maps of spaces
    \begin{align*}
    &\hocolim\big((\wt{F}_{\Kh}(K_1\amalg K_2))^+\big)\smas\{\ast,p_-,p_+\}\\
      &\qquad=(\hocolim\big((\wt{F}_{\Kh}(K_1\amalg K_2))^+\big)\times\{p_-\})\vee(\hocolim\big((\wt{F}_{\Kh}(K_1\amalg K_2))^+\big)\times\{p_+\})\\
      &\qquad\qquad\to\hocolim\big((\wt{F}_{\Kh}(K_1\amalg K_2))^+\big).
    \end{align*}
    The two maps agree on the first summand: both are the
    projection
    $\hocolim\big((\wt{F}_{\Kh}(K_1\amalg
    K_2))^+\big)\times\{p_-\}\to\hocolim\big((\wt{F}_{\Kh}(K_1\amalg
    K_2))^+\big)$. On the second summand, the two maps are
    compositions of the projection map
    $\hocolim\big((\wt{F}_{\Kh}(K_1\amalg K_2))^+\big)\times\{p_+\}\to\hocolim\big((\wt{F}_{\Kh}(K_1\amalg K_2))^+\big)$
    with the following two maps:
    \[
    \xymatrix{&\mathclap{\hocolim\big((\wt{F}_{\Kh}(K_1\amalg
        K_2))^+\big)}\\
      \hocolim\big((\wt{F}_{-*\Kh}(K_1\amalg
      K_2))^+\big)\ar[d]^\cong\ar@{<<-}[ur]+<-2ex,-2ex>&&\hocolim\big((\wt{F}_{*-\Kh}(K_1\amalg
      K_2))^+\big)\ar[d]_\cong\ar@{<<-}[ul]+<2ex,-2ex>\\
      \hocolim\big((\wt{F}_{+*\Kh}(K_1\amalg
      K_2))^+\big)\ar@{^(->}[dr]+<-2ex,3ex>&&\hocolim\big((\wt{F}_{*+\Kh}(K_1\amalg
      K_2))^+\big)\ar@{_(->}[dl]+<2ex,3ex>\\
      &\mathclap{\hocolim\big((\wt{F}_{\Kh}(K_1\amalg K_2))^+\big).}
    }
    \]
    Therefore, after composing with the map
    $\hocolim\big((\wt{F}_{\Kh}(K_1\amalg
    K_2))^+\big)\to\hocolim\big((\wt{F}_{\Kh}(K_1\# K_2))^+\big)$ from
    step~\ref{item:l3} that induces
    $\ol{F}_{\times}(f^\op_{\bul{0},0})$, we get the two maps
    \[
    \xymatrix{&\hocolim\big((\wt{F}_{\Kh}(K_1\amalg
      K_2))^+\big)\ar@{->>}[d]\\
      &\hocolim\big((\wt{F}_{--\Kh}(K_1\amalg
      K_2))^+\big)\ar[dr]+<0ex,3ex>^-\cong\ar[dl]+<0ex,3ex>_-\cong\\
      \mathclap{\hocolim\big((\wt{F}_{+-\Kh}(K_1\amalg
      K_2))^+\big)}&&\mathclap{\hocolim\big((\wt{F}_{-+\Kh}(K_1\amalg
      K_2))^+\big)}\\
      &\hocolim\big((\wt{F}_{\centernot{++}\Kh}(K_1\amalg K_2))^+\big) \ar@{->>}[d] \ar@{<-_)}[ul]+<0ex,-3ex> \ar@{<-^)}[ur]+<0ex,-3ex>\\
      &\hocolim\big((\wt{F}_{\Kh}(K_1\# K_2))^+\big).  }
    \]
    (The first vertical map is the quotient map to
    $\hocolim\big((\wt{F}_{--\Kh}(K_1\amalg K_2))^+\big)$, as the
    composition sends the rest of
    $\hocolim\big((\wt{F}_{\Kh}(K_1\amalg K_2))^+\big)$ to the
    basepoint.)  However, since $\hocolim\big((\wt{F}_{\Kh}(K_1\#
    K_2))^+\big)$ is the coequalizer (see step~\ref{item:l3}), these
    maps agree.
  \item The $n$ different morphisms $\bul{n}\to\bul{1}$ in
    $\olDeltaInj^\op$ induce the following two maps
    $\ol{F}_{\times}(\bul{n})\to\ol{F}_{\times}(\bul{1})$: first apply
    the ring multiplication on the $\KA{1}$-factors to get a map
    \[
    \phi\from\ol{F}_{\times}(\bul{n})=\KhSpace(K_1\amalg
    K_2)\smas(\bigwedge_{i=1}^{n-1}\KA{1})\to \KhSpace(K_1\amalg
    K_2)\smas\KA{1}=\ol{F}_{\times}(\bul{2}),
    \]
    and then compose with the two maps
    $\ol{F}_{\times}(f^\op_{\bul{1},0})$ and
    $\ol{F}_{\times}(f^\op_{\bul{1},1})$. We have already defined the
    map $\ol{F}_{\times}(\bul{2})\to\ol{F}_{\times}(\bul{0})$; define
    the map $\ol{F}_{\times}(\bul{n})\to\ol{F}_{\times}(\bul{0})$ to
    be its composition with $\phi$.
\end{enumerate}
The functor $\ol{F}_\times$ thus constructed satisfies the conditions
of the lemma (see steps~\ref{item:l1b-l2},~\ref{item:l1a},
and~\ref{item:l3}), thereby concluding the proof.
\end{proof}

\begin{remark}
  Steps~\ref{item:l3} and~\ref{item:bul-2-to-0} in the proof of
  \Lemma{extend-space-dtp} actually show that with the specific
  choices made in the proof, $\KhSpace(K_1\# K_2)$ is the ordinary
  tensor product of $\KhSpace(K_1)$ and $\KhSpace(K_2)$ over
  $\KA{1}$. However, these spectra and  their module structures are only defined up to
  weak equivalence (see
  \Proposition{quasi-isom-module-spectrum}); and the ordinary tensor
  product is not invariant under such equivalences, while the derived
  tensor product is.
\end{remark}

\begin{proof}[Proof of \Theorem{unred-con-sum}]
  Let $F_{\otimes}\from\DeltaInj^\op\to\CWSpectra$ be the functor from
  diagram~\eqref{eq:space-dtp}, and let
  $\ol{F}_{\times}\from\olDeltaInj^{\op}\to\CWSpectra$ be the functor
  constructed in \Lemma{extend-space-dtp}. The derived tensor product
  $\KhSpace(K_1)\otimes_{\KA{1}}\KhSpace(K_2)$ is defined to be
  $\hocolim(F_{\otimes})$. The natural transformation $\eta\from
  F_{\otimes}\to F_{\times}=\ol{F}_{\times}|_{\DeltaInj^\op}$ from
  \Lemma{extend-space-dtp}~(\ref{item:is-extension-upto-nat-trans}) is a
  stable homotopy equivalence on each object; therefore, by
  property~\ref{item:hocolim-nat-trans} in \Section{colimit},
  $\hocolim(F_{\otimes})\simeq\hocolim(F_{\times})$. Since $\bul{0}$ is a
  terminal object in $\olDeltaInj^\op$, the functor $\ol{F}_\times$
  induces a map 
  \[
  \KhSpace(K_1)\otimes_{\KA{1}}\KhSpace(K_2)=\hocolim(F_{\times})\longrightarrow
  \ol{F}_{\times}(\bul{0})=\KhSpace(K_1\# K_2).
  \]
 
  Using \Lemma{dtp-cw-structure} and
  \Lemma{extend-space-dtp}~(\ref{item:is-extension-upto-nat-trans}),
  we can view $\hocolim(F_{\times})$ as a CW spectrum whose reduced
  cellular cochain complex is the complex from
  Formula~\eqref{eq:derived-cotensor}. \Lemma{extend-space-dtp}~(\ref{item:morphisms-map-correctly})
  implies that the map $\hocolim(F_{\times})\to
  \ol{F}_{\times}(\bul{0})$ induces the quasi-isomorphism from
  \Lemma{coderived} at the level of reduced cellular cochain
  complexes. Therefore, the map is a stable homotopy equivalence.
\end{proof}

\section{Khovanov homotopy type of a mirror}\label{sec:mirror}
Given a spectrum $X$, let $\SWdual{X}$ denote the Spanier-Whitehead
dual of $X$, which is the internal function object parametrizing maps
$X \to \SphereS$. If $X$ is a finite spectrum, the spectrum $\SWdual{X}$
is also finite, and is characterized by the existence of map of spectra 
\[
\mu\co X\smas \SWdual{X}\to \SphereS
\]
such that the slant product map 
\begin{equation}\label{eq:slant}
\mu^*(\gamma)/\cdot \co \wt{H}_*(X)\to \wt{H}^*(\SWdual{X})
\end{equation}
is an isomorphism. (Here, $\gamma\in \wt{H}^0(\SphereS)$ is the fundamental class.) 
We will say that \emph{$\mu$ witnesses $S$-duality between $X$ and $\SWdual{X}$.}
See Switzer~\cite[Proposition 14.37]{Switzer-top-book} and Adams~\cite[Section III.9]{Adams-top-book} for more details.

\begin{theorem}\label{thm:mirror}\cite[Conjecture~10.1]{RS-khovanov}
  Let $L$ be a link and let $\mirror{L}$ denote the mirror of $L$. Then
  \[
  \KhSpace^j(\mirror{L})\simeq\SWdual{\KhSpace^{-j}(L)}.
  \]
\end{theorem}

Before proving the theorem we recall a reformulation of the Spanier-Whitehead duality criterion, after some homological algebra.

\begin{lemma}\label{lem:qi}
  Let $C_*$ and $D_*$ be finitely-generated chain complexes of abelian groups and let $f\co C_*\to D_*$ be a chain map. Suppose that for any field $k$, the induced map $f\otimes \Id_k\co C_*\otimes k\to D_*\otimes k$ is a quasi-isomorphism. Then $f$ is a quasi-isomorphism.
\end{lemma}
\begin{proof}
  {(Compare \cite[3.A.7]{Hatcher-top-book})} It suffices to prove that the mapping cone $\Cone(f)$ of $f$ is acyclic. Observe that $\Cone(f\otimes \Id_k)=\Cone(f)\otimes k$, and by assumption $\Cone(f\otimes \Id_k)$ is acyclic for any field $k$. It follows from the universal coefficient theorem that if $E_*\otimes \mathbb{F}_p$ is acyclic for all primes $p$ then $E_*$ is acyclic.
\end{proof}

\begin{lemma}\label{lem:perfect-pairing}
  Let $C_*$ and $D_*$ be finitely-generated chain complexes of free abelian groups and $F\co C_*\otimes D_*\to \ZZ$ a chain map. Write $D^*=\Hom(D_*,\ZZ)$ for the dual complex to $D_*$. Then the following conditions are equivalent:
  \begin{enumerate}
  \item\label{item:pp-1} The map $C_*\to D^*$ induced by $F$ is a quasi-isomorphism.
  \item\label{item:pp-2} For any field $k$, the map $H_*((C_*\otimes k))\otimes H_*((D_*\otimes k))\to k$ induced by $F$ is a perfect pairing.
  \end{enumerate}
\end{lemma}
\begin{proof}
  (\ref{item:pp-1})$\implies$(\ref{item:pp-2}) Tensoring with $k$, the
  map $f\co C_*\otimes k\to D^*\otimes k = \Hom(D_*,k)$ induced by $F$
  is a quasi-isomorphism. Suppose $\alpha$ is a non-trivial element of
  $H_*(C_*\otimes k)$. Then $f_*(\alpha)\neq 0\in H_*(D^*\otimes
  k)=\Hom(H_*(D_*\otimes k),k)$. Let $\beta\in H_*(D_*\otimes k)$ be an element
  on which $f_*(\alpha)$ evaluates non-trivially. Then
  $F_*(\alpha\otimes \beta)=f_*(\alpha)(\beta)\neq 0$.

  Similarly, if $\beta\neq 0 \in H_*(D_*\otimes k)$ then there is an
  element $b\in H_*(D^*\otimes k)$ such that $b(\beta)\neq 0$. Since
  $f$ is a quasi-isomorphism there is an $\alpha\in H_*(C_*\otimes k)$
  so that $b=f(\alpha)$. Then $F_*(\alpha\otimes
  \beta)=f_*(\alpha)(\beta)=b(\beta)\neq 0$.

  (\ref{item:pp-2})$\implies$(\ref{item:pp-1}) We start by checking
  that for any field $k$, the map $f\co C_*\otimes k\to D^*\otimes k$
  is a quasi-isomorphism. To this end, suppose $\alpha\neq 0\in
  H_*(C_*\otimes k)$. Then there exists $\beta\in H_*(D_*\otimes k)$
  so that $F_*(\alpha\otimes \beta)\neq 0\in k$.  Then
  $f_*(\alpha)(\beta)=F_*(\alpha\otimes \beta)\neq 0$, so
  $f_*(\alpha)\neq 0$. Thus, $f_*$ is injective. Next, given any
  element $b\in H_*(D^*\otimes k)=\Hom(H_*(D_*\otimes k),k)$, since
  $F_*$ is a perfect pairing there is an element $\alpha\in
  H_*(C_*\otimes k)$ so that $b(\cdot)=F_*(\alpha\otimes
  \cdot)$. Since $F_*(\alpha\otimes \cdot)=f_*(\alpha)(\cdot)$, we get
  $b=f_*(\alpha)$. So, $f_*$ is surjective.

  Thus, the map $f\otimes\Id_k\co C_*\otimes k\to D^*\otimes k$ is a
  quasi-isomorphism for any field $k$. So, by \Lemma{qi}, the map
  $f\co C_*\to D^*$ induced by $F$ is a quasi-isomorphism, as desired.
\end{proof}

\begin{proposition}\label{prop:detect-SW}
  Let $X$ and $Y$ be finite CW complexes and let $\mu\co X\smas Y\to S^n$ be a continuous map. Then $\mu$ witnesses $S$-duality between $X$ and $\Sigma^{-n}Y$ if and only if for every field $k$ and integer $i$, the induced map
  \[
  \wt{H}_i(X;k)\otimes \wt{H}_{n-i}(Y;k)\to \wt{H}_n(S^n;k)=k
  \]
  is a perfect pairing.
\end{proposition}
\begin{proof}
  Let $\cellC{*}(X)$ denote the reduced cellular chain complex of $X$ and $\cocellC{*}(X)$ the reduced cellular cochain complex. The slant product map of Formula~\eqref{eq:slant} is the map on homology induced by 
  \begin{align*}
    \cellC{*}(X)&\to \cocellC{n-*}(Y)=\Hom(\cellC{n-*}(Y),\ZZ)\\
    x&\mapsto \bigl(y\mapsto \gamma(\mu_\#(x\smas y))\bigr).\\
    \shortintertext{In other words, this is the map induced by the pairing }
    \cellC{*}(X)\otimes \cellC{n-*}(Y)&\to \cellC{n}(S^n)=\ZZ\\
    x\otimes y &\mapsto \gamma(\mu_\#(x\smas y)).
  \end{align*}
  So, the result is immediate from \Lemma{perfect-pairing}.
\end{proof}

The other ingredients in the proof of \Theorem{mirror} are some facts about functoriality of Khovanov homology and the Khovanov stable homotopy type.

Suppose that $F\subset [0,1]\times \RR^3$ is a link cobordism from $L_0$ to $L_1$. Associated to $F$ is a map $\Phi_F\co \Kh(L_0)\to \Kh(L_1)$, well-defined up to multiplication by $\pm 1$; see~\cite{Jac-kh-cobordisms,Kho-kh-cobordism,Bar-kh-tangle-cob}. The following properties of $\Phi$ are immediate from the definition:
\begin{enumerate}
\item\label{item:cob-isotopy} If $F$ and $F'$ are isotopic link
  cobordisms (relative boundary) then $\Phi_F=\pm\Phi_{F'}$.
\item\label{item:cob-identity} If $F=[0,1]\times L$ is the identity cobordism from $L$ to $L$ then $\Phi_F=\pm \Id\co \Kh(L)\to \Kh(L)$.
\item\label{item:cob-compose} If $F_1$ is a cobordism from $L_0$ to $L_1$ and $F_2$ is a cobordism from $L_1$ to $L_2$, and $F_2\circ F_1$ denotes the composition of $F_2$ and $F_1$, then
  \[
  \Phi_{F_2\circ F_1}=\pm \Phi_{F_2}\circ \Phi_{F_1}.
  \]
\item\label{item:cob-union} If $F$ is a cobordism from $L_0$ to $L_1$ and $F'$ is a cobordism from $L_0'$ to $L_1'$, and $F\amalg F'$ denotes the disjoint union of $F$ and $F'$, which is a cobordism from $L_0\amalg L_0'$ to $L_1\amalg L_1'$, then for any field $k$,
  \begin{align*}
  \Phi_{F\amalg F'}=\pm \Phi_F\otimes \Phi_{F'}\co \Kh(L_0;k)\otimes\Kh(L_0';k)&=\Kh(L_0\amalg L'_0;k)\\
  \to 
  \Kh(L_1\amalg L_1';k)&=\Kh(L_1;k)\otimes \Kh(L_1';k).
  \end{align*}
\end{enumerate}

Given a link $L$ there is a canonical, genus $0$ cobordism $F$ from $L\amalg \mirror{L}$ to the empty link.

\begin{proposition}\label{prop:Kh-perf-pair}
  For any field $k$, the map 
  \[
  \Phi_F\co \Kh(L;k)\otimes \Kh(\mirror{L};k)=\Kh(L\amalg \mirror{L};k)\to \Kh(\emptyset;k)=k
  \]
  associated to the canonical cobordism from $L\amalg \mirror{L}$ to the empty link is a perfect pairing.
\end{proposition}
\begin{proof}
  This follows from Properties~(\ref{item:cob-isotopy}), (\ref{item:cob-identity}), (\ref{item:cob-compose}) and~(\ref{item:cob-union}) above via the usual snake-straightening argument in topological field theory (see, for instance,~\cite[Lecture 7]{Quinn-top-TQFT}).
\end{proof}

As noted above, the cohomology groups of $\KhSpace(L)$ are the
Khovanov homology of $L$. Since we are viewing Khovanov homology as
covariant in the cobordism, it is more convenient to work with the
homology groups of $\KhSpace(L)$. These can be understood as follows:
\begin{lemma}\label{lem:kh-homology}
  Let $L$ be a link diagram. Then the cellular chain complex for
  $\KhSpace(L)$ agrees with the Khovanov complex for $\mirror{L}$. In
  particular, $\wt{H}_i(\KhSpace^j(L))=\Kh^{-i,-j}(\mirror{L})$.
\end{lemma}
\begin{proof}
  This is immediate from the definitions. To wit, $\KhSpace(L)$ can be
  constructed as a CW complex whose reduced cellular cochain complex
  $\cocellC{*}(\KhSpace(L))$ is isomorphic to the Khovanov complex
  $\KhCx(L)$ from \Section{khovanov-basic}. Therefore, the reduced
  cellular chain complex $\cellC{*}(\KhSpace(L))$ is isomorphic to the
  dual complex $\Hom(\KhCx(L),\Z)$. However, the dual complex is
  isomorphic to $\KhCx(m(L))$~\cite[Proposition
  32]{Kho-kh-categorification}. In the language
  of \Section{khovanov-basic}, this isomorphism takes a Khovanov
  generator in $F_{L}(v)$ to a Khovanov generator in
  $F_{m(L)}(\vec{1}-v)$ by changing the labels on the circles of
  $\mc{P}_{L}(v)=\mc{P}_{m(L)}(\vec{1}-v)$ from $x_+$ to $x_-$ and
  vice versa. The gradings work out, and we get
  $\wt{H}_i(\KhSpace^j(L))=\Kh^{-i,-j}(\mirror{L})$.
\end{proof}

Functoriality for the Khovanov spectrum has not yet been verified, but in~\cite{RS-s-invariant} we did associate maps to elementary cobordisms:
\begin{proposition}\label{prop:cob-map-spectra}
  Let $L_1$ and $L_2$ be links in $\RR^3$ and $F$ a cobordism from $L_1$ to $L_2$. Then there is a map of spectra 
  \[
  \Psi_{\mirror{F}}\co \KhSpace(\mirror{L_1})\to \KhSpace(\mirror{L_2})
  \]
  so that the induced map $\Psi_{\mirror{F},*}\co \wt{H}_*(\KhSpace(\mirror{L_1}))=\Kh(L_1)\to \Kh(L_2)=\wt{H}_*(\KhSpace(\mirror{L_2}))$ agrees with the cobordism map $\Phi_F$ up to sign.
\end{proposition}
\begin{proof}
  The corresponding statement in cohomology is immediate from~\cite[Proposition 3.4, Lemma 3.6 and Lemma 3.7]{RS-s-invariant}; the homology statement comes from dualizing~\cite[Proposition 3.4, Lemma 3.6 and Lemma 3.7]{RS-s-invariant} (cf.~\Lemma{kh-homology}).
\end{proof}

With these ingredients, we are now ready to verify the mirror theorem.

\begin{proof}[Proof of \Theorem{mirror}]
  The following argument was suggested to us by the referee for~\cite{RS-steenrod}.

  By \Proposition{detect-SW}, it suffices to construct a map $\KhSpace(\mirror{L})\smas\KhSpace(L)\to \SphereS$ so that for any field $k$, the induced map
  \[
  \wt{H}_i(\KhSpace(\mirror{L});k)\otimes   \wt{H}_{-i}(\KhSpace(L);k)\to \wt{H}_0(\SphereS)=k
  \]
  is a perfect pairing. By \Proposition{Kh-perf-pair}, applying \Proposition{cob-map-spectra} to the canonical cobordism from $L\amalg \mirror{L}$ to the empty link gives such a map.
\end{proof}


\begin{remark}
  \Theorem{mirror} (perhaps) gives an obstruction to knots being
  amphichiral: if $K$ is amphichiral then $\KhSpace^j(K)$ is
  Spanier-Whitehead dual to $\KhSpace^{-j}(K)$. Using
  KnotKit~\cite{KKI-kh-knotkit} (and the technique
  in~\cite{RS-steenrod}), it is possible to verify that this
  obstruction does not give any additional restrictions on
  amphichirality for prime knots up to $15$ crossings, beyond those implied
  by Khovanov homology itself. Indeed, the only Khovanov
  homology-symmetric knots with $15$ or fewer crossings for which
  $\Sq^2$ is non-vanishing are $14^n_{8440}$, $14^n_{9732}$,
  $14^n_{21794}$, $14^n_{22073}$, and $15^n_{139717}$. In each of
  these cases, the non-Moore space summands of $\KhSpace$ are copies
  of (various suspensions and desuspensions of) $\RR P^5/\RR P^2$
  and $\RR P^4/\RR P^1$, and these summands (with appropriate
  grading shifts) are exchanged by $(i,j)\to (-i,-j)$. 

  It remains open whether the stable homotopy type gives nontrivial
  restrictions for larger prime knots, although we expect that it
  does. As a proof of concept, look at \cite[Table~1]{Seed-Kh-square}
  and consider the link $L=14^n_{11196}\,\amalg\,
  m(14^n_{17546})$. Its Khovanov homology is symmetric since Khovanov
  homologies of $14^n_{11196}$ and $14^n_{17546}$ are isomorphic (as
  bigraded groups); the homology of either knot is supported on three
  adjacent diagonals, the only torsion present is $2$-torsion, and the
  lowest and highest quantum gradings where $2$-torsion appears are
  $-11$ and $5$. However, $14^n_{17546}$ has no $\Sq^2$, but
  $14^n_{11196}$ has a single $\Sq^2$ in quantum grading $-1$. Using
  Cartan formula and \Theorem{disjoint-union}, the lowest and highest
  quantum gradings where $\Kh(L)$ has non-trivial $\Sq^3$ are $-6$ and
  $10$, and therefore, its Khovanov homotopy type is not
  symmetric. (Of course this is not our recommended proof for showing
  $L$ is achiral!)
\end{remark}

\section{Applications}\label{sec:applications}
In this section we give an application of the K\"unneth theorem for
the Khovanov stable homotopy type to knot concordance. We begin with some
background, continuing from \Section{khovanov-basic}. Recall that
Rasmussen~\cite{Ras-kh-slice} used the Lee deformation of the Khovanov
complex~\cite{Lee-kh-endomorphism} (the specialization $(h,t)=(0,1)$
of \Definition{KhCx-diffKh}) to define a concordance invariant
$s_K^\QQ\in 2\ZZ$. As the notation suggests, to define $s^\QQ$,
Rasmussen used Khovanov homology with coefficients in $\QQ$, though
any field of characteristic different from $2$ would work as
well. Bar-Natan~\cite{Bar-kh-tangle-cob} gave an analogue
$\BNcx$ of the Lee
deformation that also works over $\F_2$ (see
also~\cite{Nao-kh-universal,Turner-kh-BNseq}). The Bar-Natan complex
$\BNcx$ is the specialization $(h,t)=(1,0)$ of the universal complex
$\Cx$ of \Definition{KhCx-diffKh}.  Let $s_K=s_K^{\F_2}$ denote the
corresponding Rasmussen-type invariant; see also~\cite[Section
2.2]{RS-s-invariant}.

Since the formal variable $h$ has quantum grading $-2$, the
differential in the Bar-Natan chain complex either preserves the
quantum grading or increases it. Let $\Filt_q\BNcx$ denote the
subcomplex of $\BNcx$ supported in quantum grading $q$ or higher. This
defines a filtration on the Bar-Natan complex, and the associated
graded object $\bigoplus_q\Filt_q\BNcx/\Filt_{q+2}\BNcx$ is the
Khovanov chain complex $\KhCx$.

The second Steenrod square $\Sq^2$ for the Khovanov spectrum
$\KhSpace(K)$ gives a map
\[
\Sq^2\from\Kh^{i,j}(K;\F_2)\to\Kh^{i+2,j}(K;\F_2)
\]
for each $i,j$.  Using this, in \cite{RS-s-invariant} we constructed
a refined $s$-invariant $s^{\Sq^2}_+(K)$, as follows:

\begin{definition}\cite[Definition 1.2 and Lemma 4.2]{RS-s-invariant}
  Consider configurations of the following form:
  \begin{equation}\label{eq:def-refined-s}
    \xymatrix@C=3ex{
      \langle\wt{a},\wt{b}\rangle\ar[r]\ar@{^(->}[d]&\langle\wh{a},\wh{b}\rangle\ar@{<-}[r]\ar@{^(->}[d]&
      \langle a,b\rangle\ar[r]\ar@{^(->}[d] & \langle\ol{a},\ol{b}\rangle\ar@{=}[d]\\
      \Kh^{-2,s_K-1}(K;\F_2)\ar[r]^-{\Sq^2}&\Kh^{0,s_K-1}(K;\F_2)\ar@{<-}[r]&
      H_0(\Filt_{s_K-1}\BNcx(K);\F_2)\ar[r] &H_0(\BNcx(K);\F_2).  }
  \end{equation}
  Define
  \[
    s^{\Sq^2}_+(K)=
    \begin{cases}
      s_K&\text{if there does not exist such a configuration,}\\
      s_K+2&\text{otherwise.}
    \end{cases}
  \]
\end{definition}
\begin{theorem}\cite[Theorem 1]{RS-s-invariant}
  The integer $s^{\Sq^2}_+(K)$ is a concordance invariant, and
  \[
    2g_4(K)\geq |s^{\Sq^2}_+(K)|
  \]
  where $g_4(K)$ is the $4$-ball genus of $K$.
\end{theorem}

\begin{lemma}\label{lem:connect-sum-s}
Let $K$ and $L$ be two knots such that the following holds:
\begin{enumerate}
\item $s^{\Sq^2}_+(K)=s_K+2$.
\item $\Kh^{0,s_L+1}(L;\F_2)\stackrel{\Sq^2}{\longrightarrow}\Kh^{2,s_L+1}(L;\F_2)$ is
the zero map.
\item Either
  $\Kh^{0,s_L+1}(L;\F_2)\stackrel{\Sq^1}{\longrightarrow}\Kh^{1,s_L+1}(L;\F_2)$
  is the zero map or
  $\Kh^{-2,s_K-1}(K;\F_2)\allowbreak\stackrel{\Sq^1}{\longrightarrow}\allowbreak\Kh^{-1,s_K-1}(K;\F_2)$
  is the zero map.
\end{enumerate}
Then $s^{\Sq^2}_+(K\# L)=s_{K\# L}+2=s_K+s_L+2$.
\end{lemma}

\begin{proof}
  \begin{figure}
    \centering
    \includegraphics{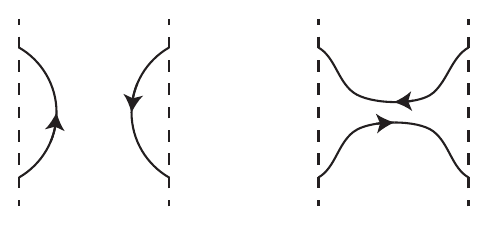}
    \caption{\textbf{Oriented connected sum.} Pieces of two knots and their connected sum, oriented coherently with the obvious saddle cobodism, are shown. }
    \label{fig:orient-connected-sum}
  \end{figure}
Consider the saddle cobordism from $K\amalg L$ to $K\# L$. Choose some
orientations $o_K$, $o_L$, and $o_{K\# L}$, of the knots $K$, $L$, and
$K\# L$, which are coherent with respect to the saddle cobordism (see
\Figure{orient-connected-sum}). Let $\{\pm o_K\}$, $\{\pm o_L\}$, and $\{\pm o_{K\#
L}\}$ denote the sets of orientations of $K$, $L$, and $K\#L$. For any
orientation $o$ there is a corresponding generator $g(o)$ of
$H_0(\BNcx;\F_2)$~\cite[Theorem 4.2]{Lee-kh-endomorphism},~\cite[Theorem 1]{Tur-kh-BN}.  
We will use the following facts.
\begin{enumerate}[labelindent=0pt,itemindent=40pt,labelwidth=0pt,leftmargin=0pt,
noitemsep, nolistsep]
\item\label{item:proof-identification} $\BNcx(K\amalg L)$ is canonically identified with
$\BNcx(K)\otimes\BNcx(L)$. The induced identification on the
associated graded objects, $\Kh(K\amalg L)\cong \Kh(K)\otimes\Kh(L)$, is
the one induced from the equivalence from \Theorem{disjoint-union}. (This is immediate from the definitions.)

\item\label{item:proof-basis} The sets $\{g(o_K),g(-o_K)\}$, $\{g(o_L),g(-o_L)\}$,
$\{g(o_{K\# L}),g(-o_{K\# L})\}$, and $\{g(o_K)\otimes
g(o_L),\allowbreak g(-o_K)\otimes g(o_L),\allowbreak g(o_K)\otimes g(-o_L),g(-o_K)\otimes
g(-o_L)\}$ form bases of $H_0(\BNcx(K))$, $H_0(\BNcx(L))$,
$H_0(\BNcx(K\# L))$, and $H_0(\BNcx(K\amalg L))$,
respectively~\cite[Section 4.4.3]{Lee-kh-endomorphism},~\cite[Theorem 1]{Tur-kh-BN}.

\item\label{item:proof-q-grading} $g(o_L)+g(-o_L)$ has a cycle
  representative in $\Filt_{s_L+1}\BNcx(L)$~\cite[Proposition 2.6]{RS-s-invariant}.

\item\label{item:proof-map-properties} The saddle cobordism map
  $\BNcx(K\amalg L)\to\BNcx(K\# L)$ preserves the homological grading
  and either increases the quantum grading or decreases it by exactly
  one.  The induced map on $\Kh$ commutes with the map
  $\Sq^2$~\cite[Theorem~4]{RS-s-invariant}. The induced map on
  $H_0(\BNcx;\F_2)$ is the following:
\begin{align*}
  g(o_K)\otimes g(o_L)&\mapsto g(o_{K\#L})& g(-o_K)\otimes g(o_L)&\mapsto
  0\\ 
  g(o_K)\otimes g(-o_L)&\mapsto 0& g(-o_K)\otimes g(-o_L)&\mapsto
  g(-o_{K\#L}).
\end{align*}
(This follows from the same argument as~\cite[Proposition
4.1]{Ras-kh-slice}, since Turner's change of basis diagonalizes the
Bar-Natan Frobenius algebra.
\end{enumerate}

Since $s^{\Sq^2}_+(K)=s_K+2$, there is a configuration as in Formula~(\ref{eq:def-refined-s}).
Since $\{g(o_K),g(-o_K)\}$ also
form a basis for $H_0(\BNcx(K);\F_2)$
(Fact~(\ref{item:proof-basis})), after performing a change of basis
if necessary, we may assume that $\ol{a}=g(o_K)$ and $\ol{b}=g(-o_K)$.

Using Fact~(\ref{item:proof-q-grading}), choose some configuration of
the following form
\[
\xymatrix@C=3ex{
  \langle\wh{c}\rangle\ar@{<-}[r]\ar@{^(->}[d]&
  \langle c\rangle\ar[r]\ar@{^(->}[d] & \langle g(o_L)+g(-o_L)\rangle\ar@{^(->}[d]\\
  \Kh^{0,s_L+1}(L;\F_2)\ar@{<-}[r]&
  H_0(\Filt_{s_L+1}\BNcx(L);\F_2)\ar[r] &H_0(\BNcx(L);\F_2).  }
\]

Combining the two configurations and using the identification from
Fact~(\ref{item:proof-identification}), we get the configuration
\[
\small
\xymatrix@C=3ex{
  \langle\wt{a}\otimes\wh{c},\wt{b}\otimes\wh{c}\rangle\ar[r]\ar@{^(->}[d]&\langle\wh{a}\otimes\wh{c},\wh{b}\otimes\wh{c} \rangle\ar@{<-}[r]\ar@{^(->}[d]& \langle
  a\otimes c,b\otimes c\rangle\ar[r]\ar@{^(->}[d] & \txt{$\langle
  g(o_K)\otimes (g(o_L)+g(-o_L)),$\\$g(-o_K)\otimes
  (g(o_L)+g(-o_L))\rangle$}\ar@{^(->}[d]\\ 
\Kh^{-2,s_K+s_L}(K\amalg
  L;\F_2)\ar[r]^-{\Sq^2}&\Kh^{0,s_K+s_L}(K\amalg L;\F_2)\ar@{<-}[r]&
  H_0(\Filt_{s_K+s_L}\BNcx(K\amalg L);\F_2)\ar[r]
  &H_0(\BNcx(K\amalg L);\F_2).
  }
\]
We should justify the leftmost horizontal arrow; that is, assuming
$\Sq^2(\wt{a})=\wh{a}$ and $\Sq^2(\wt{b})=\wh{b}$, we need to show
that $\Sq^2(\wt{a}\otimes\wh{c})=\wh{a}\otimes\wh{c}$ and
$\Sq^2(\wt{b}\otimes\wh{c})=\wh{b}\otimes\wh{c}$. Since the
identification $\Kh(K\amalg L)\cong\Kh(K)\otimes\Kh(L)$ is induced
from the identification $\KhSpace(K\amalg
L)\cong\KhSpace(K)\wedge\KhSpace(L)$ of \Theorem{disjoint-union},
cf.~Fact~(\ref{item:proof-identification}),
\begin{align*}
\Sq^2(\wt{a}\otimes\wh{c})&=\Sq^2(\wt{a})\otimes\wh{c}+\Sq^1(\wt{a})\otimes\Sq^1(\wh{c})+\wt{a}\otimes\Sq^2(\wh{c})\\
&=\wh{a}\otimes\wh{c}.
\end{align*}
The first equality is the Cartan formula; the second
uses the lemma's hypotheses.
Similarly, $\Sq^2(\wt{b}\otimes\wh{c})=\wh{b}\otimes\wh{c}$.

Now consider the image of this configuration under the saddle cobordism
map. Using Fact~(\ref{item:proof-map-properties}), we get a
configuration
\[
\small
\xymatrix@C=3ex{
  \langle\wt{p},\wt{q}\rangle\ar[r]\ar@{^(->}[d]&\langle\wh{p},\wh{q}\rangle\ar@{<-}[r]\ar@{^(->}[d]& \langle
  p,q\rangle\ar[r]\ar@{^(->}[d] & \langle g(o_{K\#L}),g(-o_{K\#L})\rangle\ar@{^(->}[d]\\ 
\Kh^{-2,s_{K\# L}-1}(K\#
  L;\F_2)\ar[r]^-{\Sq^2}&\Kh^{0,s_{K\# L}-1}(K\# L;\F_2)\ar@{<-}[r]&
  H_0(\Filt_{s_{K\# L}-1}\BNcx(K\# L);\F_2)\ar[r]
  &H_0(\BNcx(K\# L);\F_2).
  }
\]
By Fact~(\ref{item:proof-basis}), $\langle
g(o_{K\#L}),g(-o_{K\#L})\rangle = H_0(\BNcx(K\# L);\F_2)$; therefore,
$s^{\Sq^2}_+(K\# L)=s_{K\# L}+2$.
\end{proof}

We are almost ready to prove
\Corollary{top-appl}. First, we tabulate some invariants
of the knots $K$ that appear in its statement.

\begin{table}[ht]
  \centering
  \begin{tabular}{lccccccc}
    $K$&$\si(K)$&$s_K$&$\tau(K)$&$s^{Sq^2}_+(K)$&$g_4(K)$&$s^{\QQ}_K$&$u(K)$\\
    \toprule
    $9_{42}$&$-2$&$0$&$0$&$2$&$1$ &$0$& $1$\\
    $10_{136}$&$-2$&$0$&$0$&$2$&$1$&$0$& $1$\\
    $m(11^n_{19})$&$-4$&$2$&$1$&$4$&$2$ &$2$& $2$\\
    $m(11^n_{20})$&$-2$&$0$&$0$&$2$&$1$&$0$& $1$\\
    $11^n_{70}$&$-4$&$2$&$1$&$4$&$2$&$2$& $2$\\
    $11^n_{96}$&$-2$&$0$&$0$&$2$&$1$ &$0$& $1$
  \end{tabular}
  \caption{}\label{table:invariants}
\end{table}
There is some confusion in the nomenclature of knots regarding
mirrors, stemming primarily from the fact that Gauss codes do not
distinguish between a knot and its mirror. See \Figure{knot-diagrams}
for our convention. We have followed the convention from the Knot
Atlas \cite{KAT-kh-knotatlas}, using their planar diagram
presentations (which do detect chirality). The knots in
\Figure{knot-diagrams} are produced in Knotilus \cite{KIT-kh-knotilus}
using Gauss codes from the Knot Atlas; however, for
$11^n_{70}$, our knot diagram differs from the one produced by Knotilus
(as well as the one in the Knot Atlas)
by mirroring. Alternatively, one can deduce our convention from the
value of the signature $\si(K)$ in \Table{invariants} (which, for us,
is negative for positive knots).

The values of $s_K=s^{\F_2}_K$ and $s^{\Sq^2}_+(K)$ are imported
from~\cite{RS-steenrod}; $s_K$ can also be computed independently by
Knotkit \cite{KKI-kh-knotkit}. The values of the four-ball genus
$g_4(K)$ and the unknotting number $u(K)$ are extracted from Knotinfo
\cite{KIN-kh-knotinfo}. The absolute values of the signature $\si(K)$
and $s^{\QQ}_K$ are extracted from \cite{KAT-kh-knotatlas} and the
absolute value of Ozsv\'ath-Szab\'o's invariant $\tau(K)$ come
from~\cite{BG-hf-computations}. We deduce the signs of $\si(K)$,
$s^{\QQ}_K$, and $\tau(K)$ by noticing that all the unknotting
crossings in \Figure{knot-diagrams} are positive, and changing a positive
crossing to a negative one decreases $-\si(K)/2$, $s^{\QQ}/2$, and
$\tau(K)$ by $0$ or $1$. We list $\tau(K)$, $s^{\QQ}_K$, and $u(K)$
purely for the reader's interest.

Finally, we deduce \Corollary{top-appl} which, to recall, was the following:

\begin{repcorollary}{cor:top-appl}
  Let $K$ be one of the knots $9_{42}$, $10_{136}$, $m(11^n_{19})$,
  $m(11^n_{20})$, $11^n_{70}$, or $11^n_{96}$.  (Here $m$ denotes the
  mirror.)  Let $L$ be a knot which is the closure of a positive
  braid. Letting $g_4$ denote the four-ball genus, we have
  \[
  g_4(K\#L)=g_4(K)+g_4(L).
  \]
\end{repcorollary}
\begin{proof}
Certainly $g_4(K\# L)\leq g_4(K)+g_4(L)$, so we only need
to show $g_4(K\# L)\geq g_4(K)+g_4(L)$.

First consider the knot $K\in \{ 9_{42}, 10_{136}, m(11^n_{19}), m(11^n_{20}), 11^n_{70}, 11^n_{96}\}$.  In all cases
$s^{\Sq^2}_+(K)=s_K+2=2g_4(K)$ (see \Table{invariants}).

Next consider the knot $L$. Draw $L$ as the closure of a positive braid
with $n$ crossings and $m$ strands. We will classify the
generators of the Khovanov chain complex $\KhCx(L)$ in quantum
grading $n+2-m$ or less.

Let $v$ be some vertex in the cube of resolutions for $L$. Assume
there are $c_v$ circles in the resolution at $v$. The smallest quantum
grading over $v$ is achieved by the Khovanov generator which labels
all the $c_v$ circles by $x_-$. The value of this smallest quantum
grading is
\(
n+\card{v}-c_v.
\)

Since all the crossings of $L$ are positive, the resolution at the
zero vertex $\vect{0}$ is the oriented resolution; therefore
$c_{\vect{0}}=m$. If $u$ is some vertex of weight one, i.e.,
$\card{u}=1$, then the resolution at $u$ is obtained from the oriented
resolution by a merge; therefore $c_{u}=m-1$. The resolution at any
other vertex $v$ is obtained from some weight one vertex by
$\card{v}-1$ merges or splits; therefore $c_v\leq
m-1+\card{v}-1=m+\card{v}-2$, with equality holding if and only if the
resolution at $v$ can be obtained from some weight one resolution by
splits only.

It follows that the minimum quantum grading over any
vertex $v\neq\vect{0}$ is at least $n-m+2$; and over $\vect{0}$, the
minimum quantum grading is $n-m$, which is attained only by the
Khovanov generator that labels all the $m$ circles by $x_-$. Observe that
this is enough to compute the values of $s_L$, $g_4(L)$ and $g(L)$:
$s_L\geq n-m+1$, and Seifert's algorithm yields a surface of genus
$(n-m+1)/2$; therefore, the inequality
\begin{align*}
n-m+1\leq s_L\leq 2g_4(L)&\leq 2g(L)\leq n-m+1\\
\shortintertext{leads to the equality}
s_L= 2g_4(L)&= 2g(L)= n-m+1.
\end{align*}

Next, we compute the Khovanov homology in these quantum gradings. 
We have $\Kh^{*,q}(L;\Z)=0$ for all $q<n-m$ and
$\Kh^{*,n-m}(L;\Z)$ is $\Z$ supported in homological grading
zero. However, our interest lies in quantum grading $n-m+2$; we show
that $\Kh^{*,n-m+2}(L;\Z)$ is $\Z$ supported in homological grading
zero as well.

Number the $m$ strands in the braid diagram from left to right. For
$1\leq i<m$, let $n_i$ be the number of crossings in the diagram
between the $i\th$ and the $(i+1)\st$ strands, i.e., the number of times $\sigma_i$ occurs in the braid word.
(Then $n=\sum_i n_i$,
and since $L$ is a knot, $n_i\geq 1$ for all $i$.) Number the
$n_i$ crossings from top to bottom.  For $1\leq i\leq m$, let
$\ob{x}_i=(\vect{0},x_i)$ be the Khovanov generator where $x_i$ labels
the $i\th$ circle in the oriented resolution by $x_+$, and the rest by
$x_-$.  For $1\leq i<m$ and $\varnothing\neq J\subseteq \{1,\dots,n_i\}$,
let $u_{i,J}$ be the weight $\card{J}$ vertex in the cube of
resolutions where the $1$-resolution is taken only at the crossings
that appear in $J$
between the $i\th$ and $(i+1)\st$ strands; and let
$\ob{y}_{i,J}$ be the Khovanov generator living over $u_{i,J}$ where
all the circles are labeled by $x_-$. From the discussion above it is
clear that the Khovanov chain group $\KhCx^{*,n+2-m}(L)$ is generated
by these generators $\ob{x}_i$ and $\ob{y}_{i,J}$. The differential is
fairly straightforward:
\begin{align*}
\diff \ob{x}_i&=\begin{cases}
\sum_{j=1}^{n_1}\ob{y}_{1,\{j\}}&i=1\\
\sum_{j=1}^{n_i}\ob{y}_{i,\{j\}}+\sum_{j=1}^{n_{i-1}}\ob{y}_{i-1,\{j\}}&1< i<m\\
\sum_{j=1}^{n_{m-1}}\ob{y}_{m-1,\{j\}}&i=m
\end{cases}\\
\diff \ob{y}_{i,J}&=\sum_{\substack{J'\supset J\\\card{J'\setminus J}=1}} \pm\ob{y}_{i,J'}.
\end{align*}
Here we are using the standard sign assignment on the cube, and the
signs are determined by the ordering of the $n$ crossings.

Using the change of basis replacing $\ob{x}_k$ by
$\sum_{i=1}^k(-1)^i\ob{x}_i$ for all $1\leq k\leq m$, it is easy to
see that the chain complex $\KhCx^{*,n+2-m}(L)$ is isomorphic to the
following direct sum of cube complexes:
\[
\KhCx^{*,n+2-m}(L;\Z)\cong \Z\oplus(\mathop\oplus_{i=1}^{m-1}(\mathop\otimes_{j=1}^{n_i}(\Z\to\Z))).
\]
Therefore, $\Kh^{*,n+2-m}(L;\Z)\cong \Z$, supported in homological
grading zero (and is generated by the cycle
$\sum_{i=1}^m(-1)^i\ob{x}_i$).

As an immediate consequence, we have that both the maps
\begin{align*}
\Sq^1&\from\Kh^{0,s_L+1}(L;\F_2)\to\Kh^{1,s_L+1}(L;\F_2)=0\\\shortintertext{and}
\Sq^2&\from\Kh^{0,s_L+1}(L;\F_2)\to\Kh^{2,s_L+1}(L;\F_2)=0
\end{align*}
vanish. Therefore, $K$ and $L$ satisfy the conditions
of \Lemma{connect-sum-s}, so
\[
2g_4(K\# L)\geq s^{\Sq^2}_+(K\# L)=s_K+s_L+2=2g_4(K)+2g_4(L).\qedhere
\]
\end{proof}

\appendix

\section{Flow chart of sections}\label{app:flow-chart}

\noindent
\begin{tikzpicture}[yscale=-1,xscale=4]
  \node at (3,0) (cube) {\S\ref{sec:cube}: \textcolor{blue}{Cube cat.}};
  \node at (0,0) (khovanov-basic) {\S\ref{sec:khovanov-basic}: \textcolor{blue}{Khovanov constr.}};
  \node at (0,1) (manifold-corners) {\S\ref{sec:manifold-corners}: \textcolor{blue}{Corners}};
  \node at (1,1) (flow-cat) {\S\ref{sec:flow-cat}: \textcolor{blue}{Flow cats.}};
  \node at (2,1) (cube-flow-cat) {\S\ref{sec:cube-flow-cat}: \textcolor{blue}{Cube flow cat.}};
  \node at (0,2) (permutohedra) {\S\ref{sec:permutohedra}: \textcolor{blue}{Permutohedra}};
  \node at (0,3) (cubical-neat) {\S\ref{sec:cubical-neat}: Cubical neat emb.};
  \node at (0,4) (cubical-realize) {\S\ref{sec:cubical-realize}: Cubical realization};
  \node at (0,5) (cubical-CJS-same) {\S\ref{sec:cubical-CJS-same}: Cubical is CJS};  
  \node at (1,2) (cubical-flow-cat) {\S\ref{sec:cubical-flow-cat}: Cubical flow cat.};
  \node at (2,2) (Burnside) {\S\ref{sec:Burnside}: Burnside 2-cat.};
  \node at (3,3) (colimit) {\textcolor{blue}{\S\ref{sec:colimit}: Hocolim}};
  \node at (3,1) (cubical-to-burnside) {\S\ref{sec:cubical-to-burnside}: Cube to Burnside};
  \node at (3,2) (thicken) {\S\ref{sec:thicken}: Thickening};
  \node at (3,4) (subsec-realize) {\S\ref{sec:subsec-realize}: Thick realization};
  \node at (3,6) (thick-invariance) {\S\ref{sec:thick-invariance}: Thick invariance};
  \node at (3,5) (products) {\S\ref{sec:products}: Products};
  \node at (1,3) (box-maps) {\S\ref{sec:box-maps}: Box maps};
  \node at (1,4) (box-refinement) {\S\ref{sec:box-refinement}: Box refinements};
  \node at (2,4) (coherent-box-maps) {\S\ref{sec:coherent-box-maps}: Small realization};
  \node at (2,5) (small-is-big) {\S\ref{sec:small-is-big}: Small is big};
  \node at (1,5) (cubical-to-cubes) {\S\ref{sec:cubical-to-cubes}: Small is cubical};
  \node at (1,6) (Kh-htpy) {\S\ref{sec:Kh-htpy}: The Khovanov stable homotopy type and its many realizations};
  \node[align=center] at (0,7) (HKK) {\S\ref{sec:HKK}: HKK Khovanov htpy\\type agrees with ours};
  \node[align=center] at (2,7) (Kh-disj-union) {\S\ref{sec:Kh-disj-union}: Disjoint unions\\and connected sums};
  \node at (1,7) (mirror) {\S\ref{sec:mirror}: Mirrors};
  \node at (3,7) (applications) {\S\ref{sec:applications}: Applications};

  \draw[->] (cube) to (khovanov-basic);
  \draw[->] (manifold-corners) to (flow-cat);
  \draw[->] (manifold-corners) to (permutohedra);
  \draw[->] (permutohedra) to (cube-flow-cat);
  \draw[->] (flow-cat) to (cube-flow-cat);
  \draw[->] (cube-flow-cat) to (cubical-flow-cat);
  \draw[->] (cube) to (cubical-to-burnside);
  \draw[->] ($(cubical-to-burnside.south)+(-1ex,0)$) to (coherent-box-maps);
  \draw[->] (cubical-flow-cat) to (cubical-neat);
  \draw[->] (cubical-neat) to (cubical-realize);
  \draw[->] (cubical-realize) to (cubical-CJS-same);
  \draw[->] (Burnside) to (cubical-to-burnside);
  \draw[->] (cubical-flow-cat) to (cubical-to-burnside);
  \draw[->] (cubical-to-burnside) to (thicken);
  \draw[->] ($(thicken.south)+(2ex,0)$) to ($(subsec-realize.north)+(2ex,0)$);
  \draw[->] ($(subsec-realize.south)+(2ex,0)$) to ($(thick-invariance.north)+(2ex,0)$);
  \draw[->] (subsec-realize) to (products);
  \draw[->] (products) to (Kh-disj-union);
  \draw[->] (colimit) to (subsec-realize);
  \draw[->] (box-maps) to (box-refinement);
  \draw[->] (Burnside) to (box-refinement);
  \draw[->] (box-refinement) to (coherent-box-maps);
  \draw[->] (colimit) to (coherent-box-maps);
  \draw[->] (coherent-box-maps) to (small-is-big);  
  \draw[->] (subsec-realize) to (small-is-big);  
  \draw[->] (coherent-box-maps) to (cubical-to-cubes);
  \draw[->] (Kh-htpy) to (HKK);
  \draw[->] (Kh-htpy) to (Kh-disj-union);
  \draw[->] (Kh-disj-union) to (mirror);
  \draw[->] (Kh-disj-union) to (applications);
  \draw[->] (cubical-realize) to (cubical-to-cubes);
  \draw[->] (cubical-CJS-same) to (Kh-htpy);
  \draw[->] (small-is-big) to  (Kh-htpy);
  \draw[->] (cubical-to-cubes) to (Kh-htpy);
  \draw[->] (khovanov-basic.west) --++(-1ex,0) -- ($(khovanov-basic.west|-Kh-htpy)+(-1ex,0)$)-- (Kh-htpy.west);

\end{tikzpicture}

Background material is indicated by a {\color{blue}different
  shade}. Superficial dependencies (e.g., for examples or small pieces
of notation) have been omitted.

\bibliographystyle{amsalpha}
\bibliography{newbibfile}

\end{document}